%% file: paper.tex
\documentclass[a4paper,parskip,11pt,listof=totoc,bibliography=totoc, index=totoc,twoside]{scrartcl}

\newcommand{\CommonPath}{.}
\input{\CommonPath/preamble}
\usepackage[backend=biber, style=numeric, giveninits=true, sorting=nyt, url=false, doi=false, eprint=true, isbn=false]{biblatex}
\KOMAoptions{abstract=true}
\usetikzlibrary{external}
\numberwithin{equation}{section}
\title{Long Time Behaviour of the Two-Component Becker--Döring System}
\author{Jens Scholten}
\date{\today}

\begin{document}
  \maketitle
  \subfile{\CommonPath/abstract}  \tableofcontents  
  \subfile{\CommonPath/Introduction}
\subfile{\CommonPath/DissipationEstimates}
\subfile{\CommonPath/LogSobolevInequality}
\subfile{\CommonPath/mInfty}
\subfile{\CommonPath/Supercriticality}
\subfile{\CommonPath/MainTheorem}
\subfile{\CommonPath/appendix}  \subfile{\CommonPath/acknowledgements}
  \printbibliography
\end{document}

%% file: preamble.tex
\usepackage[english]{babel}
\usepackage{csquotes}

\usepackage{aligned-overset}
\usepackage[right=1.8cm,left=1.9cm]{geometry}
\usepackage{amsmath,amssymb,amsfonts,amsthm,mathtools}
\mathtoolsset{ showonlyrefs, showmanualtags }
\usepackage{dsfont}
\usepackage{multicol}
\usepackage{enumitem}
\setlist[enumerate]{label=\roman*),ref=\roman*)}
\usepackage{derivative}
\usepackage{array} 
\newcolumntype{L}{>{$}l<{$}}
\usepackage{subfiles}

\usepackage[svgnames]{xcolor}
\usepackage{subcaption}
\usepackage{graphicx}
\usepackage{tikz}
\usetikzlibrary{arrows.meta}
\usetikzlibrary{shadows}
\usepackage{intcalc}
\usepackage{float}

\usepackage{hyperref}
\hypersetup{
    colorlinks,
    citecolor=black,
    filecolor=black,
    linkcolor=black,
    urlcolor=black,
    linktoc=all,  
}

\theoremstyle{plain}
\newtheorem{theorem}{Theorem}[section]

\newtheorem{lemma}[theorem]{Lemma}

\newtheorem{proposition}[theorem]{Proposition}

\newtheorem{corollary}[theorem]{Corollary}

\newtheorem{definition}[theorem]{Definition}

\theoremstyle{definition}
\newtheorem{notation}[theorem]{Notation}

\newtheorem{example}[theorem]{Example}

\newtheorem{remark}[theorem]{Remark}

\newcommand{\eqLabeledLine}[2]{\noindent\refstepcounter{equation}#1\hfill(\theequation)  \label{#2}}
\newcommand{\Jr}{J}
\newcommand{\Js}{\dot{J}}

\newcommand{\ar}{a}
\newcommand{\as}{\dot{a}}
\newcommand{\br}{b}
\newcommand{\bs}{\dot{b}}
\newcommand{\er}{e}
\newcommand{\es}{\dot{e}}
\newcommand{\N}{\mathbb{N}}
\newcommand{\R}{\mathbb{R}}

\newcommand{\T}{\mathcal{T}}
\newcommand{\eps}{\varepsilon}

\newcommand{\indicator}[1]{\mathds{1}_{\left\{#1\right\}}}
\newcommand{\smallo}{\mathchoice
    {{\scriptstyle\mathcal{O}}}%
    {{\scriptstyle\mathcal{O}}}%
    {{\scriptscriptstyle\mathcal{O}}}%
    {\scalebox{.6}{$\scriptscriptstyle\mathcal{O}$}}%
  }
\newcommand{\bigO}{\mathcal{O}}
\newcommand{\interior}[1]{\mathring{#1}}
\DeclarePairedDelimiter{\norm}{\lVert}{\rVert}
\DeclareMathOperator*{\argmax}{arg\,max}

\newcommand{\weakstar}{\overset{\ast}{\rightharpoonup}}
\newcommand{\loc}{\text{loc}}
\newcommand{\crit}{\text{crit}}
\newcommand{\const}{\text{const}}
\newcommand{\dd}{\mathop{}\!\mathrm{d}}
\newcommand{\del}{\Delta\,}
\newcommand{\delr}{\Delta_r\,}
\newcommand{\dels}{\Delta_s\,}
\newcommand{\delk}{\Delta_k\,}
\newcommand{\dell}{\Delta_l\,}
\NewDocumentCommand{\summ}{O{1} O{\infty} }{\sum\limits_{m=#1}^{#2}}
\NewDocumentCommand{\sumrs}{O{1} O{\infty} O{r+s}}{\sum\limits_{#3=#1}^{#2}}
\NewDocumentCommand{\sumkl}{O{1} O{N} O{k+l}}{\sum\limits_{#3=#1}^{#2}}
\NewDocumentCommand{\sumRM}{O{\le} O{\le}}{\sum\limits_{r+s\ge 2}^{\substack{r#1R \\ s#2M}}}
\NewDocumentCommand{\sumgamma}{O{\Gamma}}{\sum\limits_{(r,s) \in #1}}
\NewDocumentCommand{\sumbar}{O{r} O{r+s} }{\sum\limits_{k=#1}^{#2} \cbar_{k,#2-k}}
\NewDocumentCommand{\sumhat}{O{r} O{r+s} }{\sum\limits_{k=#1}^{#2} \chat_{k,#2-k}}
\newcommand{\monr}{c_{1,0}}
\newcommand{\mons}{c_{0,1}}
\newcommand{\rs}{_{r,s}}
\newcommand{\kl}{_{k,l}}
\newcommand{\ii}{_{i,j}}
\newcommand{\rps}{_{r+1,s}}
\newcommand{\kpl}{_{k+1,l}}
\newcommand{\kml}{_{k-1,l}}
\newcommand{\klm}{_{k,l-1}}
\newcommand{\kmlp}{_{k-1,l+1}}
\newcommand{\kplm}{_{k+1,l-1}}
\newcommand{\rms}{_{r-1,s}}
\newcommand{\rsp}{_{r,s+1}}
\newcommand{\klp}{_{k,l+1}}
\newcommand{\rsm}{_{r,s-1}}

\newcommand{\rpsm}{_{r+1,s-1}}
\newcommand{\rmsp}{_{r-1,s+1}}
\newcommand{\rmmsp}{_{r-2,s+1}}
\newcommand{\rpsmm}{_{r+1,s-2}}
\newcommand{\rmspp}{_{r-1,s+2}}
\newcommand{\kn}{_{k,n-k}}

\newcommand{\kpn}{_{k+1,n-k-1}}
\newcommand{\knp}{_{k,n+1-k}}
\newcommand{\rn}{_{r,n-r}}
\newcommand{\rnp}{_{r,n+1-r}}
\newcommand{\rpn}{_{r+1,n-r-1}}

\newcommand{\rmn}{_{r-1,n-r+1}}
\newcommand{\sn}{_{n-s,s}}
\newcommand{\Vzw}{V}

\newcommand{\D}{D}
\newcommand{\Ex}{\mathcal{E}}

\newcommand{\Dnlin}{D^{N,1}}
\newcommand{\cinfzw}{\bar{c}^{z,w}}

\newcommand{\cbar}{\bar{c}}
\newcommand{\chat}{\hat{c}}

\newcommand{\ccirc}{\mathring{c}}
\newcommand{\cbargamma}{\bar{c}^\Gamma}

\newcommand{\Prs}{P^{r,s}}
\newcommand{\Pkl}{P^{k,l}}
\newcommand{\Trs}{T\rs}

\newcommand{\constructioni}{\hyperref[chatConstruction]{construction.i}}
\newcommand{\constructionii}{\hyperref[chatConstruction]{construction.ii}}
\newcommand{\constructioniii}{\hyperref[chatConstruction]{construction.iii}}

%% file: abstract.tex
\begin{abstract}
  \noindent The classical Becker--Döring equations describe the formation of clusters by aggregation and fragmentation of monomers. If the total amount of mass is supercritical, larger and larger clusters are formed, leading to an asymptotic loss of mass in the long time limit. The two-component Becker--Döring system arises when clusters are built from two types of monomers. Here, we study the natural extension of the one-component system, where no energy or entropy is leaving or entering the system---the so called detailed balance assumption. We rigorously prove that all initial conditions admit a solution minimising the relative entropy as time approaches infinity. This shows, that the long time limit in the weak* topology is selected through the initial Type I and Type II masses. Furthermore, the proportion of lost Type I mass is determined by the limit point via the mixing ratio of Type I and Type II monomers that maximises the binding energy. The main difficulty of the two-component system is that no relative entropy is weak* continuous, which is crucial for the classical argument \cite{penrose:foundations}. Instead, our proof is based on a discrete two-dimensional logarithmic Sobolev inequality, that bounds the entropy dissipation on the correct timescale. Our approach also improves current assumptions to determine the long time behaviour for the one-component system.
\end{abstract} 

%% file: Introduction.tex
\section{Introduction}
\subsection{One-Component System}
In 1935 Richard Becker and Werner Döring introduced possibly the simplest coagulation--fragmentation model capable of describing phase-transitions in supersaturated vapours \cite{beckerDoering}. They characterise particles by the number of monomers from which they are composed and assume that coagulation and fragmentation only occurs through monomer interaction. Then, the densities of particles containing exactly $i$ monomers at time $t$---denoted by $c_i(t)$---evolve according to 
\begin{equation}\label{ocBD}
    \odv{}{t} c_i(t) = - J_i(t) + J_{i-1}(t) \text{ if }i>1,
\end{equation}
where $J_i(t)$ is the flux between particles of size $i$ and those of size $i+1$. If we specify the rate at which a monomer attaches to a particle of size $i$ as $a_i$ and the rate of detachment from a particle of size $i+1$ as $b_{i+1}$, we can describe 
\[ J_i(t) = a_i c_1 (t) c_i(t) - b_{i+1} c_{i+1}(t).\]
In their original paper, Becker and Döring considered $c_1$ to be constant, making \eqref{ocBD} linear. The nonlinear version of \eqref{ocBD}, where the total mass is assumed to be constant, i.e. $\sum\limits_{n=1}^{\infty} n c_n(t) = \sum\limits_{n=1}^{\infty} n c_n(0)$, was introduced in \cite{penrose:first} and yields the additional equation 
\begin{equation}\label{ocBDMass}
  \odv{}{t} c_1(t) = -J_1 - \sum\limits_{n=1}^{\infty} J_n(t).
\end{equation}
System \eqref{ocBD}--\eqref{ocBDMass} offers a remarkably rich mathematical study \cite{canizoMomentBounds,colletLSW,laurencotLSW,penrose:metastability,schlichtingGradientFlow}. Most notably, if the initial conditions are supersaturated, the system undergoes a phase transition. In the seminal paper~\cite{penrose:foundations}, the long time behaviour entailing the phase transition was determined rigorously. They exploit that solutions to \eqref{ocBD}--\eqref{ocBDMass} dissipate an entropy $V(c)$ according to 
\begin{equation*}
  \odv{}{t} \underbrace{\sum\limits_{n=1}^{\infty} c_n\left(\ln\left(\frac{c_n}{Q_n}\right) -1\right)}_{\eqqcolon V(c)}
  =-\sum\limits_{n=1}^{\infty} (a_n c_1c_n - b_{n+1}c_{n+1})\ln\left(\frac{a_n c_1 c_n}{b_{n+1} c_{n+1}}\right) \le 0,
\end{equation*}
where the partition function $Q_i$ is defined as
\begin{equation}\label{ocDetailedBalance}
  Q_i \coloneqq \prod\limits_{k=1}^{i-1} \frac{a_k}{b_{k+1}}.
\end{equation}
The behaviour of the coefficients $Q_i$ depends on the physical system under consideration. For example, droplets in a vapour \cite{niethammer:LswLimit} or a binary alloy on a cubic lattice \cite{kalos:BinaryAlloyCoefficients} will yield
\[ Q_i \approx \frac{1}{z_s^i} e^{-i^{\frac 2 3}} \text{ for } i\to\infty,\text{ with } 0<z_s <\infty.\]
From the continuity properties of $V(c)$ (w.r.t. weak* convergence), they deduce in \cite{penrose:foundations} that $c(t)$ converges to an equilibrium. All equilibria are fully characterised by their monomer densities $z>0$ and are given by $\bar{c}^z_i \coloneqq z^i Q_i.$ Hence, there is a bijection 
\[z_\rho \colon [0,\rho_s] \to [0,z_s], \quad \rho \mapsto z_\rho \text{ with } \sum\limits_{n=1}^{\infty} n z_\rho^n Q_n = \rho,\]
where  $\rho_s \coloneqq \sum\limits_{n=1}^{\infty} n z_s^n Q_n<\infty$.
The exact equilibrium that a solution $c(t)$ converges to is determined through a maximum principle \cite{ball:refinedMaxPrinciple} and depends only on the initial mass of the system $\rho \coloneqq \sum\limits_{n=1}^{\infty} nc_n(0)$ via 
\begin{equation}\label{ocLongTime}
  c_i(t) \to 
  \begin{cases}
    z_\rho^i Q_i & \text{ if } \rho \le \rho_s\\
    z_s^i Q_i & \text{ else}.
  \end{cases}
\end{equation}
If $\rho>\rho_s$, the system is supersaturated and even though $\rho(t) = \rho>\rho_s$ for all times, in the limit $t=\infty$, some mass is lost. The excess mass is getting pushed into larger and larger clusters, building the new macroscopic phase.

The main goal of the present work is to establish an analogous result of \eqref{ocLongTime} for the two-component extension of \eqref{ocBD}--\eqref{ocBDMass}. To do so, we need to follow a novel approach, because $V(c)$ will not be weak* continuous any more \cite{paperBasics}. Furthermore, the dichotomy of either $c_1(t) < z_s$ or $c_1(t)\ge z_s$ is not present in the two-component system, which was crucially needed for the maximum principle.

For a detailed overview of the one-component Becker--Döring system we refer the reader to the well written review \cite{hingant:Overview}. Over the years, many variants and extensions of \eqref{ocBD} have been proposed and studied, e.g. \cite{laurencotDiffusionI,laurencotDiffusionII} incorporates spatial diffusion, \cite{pegoAtomization} studies a finite system with atomization of the largest cluster and \cite{niethammerInjectionDepletion} allows injection of monomers and depletion of large clusters.\\
With recent interest and advances in multidimensional coagulation equations \cite{cristianSedimentation,cristianCoagulationNonSphere,ferreiramcCoagulationSteadyStates,ferreiramcCoagulation,kharchandymcCoagFrag,}, a study of the two-component version of \eqref{ocBD} seems warranted. An early study around the turn of the century \cite{dunwell:phd,soheiliNumericalmcBD,soheiliContinuummcBD} was brief and without a satisfactory theoretical conclusion.
While \cite{doumicBiMonomeric} also considers a Becker--Döring type system with two different monomers, it describes a polymerisation--depolymerisation reaction systems, which is still characterised by a one-dimensional parameter $j\in\N$ and in particular very different form the system considered in this work.

\subsection{Two-Component System}\label{section:twoCompomentSystem}

When particles are built from two types of monomers, but still interact only through monomers, then their evolution is described by the two-component Becker--Döring equations.
These equations govern the dynamics of the concentrations $c\rs$ of particles made up of $r$ Type I and $s$ Type II monomers and read
\begin{equation}\label{mcBD}
  \begin{cases}
    \odv{}{t}c\rs  &= -\delr \Jr\rs - \dels \Js\rs\quad \text{for }r+s\ge2,\\
    \odv{}{t}c_{1,0} &= -\Jr_{1,0}-\Js_{1,0} - \sum\limits_{r+s=1}^\infty  \Jr\rs,\\
    \odv{}{t}c_{0,1} &= -\Jr_{0,1}-\Js_{0,1} - \sum\limits_{r+s=1}^\infty \Js\rs,
  \end{cases}
\end{equation}
where $\Jr\rs$ denotes the flux between particles of size $(r,s)$ and $(r+1,s)$ and $\Js\rs$ the flux between those of size $(r,s)$ and $(r,s+1).$ The two-dimensional sums have to be understood as ranging over subsets of $\Omega\coloneqq \{(r,s) \in \N^2_0 \,|\, r+s\ge 1\}$. Furthermore, the discrete backwards derivative operators $\delr$ and $\dels$ are defined for a sequence $(f\rs)_{(r,s)\in\Omega}$ via
\[ \delr f\rs \coloneqq f\rs - f\rms \text{ and }\dels f\rs \coloneqq f\rs - f\rsm, \text{ with the convention }f\rs = 0 \text{ for all }(r,s) \not\in \Omega.\]
Once we fix coagulation coefficients $\ar\rs,\as\rs\ge 0$ and fragmentation coefficients $\br\rps,\bs\rsp\ge 0$ for all $(r,s)\in\Omega$, we can describe the fluxes as
\begin{equation}\label{fluxes}
    \Jr\rs(c) \coloneqq \ar\rs c_{1,0}c\rs  - \br\rps c\rps
    \quad\text{ and }\quad
    \Js\rs(c) \coloneqq \as\rs c_{0,1}c\rs  - \bs\rsp c\rsp.
\end{equation}
If we multiply \eqref{mcBD} with a test-sequence $(g\rs)_{(r,s)\in\Omega},$ integrate it up and do a summation by parts, we formally find the weak form
\begin{equation}\label{mcBDWeak}
    \sumrs g\rs c\rs(t)  = \sumrs g\rs c\rs(0) + \int_0^t \sumrs (\delr g\rps - g_{1,0}) \Jr\rs + (\dels g\rsp - g_{0,1}) \Js\rs\dd \tau.
\end{equation}
\subsection{Basic Properties}\label{section:basicProperties}

In the applied literature \eqref{mcBD} is usually referred to as birth--death equations and has been used extensively to calculate and estimate nucleation rates for atmospheric droplets \cite{elm:applicationReview,vehkamaeki:CNT}. We initiated the general study of \eqref{mcBD} from a mathematical perspective in \cite{paperBasics}. In this section we want to summarise the main results needed for the present work. Remember that $\Omega = \{(r,s) \in \N^2_0 \,|\, r+s\ge 1\}$ and let 
\[ X^+\coloneqq \left\{(c\rs)_{(r,s)\in\Omega} \subset \R_{\ge 0} \;\middle|\; \sumrs r c\rs < \infty \text{ and } \sumrs s c\rs < \infty \right\} \text{ and } ||c|| \coloneqq \sumrs (r+s) |c\rs|.\]
Furthermore, for $c\in X^+,$ we define its Type I and Type II mass as
\[ \rho(c) \coloneqq \sumrs r c\rs \text{ and }\sigma(c) \coloneqq \sumrs s c\rs \text{ respectively}.\]
    \begin{definition}\textit{(Solutions to \eqref{mcBD})}\label{mcBDSolutions}\\
        \input{\CommonPath/definitions/mcBDSolutions}
\end{definition}
Then, we found the following existence result 
\begin{theorem}\label{mcBDCauchy}(Cauchy problem \eqref{mcBD}, \cite{paperBasics} Theorems 2.8 and 2.9)\\
  Assume that $\ar\rs + \as\rs \in \bigO(r+s)$. Then, there exists a solution to \eqref{mcBD} for all times. Furthermore, any solution will conserve Type I and Type II mass for all times.
\end{theorem}
Even though Theorem \ref{mcBDCauchy} allows for arbitrary fragmentation coefficients $\br\rps,\bs\rsp \ge0,$ physical considerations (see Subsection \ref{subsection:physicalConsiderations}) constrain them via the following condition. There exists a sequence $Q\rs >0$ with $Q_{1,0} = 1 = Q_{0,1}$ such that 
\begin{equation}\label{detailedBalanceBasic}
  \Jr\rs(Q) =  0 = \Js\rs(Q) \text{ for all } r+s\ge1,
\end{equation}
which is equivalent to \eqref{assumption:detailedBalance} below.
Under this assumption, we made a rigorous connection between the entropy $V(c)$ and the dissipation $D(c),$ that are given by 
\begin{equation}\label{defVD}
  \begin{split}
      V(c) &\coloneqq \sumrs c\rs \left[\ln\left(\frac{c\rs}{Q\rs}\right)-1\right],\\
        \D(c) &\coloneqq \sumrs \Jr\rs \ln\left(\frac{\ar\rs\monr c\rs}{\br\rps c\rps}\right) + \Js\rs \ln\left(\frac{\as\rs\mons c\rs}{\bs\rsp c\rsp }\right).
  \end{split}
\end{equation}
\begin{theorem}\label{entropyInequality}(Entropy dissipation, \cite{paperBasics} Theorem 2.19 and \cite{dunwell:phd} Theorem 2.5.4)\\
  Let $c^0\in X^+$ with $\rho(c^0) \neq 0 \neq \sigma(c^0)$ and assume that $\ar\rs + \as\rs \in \bigO(r+s)$ and $\ar\rs,\as\rs,\br\rps,\bs\rsp >0$ satisfy \eqref{detailedBalanceBasic} as well as 
  \begin{equation}\label{assCoeff}
     \begin{split}
         \lim\limits_{N\to\infty}\inf\limits_{r+s\ge N} Q\rs^\frac 1 {r+s} >0 \quad\text{ and }\quad
         \lim\limits_{N\to\infty}\sup\limits_{r+s\ge N} Q\rs^\frac 1 {r+s} <\infty.
     \end{split}
  \end{equation}
  Then, there exists a solution $c$ to \eqref{mcBD} with initial data $c^0$, which is strictly positive on $(0,\infty)$, satisfying 
  \begin{equation*}
    \Vzw(c(t_2)) \le \Vzw(c(t_1)) - \int_{t_1}^{t_2} D(c(\tau)) \dd \tau \quad\text{ for all }0\le t_1 < t_2 < \infty.
  \end{equation*}
\end{theorem}
As mentioned above, the goal of the present work is to find and prove a generalisation of \eqref{ocLongTime} for the two-component system. For the classical Becker--Döring equations it was shown in \cite{penrose:foundations}, that \eqref{ocLongTime} is equivalent to 
\begin{equation}
  c(t) \text{ minimises }V(c) \text{ on } X^+_\rho \coloneqq \left\{(c_i)_{i\in\N} \subset \R_{\ge 0} \,\middle|\, \sum\limits_{n=1}^{\infty} n c_n = \rho\right\}.
\end{equation}
This formulation carries over nicely to the two-component case and yields a complete description of the limit in terms of the initial masses $\rho$ and $\sigma.$ To be precise, we have the following.
    \begin{definition}\textit{(Region of existence)}\label{regionOfExistence}\\
        \input{\CommonPath/definitions/regionOfExistence}
\end{definition}
If we call a sequence $\cbar\rs$ satisfying $\Jr\rs(\cbar) = 0 = \Js\rs(\cbar)$ for all $(r,s)\in\Omega$ an equilibrium of \eqref{mcBD}, then all equilibria with finite mass are exactly $(\cinfzw)_{(z,w)\in \Ex}.$ With this, we were able to map each pair of Type I and Type II masses $(\rho,\sigma)$ to a unique equilibrium $(z,w) \in \Ex$.
\begin{theorem}\label{steadyStateLimit}(\cite{paperBasics} Theorem 3.7)\\
  Assume $Q\rs$ satisfies \eqref{assCoeff}. Then, for any fixed $\rho,\sigma > 0$, there is a pair $(z,w)\in\Ex$ uniquely determined by 
  \begin{equation}\label{steadyStateLimitEq}
    \ln\left(\frac z u\right)\left(\rho - \rho(\cinfzw)\right) + \ln\left(\frac w v\right) \left(\sigma - \sigma(\cinfzw)\right) \ge 0 \text{ for all } 0< u,v \text{ with } (u,v)\in\Ex.
  \end{equation}
\end{theorem}
Before we characterise the minimising sequences of the entropy, let us first introduce the relative entropy, which is just $V + \const.$
    \begin{definition}\textit{(Relative entropy)}\label{relativeEntropy}\\
        \input{\CommonPath/definitions/relativeEntropy}
\end{definition}
Note that $\Psi \ge 0$ and hence $H \ge 0.$ Finally, we proved 
\begin{theorem}\label{steadyStateMinimising}(Minimising sequences, \cite{paperBasics} Theorem 4.3 and Corollaries 4.4 and 4.5)\\
  Assume $Q\rs$ satisfies \eqref{assCoeff},  fix $\rho,\sigma>0$ and define 
  \[X^+_{\rho,\sigma} \coloneqq \Big\{ c\in X^+ \,\Big|\, \sumrs r c\rs = \rho \text{ and } \sumrs s c\rs = \sigma\Big\}.\]
  For $\cinfzw$ determined by \eqref{steadyStateLimitEq} we have 
  \[ \inf\limits_{c\in X^+_{\rho,\sigma}} H[c|\cinfzw] = 0\]
  and any minimising sequence $c^n$ converges to $\cinfzw$ in the weak* topology, i.e. $\lim\limits_{n \to \infty} c^n\rs = \cinfzw\rs$ for all $r+s\ge 1.$ Furthermore, for any $M\in \N$ and $\delta>0$, we have as $n\to\infty$
\begin{equation}
  \sumgamma[\Gamma_{M,\delta}] (r+s) |c^n\rs - \cinfzw\rs| \to 0, \text{ where }
  \Gamma_{M,\delta} \coloneqq \left\{(r,s)\in \Omega\,\middle|\, r+s\le M \text{ or }\frac{\ln(\cinfzw\rs)}{r+s} \le - \delta\right\}.
\end{equation}
\end{theorem}

\subsection{Coefficients}\label{section:Coefficients}

The behaviour of solutions to \eqref{mcBD} depends heavily on the parameters $\ar\rs,\as\rs,\br\rps,\bs\rsp$. The derivation of accurate coefficients for a given physical situation is still active research and believed to be one of the main difficulties, when predicting nucleation rates \cite{elm:applicationReview}. Luckily, for our study, we do not need to know these parameters precisely. However, they do need to obey certain asymptotic bounds. To make sure that we are studying a physically relevant context, let us quickly lay out a standard approach to estimate these parameters and then formulate the assumptions of the main theorem in the physical context.

\subsubsection{Physical Considerations}\label{subsection:physicalConsiderations}

 We start with the coagulation coefficients. Assuming that the monomers move through space in a diffusive manner and that larger particles are very slow, the likelihood of a monomer colliding with a particle of size $(r,s)$ should be proportional to the surface area of the $(r,s)$ particle. Since $\ar\rs$ and $\as\rs$ should be given as collision rate multiplied with an accommodation coefficient, we expect for $r+s$ large $\ar\rs \approx (r+s)^{\frac 2 3} \approx \as\rs$. A sharper analysis, as in the appendix of \cite{niethammer:LswLimit} reveals the general form
\begin{equation}\label{coagulationPhysical}
  \ar\rs \approx (r+s)^\alpha \approx \as\rs \text{ with } \alpha \in [0,1].
\end{equation}
The surface attachment $\alpha = \frac 2 3$ corresponds to interface-reaction-controlled systems, but the growth may also be controlled by diffusion which results in $\alpha = \frac 1 3.$ In the latter we assume that the attachment of monomers to a cluster depletes the monomers locally in space, leading to a spatial gradient in the monomer concentrations and hence to the lower value of $\alpha$.
With the coagulation parameters given, we can calculate the condensation rates from the Gibbs free energy. To be precise, unless there is energy or entropy going into or out of the system, it should satisfy the detailed balance condition, i.e. there exists a sequence $Q\rs>0$ (given by the Gibbs free energy), such that 
\begin{equation}
  \br\rps = \ar\rs \frac{Q\rs}{Q\rps} \text{ and }\bs\rsp = \as\rs \frac{Q\rs}{Q\rsp} \text{ and } Q_{1,0} = Q_{0,1}.
\end{equation}
We can assume without loss of generality $Q_{1,0} = 1 =Q_{0,1},$ which we will do throughout this entire work. If we denote the Gibbs free energy of a $(r,s)$ cluster as $F(r,s)$, we have up to some physical constants $Q\rs = e^{F(r,s)}$. In particular, the coefficient $Q\rs$ can in theory be calculated from the quantum mechanic states of the $(r,s)$ cluster, which is unfortunately not feasible for large clusters \cite{jensen:computationalChemistry}. Instead some approximations (e.g. liquid droplet approximation and capillarity approximation) have to be made, to estimate $Q\rs.$ The idealised Kelvin model is probably the simplest (non-trivial) example, which yields \cite{wilemski:Coefficients}
\begin{equation}\label{idealizedKelvin}
  Q\rs \approx \lambda^r \mu^s \binom{r+s}{r} e^{- \kappa \sum\limits_{n=1}^{r+s} n^{-\frac 1 3}},
\end{equation}
where $0<\lambda,\mu,\kappa<\infty$ are constants (given by the equilibrium monomer concentrations of a gas over a pure liquid and the surface tension).\\
Classically, one considers the Gibbs free energy of a droplet of size $(r,s)$ to be given by the surface area $A\rs$, the surface tension $\kappa(\cdot)$, and two functions $\phi_1(\cdot)$ and $\phi_2(\cdot)$, that are determined by the liquid phase activity and gas phase activity \cite{vehkamaeki:CNT}. The Gibbs free energy $F(r,s)$ then takes the form 
\[ F(r,s) = -k T \Big[ r \ln(\phi_1(\xi)) + s \ln(\phi_2(1-\xi))\Big] + \kappa(\xi) A\rs,\]
where $k$ is the Boltzmann constant, $T$ the temperature and $\xi$ is the fraction of Type I monomers in the bulk of the droplet. If we assume perfect mixture of the Type I and Type II monomers in the droplet, so that the fraction of Type I monomers in the bulk is equal to the fraction on the surface, we have $\xi = \frac{r}{r+s}.$ This leads us to the general form, we will consider in this work 
\begin{equation}\label{QThermoDom}
  Q\rs = e^{(r+s) \Phi\left(\frac{r}{r+s}\right) + \smallo(r+s)}.
\end{equation}
This structure immediately carries over to all equilibria $\cinfzw\rs = \exp\left((r+s) G_{z,w}\big(\frac r {r+s}\big) +\smallo(r+s)\right),$ where $G_{z,w}$ is the binding energy for given monomer concentrations 
\begin{equation}\label{defBindingEnergy}
  G_{z,w}(\xi) \coloneqq \xi \ln(z) + (1-\xi)\ln(w) + \Phi(\xi).
\end{equation}
A nice discussion regarding the applicability of the equations \eqref{mcBD} to physical systems can be found in \cite{elm:applicationReview}.

Finally, we want to mention, that the appearance of the binomial coefficient in \eqref{idealizedKelvin} is very natural, which was already pointed out in \cite{dunwell:phd}.
To see why, consider a one-component system with coagulation and fragmentation parameters $A_i,B_i$. If we track a portion of the monomers, e.g. by colouring them, we obtain an artificial two-component system. The coagulation parameters of that two-component system are then exactly $\ar\rs = \as\rs = A_{r+s}$, because there is no physical difference between Type I and Type II monomers. On the other hand, the probability that the evaporation of a monomer from an $(r,s)$ cluster is of Type I is $\frac r {r+s}$ and hence 
\begin{equation}\label{combinatoricB}
  \br\rs = \frac{r}{r+s} B_{r+s} \text{ and } \bs\rs = \frac s {r+s} B_{r+s}.
\end{equation}
Now, this is exactly equivalent to \eqref{assumption:detailedBalance} with $Q\rs = \binom{r+s}{r} Q_{r+s}$, where $Q_{r+s}$ are the one-dimensional $Q_i$ from \eqref{ocDetailedBalance}. In this case, one can even see, that $c_i(t)$ is a solution to \eqref{ocBD}, if and only if $c\rs = q^r (1-q)^s \binom{r+s}{r} c_{r+s}$ is a solution to \eqref{mcBD} for $q\in(0,1).$

\subsubsection{Assumptions Needed for the Main Theorem}

In agreement with the physical considerations, we split our assumptions on the coefficients into a kinetic part on $\ar\rs,\as\rs>0$ and a thermodynamic part on $\br\rps,\bs\rsp>0$. The assumptions on the coagulation coefficients are straight from the previous subsection.
\begin{enumerate}
  \item (Isotropic kinetics for Type I and Type II) There is an $\alpha \in[0,1]$ and two constants $0<c_\alpha,C_\alpha<\infty$, such that
\begin{equation}\label{assumption:coagulationGrowth}
  c_\alpha (r+s)^\alpha \le \ar\rs,\as\rs \le C_\alpha (r+s)^\alpha \text{ for all }r+s\ge 1
\end{equation}
\end{enumerate}
As discussed above, the fragmentation parameters should be given by the thermodynamics.
\begin{enumerate}[resume*]
  \item (Detailed Balance) There is a sequence $Q\rs >0$ with $Q_{1,0} = 1 =Q_{0,1}$, such that
\begin{equation}\label{assumption:detailedBalance}
  \br\rps = \ar\rs \frac{Q\rs}{Q\rps} \text{ and }\bs\rsp = \as\rs \frac{Q\rs}{Q\rsp}.
\end{equation}
\end{enumerate}
Throughout this work, we will mainly be working with assumptions on the level of $Q\rs$.
However, in Section \ref{section:reductionToPhi} we will derive these assumptions from the Gibbs free energy through the connection \eqref{QThermoDom} and express our results in terms of $\Phi$.
We are asking for the following.
\begin{enumerate}[resume*]
\item(Thermodynamically dominating part) There is a $\Phi \in C^2((0,1))\cap C^0([0,1])$, such that $\lim\limits_{\xi \to 1} (1-\xi)\Phi^\prime(\xi)$ and $\lim\limits_{\xi \to 0} \xi \Phi^\prime(\xi)$ exist in $\R$ and 
  \begin{equation}
    \frac{Q\rps}{Q\rs} = \varphi\left(\frac{r+1}{r+s+1}\right)f\rs \quad\text{ and }\quad
    \frac{Q\rsp}{Q\rs} = \psi\left(\frac{r}{r+s+1}\right)g\rs, \label{assumption:thermo}
  \end{equation}
  where $\varphi(\xi)\coloneqq e^{\Phi(\xi) + (1-\xi)\Phi^\prime(\xi)}$ and $\psi(\xi)\coloneqq e^{\Phi(\xi) - \xi \Phi^\prime(\xi)}$. Furthermore, for any $\eps>0$, there exists an $N\in\N$, such that for all $r+s\ge N$ we have $|1-f\rs|<\eps$ and $|1-g\rs|<\eps.$
\item \eqLabeledLine{(Concavity) There is a $\gamma>0$, such that $(1-\xi)\Phi^{\prime\prime}(\xi) \le - \gamma <0$ and $\xi \Phi^{\prime\prime}(\xi) \le -\gamma < 0.$}{assumption:thermo:concave}
\item\eqLabeledLine{(Non-degeneracy at the boundary) The limits $\displaystyle \lim\limits_{\xi \to 0} \xi \varphi(\xi)$ and $\displaystyle \lim\limits_{\xi \to 1}  (1-\xi) \psi(\xi)$ exist in $\R$.}{assumption:thermo:boundary}
\item(Continuity of higher order terms) The bounds 
  \begin{equation} \label{assumption:thermo:higherOrder} 
    \limsup\limits_{n\to \infty}\max\limits_{0<r\le n}\Big(\frac{f\rn}{f\rmn}\Big)^{n+1} \le 1\text{ and }\liminf\limits_{n\to \infty}\min\limits_{0<r\le n}\Big(\frac{g\rpn}{g\rn}\Big)^{n+1} \ge 1\text{ hold true.}
  \end{equation}
  \end{enumerate}
To see how \eqref{QThermoDom} connects to \eqref{assumption:thermo}, let us ignore for the moment the higher order terms in \eqref{QThermoDom}. Then we find from the mean value theorem a $\xi \in \big(\frac{r+1}{r+s+1},\frac r {r+s}\big)$, such that
  \begin{equation}
    \begin{split}
      \frac{Q\rps}{Q\rs} = e^{\delr (r+s+1)\Phi\left(\frac {r+1}{r+s+1}\right)}
    = e^{\Phi\left(\frac {r+1}{r+s+1}\right) + (r+s) \delr \Phi\left(\frac {r+1}{r+s+1}\right)}
    = e^{\Phi\left(\frac {r+1}{r+s+1}\right) + \frac{s}{r+s+1} \Phi^\prime(\xi)}.
    \end{split}
  \end{equation}
  However, the argument \eqref{combinatoricB} shows that the correct way to resolve $\Phi(\xi) = -\xi \ln(\xi) - (1-\xi)\ln(1-\xi)$ is the binomial $Q\rs = \binom{r+s}{r}$, for which \eqref{assumption:thermo} holds without higher order terms, i.e. $f\rs = 1 =g\rs.$ This suggests that \eqref{assumption:thermo} is the cleaner way to separate the higher order terms from the Gibbs free energy. Finally, we want to give an example of a class of coefficients, that satisfy \eqref{assumption:thermo}--\eqref{assumption:thermo:higherOrder}.
    \begin{example}\textit{(Standard coefficients are applicable)}\label{standardCoeffApplicable}\\
        \input{\CommonPath/examples/standardCoeffApplicable}
    \end{example}

\subsection{Main Theorem}\label{subsection:mainTheorem}

From Section \ref{section:basicProperties}, one can easily deduce that for any initial datum, there exists a solution $c(t)$ to \eqref{mcBD}, for which the weak* distance to the set of equilibria vanishes as $t\to \infty.$ The difficulty arises, when we try to single out a unique equilibrium $(z,w)\in\Ex,$ to which $c(t)$ is supposed to converge to. The main body of the present work is dedicated to prove a stronger result, i.e. $c(t)$ minimises the entropy as $t\to\infty$. It then immediately follows from Theorems \ref{mcBDCauchy}, \ref{steadyStateLimit} and \ref{steadyStateMinimising}, that $c\rs(t)\to\cinfzw\rs$, where $(z,w)\in\Ex$ is uniquely determined by \eqref{steadyStateLimitEq}. Therefore, our main theorem reads
\begin{theorem}\label{mainTheorem}(Main theorem)\\
  Assume \eqref{assumption:coagulationGrowth}--\eqref{assumption:thermo:higherOrder}.
  Then, for any initial data with masses $\rho,\sigma>0,$ there is a solution $c$ to \eqref{mcBD} satisfying 
  $ H[c(t)|\cinfzw] \to 0 \text{ as }t\to\infty,$
  where $\cinfzw$ is the steady state satisfying \eqref{steadyStateLimitEq}.
  Furthermore, for the binding energy of the limit point $G_{z,w}(\xi) = \xi \ln z + (1-\xi)\ln w + \Phi(\xi)$, we have
\[\text{either  }\rho(\cinfzw) = \rho \text{ and }\sigma(\cinfzw) = \sigma 
\text{ or } G_{z,w}\left(\frac{\rho-\rho(\cinfzw)}{\rho+\sigma - \rho(\cinfzw) - \sigma(\cinfzw)}\right) = 0.\]
\end{theorem}
The last part of the theorem determines the ratio of the lost Type I mass, providing a full description of the long time behaviour with respect to the initial masses $\rho,\sigma$ and the free energy $\Phi$, which we can express in a picture. We have done so for the idealised Kelvin model in Figure \ref{longtimePicture}. Note the astonishing fact, that the selection of the limit point is purely through thermodynamics, i.e. $Q\rs$ and completely independent from the exact kinetics $\ar\rs,\as\rs.$
\input{\CommonPath/pictures/longTime.tex}
Next, we want to sketch the proof of Theorem \ref{mainTheorem}, which also outlines the rest of this paper. To show $H[c(t)|\cinfzw] \to 0$ as $t\to\infty$, we will argue by contradiction. By Theorem \ref{entropyInequality}, the function $H[c(t)|\cinfzw]$ is monotonically decreasing in $t$, so let as assume that $H[c(t)|\cinfzw] \ge \delta >0$ for all times.
Now, a major part of the proof is the following dissipation estimate.
\begin{theorem}\label{dissipationEstimateIntro}(Dissipation estimate)\\
  Assume \eqref{assumption:coagulationGrowth}--\eqref{assumption:thermo:higherOrder} and fix $\rho,\sigma>0$ and $\delta>0.$
Then, there exist constants $K<\infty$ and $\delta_2>0$ as well as $M<\infty$, such that for all $c\in X^+_{\rho,\sigma}$ and $M\le N\in \N$ with  $\displaystyle \sumrs[N+1] (r+s) c\rs \le \delta_2 \text{ and } H[c|\cinfzw] > \delta,$
where $\cinfzw$ satisfies \eqref{steadyStateLimitEq}, we have \[\displaystyle \D(c) \ge \frac {N^{\alpha-1}} K.\]
\end{theorem}
Miraculously, the timescale at which mass can escape to infinity is also of the order $N^{1-\alpha}$. In Lemma \ref{timescale}, we show that there are constants $C_1,C_2<\infty$, such that if $\sumrs[\frac{N+1}3] (r+s)c\rs(T) \le \frac{\delta_2}{C_1}$ at some time $T$, then $\sumrs[N+1](r+s)c\rs(t) < \delta_2$ for all $T\le t\le T+\frac{N^{1-\alpha}}{C_2}.$ In particular $c(t)$ satisfies the assumptions of Theorem \ref{dissipationEstimateIntro} on this time interval and we find 
\[ H[c(T)|\cinfzw] - H[c(T+\frac{N^{1-\alpha}}{C_2})|\cinfzw] \overset{\text{Theorem \ref{entropyInequality}}} \ge \int_T^{T+\frac{N^{1-\alpha}}{C_2}} D(c(\tau))\dd\tau \overset{\text{Theorem \ref{dissipationEstimateIntro}}} \ge \frac{1}{C_2 K},\]
in contradiction to the convergence of $H[c(t)|\cinfzw].$

The proof of Theorem \ref{dissipationEstimateIntro} is contained in Chapter \ref{chapter:dissipationEstimate} in combination with Proposition \ref{mainAssumptions}.
If we denote 
\begin{equation}\label{defCbarU}
  \cbar\rs \coloneqq \monr^r\mons^s Q\rs \text{ and } u\rs \coloneqq \frac{c\rs}{\cbar\rs},
\end{equation}
we can write as in Section 3 of \cite{paperBasics}
\[ \D(c) =\sumrs \ar\rs \monr \cbar\rs \Big(\delr u\rps \delr \ln(u\rps)\Big)
+ \as\rs \mons \cbar\rs \Big(\dels u\rsp \dels \ln(u\rsp)\Big).\]
Hence, we might as well bound 
\begin{equation}\label{defDnlin}
  \begin{split}
    \Dnlin \coloneqq \sumrs[1][N-1] &(r+s+1) \monr \cbar\rs \Big(\delr u\rps \delr \ln(u\rps)\Big)\\
    &+ (r+s+1) \mons \cbar\rs \Big(\dels u\rsp \dels \ln(u\rsp)\Big)
  \end{split}
\end{equation}
from below, because \eqref{assumption:coagulationGrowth} implies $D\gtrsim N^{\alpha -1} \Dnlin$.
Now, the bound $\Dnlin \ge \frac 1 K$ is obtained through three separate cases---small monomer densities, subcritical monomer densities and supercritical monomer densities.
The first case is treated in Theorem \ref{dissipationSmallMon} by a direct calculation related to \eqref{assumption:thermo:boundary}, which shows $\Dnlin \ge \frac 1 K$ whenever $\min(\monr,\mons) \le \eta$. Next, we study in Subsection \ref{subsection:abstractResult} a discrete two-dimensional logarithmic Sobolev inequality, found in Theorem \ref{hardyInequality}, which roughly reads
\[ H^N[c|\cbar] \coloneqq \sumrs[2][N] \cbar\rs \Psi(u\rs) \le K_N \Dnlin,\]
with an explicit constant $K_N$ and holds as an abstract result independent of any assumptions on $\cbar\rs,u\rs$. In Subsection \ref{subsection:application}, we show, that under the concavity assumption \eqref{assumption:thermo:concave}, we can express $\sup\limits_{N\in\N} K_N$ in terms of $\max\limits G_{\monr,\mons}(\xi)$. From this, we deduce in Section \ref{section:subcritical} the estimate $\Dnlin \ge \frac 1 K$, whenever $\max\limits_{\xi\in[0,1]} G_{\monr,\mons}(\xi)\le -\eps$, which you may take as a definition for subcriticality.\\
Finally, we reduce the case $\max\limits_{\xi\in[0,1]} G_{\monr,\mons}(\xi)> -\eps$ to the subcritical one. The semicontinuity of $H[c|\cbar]$ (Lemma \ref{relativeEntropySemicontinuity}), allows us to replace the $\cbar$ in the logarithmic Sobolev inequality by a subcritical $\chat.$ The construction of such a $\chat$ is a delicate matter, because it is tightly constrained by Lemma \ref{relativeEntropySemicontinuity} and the conditions from Subsection \ref{subsection:application}. We present an inductive construction in Subsection \ref{subsection:construction} and then analyse it in Subsections \ref{subsection:chatResults} and \ref{subsection:constructionApplication} to be good enough to prove Theorem \ref{dissipationSupercritical}, i.e. $\Dnlin \ge \frac 1 K$ in the supercritical case. This finishes the proof that $c(t)$ minimises the entropy.

Note that we are using sub/supercriticality to discern $\max\limits_{\xi\in[0,1]} G_{\monr,\mons}(\xi) \substack{\le \\ >} -\eps$, which somewhat abuses the usual naming convention that defines
\begin{equation}
  \begin{split}
    \monr,\mons \text{ subcritical} \quad &\overset{\text{def.}}\iff (\monr,\mons)\in \interior{\Ex} \overset{\text{in our case}}\iff\max\limits_{\xi\in[0,1]} G_{\monr,\mons}(\xi)< 0 \\
    \monr,\mons \text{ critical} \quad &\overset{\text{def.}}\iff (\monr,\mons) \in \partial{\Ex} \overset{\text{in our case}}\iff\max\limits_{\xi\in[0,1]} G_{\monr,\mons}(\xi)= 0 \\
    \monr,\mons \text{ supercritical} \quad &\overset{\text{def.}}\iff (\monr,\mons) \not\in{\Ex} \overset{\text{in our case}}\iff\max\limits_{\xi\in[0,1]} G_{\monr,\mons}(\xi)> 0.
  \end{split}
\end{equation}
However, since $\eps$ is basically arbitrary and we are not trying to push the estimate to the critical case (which is only possible from below anyway), our naming seems justified.

We end our discussion of the two-component system, by characterising the escaping mass and the concentration on the critical line, which are general statements for all minimising sequences of $H[\cdot|\cinfzw]$ on $X^+_{\rho,\sigma}$ and depend only on the binding energy $G_{z,w}.$

Lastly, we discuss in Section \ref{section:improvedOC}, how our result (or our method of proof) can be used to improve the assumptions to determine the long time behaviour for the one-component system.

Before we start with Chapter \ref{chapter:dissipationEstimate}, we want to mention, that putting $D$ into relation with $H[c|\cbar],$ where $\cbar\rs \coloneqq \monr(t)^r \mons(t)^s Q\rs$, was first done in \cite{jabin:convergenceRate} to obtain a stretched exponential convergence to the equilibrium for subcritical initial data. Our estimate is inspired by \cite{canizo:EDEstimate}.

\subsection{List of Symbols}

\begin{center}
  \begin{tabular}{L|l|l}
    \text{Symbol} & Meaning & Introduced in \\ 
    \hline
    \delr, \dels & discrete backwards derivative operator & Section \ref{section:twoCompomentSystem}\\
    \Omega & phase space & Section \ref{section:twoCompomentSystem}\\
    \ar\rs,\as\rs & coagulation parameters & Section \ref{section:twoCompomentSystem}\\
    \br\rps,\bs\rsp & fragmentation parameters & Section \ref{section:twoCompomentSystem}\\
    \Jr\rs,\Js\rs & fluxes & Section \ref{section:twoCompomentSystem}\\
    \rho(c),\sigma(c) & Type I and II masses & Section \ref{section:basicProperties}\\
    X^+,X^+_{\rho,\sigma} & solution space & Section \ref{section:basicProperties} and Theorem \ref{steadyStateMinimising}\\
    V(c) & entropy & \eqref{defVD}\\
    D(c) & entropy dissipation & \eqref{defVD}\\
    Q\rs & partition function & \eqref{assumption:detailedBalance}\\
    \Ex & region of existence & Definition \ref{regionOfExistence}\\
    H[c|\cinfzw] & relative entropy & Definition \ref{relativeEntropy}\\
    \Psi(x) &  $\Psi(x) = x \ln (x) + 1 -x$ & Definition \ref{relativeEntropy}\\
    \Phi(\xi) & first order term of the Gibbs free energy & \eqref{QThermoDom} or more specifically \eqref{assumption:thermo}\\
    G_{z,w} & binding energy & \eqref{defBindingEnergy}\\
    \varphi(\xi) & $\varphi(\xi) = \exp(\Phi(\xi) + (1-\xi)\Phi^\prime(\xi))$ & \eqref{assumption:thermo}\\
    \psi(\xi) & $\psi(\xi) = \exp(\Phi(\xi) - \xi\Phi^\prime(\xi))$ & \eqref{assumption:thermo}\\
    \cbar\rs & quasi steady state & \eqref{defCbarU}\\
    u\rs & $u\rs = \frac{c\rs}{\cbar\rs}$ & \eqref{defCbarU}\\
    \Dnlin & $D$ cut off at $N$ with $\ar\rs = (r+s-1) = \as\rs$ & \eqref{defDnlin}\\
    G = (V,E) & directed Graph with nodes $V$ and edges $E$ & Notation \ref{treeStructure}\\
    \er\rs,\es\rs & edge weights & Notation \ref{treeStructure}\\
    \T & set of spanning trees with $(1,0),(0,1)$ fused & Notation \ref{treeStructure}\\
    \nu & probability distribution on $\T$ & Notation \ref{treeStructure}\\
    \mathbb E[\cdot] & expected value & Notation \ref{treeStructure}\\
    \Trs(T) & maximal connected subgraph with root $(r,s)$ & Notation \ref{treeStructure}\\
    \Prs(T) & unique path from root to $(r,s)$ in $T$ & Notation \ref{treeStructure}\\
    \Prs h, \Trs g & specific random variables& Notation \ref{treeStructure}\\
    l(x) & $l(x) = \frac{x\ln^2(x)}{\Psi(x)}$& Lemma \ref{derivativeSqrtPsi}\\
  m_n & $m_n = \frac{\sumrs[n][] \cbar\rs }{\sumrs[n-1][]\cbar\rs}$ & Lemma \ref{eres}\\
    r_*(n) & $r_*(n) = \min\big\{r \,\big|\, \cbar\rn\ge e^{-\delta_1 n}\big\}$ & Subsection \ref{subsection:construction}\\
    r^*(n) & $r_*(n) = \max\big\{r \,\big|\, \cbar\rn\ge e^{-\delta_1 n}\big\}$ & Subsection \ref{subsection:construction}\\
    \chat\rs & cut off construction & \eqref{chatConstruction}\\
    \ccirc\rs & cut off construction & \eqref{ccircConstruction}\\
    r^m(n) & $r^m(n) = \max\big\{r \,\big|\, \cbar\rn \ge \cbar\kn\big\}$ & Subsection \ref{subsection:construction}\\
    r_m(n,\chat) & $r_m(n) = \min\big\{r \,\big|\, \chat\rn \ge \chat\kn\big\}$ & Subsection \ref{subsection:chatResults}\\
    q_n & $q_n=\frac{\max\limits_{r+s=n}\cbar\rs}{\max\limits_{r+s=n-1}\cbar\rs}$ & \eqref{assumptionQn}
  \end{tabular}
\end{center}

%% file: definitions/mcBDSolutions.tex
Fix $0<T\le\infty$. We call a function $c \colon [0,T) \to X^+$ a solution to \eqref{mcBD} with initial data $c^0 \in X^+$, if 
\begin{enumerate}
    \item for all $(r,s)\in\Omega$ the function $c\rs \colon [0,T)\to\R_{\ge0}$ is continuous and  $\sup\limits_{[0,T)}\norm{c(t)}<\infty$,
    \item the transitions are in $L^1_\loc$, i.e.
        \begin{equation*}
            \int_0^t \sumrs (\ar\rs+\as\rs + \br\rps + \bs\rsp) c\rs(\tau) \dd\tau < \infty \text{ for all } t\in[0,T)\text{ and }
        \end{equation*}
    \item the equations are satisfied in a mild sense, i.e. for any $t\in[0,T)$
        \begin{equation}
            \begin{split}
                c_{1,0}(t) &= c^0_{1,0} - \int_0^t \Jr_{1,0}(c(\tau)) + \Js_{1,0}(c(\tau)) + \sumrs \Jr\rs(c(\tau)) \dd\tau,\\
                c_{0,1}(t) &= c^0_{0,1} - \int_0^t \Jr_{0,1}(c(\tau)) + \Js_{0,1}(c(\tau)) + \sumrs \Js\rs(c(\tau)) \dd \tau \text{ and}\\
                c\rs(t) &= c^0\rs - \int_0^t \delr \Jr\rs(c(\tau)) + \dels \Js\rs(c(\tau))\dd\tau \text{ for all }r+s\ge2.
            \end{split}
        \end{equation}
\end{enumerate}

%% file: definitions/regionOfExistence.tex
Let us denote $\cinfzw\rs \coloneqq z^r w^s Q\rs$. Then, we define the region of existence as 
\[\Ex \coloneqq\left\{(z,w) \mid z\ge 0, w\ge 0 \text{ and }\rho(\cinfzw) + \sigma(\cinfzw)<\infty\right\} .\]

%% file: definitions/relativeEntropy.tex
Let $0<z,w$ with $(z,w)\in\Ex$. We denote the relative entropy of $c\rs \in X^+$ with respect to $\cinfzw$ as
\begin{equation*}
    H[c|\cinfzw] \coloneqq \sumrs \cinfzw\rs \Psi\left(\frac{c\rs}{\cinfzw\rs}\right),
    \text{ where }\Psi(x) \coloneqq x\ln(x) + 1 -x.
\end{equation*}

%% file: examples/standardCoeffApplicable.tex
For a fixed $\beta \in (0,1]$ and $\lambda,\mu, v_1,v_2 \in (0,\infty)$, let 
\[ Q\rs \coloneqq \lambda^r \mu^s \binom{r+s}{r}^\beta e^{-(v_1 r+v_2 s)^{\frac 2 3}} \text{ for all }r+s > M,\]
with $M\ge 1$ and $Q_{1,0} = Q_{0,1} =1.$ Then $Q\rs$ satisfies \eqref{assumption:thermo}--\eqref{assumption:thermo:higherOrder} with 
\begin{equation*}
  \begin{split}
    \Phi(\xi) &= \xi \ln \lambda + (1-\xi)\ln\mu + \beta\Big( -\xi\ln\xi - (1-\xi) \ln(1-\xi)\Big),\\
    \varphi(\xi) &= \frac{\lambda}{\xi^\beta}, \qquad\psi(\xi) = \frac{\mu}{(1-\xi)^\beta} \\
    f\rs &= e^{(v_1 r+v_2 s)^{\frac 2 3} - (v_1(r+1)r+v_2 s)^{\frac 2  3}}\quad\text{ and }\quad
    g\rs = e^{(v_1r+v_2s)^{\frac 2 3} - (v_1r+v_2(s+1))^{\frac 2  3}} \quad\text{ for r+s >M.}
  \end{split}
\end{equation*}
The case $\beta = 1$, is the idealised Kelvin model \eqref{idealizedKelvin} with a slightly more general surface energy.

%% file: pictures/longTime.tex
\newcommand{\oneDMass}{1.964838} 
\newcommand{\oneDMassGreen}{1.1219} 
\newcommand{\RRR}{5}
\newcommand{\SSS}{8}
\newcommand{\ra}{3}
\newcommand{\sa}{5}
\newcommand{\greenDot}{1/3} 
\newcommand{\blueDot}{0.25} 
\newcommand{\redDot}{0.65}

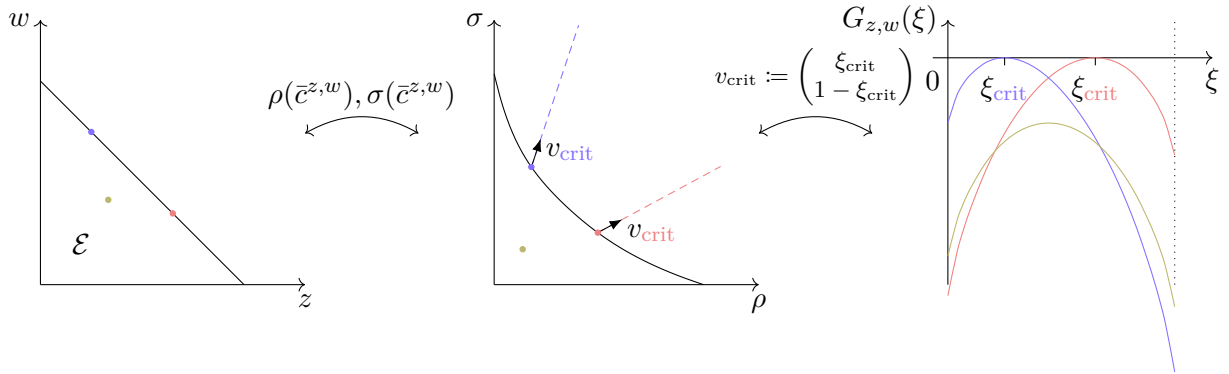
\begin{figure}[!htb]
  \centering 
  \begin{tikzpicture}[
          scale=1,
          declare function={
              xln(\x) = ifthenelse(\x == 0, 0, \x*ln(\x));
              rho(\x) = \oneDMass*\x + \RRR*(\ra -1 )*exp(-(\RRR-1)^(2/3))*\x^\RRR;
              sigma(\x) = \oneDMass*(1-\x) + \SSS*(\sa-1)*exp(-(\SSS-1)^(2/3))*(1-\x)^\SSS;
              rhoGreen(\x) = \oneDMassGreen*\x + \RRR*(\ra -1 )*exp(-(\RRR-1)^(2/3))*\x^\RRR;
              sigmaGreen(\x) = \oneDMassGreen*(0.75-\x) + \SSS* (\sa -1 )*exp(-(\SSS-1)^(2/3))*(0.75-\x)^\SSS;
          }
      ]
      \begin{scope}[shift={(6,0)}]
        \draw[->] (0, 0) -- (3.5, 0) node[below] {$\rho$};
        \draw[->] (0, 0) -- (0, 3.5) node[left] {$\sigma$};
        \draw[scale=1, domain=0:1, smooth, variable=\x, black] plot ( {rho(\x)}, {sigma(\x)});
        \draw[densely dashed, LightSlateBlue] ({rho(\blueDot)},{sigma(\blueDot)}) -- ({rho(\blueDot)+2.5*\blueDot},{sigma(\blueDot)+2.5*(1-\blueDot)});
        \draw[-{Latex}] ({rho(\blueDot)},{sigma(\blueDot)}) -- ({rho(\blueDot)+0.5*\blueDot},{sigma(\blueDot)+0.5*(1-\blueDot)}) node[below right= -2pt] {$v_{\text{\color{LightSlateBlue}crit}}$};
        \filldraw[LightSlateBlue] ({rho(\blueDot)},{sigma(\blueDot)}) circle (1pt);
        \draw[densely dashed, LightCoral] ({rho(\redDot)},{sigma(\redDot)}) -- ({rho(\redDot)+2.5*\redDot},{sigma(\redDot)+2.5*(1-\redDot)});
        \draw[-{Latex}] ({rho(\redDot)},{sigma(\redDot)}) -- ({rho(\redDot)+0.5*\redDot},{sigma(\redDot)+0.5*(1-\redDot)}) node[below right = -2pt] {$v_{\text{\color{LightCoral}crit}}$};
        \filldraw[LightCoral] ({rho(\redDot)},{sigma(\redDot)}) circle (1pt);
        \filldraw[DarkKhaki] ({rhoGreen(\greenDot)},{sigmaGreen(\greenDot)}) circle (1pt);
      \end{scope}
      \draw[<->] (3.5,2) to[bend left=30] (5,2);
      \node at (4.25,2.5){$\rho(\cinfzw),\sigma(\cinfzw)$};
      \begin{scope}[shift={(0,0)},scale=3.5/1.3]
          \node at (0.2,0.2) (a) {$\Ex$}; 
          \draw[->] (0, 0) -- (1.3, 0) node[below] {$z$};
          \draw[->] (0, 0) -- (0, 1.3) node[left] {$w$};
          \draw[scale=1, domain=0:1, smooth, variable=\x, black] plot ({\x}, {1-\x});
          \filldraw[DarkKhaki] (\greenDot,0.75-\greenDot) circle (1.3/3.5pt);
          \filldraw[LightSlateBlue] (\blueDot,1-\blueDot) circle (1.3/3.5pt);
          \filldraw[LightCoral] (\redDot,1-\redDot) circle (1.3/3.5pt);
      \end{scope}
      \draw[<->] (9.5,2) to[bend left=30] (11,2);
      \node at (10.25,2.75){\footnotesize$v_\crit \coloneqq \begin{pmatrix}
        \xi_\crit \\ 1-\xi_\crit
      \end{pmatrix}$};
      \begin{scope}[shift={(12,3)}]
          \node at (-0.2,-0.3) (a) {$0$}; 
          \draw[->] (-0.2, 0) -- (3.5, 0) node[below] {$\xi$};
          \draw[->] (0, -3) -- (0, 0.5) node[left] {$G_{z,w}(\xi)$};
          \draw[dotted] (3, -3) -- (3, 0.5);
          \draw[scale=3, domain=0.0001:1, smooth, variable=\x, LightCoral] plot ({\x}, {(-xln(\x) - xln(1-\x) + \x*ln(\redDot) + (1-\x)*ln(1-\redDot))});
          \draw[scale=3, domain=0.0001:1, smooth, variable=\x, LightSlateBlue] plot ({\x}, {(-xln(\x) - xln(1-\x) + \x*ln(\blueDot) + (1-\x)*ln(1-\blueDot))});
          \draw[scale=3, domain=0.0001:1, smooth, variable=\x, DarkKhaki] plot ({\x}, {(-xln(\x) - xln(1-\x) + \x*ln(\greenDot) + (1-\x)*ln(0.75-\greenDot))});
          \draw[-] ({3*\blueDot},0) -- ({3*\blueDot},-0.1) node[below] {$\xi_\text{\color{LightSlateBlue}crit}$};
          \draw[-] ({3*\redDot},0) -- ({3*\redDot},-0.1) node[below] {$\xi_\text{\color{LightCoral}crit}$};
      \end{scope}
  \end{tikzpicture}
  \pgfmathtruncatemacro{\Rminus}{\RRR-1}
  \pgfmathtruncatemacro{\Sminus}{\SSS-1}
  \caption{All initial conditions with masses on the dashed lines will converge to the anchor of their line. The slope of the dashed line is exactly given by $\frac{1-\xi_\crit}{\xi_\crit}$, where $\xi_\crit$ is the maximum point of $G_{z,w}$. We have plotted the situation for $Q\rs$ defined via $Q_{\RRR,0} \coloneqq \ra e^{-\Rminus^{\frac{2}{3}}}$, $Q_{0,\SSS} \coloneqq \sa e^{-\Sminus^{\frac{2}{3}}}$ and $Q\rs = \binom{r+s}{r} e^{-(r+s-1)^{\frac 2 3}}$ otherwise, i.e. $\beta = 1$ in Example \ref{standardCoeffApplicable}.}
  \label{longtimePicture}
\end{figure}

%% file: DissipationEstimates.tex
\section{Dissipation Estimates}\label{chapter:dissipationEstimate}
The goal of this chapter is to prove Theorem \ref{dissipationEstimateIntro}. However, even though we outlined its proof in the introduction at the level of $\Phi$ and $\xi$, we have to work with more general $Q\rs$ for the rigorous part throughout this chapter. Therefore, at the end of this chapter, we will have proved a slightly stronger statement (Theorem \ref{dissipationEstimate}) at the cost of having less readable assumptions. The reason for this is two fold. Firstly, there are coefficients of interest (see Example \ref{multiplicativeCoeff}), which are not covered by \eqref{assumption:thermo}--\eqref{assumption:thermo:higherOrder}. The second and arguably more important reason is that the construction of $\chat\rs$ in Subsection \ref{subsection:construction} does not fall into the assumptions \eqref{assumption:thermo}--\eqref{assumption:thermo:higherOrder}, because it is not regular enough. That is why, we have to build our analysis in detail at the level of $(r,s).$ In Proposition \ref{mainAssumptions}, we will then deduce all the assumptions introduced in this chapter from \eqref{assumption:thermo}--\eqref{assumption:thermo:higherOrder}.

As mentioned in the introduction, we want to establish---under the assumption $H[c|\cinfzw] \ge \delta > 0$---the estimate $\Dnlin \ge \frac 1 K$. Our strategy to prove that bound is roughly as follows.
\begin{enumerate}[label=\arabic*)]
  \item If $\min(\monr,\mons)$ is small, a direct calculation yields a universal $K$ independent of $\delta.$
  \item If $\cbar\rs$ is subcritical, then we estimate $\Dnlin \overset{\text{Section \ref{section:logSobolevInequality}}}\gtrsim H^N[c|\cbar] \overset{\text{continuity}}\gtrsim H[c|\cinfzw]$.
  \item If $\cbar\rs$ is supercritical, we reduce the estimate to 2) via the semicontinuity of $H^\Gamma[c|\cbar],$ by constructing a suitable subcritical $\chat\rs.$
\end{enumerate}
Here $\Ex$ is the region of existence from Definition \ref{regionOfExistence} and $H^N$ and $H^\Gamma$ are the relative entropy, where the sum is only taken over $\{r+s\le N\}$ and $\Gamma\subset \Omega$ respectively.

\subsection{Timescale and Small Monomer Densities}

Since we want to have $H^N[c|\cinfzw] \approx H[c|\cinfzw]$ and $H[\cdot|\cinfzw]$ is continuous only with respect to the strong and not the weak* topology on $X^+$, the correct timescale to consider is the speed at which mass gets transferred to larger and larger regions. Luckily, if $\ar\rs \approx (r+s)^\alpha \approx \as\rs$ this happens to be $N^{1-\alpha}$, as the following Lemma shows.
    \begin{lemma}\textit{(Timescale)}\label{timescale}\\
        \input{\CommonPath/lemmas/timescale}

    \end{lemma}
\begin{proof}
    \input{\CommonPath/proofs/timescale.tex}
\end{proof}
Next, we want to show that the dissipation must be large whenever $\monr$ or $\mons$ is small. This can be done, if we assume that $\frac{Q\rs}{Q\rps}$ and $\frac{Q\rs}{Q\rsp}$ do not vanish any faster than is given by the combinatorial argument \eqref{combinatoricB}. To be precise, we will assume that there exist constants $\lambda,\mu < \infty$, such that 
\begin{equation}\label{assumption:boundaryComb}
  \frac{Q\rs}{Q\rps} \ge \frac 1 \lambda \frac {r+1}{r+s+1} \quad\text{ and }\quad
  \frac{Q\rs}{Q\rsp} \ge \frac 1 \mu \frac {s+1}{r+s+1} \quad\text{ for all }r+s\ge 1.
\end{equation}
Note that \eqref{assumption:boundaryComb} is not optimal. Even in the one-dimensional case, one can recover a dissipation estimate for small monomers, for some $Q_i$ that satisfy $\frac{Q_i}{Q_{i+1}} \to 0$ along some subsequence (if this happens slow enough). However, with the considerations of Subsection \ref{subsection:physicalConsiderations} in mind, we are happy to accept \eqref{assumption:boundaryComb}.
    \begin{theorem}\textit{(Large dissipation for small monomers)}\label{dissipationSmallMon}\\
        \input{\CommonPath/theorems/dissipationSmallMon}
    \end{theorem}
\begin{proof}
    \input{\CommonPath/proofs/dissipationSmallMon.tex}
\end{proof}

%% file: lemmas/timescale.tex
Assume that $\ar\rs+\as\rs\in \mathcal{O}(r+s)$. Then, for any solution $c \colon [0,T)\to X^+$ to \eqref{mcBD}, with non zero Type I and Type II mass, we have for $0<t_1<t_2<T$ and all $N\ge 1$
\begin{equation*}
  \sumrs[3N] (r+s) c\rs(t_2) \le \sumrs[N] (r+s) c\rs(t_1) + 2(\rho + \sigma)\sqrt{\int_{t_1}^{t_2} D(c(\tau))\dd \tau} \sqrt{(t_2-t_1)\sup\limits_{r+s+1\ge N}\left\{\frac {\ar\rs}{r+s},\frac{\as\rs}{r+s}\right\}}.
\end{equation*}

%% file: proofs/timescale.tex
Let $N_2 \coloneqq \max\left\{n \ge N \colon \sum\limits_{k=N}^n \frac 1 k \le 1\right\}$. Then we define the sequence 
\[ g\rs \coloneqq
  \begin{cases}
    0 &\text{ for } r+s < N\\
    \sum\limits_{k=N}^{r+s} \frac 1 k &\text{ for } N \le r+s \le N_2\\
    1 &\text{ for }N_2<r+s.
  \end{cases}\]
  By definition $(r+s)g\rs = (r+s)$ for $r+s > N_2$, so it is a valid test-sequence for \eqref{mcBDWeak}. Therefore, we find 
  \begin{equation*}
    \begin{split}
      \sumrs &(r+s)g\rs c\rs(t_2) - \sumrs (r+s)g\rs c\rs(t_1) \\
      \overset{\text{equation}} &= 
      \int_{t_1}^{t_2} \sumrs \Jr\rs \underbrace{\delr\left((r+s+1)g\rps\right)}_{\ge 0}+ \Js\rs \underbrace{\dels \left((r+s+1)g\rsp\right)}_{\ge 0} \dd \tau \\
       &\le \int_{t_1}^{t_2} \sum\limits_{\substack{r+s\ge 1 \\ \Jr\rs > 0}} \Jr\rs \underbrace{\delr\left((r+s+1)g\rps\right)}_{\le 2\indicator{r+s+1\ge N}}
      + \sum\limits_{\substack{r+s\ge 1 \\ \Js\rs > 0}} \Js\rs \underbrace{\dels \left((r+s+1)g\rsp\right)}_{\le 2 \indicator{r+s+1\ge N}} \dd \tau\\
       &\le  2 \int_{t_1}^{t_2} \sumrs[N-1] \Jr\rs \indicator{\Jr\rs>0} + \Js\rs \indicator{\Js\rs>0} \dd \tau \eqqcolon (*).
    \end{split}
  \end{equation*}
  Now we use the Cauchy--Schwarz inequality twice, first in the sum and then in the integral to obtain
  \begin{equation*}
    \begin{split}
      (*)\overset{\text{C.S. in the sum}}&\le 2 \int_{t_1}^{t_2} \Bigg(\sumrs[N-1] \Jr\rs \ln\left(\frac{\ar\rs\monr c\rs}{\br\rps c\rps}\right) \indicator{\Jr\rs>0} + \Js\rs  \ln\left(\frac{\as\rs\mons c\rs}{\bs\rsp c\rsp}\right)\indicator{\Js\rs>0}\Bigg)^{\frac 1 2}\\
       &\qquad\qquad\qquad \cdot\Bigg(\sumrs[N-1] \frac{\Jr\rs} {\ln\left(\frac{\ar\rs\monr c\rs}{\br\rps c\rps}\right)} \indicator{\Jr\rs>0} + \frac{\Js\rs} {\ln\left(\frac{\as\rs\mons c\rs}{\bs\rsp c\rsp}\right)}\indicator{\Js\rs>0}\Bigg)^{\frac 1 2}\dd \tau\\
       \overset{\frac{x-y}{\ln x - \ln y} \le \max(x,y)}&\le 2 \int_{t_1}^{t_2} \Bigg(\sumrs[N-1] \Jr\rs \ln\left(\frac{\ar\rs\monr c\rs}{\br\rps c\rps}\right) \indicator{\Jr\rs>0} + \Js\rs  \ln\left(\frac{\as\rs\mons c\rs}{\bs\rsp c\rsp}\right)\indicator{\Js\rs>0}\Bigg)^{\frac 1 2}\\
       &\qquad\qquad\qquad \cdot\Bigg(\sumrs[N-1] \ar\rs\monr c\rs +\as\rs\mons c\rs\Bigg)^{\frac 1 2}\dd \tau\\
       \overset{\text{C.S. in the integral}}&\le 2 \Bigg(\int_{t_1}^{t_2} \sumrs[N-1] \Jr\rs \ln\left(\frac{\ar\rs\monr c\rs}{\br\rps c\rps}\right) + \Js\rs  \ln\left(\frac{\as\rs\mons c\rs}{\bs\rsp c\rsp}\right)\dd \tau\Bigg)^{\frac 1 2}\\
       &\qquad\qquad\qquad \cdot\Bigg(\int_{t_1}^{t_2} \sumrs[N-1] \ar\rs\monr c\rs +\as\rs\mons c\rs\dd\tau\Bigg)^{\frac 1 2}\eqqcolon (**).
    \end{split}
  \end{equation*}
  Putting in the definition of the dissipation \eqref{defVD}, we finally arrive at
  \begin{equation*}
    \begin{split}
      (**)
       &\le 2 \Bigg(\int_{t_1}^{t_2} D(c(\tau)) \dd \tau\Bigg)^{\frac 1 2}
       \cdot\Bigg(\sup\limits_{r+s \ge N-1}\left\{\frac {\ar\rs}{r+s}, \frac{\as\rs}{r+s}\right\} \int_{t_1}^{t_2} \sumrs[N-1] (\monr+\mons) (r+s) c\rs \dd\tau\Bigg)^{\frac 1 2}\\
       \overset{\mons+\monr \le \rho + \sigma}&\le 2(\rho+\sigma) \Bigg(\int_{t_1}^{t_2} D(c(\tau)) \dd \tau\Bigg)^{\frac 1 2}
       \cdot\Bigg(\sup\limits_{r+s \ge N-1}\left\{\frac {\ar\rs}{r+s}, \frac{\as\rs}{r+s}\right\} (t_2-t_1)\Bigg)^{\frac 1 2}.
    \end{split}
  \end{equation*}
  Because $(r+s) \indicator{r+s>N_2}\le (r+s)g\rs \le (r+s) \indicator{r+s\ge N}$, it remains to show that $N_2 < 3N$. However, this follows from 
  \begin{equation*}
      \sum\limits_{k=N}^{3N}  \frac 1 k
      \ge \int_N^{3N} \frac 1 x \dd x = \ln(3)>1. \qedhere
  \end{equation*}

%% file: theorems/dissipationSmallMon.tex
Fix $\rho,\sigma > 0$ and assume $Q\rs$ satisfy \eqref{assumption:boundaryComb}
Then, there exist two constant $\eta > 0$ and $K<\infty$, such that for all $c\in X^+_{\rho,\sigma}$ with $\rho^N(c) \ge \frac \rho 2$, $\sigma^N(c)\ge \frac \sigma 2$ and $\min(\monr,\mons) \le \eta$, we have 
\[ \Dnlin \ge \frac 1 K.\]

%% file: proofs/dissipationSmallMon.tex
\newcommand{\sumxi}{\sum\limits_{\substack{r+s = 1 \\ \frac{r+1}{r+s+1}\ge \bar\xi}}^{N-1} }
Let $N\in\N$ be arbitrary and assume $c\in X^+_{\rho,\sigma}$ satisfies $\rho^N(c) \ge \frac \rho 2$, $\sigma^N(c)\ge \frac \sigma 2$ and $\monr \le \eta$ with $\eta$ to be determined.
Fix any $\theta > 1$ and let $\bar\xi \coloneq \frac{\rho}{4(\rho + \sigma)}$. Then, we have 
\begin{equation}
  \begin{split}
    \Dnlin 
    &\ge \sumrs[1][N-1] (r+s+1) \monr \cbar\rs (\delr u\rps \delr \ln u\rps) \\
    &= \sumrs[1][N-1] (r+s+1) \frac{Q\rs}{Q\rps} c\rps \left(\Big(1-\frac{u\rs}{u\rps}\Big) \delr \ln u\rps\right)\\
    \overset{\frac{Q\rs}{Q\rps} \ge \frac 1 \lambda \frac{r+1}{r+s+1}}&\ge \frac 1 \lambda \sumxi  \indicator{\frac{u\rps}{u\rs} \ge \theta} (r+1)  c\rps \left(\Big(1-\frac{u\rs}{u\rps}\Big) \delr \ln u\rps\right).
  \end{split}
\end{equation}
Because $(1-\frac 1 x)\ln x$ is monotone on $x\ge 1$, this yields
\begin{equation}\label{dissipationSmallMon:eq1}
    \Dnlin \ge \frac{(1-\frac 1 \theta)\ln \theta} \lambda \sumxi  \indicator{\frac{u\rps}{u\rs} \ge \theta} (r+1)  c\rps.
\end{equation}
In order to show that $\sumxi  \indicator{\frac{u\rps}{u\rs} \ge \theta} (r+1)  c\rps$ is large, we first note
\begin{equation*}
  {\sum\limits_{\substack{r+s = 1 \\ \frac{r+1}{r+s+1}< \bar\xi}}^{N-1} } (r+1) c\rps 
  = {\sum\limits_{\substack{r+s = 1 \\ \frac{r+1}{r+s+1}< \bar\xi}}^{N-1} } \frac{r+1}{r+s+1}(r+s+1) c\rps 
  \le \bar\xi (\rho + \sigma) \overset{\text{choice of }\bar\xi} = \frac \rho 4.
\end{equation*}
Combining this with $\rho^N(c) \ge \frac \rho 2$, we find $\displaystyle \sumxi (r+1)c\rps \ge \frac \rho 4$.
On the other hand, we also have 
\begin{equation*}
  \begin{split}
    \sumxi  \indicator{\frac{u\rps}{u\rs} < \theta} (r+1)  c\rps
    &= \sumxi  \indicator{\frac{u\rps}{u\rs} < \theta} (r+s) c\rs \underbrace{\frac { (r+1) \cbar\rps u\rps}{(r+s) \cbar\rs u\rs}}_{\le 2 \monr \lambda \theta}\\
    & \le (\rho+\sigma) \cdot {2 \monr \lambda \theta}.
  \end{split}
\end{equation*}
If $\eta$ is small enough to ensure $2\monr \lambda \theta \le \frac {\rho}{8 (\rho+\sigma)}$ for all $\monr \le \eta$, we finally arrive at 
\[ \sumxi  \indicator{\frac{u\rps}{u\rs} \ge \theta} (r+1)  c\rps
=  \sumxi  (r+1)  c\rps
-  \sumxi  \indicator{\frac{u\rps}{u\rs} < \theta} (r+1)  c\rps \ge \frac \rho 8 .\]
Plugging this into \eqref{dissipationSmallMon:eq1}, we obtain $\displaystyle \Dnlin \ge \frac {\rho (1-\frac 1 \theta)\ln \theta}{8 \lambda}.$
The case of small $\mons$ is proved analogously, after setting $\hat\xi = \frac{\sigma}{4(\rho+\sigma)}$ and summing over $\frac{s+1}{r+s+1} \ge \hat\xi.$

%% file: LogSobolevInequality.tex
\subsection{Logarithmic Sobolev Inequality}\label{section:logSobolevInequality}
In this section, we want to prove and understand the inequality $\Dnlin \gtrsim H^N[c|\cbar]$. It turns out that the independence of this inequality on $N$ can be taken as a definition of subcriticality in accordance with the naming convention introduced in Section \ref{subsection:mainTheorem}.
Let us first sketch the idea behind said inequality.  Given that $(x-y)(\ln x -\ln y) \ge (\sqrt x - \sqrt y)^2$, we roughly find 
\begin{equation*}
  \Dnlin \gtrsim \sumrs[1][N] (r+s)\cbar\rs \left((\delr \sqrt{u\rs})^2 + (\dels \sqrt{u\rs})^2\right).
\end{equation*}
On the other hand we also see $\Psi(x) = x\ln x + 1 -x \approx \ln(x+2)(1-\sqrt x)^2$ and therefore
\begin{equation*}
  H^N[c|\cbar] = \sumrs[1][N] \cbar\rs \Psi(u\rs) \approx \sumrs[1][N] \cbar\rs \ln(2+u\rs) (1-\sqrt{u\rs})^2 \approx \sumrs[1][N] (r+s)\cbar\rs (1-\sqrt{u\rs})^2.
\end{equation*}
So an estimate of the form $H^N[c|\cbar] \le \frac 1 C \Dnlin$ is roughly a Poincaré inequality on $l^2$ with measure $(r+s)\cbar\rs$. In the continuous setting we can see, that such an inequality is expected if $(r+s)\cbar\rs$ decays exponentially.
    \begin{remark}\textit{(Continuous Poincaré inequality)}\label{hardy}\\
        \input{\CommonPath/remarks/hardy}
    \end{remark}
There are two obvious problems, when applying this reasoning to our situation. Firstly, we can not do the change of variables $(x+y) \mapsto (\frac {x}{x+y}, x+y)$ in the discrete setting and secondly we can not expect $\cbar$ to decay exponentially during the evolution of our solution (especially if we start with large initial masses).

\subsubsection{Abstract Result}\label{subsection:abstractResult}
This subsection is devoted to pull the calculation from Remark \ref{hardy} into the discrete setting. This yields an abstract inequality for general sequences $\cbar\rs,u\rs$, i.e. they do not have to be given by \eqref{defCbarU}.
The main hurdle is to mimic the transformation $(x,y) \mapsto (\frac{x}{x+y},x+y)$ and integrate along lines in the discrete setting. Interestingly, we will see that we need to (approximately) integrate along curved lines, with the curvature depending on the Gibbs free energy. We start with some Notation.
    \begin{notation}\textit{}\label{treeStructure}\\
        \input{\CommonPath/notations/treeStructure}
    \end{notation}
We want to leverage this graph structure to make the calculation from Remark \ref{hardy} possible in the discrete setting. As we will see below, this framework is very well adapted to handle the operations used in Remark \ref{hardy}, in particular integrating up a derivative, using the Cauchy--Schwarz inequality and changing the order of integration. In the appendix, we show how $E[\Prs]$ corresponds to integrating from $0$ to $(r,s)$, $E[\Trs]$ can be interpreted as integrating from $(r,s)$ to infinity and $\er,\es$ yields a coordinate transformation. In particular $(x,y)\mapsto (\frac x {x+y},x+y),$ equates to $\er\rs = \frac{r}{r+s}$ and $\es\rs = \frac{s}{r+s}$, which is surprisingly not the appropriate choice in general. We begin with a helpful Lemma.
    \begin{lemma}\textit{}\label{treeExpressions}\\
        \input{\CommonPath/lemmas/treeExpressions}
    \end{lemma}
\begin{proof}
    \input{\CommonPath/proofs/treeExpressions.tex}
\end{proof}
Next, we study the analogue of $\int_0^\sigma \frac{1}{\tau^2 \cbar(\xi,\tau) N(\xi,\tau)}\dd \tau = 2 N(\xi,\sigma)$ and $\int_\tau^\infty \frac{\sigma^2 \cbar(\xi,\sigma)}{M(\xi,\sigma)} \dd \sigma = 2 M(\xi,\tau).$ Note, that we have not considered the local volume deformation of the coordinate transformation yet. This is given by the $f\rs$ in the next Lemma.
    \begin{lemma}\textit{}\label{treeEstimates}\\
        \input{\CommonPath/lemmas/treeEstimates}
    \end{lemma}
\begin{proof}
    \input{\CommonPath/proofs/treeEstimates.tex}
\end{proof}
Changing the order of integration is no problem.
    \begin{lemma}\textit{(Fubini)}\label{fubini}\\
        \input{\CommonPath/lemmas/fubini}
    \end{lemma}
\begin{proof}
    \input{\CommonPath/proofs/fubini.tex}
\end{proof}
With this we have everything to prove the Poincaré inequality in the discrete setting.
    \begin{theorem}\textit{(Poincaré inequality)}\label{hardyInequality}\\
        \input{\CommonPath/theorems/hardyInequality}
    \end{theorem}
\begin{proof}
    \input{\CommonPath/proofs/hardyInequality.tex}
\end{proof}
The attentive reader will have noticed, that we have proved a Poincaré inequality with measure $\cbar\rs$ instead of the earlier proclaimed $(r+s)\cbar\rs$. Our result is just a little cleaner and more in the spirit of a logarithmic Sobolev inequality, as it is less dependent on the space dimension. If we were to go to three components and hence three dimensions, we do not want to have $(r+s)^2\cbar\rs$ as a measure. However, we now need to check that the right hand side of Theorem \ref{hardyInequality} is actually close to $\Dnlin.$
    \begin{lemma}\textit{}\label{derivativeSqrtPsi}\\
        \input{\CommonPath/lemmas/derivativeSqrtPsi}

    \end{lemma}
\begin{proof}
    \input{\CommonPath/proofs/derivativeSqrtPsi.tex}

\end{proof}
Finally, the factor $l(\max(u\rs,u\rms))$ still needs to be treated. Clearly, $l$ is a monotone function, that behaves roughly like $\ln(x)$ for large $x$. So our concern lies with large values of $u\rs$. In our situation, we have $u\rs = \frac{c\rs}{\cbar\rs} \le \frac{\rho+\sigma}{\min(\mons,\monr)^{r+s} Q\rs}$. The next Lemma tells us, that we can pull out the exponent to find a growth like $r+s$. In fact, the growth like $r+s$ is optimal, as is known in the one-dimensional case (see \cite{canizo:EDEstimate} for an example with infinite mass, which can be modified to have finite mass).
    \begin{lemma}\textit{}\label{growthL}\\
        \input{\CommonPath/lemmas/growthL}
    \end{lemma}
\begin{proof}
    \input{\CommonPath/proofs/grwothL.tex}
\end{proof}

\subsubsection{Application to \texorpdfstring{$\cbar\rs$}{cbar}}\label{subsection:application}
The previous subsection proves the desired $H^N[c|\cbar] \lesssim \Dnlin$ if we can apply Theorem \ref{hardyInequality} to $H[c|\cbar]$, i.e. if we can bound
\begin{equation}\tag*{\eqref{hardyAss}}
  \sup\limits_{r,s \in V} \mathbb{E}\left[\Prs\frac 1 {f\kl \cbar\kl}\right] f\rs \mathbb{E}[\Trs\cbar\kl ], \text{ where } \cbar\rs = \monr^r \mons^s Q\rs.
\end{equation}
To do so, we need to find appropriate values for $\er\rs,\es\rs$ and $f\rs.$ Regardless of that choice, we always have $\mathbb{E}\left[\Prs\frac 1 {f\kl \cbar\kl}\right]\ge \frac {1}{f\rs \cbar\rs}$ and $\mathbb{E}[\Trs\cbar\kl ]\ge \cbar\rs.$ So the best we can hope for is a constant $K$, such that 
\begin{equation}\label{exponentialBehaviourOfPrsTrs}
  \mathbb{E}\left[\Prs\frac 1 {f\kl \cbar\kl}\right]\overset{!}\le K \frac {1}{f\rs \cbar\rs} \text{ and } \mathbb{E}[\Trs\cbar\kl ]\overset{!}\le K \cbar\rs,
\end{equation}
which is exactly the exponential behaviour needed in Remark \ref{hardy}. By Lemma \ref{treeExpressions}.\ref{treeExpressions:derivativeTrs} we have 
\begin{equation}\label{formulaTrs}
  \begin{split}
    \mathbb{E}[\Trs \cbar\kl] &= \cbar\rs + \frac{\er\rps}{\er\rps+\es\rps} \underbrace{\mathbb{E}[T\rps \cbar\kl]}_{\eqqcolon K\rps \cbar\rps}  + \frac{\es\rsp}{\er\rsp+\es\rsp} \underbrace{\mathbb{E}[T\rsp \cbar\kl]}_{\eqqcolon K\rsp \cbar\rsp}\\
    & = \cbar\rs \left(1 + \frac{\er\rps}{\er\rps+\es\rps} K\rps \frac{\cbar\rps}{\cbar\rs}  + \frac{\es\rsp}{\er\rsp+\es\rsp} K\rsp \frac{\cbar\rsp}{\cbar\rs}\right).
  \end{split}
\end{equation}
In order to globally bound $K\rs$, it is therefore reasonable to try and bound 
\[\frac{\er\rps}{\er\rps+\es\rps} \frac{\cbar\rps}{\cbar\rs}  + \frac{\es\rsp}{\er\rsp+\es\rsp} \frac{\cbar\rsp}{\cbar\rs} \overset{!}<  1, \text{ whenever } \monr,\mons \in \interior{\Ex},\]
where $\Ex$ is the region of existence from Definition \ref{regionOfExistence}. Note, that there is no hope to bound \eqref{hardyAss} if $(\monr,\mons)\not \in \Ex$, even in the one-dimensional case.
The main inside of this subsection, is that we can essentially make $K\rs$ one-dimensional, i.e. $K\rs = K_{r+s}$ under the following concavity condition 
\begin{equation}\label{concavityAss:old}
    \frac{Q\rnp Q\kn}{Q\rn Q\knp} \ge 1 \quad\text{ and }\quad \frac{Q\rpn Q\kn}{Q\rn Q\kpn} \le 1 \quad\text{ for all } 0\le k< r < n+1.
\end{equation} 
This is a weaker version of \eqref{assumption:thermo:concave} formulated at the level of $(r,s).$ To see this, set in \eqref{assumption:thermo} the higher order terms $f\rs = g\rs =1$ which yields 
\[ \frac{Q\rnp Q\kn}{Q\rn Q\knp} = \frac{\psi\left(\frac r {n+1}\right)}{\psi\left(\frac k {n+1}\right)}\quad\text{ and }\quad
\frac{Q\rpn Q\kn}{Q\rn Q\kpn} =  \frac{\varphi\left(\frac r {n+1}\right)}{\varphi\left(\frac k {n+1}\right)}.\]
But $\psi$/$\varphi$ are monotonically increasing/decreasing if and only if $\Phi^{\prime\prime}\le 0.$
    \begin{lemma}\textit{($\er\rs,\es\rs$)}\label{eres}\\
        \input{\CommonPath/lemmas/eres}
    \end{lemma}
\begin{proof}
    \input{\CommonPath/proofs/eres.tex}
\end{proof}
Some further remarks on the assumption \eqref{concavityAss:old} and the necessity to allow such general $\er\rs,\es\rs$ can be found in the appendix.
Coming back to the bound \eqref{exponentialBehaviourOfPrsTrs}; Lemma \ref{eres} together with \eqref{formulaTrs} yields a bound on $\frac{\mathbb{E}[\Trs\cbar\kl]}{\cbar\rs}$ in terms of $m_n.$ With the right choice of $f\rs$, we can also bound $f\rs \cbar\rs \mathbb{E}[\Prs \frac{1}{f\kl \cbar\kl}]$ in terms of $m_n.$ For a given edge weighing $\er\rs,\es\rs$ we can inductively set 
\begin{equation}\label{frs}
  \frac 1 {f\rs} = 
  \begin{cases}
    \cbar\rs &\text{ if }r+s = N\\
    \er\rps \frac 1 {f\rps} + \es\rsp \frac 1 {f\rsp} &\text{ else.}
  \end{cases}
\end{equation}
If $\er\rs,\es\rs$ satisfies \eqref{eres:equation}, then $\frac{1}{f\rs\cbar\rs} = \prod\limits_{n=r+s+1}^N m_n$, which is easily seen inductively. From this, we can deduce the desired bound on $\mathbb{E}[\Prs \frac{1}{f\kl\cbar\kl}].$ In the next Lemma we combine these observations with an estimate for small values of $n$, where we can not expect \eqref{concavityAss:old} to hold, because for small values of $r+s$ the $\smallo(r+s)$ terms from \eqref{QThermoDom} are not negligible.
    \begin{lemma}\textit{(Estimates on $N\rs,M\rs$)}\label{MNEstimates}\\
        \input{\CommonPath/lemmas/MNEstimates}
    \end{lemma}
\begin{proof}
    \input{\CommonPath/proofs/MNEstimates.tex}
\end{proof}

%% file: remarks/hardy.tex
\newcommand{\intR}{\int_0^\infty}
The idea of the proof of the Poincaré inequality is best seen in the continues setup. Define 
\begin{equation*}
    N(\xi,\omega) = \sqrt{\int_0^\omega \frac {1} {\tau^2 \cbar(\xi,\tau)}\dd\tau} \text{ and } 
    M(\xi,\omega) = \sqrt{\int_\omega^\infty \tau^2 \cbar(\xi,\tau)\dd \tau}
\end{equation*}
and assume $MN\le \sqrt K$. Then we can find
\begin{equation*}
    \begin{split}
        \intR& \intR (x+y)\cbar(x,y) (\sqrt{u(x,y)}-1)^2 \dd x \dd y
         \overset{(\frac{x}{x+y},x+y)=(\xi,\omega)}{=} \int_0^1 \intR \omega\cbar(\xi,\omega) (\sqrt{u(\xi,\omega)}-1)^2 \omega\dd\omega \dd \xi  \\
         &= \int_0^1 \intR \omega^2 \cbar(\xi,\omega) \left(\int_0^\omega \partial_\tau \sqrt{u(\xi,\tau)}\dd\tau\right)^2 \dd\omega \dd \xi  \\
        \overset{\text{C.S.}}&{\le} \int_0^1 \intR \omega^2 \cbar(\xi,\omega) \left(\int_0^\omega (\partial_\tau \sqrt{u(\xi,\tau)})^2 \tau^2 \cbar(\xi,\tau) N(\xi,\tau)\dd\tau\right) 
        \int_0^\omega \underbrace{\frac{1}{\tau^2 \cbar(\xi,\tau) N(\xi,\tau)}}_{= 2 \partial_\tau N(\xi,\tau)}\dd \tau \omega\dd \dd \xi  \\
        &= 2 \int_0^1 \intR \omega^2 \cbar(\xi,\omega)  N(\xi,\omega)\left(\int_0^\omega (\partial_\tau \sqrt{u(\xi,\tau)})^2 \tau^2 \cbar(\xi,\tau) N(\xi,\tau)\dd\tau\right) \dd\omega \dd\xi\\
        \overset{\text{Fubini}}&{=} 2\int_0^1 \intR (\partial_\tau \sqrt{u(\xi,\tau)})^2 \tau^2 \cbar(\xi,\tau) N(\xi,\tau) \int_\tau^\infty \omega^2 \cbar(\xi,\omega)  \underbrace{N(\xi,\omega)}_{\le \frac{\sqrt K} M} \dd\omega \dd\tau\dd\xi \\
        &\le 2\sqrt{K}\int_0^1 \intR (\partial_\tau \sqrt{u(\xi,\tau)})^2 \tau^2 \cbar(\xi,\tau) N(\xi,\tau) \int_\tau^\infty \underbrace{\frac{\omega^2 \cbar(\xi,\omega) } {M(\xi,\omega)}}_{=-2\partial_\omega M(\xi,\omega)}\dd\omega \dd\tau\dd\xi \\
        & = 4\sqrt{K}\int_0^1 \intR (\partial_\tau \sqrt{u(\xi,\tau)})^2 \tau^2 \cbar(\xi,\tau) N(\xi,\tau) M(\xi,\tau)\dd\tau\dd\xi \\
        &\le 4K \intR\intR (\partial_\tau \sqrt{u(x,y)})^2 (x+y) \cbar(x,y) \dd x\dd y.
    \end{split}
\end{equation*}

%% file: notations/treeStructure.tex
Let $G = (V,E)$ be the directed Graph with nodes $V$ and edges $E$ given by
\begin{equation*}
  \begin{split}
    V &\coloneqq \left\{(r,s) \,\mid\, r\in\N_0,s\in \N_0, 1\le r+s\le N\right\} \text{ and}\\
    E &\coloneqq \left\{((r,s),(r+1,s)) \,\mid\, (r,s) \in V \text{ and } r+s<N\right\} \cup \left\{((r,s),(r,s+1)) \,\mid\, (r,s) \in V \text{ and } r+s<N\right\}.
  \end{split}
\end{equation*}
For some weighing of the edges $e : E \to \R_{>0}$ we write
\begin{equation*}
  \er\rs \coloneqq e((r-1,s),(r,s)) \text{ and }\es\rs \coloneqq e((r,s-1),(r,s)),
\end{equation*}
with the convention that $\er_{0,s} = 0 = \es_{r,0}$.
Furthermore, let the set of spanning trees of $G$ with $(1,0)$ and $(0,1)$ fused be denoted by $\T$.

\input{\CommonPath/pictures/GraphWithSpanningTree.tex}

Then we obtain a probability distribution on $\T$ via
\begin{equation*}
    \nu : \T \to \R_{>0}, \quad T \mapsto \prod\limits_{f\in E(T)} e(f)
\end{equation*}
and appropriate normalisation. For any random variable $X \colon \T \to \R$ we denote its expected value as $\mathbb E[X]$.
For $(r,s)\in V$ and any $T\in \T$ we denote $\Trs(T)$ as the maximal connected subgraph in $T$ with root $(r,s)$ and $\Prs(T)$ as the subgraph of $T$ containing only the unique path from the root of $T$ to (r,s).

For a function $h : E \to \R$, we then define the random variables
\begin{equation*}
        \Prs h: \T \to \R, \quad T\mapsto \sum\limits_{ f \in \Prs(T)} h(f),
\end{equation*}
where $P^{1,0}h (T) = 0 = P^{0,1}h(T)$. For a function $g : V \to \R$ we define 
\begin{equation*}
    \begin{split}
        &\Trs g: \T \to \R, \quad T\mapsto \sum\limits_{(k,l) \in \Trs(T)} g\kl \quad\text{ and denote }\\
        &\Prs g\kl: \T \to \R, \quad T\mapsto \sum\limits_{f = (v,w) \in \Prs(T)} g_w.
    \end{split}
\end{equation*}
For any function $h : E \to \R$ we denote 
\begin{equation*}
    h\rs \coloneqq h((r-1,s),(r,s)) \quad\text{ and }\quad \dot{h}\rs \coloneqq h((r,s-1),(r,s)).
\end{equation*}
Lastly, for a sequence $g\rs$, we define 
\begin{equation*}
    \del g : E(G) \to \R, \quad f=(v,w)\mapsto g_w - g_v. 
\end{equation*}

%% file: pictures/GraphWithSpanningTree.tex
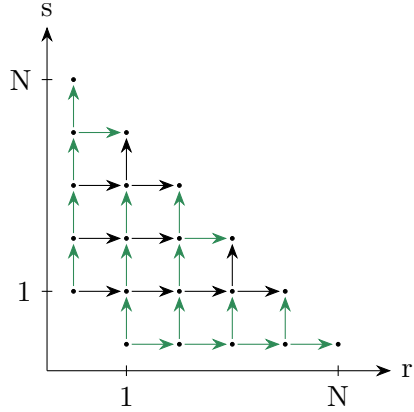
\begin{figure}[ht]
  \centering
  \begin{tikzpicture}[scale=0.7,>={Stealth[scale=1.2]}]
    \draw[black, ->] (-0.5,-0.5) -- (-0.5,6) node[above]{s};
    \draw[black, ->] (-0.5,-0.5) -- (6,-0.5) node[right]{r};
    \draw[black] (-0.4,5) -- (-0.6,5) node[left]{N};
    \draw[black] (-0.4,1) -- (-0.6,1) node[left]{1};
    \draw[black] (5,-0.4) -- (5,-0.6) node[below]{N};
    \draw[black] (1,-0.4) -- (1,-0.6) node[below]{1};
    \foreach \r in {1,...,4}{
      \pgfmathsetmacro{\z}{5-\r}
      \foreach \s in {1,...,\z}{
        \ifthenelse{\r < 3}{
          \filldraw[black] (\r,\s) circle (1pt) node {};}{
          \filldraw[black] (\r,\s) circle (1pt) node {};}}
    }
    \foreach \r in {1,...,5}{
      \filldraw[black] (\r,0) circle (1pt) node {};
      \filldraw[black] (0,\r) circle (1pt) node {};
    }
    \foreach \r in {1,...,3}{
      \pgfmathsetmacro{\z}{4-\r}
      \foreach \s in {1,...,\z}{
        \ifthenelse{\r < 3}{
          \ifthenelse{\s <3}{
            \ifthenelse{\r =2 }{
              \ifthenelse{\s>1}{
                \draw[SeaGreen,->] (\r+0.1,\s) -- (\r+0.9,\s);
                \draw[SeaGreen,->] (\r,\s+0.1) -- (\r,\s+0.9);
              }{
                \draw[black,->] (\r+0.1,\s) -- (\r+0.9,\s);
                \draw[SeaGreen,->] (\r,\s+0.1) -- (\r,\s+0.9);
              }
            }{
              \draw[black,->] (\r+0.1,\s) -- (\r+0.9,\s);
              \draw[SeaGreen,->] (\r,\s+0.1) -- (\r,\s+0.9);
            }
          }{
            \ifthenelse{\s = 3}{
              \draw[black,->] (\r+0.1,\s) -- (\r+0.9,\s);
              \draw[black,->] (\r,\s+0.1) -- (\r,\s+0.9);
            }{
              \draw[SeaGreen,->] (\r+0.1,\s) -- (\r+0.9,\s);
              \draw[black,->] (\r,\s+0.1) -- (\r,\s+0.9);
            }
          }
        }
        {
          \ifthenelse{\s>1}{
            \ifthenelse{\r = 3}{
              \ifthenelse{\s = 2}{
                \draw[SeaGreen,->] (\r+0.1,\s) -- (\r+0.9,\s);
                \draw[SeaGreen,->] (\r,\s+0.1) -- (\r,\s+0.9);
              }{
                \draw[SeaGreen,->] (\r+0.1,\s) -- (\r+0.9,\s);
                \draw[black,->] (\r,\s+0.1) -- (\r,\s+0.9);
              }
            }{
              \draw[SeaGreen,->] (\r+0.1,\s) -- (\r+0.9,\s);
              \draw[black,->] (\r,\s+0.1) -- (\r,\s+0.9);
            }
          }{
            \draw[black,->] (\r+0.1,\s) -- (\r+0.9,\s);
            \draw[black,->] (\r,\s+0.1) -- (\r,\s+0.9);
          }
        }
      }
    }
    \foreach \r in {1,...,4}{
      \draw[SeaGreen,->] (\r+0.1,0) -- (\r+0.9,0);
      \draw[SeaGreen,->] (\r,0.1) -- (\r,0.9);
      \draw[SeaGreen,->] (0,\r+0.1) -- (0,\r+0.9);
      \ifthenelse{\r < 4}{
        \draw[black,->] (0.1,\r) -- (0.9,\r);
      }{
        \draw[SeaGreen,->] (0.1,\r) -- (0.9,\r);
      }

    }
  \end{tikzpicture}
\caption{A picture of $G$ with a spanning tree $T\in\T$ coloured in green.}
\end{figure}

%% file: lemmas/treeExpressions.tex
\newcommand{\eer}{\frac{\er\rs}{\er\rs+\es\rs}}
\newcommand{\eerp}{\frac{\er\rps}{\er\rps+\es\rps}}
\newcommand{\ees}{\frac{\es\rs}{\er\rs+\es\rs}}
\newcommand{\eesp}{\frac{\es\rsp}{\er\rsp+\es\rsp}}
\newcommand{\eek}{\frac{\er\kl}{\er\kl+\es\kl}}
\newcommand{\eel}{\frac{\es\kl}{\er\kl+\es\kl}}
For $h\colon E(G) \to \R$, we define 
\[ h_{\delta^-(r,s)} \colon \T \to \R, \quad T\mapsto h(f), \text{ where }f\in T \text{ has the end node }(r,s).\]
For $g : V \to \R$  and $h : E(G) \to \R$ the following equations hold for $(r,s)\in V$ with the convention $0\cdot \text{undefined} = 0$
\begin{enumerate}
  \item \label{treeExpressions:derivativePrs}
    \begin{equation*}
      \mathbb{E}[h_{\delta^-(r,s)}] = \eer \delr \mathbb{E}[\Prs h] +\ees \dels \mathbb{E}[\Prs h],
    \end{equation*}
  \item \label{treeExpressions:hdelta}
    \begin{equation*}
      \mathbb{E}[h_{\delta^-(r,s)}] = \eer h\rs +\ees \dot h\rs,
    \end{equation*}
  \item \label{treeExpressions:PrsExpandh}
    \begin{equation*}
        \mathbb{E}[\Prs h] = \mathbb{E}\left[\Prs \left(\eek h\kl + \eel \dot{h}\kl\right)\right],
    \end{equation*}
  \item \label{treeExpressions:derivativeTrs}
    \begin{equation*}
        g\rs = \mathbb{E}[\Trs g\kl] - \eerp \mathbb{E}[T\rps g\kl] - \eesp \mathbb{E}[T\rsp g\kl]
    \end{equation*}
    and
  \item  \label{treeExpressions:TrsExpandg}
    \begin{equation*}
        g\rs = \mathbb{E}[\Trs (g\kl - \eerp g\kpl - \eesp  g\klp)].
    \end{equation*}
\end{enumerate}

%% file: proofs/treeExpressions.tex
\newcommand{\eer}{\frac{\er\rs}{\er\rs+\es\rs}}
\newcommand{\eerp}{\frac{\er\rps}{\er\rps+\es\rps}}
\newcommand{\ees}{\frac{\es\rs}{\er\rs+\es\rs}}
\newcommand{\eesp}{\frac{\es\rsp}{\er\rsp+\es\rsp}}
\newcommand{\eek}{\frac{\er\kl}{\er\kl+\es\kl}}
\newcommand{\eel}{\frac{\es\kl}{\er\kl+\es\kl}}
In this proof we will use the notation $\er\rs \in T$ to denote $((r-1,s),(r,s))\in T$ and $\es\rs\in T$ to denote $((r,s-1),(r,s))\in T.$\\
We see, that \ref{treeExpressions:derivativePrs} is true if either $\er\rs=0$ or $\es\rs=0$, because then $(r,s)$ has only one incoming edge. So any path to $(r,s)$ needs to have that edge. If $\er\rs\neq0\neq \es\rs$ we calculate
\begin{equation*}
  \begin{split}
    &\mathbb{E}[\Prs h] = \sum\limits_{T\in\T} \sum\limits_{f\in \Prs} h(f) \frac{\nu(T)}{\sum \nu(T)}\\
    &= \sum\limits_{\substack{T\in \T \\ ((r-1,s),(r,s))\in T}} \left(h\rs + \sum\limits_{f\in P^{r-1,s}} h(f)\right) \frac{\nu(T)}{\sum \nu(T)}
     + \sum\limits_{\substack{T\in \T \\ ((r,s-1),(r,s))\in T}} \left(\dot{h}\rs + \sum\limits_{f\in P^{r,s-1}} h(f)\right) \frac{\nu(T)}{\sum \nu(T)}\\
     & = \mathbb{E}[h_{\delta^-(r,s)}] + \eer \sum\limits_{\substack{T\in \T \\ \er\rs\in T}} \sum\limits_{f\in P^{r-1,s}} h(f) \frac{\nu(T)}{\sum \nu(T)}  + \ees \sum\limits_{\substack{T\in \T \\ \er\rs \in T}} \sum\limits_{f\in P^{r-1,s}} h(f) \frac{\nu(T)}{\sum \nu(T)}\\
     & \qquad\qquad+ \eer \underbrace{\sum\limits_{\substack{T\in \T \\ \es\rs\in T}} \sum\limits_{f\in P^{r,s-1}} h(f) \frac{\nu(T)}{\sum \nu(T)}}_{ = (*)}
     + \ees \sum\limits_{\substack{T\in \T \\ \es\rs \in T}} \sum\limits_{f\in P^{r,s-1}} h(f) \frac{\nu(T)}{\sum \nu(T)}.
  \end{split}
\end{equation*}
If we denote for a given $T\in\T$ with $\delta^-(r,s)$ the unique edge ending in $(r,s)$ and $\nu(T\setminus \delta^-(r,s)) \coloneqq \prod_{\substack{f\in E(T)\\f\neq \delta^-(r,s)}} e(f),$ we can determine $(*)$ as
\begin{equation}
  \begin{split}
    (*)&= \es\rs \sum\limits_{\substack{T\in \T \\ \es\rs\in T}} \underbrace{\sum\limits_{f\in P^{r,s-1}} h(f) \frac{\nu(T \setminus \delta^-(r,s))}{\sum \nu(T)}}_{\text{indep. of $\delta^-(r,s)$}} \overset{\text{swap }\es\rs \to\er\rs\text{ in }T}{=}  \es\rs \sum\limits_{\substack{T\in \T \\ \er\rs\in T}} \sum\limits_{f\in P^{r,s-1}} h(f) \frac{\nu(T \setminus \delta^-(r,s))}{\sum \nu(T)}\\
      &=\frac{\es\rs}{\er\rs} \sum\limits_{\substack{T\in \T \\ \er\rs\in T}} \sum\limits_{f\in P^{r,s-1}} h(f) \frac{\nu(T)}{\sum \nu(T)}
  \end{split}
\end{equation}
and hence with the analogous statement for the terms involving $P^{r-1,s}$
\begin{equation*}
    \mathbb{E}[\Prs h] = \mathbb{E}[h_{\delta^-(r,s)}] + \eer \sum\limits_{T\in \T} \sum\limits_{f\in P^{r-1,s}} h(f) \frac{\nu(T)}{\sum \nu(T)}  + \ees \sum\limits_{T \in \T } \sum\limits_{f\in P^{r,s-1}} h(f) \frac{\nu(T)}{\sum \nu(T)},
\end{equation*}
which is exactly \ref{treeExpressions:derivativePrs}. Now \ref{treeExpressions:hdelta} follows from 
\begin{equation*}
    \begin{split}
      &\mathbb{E}[h_{\delta^-(r,s)}] = \sum\limits_{\substack{ T \in \T \\\er\rs\in T}}h\rs \frac{\nu(T)}{\sum\nu(T)}
        +  \sum\limits_{\substack{T \in \T \\\es\rs\in T}}\dot{h}\rs \frac{\nu(T)}{\sum\nu(T)}\\
        =& \eer h\rs (\er\rs + \es\rs) \sum\limits_{\substack{ T \in \T \\\er\rs\in T}} \frac{\nu(T \setminus \delta^-(r,s))}{\sum \nu(T)}
        + \ees \dot{h}\rs (\er\rs + \es\rs) \sum\limits_{\substack{ T \in \T \\\es\rs\in T}} \frac{\nu(T \setminus \delta^-(r,s))}{\sum \nu(T)}\\
    \end{split}
\end{equation*}
once we realise 
\begin{equation*}
    (\er\rs + \es\rs) \sum\limits_{\substack{ T \in \T \\\er\rs\in T}} \frac{\nu(T \setminus \delta^-(r,s))}{\sum \nu(T)}= 
     \sum\limits_{\substack{ T \in \T \\\er\rs\in T}} \frac{\nu(T)}{\sum \nu(T)} +\sum\limits_{\substack{ T \in \T \\\es\rs\in T}} \frac{\nu(T)}{\sum \nu(T)}= 1
\end{equation*}
and the analogous statement with $\es\rs\in T$.
Next, we can prove \ref{treeExpressions:PrsExpandh} inductively from $r+s-1$ to $r+s$. The induction start is true by our definition $P^{1,0}h = 0 =P^{0,1}h$ for any $h$ and the induction step is
\begin{equation*}
    \begin{split}
      \mathbb{E}[\Prs h] \overset{\ref{treeExpressions:derivativePrs} \text{ and }\ref{treeExpressions:hdelta}}&{=} \eer h\rs + \ees \dot{h}\rs + \eer \mathbb{E}[P^{r-1,s} h] + \ees \mathbb{E}[P^{r,s-1} h ] \\
        \overset{\text{ind.}}&{=}\eer h\rs + \ees \dot{h}\rs + \eer \mathbb{E}\left[P^{r-1,s} \left(\eek h\kl + \eel \dot{h}\kl\right)\right] \\
        &\qquad\qquad\qquad\qquad\qquad+ \ees \mathbb{E}\left[P^{r,s-1} \left(\eek h\kl + \eel \dot{h}\kl\right)\right] \\
        \overset{\ref{treeExpressions:derivativePrs}}&{=}\mathbb{E}\left[\Prs \left(\eek h\kl + \eel \dot{h}\kl\right)\right].
    \end{split}
\end{equation*}
For \ref{treeExpressions:derivativeTrs} we abbreviate $T\rps g = T\rps g(T)$ and $T\rsp g = T\rsp g(T)$ and see 
\begin{equation*}
    \begin{split}
        \mathbb{E}[\Trs g\kl] &= \sum\limits_{T\in\T} \sum\limits_{k,l \in \Trs} g\kl \frac{\nu(T)}{\sum\nu(T)}\\
        &= g\rs + \sum\limits_{\substack{T \in \T \\\er\rps\in T}} T\rps g  \frac{\nu(T)}{\sum \nu(T)} + \sum\limits_{\substack{T \in \T \\\es\rsp\in T}} T\rsp g  \frac{\nu(T)}{\sum \nu(T)}\\
        &= g\rs + \sum\limits_{\substack{T \in \T \\\er\rps\in T}} \er\rps T\rps g  \frac{\nu(T\setminus \delta^-(r+1,s))}{\sum \nu(T)} + \sum\limits_{\substack{T \in \T \\\es\rsp\in T}}\es\rsp T\rsp g  \frac{\nu(T\setminus \delta^-(r,s+1))}{\sum \nu(T)}\\
        \overset{\text{as above}}&{=}g\rs + \eerp \sum\limits_{T \in \T } T\rps g  \frac{\nu(T)}{\sum \nu(T)} + \eesp \sum\limits_{T \in \T} T\rsp g  \frac{\nu(T)}{\sum \nu(T)}.
    \end{split}
\end{equation*}
Lastly, we can show \ref{treeExpressions:TrsExpandg} again inductively over the distance to the boundary (i.e. from $r+s$ to $r+s-1$). Let us abbreviate $\mathbb{E}[T\rs (g\kl - \eerp g\kpl - \eesp  g\klp)] \eqqcolon G\rs$ and calculate
\begin{equation*}
    \begin{split}
      g\rs &-\eerp g\rps - \eesp g\rsp\\
        \overset{\text{\ref{treeExpressions:derivativeTrs}}}&{=} G\rs
        - \eerp G\rps - \eesp G\rsp\\
        \overset{\text{ind.}}&{=} G\rs -\eerp g\rps - \eesp g\rsp,
    \end{split}
\end{equation*}
which shows the desired result.

%% file: lemmas/treeEstimates.tex
\newcommand{\eer}{\frac{\er\rs}{\er\rs+\es\rs}}
\newcommand{\eerp}{\frac{\er\rps}{\er\rps+\es\rps}}
\newcommand{\ees}{\frac{\es\rs}{\er\rs+\es\rs}}
\newcommand{\eesp}{\frac{\es\rsp}{\er\rsp+\es\rsp}}
\newcommand{\eek}{\frac{\er\kl}{\er\kl+\es\kl}}
\newcommand{\eel}{\frac{\es\kl}{\er\kl+\es\kl}}
Let $0<(f\rs)_{(r,s)\in V}$ solve 
\begin{equation*}
    \delr\left(\eerp \frac 1 {f\rps}\right) + \dels\left(\eesp \frac 1 {f\rsp}\right) = 0
    \text{ for all } 1 < r+s < N.
\end{equation*}
and define for some sequence $g\rs : V \to \R_{>0}$
\begin{equation*}
    \begin{split}
        N\rs \coloneqq \sqrt{\mathbb{E}\left[\Prs g\kl\right]} \quad\text{ and}\quad
        M\rs \coloneqq \sqrt{f\rs \mathbb{E}\left[\Trs g\kl\right]}.
    \end{split}
\end{equation*}
Then, the estimates 
\begin{equation*}
    \begin{split}
        \mathbb{E}\left[\Prs\frac{g\kl}{N\kl}\right] \le 2 N\rs  \quad\text{ and }\quad
        \mathbb{E}\left[\Trs \frac{g\kl}{M\kl}\right] \le 2 \frac{M\rs}{f\rs} \quad\text{ hold true.}
    \end{split}
\end{equation*}

%% file: proofs/treeEstimates.tex
\newcommand{\eer}{\frac{\er\rs}{\er\rs+\es\rs}}
\newcommand{\eerp}{\frac{\er\rps}{\er\rps+\es\rps}}
\newcommand{\ees}{\frac{\es\rs}{\er\rs+\es\rs}}
\newcommand{\eesp}{\frac{\es\rsp}{\er\rsp+\es\rsp}}
\newcommand{\eek}{\frac{\er\kl}{\er\kl+\es\kl}}
\newcommand{\eekp}{\frac{\er\kpl}{\er\kpl+\es\kpl}}
\newcommand{\eel}{\frac{\es\kl}{\er\kl+\es\kl}}
\newcommand{\eelp}{\frac{\es\klp}{\er\klp+\es\klp}}
We have
\begin{equation*}
    \begin{split}
        \mathbb{E}\left[\Prs\frac{g\kl}{N\kl}\right] 
        \overset{\text{Lemma \ref{treeExpressions}.\ref{treeExpressions:derivativePrs}}}&{=} \mathbb{E}\left[\Prs\frac{\eek \delk \mathbb{E}[\Pkl g\ii] +\eel \dell \mathbb{E}[\Pkl g\ii]}{N\kl}\right] \\
        \overset{\frac{x-y}{\sqrt x} \le 2(\sqrt x - \sqrt y)}&{\le} 2 \mathbb{E}\left[\Prs \frac{\er\kl}{\er\kl+\es\kl}\delk N\kl + \frac{\es\kl}{\er\kl+\es\kl}\dell N\kl \right] \\
        \overset{\text{Lemma \ref{treeExpressions}.\ref{treeExpressions:PrsExpandh}}}&{=} 2 \mathbb{E}[\Prs \del N\kl] \overset{N_{1,0} = N_{0,1} = 0}{=} 2 N\kl.
    \end{split}
\end{equation*}
Similarly, we obtain for $k+l < N$
\begin{equation*}
    \begin{split}
        \frac{g\kl}{M\kl} 
        \overset{\text{Lemma \ref{treeExpressions}.\ref{treeExpressions:derivativeTrs}}}&{=} \frac{ \mathbb{E}[T\kl g\ii] - \eekp \mathbb{E}[T\kpl g\ii] - \eelp \mathbb{E}[T\klp g\ii]}{M\kl}\\
        \overset{\text{def }f}&{=} \frac{\eekp \frac 1 {f\kpl} (-\delk (f\kpl \mathbb{E}[T\kpl g\ii ]) + \eelp \frac 1 {f\klp} (-\dell f\klp \mathbb{E}[T\klp g\ii])}{M\kl}\\
        \overset{\frac{x-y}{\sqrt x} \le 2(\sqrt x - \sqrt y)}&{\le} 2 \left(\eekp \frac 1 {f\kpl} (-\delk M\kpl) + \eelp \frac 1 {f\klp} (-\dell M\klp)\right)\\
        \overset{\text{def }f}&{=} 2\left(\frac{M\kl}{f\kl} - \eekp \frac{M\kpl}{f\kpl} - \eelp \frac{M\klp}{f\klp}\right).
    \end{split}
\end{equation*}
And for $k+l = N$, we have 
\[\frac {g\kl}{M\kl} = \frac {g\kl}{\sqrt{f\kl g\kl}}= \frac{M\kl}{f\kl} \le 2 \frac{M\kl}{f\kl} = 2\left(\frac{M\kl}{f\kl} - \eekp \frac{M\kpl}{f\kpl} - \eelp \frac{M\klp}{f\klp}\right).\]
Now we take $\mathbb{E}[\Trs]$ to obtain
\begin{equation*}
    \begin{split}
        \mathbb{E}\left[\Trs \frac{g\kl}{M\kl}\right] &\le 2 \mathbb{E}\left[\Trs \left(\frac{M\kl}{f\kl} - \eekp \frac{M\kpl}{f\kpl} - \eelp \frac{M\klp}{f\klp}\right)\right]\\
        \overset{\text{Lemma \ref{treeExpressions}.\ref{treeExpressions:TrsExpandg}}}&{=}2 \frac{M\rs}{f\rs}. \qedhere
    \end{split}
\end{equation*}

%% file: lemmas/fubini.tex
\newcommand{\eer}{\frac{\er\rs}{\er\rs+\es\rs}}
\newcommand{\eerp}{\frac{\er\rps}{\er\rps+\es\rps}}
\newcommand{\ees}{\frac{\es\rs}{\er\rs+\es\rs}}
\newcommand{\eesp}{\frac{\es\rsp}{\er\rsp+\es\rsp}}
\newcommand{\eek}{\frac{\er\kl}{\er\kl+\es\kl}}
\newcommand{\eel}{\frac{\es\kl}{\er\kl+\es\kl}}
For $g : V \to \R$  and $h : E(G) \to \R$ we have
\begin{equation*}
    \begin{split}
      \sumrs[2][N] \mathbb{E}[\Prs h] g\rs 
      = \sumrs[2][N] \left(\eer h\rs + \ees \dot{h}\rs\right) \mathbb{E}[\Trs g\kl].
    \end{split}
\end{equation*}

%% file: proofs/fubini.tex
\newcommand{\eer}{\frac{\er\rs}{\er\rs+\es\rs}}
\newcommand{\eerp}{\frac{\er\rps}{\er\rps+\es\rps}}
\newcommand{\ees}{\frac{\es\rs}{\er\rs+\es\rs}}
\newcommand{\eesp}{\frac{\es\rsp}{\er\rsp+\es\rsp}}
\newcommand{\eek}{\frac{\er\kl}{\er\kl+\es\kl}}
\newcommand{\eel}{\frac{\es\kl}{\er\kl+\es\kl}}
\newcommand{\nuprob}{\frac{\nu(T)}{\sum \nu(T)}}
We denote $\indicator{\er\kl} \coloneqq \indicator{((k-1,l),(k,l)) \in T}$, $\indicator{\es\kl} \coloneqq \indicator{((k,l-1),(k,l)) \in T}$ and calculate
\begin{equation}
  \begin{split}
    \sumrs[2][N]&\mathbb{E}[\Prs h] g\rs 
    \overset{\text{Def $\Prs$}}{=} \sumrs[2][N] \sum\limits_{T\in \T}\sum\limits_{f\in \Prs(T)} h(f) g\rs \nuprob\\
    &= \sum\limits_{T\in\T} \sumrs[2][N] \sum\limits_{f\in\Prs(T)} h(f)g\rs \nuprob \eqqcolon (*)
\end{split}
\end{equation}
For a fixed spanning tree, we can change the order of summation, because a fixed edge $(u,(k,l))$ appears exactly in all paths $\Prs$ for all $(r,s) \in T\kl.$ Therefore, we obtain
\begin{equation}
  \begin{split}
    (*)&= \sum\limits_{T\in\T} \sum\limits_{k+l=2}^{N} \left(\indicator{\er\kl} h\kl \sum\limits_{r,s\in T\kl(T)}g\rs\right)\nuprob + \left(\indicator{\es\kl}\dot{h}\kl \sum\limits_{r,s\in T\kl(T)}g\rs\right) \nuprob\\
    &= \sum\limits_{k+l=2}^{N} h\kl \sum\limits_{T\in\T} \left(\indicator{\er\kl} \sum\limits_{r,s\in T\kl(T)}g\rs\right)\nuprob + \dot{h}\kl \sum\limits_{T\in\T} \left(\indicator{\es\kl} \sum\limits_{r,s\in T\kl(T)}g\rs\right)\nuprob\\
    &= \sum\limits_{k+l=2}^{N} h\kl \underbrace{\mathbb{E}[\indicator{\er\kl}T\kl g\rs]}_{\overset{\text{indep.}}{=} \mathbb{E}[\indicator{\er\kl}]\mathbb{E}[T\kl g\rs]} + \dot{h}\kl \mathbb{E}[\indicator{\es\kl}T\kl g\rs] \\
    &= \sumrs[2][N] \left(\eer h\rs + \ees \dot h\rs\right) \mathbb{E}[\Trs g\kl],
  \end{split}
\end{equation}
where we used $\mathbb{E}[\indicator{\er\kl}] = \eek$ and $\mathbb{E}[\indicator{\es\kl}] = \eel$ in the last line, which follows from Lemma \ref{treeExpressions}.\ref{treeExpressions:hdelta}.

%% file: theorems/hardyInequality.tex
For any positive sequences $\cbar\rs$  we define 
\begin{equation}\label{hardyAss}
    \sup\limits_{r,s \in V} \underbrace{\mathbb{E}\left[\Prs\frac 1 {f\kl \cbar\kl}\right]}_{\eqqcolon N\rs^2} \underbrace{f\rs \mathbb{E}[\Trs\cbar\kl ]}_{\eqqcolon M\rs^2} \eqqcolon K.
\end{equation}
Then, for any $u\rs$ with $u_{1,0} = 1 = u_{0,1}$ the following estimate holds true
\begin{equation*}
    \sumrs[2][N] \cbar\rs \Psi(u\rs) \le 4K \sumrs[2][N] \cbar\rs \left(\frac{\er\rs}{\er\rs+\es\rs} \Bigl(\delr \sqrt{\Psi(u\rs)}\Bigr)^2 + \frac{\es\rs}{\er\rs+\es\rs} \Bigl(\dels \sqrt{\Psi(u\rs)}\Bigr)^2\right).
\end{equation*}

%% file: proofs/hardyInequality.tex
Since $\Psi$ is non negative and $u_{1,0} = 1 = u_{0,1},$ we find 
\begin{equation}
  \begin{split}
    \sumrs[2][N] \cbar\rs \Psi(u\rs) = \sumrs[2][N] \cbar\rs \left(\sqrt{\Psi(u\rs)}\right)^2
    \overset{u_{1,0} = u_{0,1} = 1}{=} \sumrs[2][N] \cbar\rs \mathbb{E}[\underbrace{\Prs\del \sqrt{\Psi(u)}}_{= \sqrt{\Psi(u\rs)}}]^2\eqqcolon (*).
  \end{split}
\end{equation}
Now, $\mathbb{E}[\Prs]$ is just the double sum $\sum\limits_{T\in\T}\sum\limits_{f\in\Prs(T)}$, so we can use Cauchy--Schwarz to bound
\begin{equation}
  \begin{split}
    (*)\overset{\text{C.S.}}&{\le} \sumrs[2][N] \cbar\rs  \mathbb{E}[\Prs(\del \sqrt{\Psi(u)})^2 \cbar\kl f\kl N\kl]\mathbb{E}\left[\Prs\frac{1}{f\kl\cbar\kl N\kl}\right]\\
    \overset{\text{Lemma \ref{treeEstimates}}}&{\le} 2 \sumrs[2][N] \cbar\rs  \mathbb{E}[\Prs(\del \sqrt{\Psi(u)})^2 \cbar\kl f\kl N\kl] N\rs\\
    \overset{\text{Lemma \ref{fubini}}}&{=} 2\sumrs[2][N] \cbar\rs \left(\frac{\er\rs}{\er\rs+\es\rs} \Bigl(\delr \sqrt{\Psi(u\rs)}\Bigr)^2 + \frac{\es\rs}{\er\rs+\es\rs} \Bigl(\dels \sqrt{\Psi(u\rs)}\Bigr)^2\right) f\rs N\rs \mathbb{E}[\Trs \cbar\kl  N\kl]\\
    \overset{\text{Ass.}}&{\le} 2\sqrt K \sumrs[2][N] \cbar\rs \left(\frac{\er\rs}{\er\rs+\es\rs} \Bigl(\delr \sqrt{\Psi(u\rs)}\Bigr)^2 + \frac{\es\rs}{\er\rs+\es\rs} \Bigl(\dels \sqrt{\Psi(u\rs)}\Bigr)^2\right) f\rs N\rs \mathbb{E}[\Trs \frac{\cbar\kl }{M\kl}].\\
    \overset{\text{Lemma \ref{treeEstimates}}}&{\le} 4\sqrt K \sumrs[2][N] \cbar\rs \left(\frac{\er\rs}{\er\rs+\es\rs} \Bigl(\delr \sqrt{\Psi(u\rs)}\Bigr)^2 + \frac{\es\rs}{\er\rs+\es\rs} \Bigl(\dels \sqrt{\Psi(u\rs)}\Bigr)^2\right) f\rs N\rs \frac{M\rs}{f\rs}\\
    \overset{\text{Ass.}}&{\le} 4K \sumrs[2][N] \cbar\rs \left(\frac{\er\rs}{\er\rs+\es\rs} \Bigl(\delr \sqrt{\Psi(u\rs)}\Bigr)^2 + \frac{\es\rs}{\er\rs+\es\rs} \Bigl(\dels \sqrt{\Psi(u\rs)}\Bigr)^2\right).\qedhere
  \end{split} 
\end{equation}

%% file: lemmas/derivativeSqrtPsi.tex
Define the function 
\[ l \colon [0,\infty) \to [0,\infty), \quad l(x) \coloneqq \frac {x \ln^2(x)}{4\Psi(x)}.\]
and let $0<x,y<\infty$ be arbitrary. Then, we have
\begin{equation*}
  \left(\sqrt{\Psi(x)} - \sqrt{\Psi(y)}\right)^2 \le (x-y) \ln\left(\frac x y \right) l(\max(x,y)).
\end{equation*}

%% file: proofs/derivativeSqrtPsi.tex
Let us start by showing that $l(x)$ is strictly increasing. We have 
\begin{equation*}
  l^\prime(x) = \frac{\ln(x)}{4 \Psi^2(x)} \Bigl((\ln(x) +2)\Psi(x) - x\ln^2(x)\Bigr)
\end{equation*}
and 
\[{\odv{}{x} \Bigl((\ln(x) +2)\Psi(x) - x\ln^2(x)\Bigr)= \frac{\Psi(x)}{x}}\ge 0 \text{ and } = 0 \iff x=1.\]
So $\Bigl((\ln(x) +2)\Psi(x) - x\ln^2(x)\Bigr)$ is strictly monotone and has a zero at $x=1$. Together with $l^\prime(1) = \frac 1 6$ we get $l^\prime(x) > 0$.

Now because both sides of the desired inequality are symmetric in $x,y$, we assume without loss of generality $y\le x$. If $x=y$ then both sides are zero. Otherwise, since $\sqrt{\Psi(x)}$ is $C^1$ in $(0,1)\cup (1,\infty)$, we get
\begin{equation*}
  \begin{split}
     \left(\sqrt{\Psi(x)} - \sqrt{\Psi(y)}\right)^2  &=
   { \left(\int\limits_y^x \frac {\ln(s)}{2\sqrt{\Psi(s)}}\dd s\right)^2 }  \overset{\text{C.S}} \le (x-y) \int\limits_y^x \frac {\ln^2(s)}{4{\Psi(s)}}\dd s \\
                                                                                                &= (x-y) \int\limits_y^x l(s) \frac 1 s \dd s 
    \overset{l\text{ is monotone}}\le  (x-y)\left(\int\limits_y^x  \frac 1 s \dd s\right) l(\max(x,y))\\
                                                                                                &= (x-y)\ln\left(\frac x y \right) l(\max(x,y)).\qedhere
  \end{split}
\end{equation*}

%% file: lemmas/growthL.tex
For all $x\in[0,\infty)$ and all $y\ge 1$ 
\[ l(x) \le y l\left(x^{\frac 1 y}\right)\]
holds true.

%% file: proofs/grwothL.tex
It suffices to show $\odv{}y y l\left(x^{\frac 1 y}\right)\ge 0$. To this end, we calculate
\begin{equation*}
  \begin{split}
    \odv{}y y l\left(x^{\frac 1 y}\right)
    &= l\left(x^{\frac 1 y}\right) + y l^\prime\left(x^{\frac 1 y}\right)x^{\frac 1 y} \ln(x)\frac{-1}{y^2}\\
    \overset{z \coloneqq x^{\frac 1 y}} &= l(z) - z \ln(z) l^\prime(z)
  = \frac{l(z)}{\Psi(z)}z \Psi\left(\frac 1 z\right) \ge 0.\qedhere
  \end{split}
\end{equation*}

%% file: lemmas/eres.tex
Fix $\monr,\mons>0$ and $n\in \N_{\ge 2}$ and assume \eqref{concavityAss:old}.
Then, there exists a unique edge weighing $\er\rs$ with $\es\rs = 1 -\er\rs$ for $r+s=n+1$ satisfying 
\begin{equation}\label{eres:equation}
  \er\rps \frac{\cbar\rps}{\cbar\rs} + \es\rsp \frac{\cbar\rsp}{\cbar\rs} = m_{n+1} \text{ for all }r+s =n.
\end{equation}
Furthermore, we have $\er\rps,\es\rsp >0$ for all $r+s=n$, as well as 
$\displaystyle m_{n+1} = \frac{\sumrs[n+1][]\cbar\rs}{\sumrs[n][]\cbar\rs}.$

%% file: proofs/eres.tex
By definition we have $\er_{0,n+1} = 0$, so we can solve \eqref{eres:equation} inductively from $r$ to $r+1$ and obtain for $r+s = n+1$
\begin{equation}\label{proof:eres}
  \er\rs \coloneqq \frac{ \sumbar - m_{r+s} \sumbar[r][r+s-1]}{\cbar\rs}.
\end{equation}
Next, we can determine $m_n,$ by summing \eqref{eres:equation} over all $r+s = n$ to find
\[ m_{n+1} \sumrs[n][]\cbar\rs = \sumrs[n][] \er\rps \cbar\rps +\es\rsp \cbar\rsp \overset{\er_{0,n+1} = 0 = \es_{n+1,0}}= \sumrs[n+1][] \underbrace{(\er\rs+\es\rs)}_{=1}\cbar\rs.\]
Putting $(r,s) = (n+1,0)$ in \eqref{proof:eres} yields $\er_{n+1,0} = 1$ and hence $\es_{n+1,0} = 0$, so it remains to show $\er\rps,\es\rsp>0$ for $r+s = n$. We use the form of $m_n$ to write
\begin{equation*}
    \begin{split}
      \er\rps 
        &= \frac{1}{\cbar\rps \sumbar[0][r+s]} \Bigg(\sum\limits_{k=r+1}^{r+s+1} \sum\limits_{l=0}^{r}\underbrace{\cbar_{k,r+s+1-k}\cbar_{l,r+s-l}}_{\substack{\ge \cbar_{k,r+s-k} \cbar_{l,r+s+1-l} \\\text{by assumption.}}} - \sum\limits_{k=r+1}^{r+s}\sum\limits_{l=0}^{r} \cbar_{k,r+s-k} \cbar_{l,r+s+1-l} \Bigg)\\
        &\ge \frac{\cbar_{r+s+1,0} \sum\limits_{l=0}^{r}\cbar_{l,r+s-l}}{\cbar\rps \sumbar[0][r+s]}
        > 0.
    \end{split}
\end{equation*}
Similarly, we find 
\begin{equation*}
    \begin{split}
      \er\rsp 
        &= 1 + \frac{1}{\cbar\rsp \sumbar[0][r+s]} \Bigg(\sum\limits_{k=r}^{r+s} \sum\limits_{l=0}^{r-1}\underbrace{\cbar_{k+1,r+s-k}\cbar_{l,r+s-l}}_{\substack{\le \cbar_{k,r+s-k} \cbar_{l+1,r+s-l} \\\text{by assumption.}}} - \sum\limits_{k=r}^{r+s}\sum\limits_{l=-1}^{r-1} \cbar_{k,r+s-k} \cbar_{l+1,r+s-l} \Bigg)\\
        &\le 1 -  \frac{\cbar_{0,r+s+1} \sum\limits_{k=r}^{r+s}\cbar_{k,r+s-k}}{\cbar\rsp \sumbar[0][r+s]}
        <1.\qedhere
    \end{split}
\end{equation*}

%% file: lemmas/MNEstimates.tex
Assume that $Q\rs$ satisfies \eqref{concavityAss:old} for all $n\ge L$. Let $\er\rs,\es\rs$ be defined by \eqref{eres:equation} for all $r+s>L$ and arbitrarily set for $r+s\le L$. Furthermore, let $f\rs$ be defined by \eqref{frs}.
Then, there exists a constant $C<\infty$ only dependent on the values  $\er\rs,\es\rs$ and $Q\rs$ for $r+s\le L$, such that 
\begin{equation*}
    \begin{split}
      \mathbb{E}[\Trs \cbar\kl] &\le C \cbar\rs \summ[r+s][N] \prod\limits_{n=r+s+1}^m m_n \quad \text{ and }\\
      f\rs \cbar\rs \mathbb{E}\left[\Prs \frac 1 {f\kl \cbar\kl }\right] &\le C \sum\limits_{m=2}^{r+s} \prod\limits_{n=m+1}^{r+s} m_n,
    \end{split}
\end{equation*}
for any $\monr,\mons>0.$

%% file: proofs/MNEstimates.tex
Because we do not have \eqref{concavityAss:old} for $n<L$, we do not put the expected values into relation with $m_n$ directly. We will first go through 
\[\overline m_n \coloneqq\max\limits_{k+l = n-1}\left(\er\kpl \frac{\cbar\kpl}{\cbar\kl} + \es\klp \frac{\cbar\klp}{\cbar\kl}\right) \text{ and }\underline m_n \coloneqq \min\limits_{k+l = n-1}\left(\er\kpl \frac{\cbar\kpl}{\cbar\kl} + \es\klp \frac{\cbar\klp}{\cbar\kl}\right).\]
We start with an induction from $r+s+1$ to $r+s$ to bound
\begin{equation}\label{prooMNEstimates:fcbound}
    \begin{split}
        \frac 1 {f\rs \cbar\rs}  
        = \er\rps \frac{\cbar\rps}{\cbar\rs} \frac{1}{f\rps \cbar\rps} + \es\rsp \frac{\cbar\rsp}{\cbar\rs} \frac{1}{f\rsp \cbar\rsp}
        =
        \begin{cases}
            \le \prod\limits_{n=r+s+1}^{N}\max\big(\overline m_n,m_n\big) \\
            \ge \prod\limits_{n=r+s+1}^{N}\min\big(\underline m_n,m_n\big)
        \end{cases}
    \end{split}
\end{equation}
and show $\mathbb{E}[\Trs \cbar\kl]\le \cbar\rs \summ[r+s][N] \prod\limits_{n=r+s+1}^{m}\max\big(\overline m_n,m_n\big)$  through
\begin{equation}\label{MNEstimate:inductionTrs}
    \begin{split}
      &\mathbb{E}[\Trs \cbar\kl]\overset{\text{Lem \ref{treeExpressions}.\ref{treeExpressions:derivativeTrs}}}{=} 
      \cbar\rs + \er\rps \mathbb{E}[T\rps \cbar\kl] + \es\rsp \mathbb{E}[T\rsp \cbar\kl]\\
        \overset{\text{ind.}}&{\le} \cbar\rs + \er\rps \cbar\rps \summ[r+s+1][N] \prod\limits_{n=r+s+2}^{m}\max\big(\overline m_n,m_n\big)
        + \es\rsp \cbar\rsp \summ[r+s+1][N] \prod\limits_{n=r+s+2}^{m}\max\big(\overline m_n,m_n\big)\\
        &\le  \cbar\rs \summ[r+s][N] \prod\limits_{n=r+s+1}^{m}\max\big(\overline m_n,m_n\big).
    \end{split}
\end{equation}
Next, we can eliminate the $\overline m_n$ and $\underline m_n$ by noticing 
\begin{equation}
    \begin{split}
        m_n &=  \frac {\sumbar[0][n]} {\sumbar[0][n-1]} 
        = \frac{\cbar_{0,n} + \frac 1 2 \sum\limits_{k=1}^{n-1}\cbar_{k,n-k} + \cbar_{n,0} +\frac 1 2 \sum\limits_{k=1}^{n-1}\cbar_{k,n-k}}{\sumbar[0][n-1]}\\
        &=
        \begin{cases}
            \le \mons \max\limits_{r+s=n-1} \frac{Q\rsp}{Q\rs} + \monr \max\limits_{r+s=n-1}\frac{Q\rps}{Q\rs} \le (\monr+\mons) \max\limits_{r+s = n-1}\left\{\frac{Q\rsp}{Q\rs},\frac{Q\rps}{Q\rs}\right\}\\
            \ge \frac \mons 2 \min\limits_{r+s=n-1} \frac{Q\rsp}{Q\rs} + \frac \monr 2 \min\limits_{r+s=n-1}\frac{Q\rps}{Q\rs} \ge \frac{\monr+\mons} 2 \min\limits_{r+s = n-1}\left\{\frac{Q\rsp}{Q\rs},\frac{Q\rps}{Q\rs}\right\}
        \end{cases}
    \end{split}
\end{equation}
and 
\begin{equation}
    \er\kpl \frac{\cbar\kpl}{\cbar\kl} + \es\klp \frac{\cbar\klp}{\cbar\kl}=
    \begin{cases}
        \le (\monr+\mons) \max\limits_{r+s = k+l}\left\{\er\rps\frac{Q\rps}{Q\rs},\es\rsp\frac{Q\rsp}{Q\rs}\right\}\\
        \ge (\monr+\mons) \min\limits_{r+s = k+l}\left\{\er\rps\frac{Q\rps}{Q\rs},\es\rsp\frac{Q\rsp}{Q\rs}\right\}.
    \end{cases}
\end{equation}
Hence, due to our choice of $\er,\es$ there is a constant $C$ depending only on $\er\kl,\es\kl,Q\kl$ for $k+l\le L$, such that for any $r,s$ and $m$ we have
\begin{equation}\label{MNEstimate:reductionTom}
    \begin{split}
        \prod\limits_{n=r+s+1}^{m}\max\big(\overline m_n,m_n\big) \le C\prod\limits_{n=r+s+1}^{m}m_n \text{ and}
        \prod\limits_{n=r+s+1}^{m}\min\big(\underline m_n,m_n) \ge \frac 1 C \prod\limits_{n=r+s+1}^{m}m_n.
    \end{split}
\end{equation}
Putting \eqref{MNEstimate:reductionTom} together with \eqref{MNEstimate:inductionTrs} is already the first estimate we wanted to show. In order to bound $\mathbb{E}[\Prs\frac 1 {f\kl \cbar\kl}],$ we use that for any $T\in\T$, we have
\begin{equation*}
    \begin{split}
      \sum\limits_{f=(v,(k,l))\in \Prs(T)} \frac 1 {f\kl \cbar\kl}
      \overset{\text{\eqref{prooMNEstimates:fcbound} and \eqref{MNEstimate:reductionTom}}} \le C \sum\limits_{f=(v,(k,l))\in \Prs(T)} \prod\limits_{n=k+l+1}^N m_n 
    = C \sum\limits_{m=2}^{r+s} \prod\limits_{n=m+1}^N m_n.
    \end{split}
\end{equation*}
Since $\Prs\frac 1 {f\kl \cbar\kl} \overset{\text{def.}} = \sum\limits_{f=(v,(k,l))\in \Prs(T)} \frac 1 {f\kl \cbar\kl}$, we can take the expected value and use that \eqref{prooMNEstimates:fcbound} and \eqref{MNEstimate:reductionTom} implies $\cbar\rs f\rs \le C \frac 1 {\prod\limits_{n=r+s+1}^N m_n},$ to obtain the second estimate.

%% file: mInfty.tex
\subsection{Subcritical \texorpdfstring{$\cbar\rs$}{cbar}}\label{section:subcritical}
Section \ref{section:logSobolevInequality} essentially reduces the estimate $H^N[c|\cbar] \lesssim \Dnlin$ to an analysis of $m_n \coloneqq \frac{\sumrs[n] \cbar\rs}{\sumrs[n-1]\cbar\rs}$.
If we want to bound the dissipation through this estimate, then $H^N$ needs to be bounded from below (under the assumption that $H[c|\cinfzw]\ge \delta >0$), which is continuity question. The continuity properties of $H[c|\cbar]$ were analysed in \cite{paperBasics} and depend on $\monr,\mons.$ So we are left in a situation, where control on $m_n$ and $H^N[c|\cbar]$ both constrain $\monr,\mons.$ Fortunately, both place the restriction that $\cbar\rs$ needs to be subcritical, though through different statements. Roughly speaking the continuity of $H^N[c|\cbar]$ requires a bound on $m_n$, whenever $\limsup\limits_{r+s \to \infty} \sqrt[r+s]{\cbar\rs} \le e^{-\delta}.$
We can make a connection to $m_n$ under the following generalisation of $\lim\limits_{i \to \infty} \frac{Q_i}{Q_{i+1}} = z_s$:\\
For all $0\le\monr,\mons < \infty$ with $\max(\monr,\mons) >0$ we have 
\begin{equation}\label{assumptionMn}
  \frac{\sumrs[n][] \cbar\rs}{\sumrs[n-1][]\cbar\rs} = m_n = m_\infty(\monr,\mons)(1+\smallo(1))
\end{equation}
  and for all $\eta >0$ and any $\eps>0$, there is an $M<\infty$, such that $|\smallo(1)| \le \eps$ for all $r+s\ge M$ and $\eta\le \max(\monr,\mons) \le \rho+\sigma$.\\
  Note, that for the idealised Kelvin model $Q\rs = \lambda^r \mu^s \binom{r+s}{r} e^{- \kappa \sum\limits_{n=1}^{r+s} n^{-\frac 1 3}}$ one directly computes $m_n = (\lambda \monr + \mu \mons) e^{-\kappa n^{-\frac 1 3}},$ satisfying \eqref{assumptionMn}. In fact, if we replace the sums by integrals $\frac 1 n \sumrs[n][] \cbar\rs \approx \int_0^1 e^{n\Phi(\xi)}\dd \xi$, then \eqref{assumptionMn} follows from Laplace's method.
    \begin{lemma}\textit{({$m_n$} is small for subcritical $\monr,\mons$)}\label{mnLimit}\\
        \input{\CommonPath/lemmas/mnLimit}
    \end{lemma}
\begin{proof}
    \input{\CommonPath/proofs/mnLimit.tex}
\end{proof}
If we also assume \eqref{assCoeff}, we can show that the bound from Lemma \ref{MNEstimates} depends only on $m_\infty.$
    \begin{lemma}\textit{(Bound of Lemma \ref{MNEstimates} depends only on $m_\infty$)}\label{boundLemmaMNEstimate}\\
        \input{\CommonPath/lemmas/boundLemmaMNEstimate}
    \end{lemma}
\begin{proof}
    \input{\CommonPath/proofs/boundLemmaMNEstimate.tex}
\end{proof}
In order to prove the dissipation estimate for subcritical $\cbar\rs$, it remains to go from $H^N[c|\cbar]$ to $H[c|\cinfzw]$.
The precise statement was found in \cite{paperBasics}, which we restate below for the readers convenience, where for $\Gamma \subset \Omega$ finite we denote $H^\Gamma[c|\cbar] \coloneqq \sumgamma c\rs \Psi\left(\frac {c\rs}{\cbar\rs}\right)$ and for $\rho,\sigma\ge0$ we denote $\cbargamma(\rho,\sigma)\coloneqq z^rw^sQ\rs$, where $z,w\ge0$ are uniquely determined by $\sumgamma r \cbargamma\rs = \rho$ and $\sumgamma s \cbargamma\rs = \sigma.$
    \begin{lemma}\textit{(Semicontinuity)}\label{relativeEntropySemicontinuity}\\
        \input{\CommonPath/lemmas/relativeEntropySemicontinuity}
    \end{lemma}
With this, we have everything we need to bound the dissipation for subcritical values of $\monr,\mons.$
    \begin{theorem}\textit{({$\Dnlin \ge \frac 1 K$} for subcritical $\monr,\mons$)}\label{dissipationSubcritical}\\
        \input{\CommonPath/theorems/dissipationSubcritical}
    \end{theorem}
\begin{proof}
    \input{\CommonPath/proofs/dissipationSubcritical.tex}

\end{proof}

%% file: lemmas/mnLimit.tex
Assume $Q\rs$ satisfies \eqref{assumptionMn}. Then, we already have 
\[\lim\limits_{n\to\infty} m_n = \limsup\limits_{r+s \to \infty} \sqrt[r+s]{\cbar\rs}.\]

%% file: proofs/mnLimit.tex
Let $m_\infty \coloneqq \lim\limits_{n\to\infty} m_n$ and $\eps>0$ be arbitrary. Then, we find an $M<\infty$, such that $m_\infty e^{-\eps} \le m_n \le m_\infty e^\eps$ for all $n\ge M$ and therefore 
\begin{equation*}
  \max\limits_{r+s=n} \cbar\rs \le \sumrs[n][]\cbar\rs
  = \prod\limits_{k=M+1}^n m_k \sumrs[M][]\cbar\rs
  \le m_\infty^{n-M} e^{\eps(n-M)} \sumrs[M][]\cbar\rs.
\end{equation*}
Now taking $\sqrt[n]{\cdot}$ on both sides, we find $\limsup\limits_{n\to \infty} \max\limits_{r+s = n} \sqrt[r+s]{\cbar\rs} \le m_\infty e^\eps.$ Similarly, we obtain 
\begin{equation*}
  \max\limits_{r+s=n} \cbar\rs \ge \frac{1}{n} \sumrs[n][]\cbar\rs
  = \frac 1 n \prod\limits_{k=M+1}^n m_k \sumrs[M][]\cbar\rs
  \ge m_\infty^{n-M} e^{-\eps(n-M)} \frac{1}{n} \sumrs[M][]\cbar\rs.
\end{equation*}
Again we can take $n-$th root and use $\sqrt[n] n  \to 1$, to conclude $\liminf\limits_{n\to \infty}\max\limits_{r+s = n} \sqrt[n]{\cbar\rs} \ge m_\infty e^{-\eps}$.
Since $\eps$ was arbitrary, we can conclude $\lim\limits_{n\to\infty} \max\limits_{r+s = n} \sqrt[r+s]{\cbar\rs} = m_\infty.$

%% file: lemmas/boundLemmaMNEstimate.tex
Assume $Q\rs$ satisfies \eqref{assumptionMn} and \eqref{assCoeff}. Furthermore, assume that $m_\infty(\monr,\mons) \le e^{-\delta}$ and $\eta\le \max(\monr,\mons)\le \rho+\sigma.$ Then, there is a constant $C_{\delta,\eta}$, depending only on $\delta,\eta$ and $Q\rs,\rho,\sigma$, such that 
\[ \summ[r+s][N] \prod\limits_{n=r+s+1}^m m_n \le C_{\delta,\eta}\quad \text{ and }\quad \sum\limits_{m=2}^{r+s} \prod\limits_{n=m+1}^{r+s} m_n \le C_{\delta,\eta}.\]

%% file: proofs/boundLemmaMNEstimate.tex
First, we will show that there is a constant $C_{\delta,\eta,Q\rs}$, such that for any $k_1\le k_2$, we have
\begin{equation}\label{boundLemmaMNEstimate:prod}
  \prod\limits_{n= k_1}^{k_2} m_n \le C_{\delta,\eta,Q\rs} e^{-(k_2-k_1+1) \frac \delta 2}.
\end{equation}
Due to \eqref{assumptionMn} and the assumption on $m_\infty$, we can fix an $M>0$, such that $m_n \le e^{-\frac \delta 2}$ for all $n\ge M$. Now, we need to bound $m_n \le C e^{- \frac \delta 2}$ for all $n < M$. We can do this as follows
\begin{equation*}
  m_n \overset{\text{def.}} = \frac{\sumrs[n][] \cbar\rs}{\sumrs[n-1][] \cbar\rs}
  \le \frac{\max(\monr,\mons)^n \sumrs[n][] Q\rs}{\max(\monr,\mons)^{n-1} \min(Q_{n-1,0},Q_{0,n-1})}
  \le \max(\monr,\mons) \frac{\sumrs[n][] Q\rs}{\min(Q_{n-1,0},Q_{0,n-1})}.
\end{equation*}
And $\max(\monr,\mons)$ can be bounded via
\begin{equation*}
  \begin{split}
    \monr
    &= \frac 1 {\liminf\limits_{r+s \to \infty}\sqrt[r+s]{Q\rs}}{\liminf\limits_{r+s \to \infty}\sqrt[r+s]{Q\rs}}\monr 
    \le  \frac 1 {\liminf\limits_{r+s \to \infty}\sqrt[r+s]{Q\rs}} \liminf\limits_{n \to \infty}\sqrt[n]{Q_{n,0}}\monr\\
    & = \frac 1 {\liminf\limits_{r+s \to \infty}\sqrt[r+s]{Q\rs}} \liminf\limits_{n \to \infty}\sqrt[n]{\cbar_{n,0}}
    \le \frac 1 {\liminf\limits_{r+s \to \infty}\sqrt[r+s]{Q\rs}} \limsup\limits_{n \to \infty}\sqrt[n]{\cbar_{n,0}}\\
    & \le \frac 1 {\liminf\limits_{r+s \to \infty}\sqrt[r+s]{Q\rs}} \limsup\limits_{r+s \to \infty}\sqrt[r+s]{\cbar\rs}
    \overset{\text{Lemma \ref{mnLimit}}}=  \frac {m_\infty} {\liminf\limits_{r+s \to \infty}\sqrt[r+s]{Q\rs}}\\
    \overset{\text{\eqref{assCoeff} and ass. on }m_\infty}&\le C e^{-\delta}.
  \end{split}
\end{equation*}
The same argument for $\mons$ yields the desired $m_n \le C e^{-\frac \delta 2}$. Now we can plug \eqref{boundLemmaMNEstimate:prod} into the sums we want to estimate, to obtain for a new constant $C_{\delta,\eta,Q\rs}$
\begin{equation*}
  \begin{split}
    &\summ[r+s][N] \prod\limits_{n=r+s+1}^m m_n \le C \summ[r+s][N] e^{-(m-r-s)\frac \delta 2} \le C \summ[0] e^{-\frac \delta 2 m} < C_{\delta,\eta,Q\rs} \text{ and }\\
    &\sum\limits_{m=2}^{r+s} \prod\limits_{n=m+1}^{r+s} m_n \le C \summ[2][r+s] e^{- \frac \delta 2 (r+s-m)} \le C \summ[0] e^{-\frac \delta 2 m} \le  C_{\delta,\eta,Q\rs}.
  \end{split}
\end{equation*}

%% file: lemmas/relativeEntropySemicontinuity.tex
Fix $\rho,\sigma >0$ and assume \eqref{assCoeff}. Then, for any $\eps>0$, we can find $M<\infty$ and $\delta_1,\delta_2>0$, such that for any $N$ and 
\[ \Gamma \coloneqq \left\{ (r,s) \in \Omega \, \middle| \, r+s \le M \text{ or }  r+s \le N \text{ with } \ln\Big(\cbar\rs^{\frac 1 {r+s}}\Big) < -\delta_1\right\}\]
we have for any $c \in X^+_{\rho,\sigma}$ with $c_{1,0},c_{0,1} > 0$ and $\sumrs[N+1] (r+s) c\rs \le \delta_2$
\begin{equation*}
  H^\Gamma[c|\cbar] \ge H[c|\cinfzw] + H^\Gamma[\cbargamma(\rho^N,\sigma^N)|\cbar] -\eps,
\end{equation*}
where $\rho^N = \sumrs[1][N] r c\rs$ and $\sigma^N = \sumrs[1][N] s c\rs$ and $\cinfzw$ is the steady state satisfying \eqref{steadyStateLimitEq}.

%% file: theorems/dissipationSubcritical.tex
Assume $Q\rs$ satisfies  \eqref{assCoeff}, \eqref{assumptionMn} and \eqref{concavityAss:old} for all $n\ge L$ and fix $\rho,\sigma>0$ as well as $\delta,\delta_3,\eta>0.$ Then, there exist constants $K<\infty$ and $\delta_2>0$ as well as $M<\infty$, such that for all $c\in X^+_{\rho,\sigma}$ and $M\le N \in\N $ with 
\begin{equation}\label{dissipationSubcritical:Ass}
  m_\infty(\monr,\mons)\le e^{-\delta_3}, \,\min(\monr,\mons)\ge \eta,\, \sumrs[N+1] (r+s)c\rs \le \delta_2 \text{ and } H[c|\cinfzw]>\delta,
\end{equation}
where $\cinfzw$ satisfies \eqref{steadyStateLimitEq}, we have $\displaystyle\Dnlin(c) \ge \frac 1 K.$

%% file: proofs/dissipationSubcritical.tex
By Lemma \ref{relativeEntropySemicontinuity}, we can find an $M<\infty$ and $\delta_2>0$, such that 
\[ H^N[c|\cbar] \overset{u\rs = \frac{c\rs}{\cbar\rs}}= \sumrs[1][N] \cbar\rs \Psi(u\rs) \ge \frac \delta 2 \text{ for all }N\ge M\text{ and } c\rs \text{ satisfying \eqref{dissipationSubcritical:Ass},}\]
because $\Gamma \subset \{r+s\le N\}.$ If we combine Lemmas \ref{boundLemmaMNEstimate} and \ref{MNEstimates} with Theorem \ref{hardyInequality}, we find a constant $C_{\delta_1,\eta},$ such that
\begin{equation}
    \sumrs[2][N] \cbar\rs \Psi(u\rs)
    \le C_{\delta_1,\eta} \sumrs[2][N] \cbar\rs \left(\er\rs \Bigl(\delr \sqrt{\Psi(u\rs)}\Bigr)^2 +  \es\rs \Bigl(\dels \sqrt{\Psi(u\rs)}\Bigr)^2\right),
\end{equation}
where $\er\rs,\es\rs$ is given as in Lemma \ref{MNEstimates}. From \eqref{assumptionMn}, there exits an $M_1$, such that $m_n \le 2 m_\infty$ for all $n\ge M_1$. So for $r+s > \max(L,M_1)$ we have 
\[\er\rs \cbar\rs \le \cbar\rms \left(\er\rs \frac{\cbar\rs}{\cbar\rms} + \es\rmsp \frac{\cbar\rmsp}{\cbar\rms}\right) \overset{\text{\eqref{eres:equation}}}= \cbar\rms m_{r+s} \le 2 m_\infty \cbar\rms \overset{\monr \ge \eta}\le \frac{2 m_\infty}{\eta}\monr \cbar\rms \]
and for $r+s\le \max(L,M_1)$, we find $\displaystyle \er\rs \cbar\rs \overset{\er\rs \le 1} \le \monr \cbar\rms \sup_{r+s<\max(L,M_1)} \left\{\frac {Q\rps}{Q\rs}\right\}.$ The same argument for $\es\rs$ yields a constant $C_{\delta_1,\eta,Q\rs}$, such that for all $r,s$ we have 
\begin{equation}
  \er\rs \cbar\rs \le C_{\delta_1,\eta,Q\rs} \monr \cbar\rms \text{ and }
  \es\rs \cbar\rs \le C_{\delta_1,\eta,Q\rs} \mons \cbar\rsm.
\end{equation}
Together with Lemmas \ref{derivativeSqrtPsi} and \ref{growthL}, we arrive at 
\begin{equation*}
  \begin{split}
  H^N[c|\cbar] \le C_{\delta_1,\eta,Q\rs} 
  \sumrs[1][N-1]  &l\left(\max(u\rps,u\rs)^{\frac 1 {r+s+1}}\right)(r+s+1) \monr \cbar\rs \Big(\delr u\rps \delr \ln(u\rps)\Big)\\
  + &l\left(\max(u\rsp,u\rs)^{\frac 1 {r+s+1}}\right)(r+s+1) \mons \cbar\rs \Big(\dels u\rsp \dels \ln(u\rsp)\Big).
  \end{split}
\end{equation*}
Remembering, that $l$ is a monotone function and bounding $u\rs = \frac{c\rs}{\cbar\rs} \le \frac{\rho+\sigma}{\eta^{r+s} Q\rs}$, we obtain with \eqref{assCoeff}
\[ \frac \delta  2\le H^N \le C_{\delta_1,\eta,Q\rs,\rho,\sigma} \Dnlin.\]

%% file: Supercriticality.tex
\subsection{Supercritical \texorpdfstring{$\cbar\rs$}{cbar}}\label{section:supercriticality}
Theorem \ref{dissipationSmallMon} allows us to eliminate $\min(\monr,\mons)\ge \eta$ from \eqref{dissipationSubcritical:Ass} in Theorem \ref{dissipationSubcritical}.
In this section, we will also get rid of $m_\infty(\monr,\mons) \le e^{-\delta_3}.$

We see two possibilities to attack the problem. First, we could just take $\Gamma$ as in Lemma \ref{relativeEntropySemicontinuity} and apply the calculations above. That is, we would restrict the graph from Notation \ref{treeStructure} to the nodes $(r,s)$ with $\frac r {r+s} \le \xi_1$ (and in a second step to nodes with $\frac r {r+s} \ge \xi_2$). 
However, this approach does not reduce $m_n$ for all large $n$.
The problem stems from the fact that $\frac{Q\rps}{Q\rs}$ may explode for $r = 0$. For the standard coefficients \eqref{idealizedKelvin} we find $\frac{Q\rps}{Q\rs} \approx \lambda \frac{r+s+1}{r+1}$. Now, if we have $(r,s),(r+1,s) \in \Gamma$ and $(r+1,s-1)\not\in \Gamma$, then $\er\rps = 1$. And therefore, $\er\rps \frac{Q\rps}{Q\rs}$ may become large. In particular, in that situation, we have $m_{r+s+1} = \er\rps \frac{\cbar\rps}{\cbar\rs} + \es\rsp \frac{Q\rsp}{Q\rs} \ge \monr \er\rsp \frac{Q\rps}{Q\rs} > 1.$ While these hurdles may be overcome (e.g. $m_{r+s}$ is only going to be occasionally large and small most of the time), we will follow a more flexible approach.

\subsubsection{Cut-off Construction \texorpdfstring{$\chat\rs$}{chat}}\label{subsection:construction}

Our strategy is to exploit Lemma \ref{relativeEntropySemicontinuity} by constructing a $\chat\rs$ for $r+s\le N$ satisfying 
\begin{equation}\label{constructionBound}
  \chat\rs = 
  \begin{cases}
    =\cbar\rs \quad & \text{ if }\cbar\rs < e^{-\delta_1(r+s)},\\
    m_n(\chat) \le e^{-\eps} & \text{ for large }n\\ 
    \er\rps \frac {\chat\rps}{\chat\rs} + \es\rsp\frac{\chat\rsp}{\chat\rs} = m_{r+s+1} & \text{ for some } 0\le \er\rs \le 1 \text{ and } \es = 1 - \er.\\
    \er\rs \chat\rs \le C \monr\cbar\rms \text{ and }\es\rs\chat\rs \le \dot C \mons\cbar\rsm & \text{ for some } C,\dot C <\infty.
  \end{cases}
\end{equation}
Then, we can argue 
\begin{equation*}
  H^\Gamma[c|\cbar] \le \sumrs[1][N] \chat\rs \Psi(u\rs) \overset{\text{Thm. \ref{hardyInequality}}} \le K \sumrs \underbrace{\chat\rs \er\rs}_{\le C \monr \cbar\rms} \Bigl(\delr \sqrt{\Psi(u\rs)}\Bigr)^2 + \underbrace{\chat\rs \es\rs}_{\le \dot C\mons \cbar\rsm} \Bigl(\dels \sqrt{\Psi(u\rs)}\Bigr)^2.
\end{equation*}
Constructing such a $\chat\rs$ is a delicate problem. To make our lives easier, we treat the ``left'' and ``right'' case of the maximum separately. That is, we only require
\[ \chat\rs = \cbar\rs \text{ for all }r< r_*(n) \coloneqq \min\big\{r \,\big|\,\cbar\rn\ge e^{-\delta_1 n}\big\}.\]
And then we do another construction, calling it $\ccirc\rs$ with 
\[ \ccirc\rs = \cbar\rs \text{ for all }r> r^*(n) \coloneqq \max\big\{r \,\big|\, \cbar\rn \ge  e^{-\delta_1 n}\big\}.\]
In order to find the $\er\rs,\es\rs$, we want to apply Lemma \ref{eres}, which we can only do, if $\chat\rs$ satisfies the concavity condition \eqref{concavityAss:old}, which is equivalent to 
\begin{equation}\label{concavityAss}
  \frac{Q\rpsm}{Q\rs} \le \frac{Q\rps}{Q\rsp} \le \frac{Q\rs}{Q\rmsp} \text{ for all }r+s = n,
\end{equation}
with the usual convention $Q\rs =0$ if $r<0$ or $s<0$ or alternatively
\begin{equation}\label{constructionConcave}
  \chat\rs \in \left[\chat\rmsp \frac{\chat\rsm}{\chat\rms}, \chat\rmsp \frac{\chat\rms}{\chat_{r-2,s+1}}\right] \text{ for all }r+s = n+1.
\end{equation}%
So for $r\ge r^*$, we may start to construct something in accordance with (\ref{constructionConcave}). Lemma \ref{mnLimit} suggests, that $\chat\rs$ needs to be small. Now, the smallest possible choice according to \eqref{constructionConcave} is to inductively set $\chat\rs \coloneqq \chat\rmsp \frac{\chat\rsm}{\chat\rms}$. Unfortunately, that makes $\chat\rs$ so small, that we encounter the same problem as we do with cutting (i.e. setting $\chat = 0$), where $m_n$ becomes large repeatedly. In order to have both $\chat\rs$ small, but not too small, we propose to make $\chat\rs$ small in front of the right maximum point 
\[ r^m(n)\coloneqq \max\big\{r \,\big|\, \cbar\rn \ge \cbar\kn\big\},\]
but keep $r^m(n)$ also as the right maximum point of $\chat$. This leads us to the following construction, which (as we will see below) essentially makes $\chat$ constant between $r_*$ and $r^m.$\\
Given $\chat\rs$ for $r+s\le n$ and $r_*(n+1)$ well defined, we construct $\chat\rs$ for $r+s = n+1$ via
\begin{equation}\label{chatConstruction}
  \chat\rs \coloneqq
  \begin{cases}
    \cbar\rs  &\text{ for all }r<r_*(n+1)\\
    \max\left(\chat\rmsp \frac{\chat\rsm}{\chat\rms},e^{-\delta_1(r+s)}\right) &\text{ for all } r_*(n+1)\le r \le \max(r^m(n),r_*(n+1))\\
    \chat\rmsp \min\left(\frac{\cbar\rs}{\cbar\rmsp},1\right) & \text{ for all } r>\max(r^m(n),r_*(n+1)),
  \end{cases}
\end{equation}
where $\max\left(\chat\rmsp \frac{\chat\rsm}{\chat\rms},e^{-\delta_1(r+s)}\right) \coloneqq e^{-\delta_1(r+s)} $ if $r=0$ or $s=0$.

Note, that $r_*(n+1) > r^m(n)$ can only happen, if $r_*(n+1) = r^m(n+1)$, which is rather extreme and not possible in the cases, where we will apply the construction later. However, we do include it in the construction for completion sake. Also, when we refer to the construction \eqref{chatConstruction}, we may specify the three cases as ``construction.i'', ``construction.ii'' and ``construction.iii''.
    \begin{lemma}\textit{}\label{constructionIsConcave}\\
        \input{\CommonPath/lemmas/constructionIsConcave}
    \end{lemma}
\begin{proof}
    \input{\CommonPath/proofs/constructionIsConcave.tex}
\end{proof}

\subsubsection{Results for \texorpdfstring{$\chat\rs$}{chat}}\label{subsection:chatResults}

Lemma \ref{constructionConcave} tells us, that we can use Lemmas \ref{eres} and \ref{MNEstimates} for $\chat\rs$. Therefore, our next task is to bound $m_n(\chat).$ To do so, we will first show that $\chat$ becomes almost constant between $r_*$ and $r^m$, which then allows us to bound $m_n(\chat).$
We can understand how $\chat\rs$ becomes constant if we consider the left maximum point
\begin{equation*}
  r_m(n,\chat) \coloneqq \min\big\{r \,\big|\, \chat\rn \ge \chat\kn \big\}.
\end{equation*}
    \begin{lemma}\textit{}\label{constructionLeftMaximum}\\
        \input{\CommonPath/lemmas/constructionLeftMaximum}
    \end{lemma}
\begin{proof}
    \input{\CommonPath/proofs/constructionLeftMaximum.tex}
\end{proof}
\input{\CommonPath/pictures/constructionChat.tex}
From here, we can estimate $m_n(\chat)$ nicely. As we can see in Figure \ref{figure:constructionChat}, $m_n(\chat)$ may jump down every now and then. This is not a problem for us, because it even improves the estimates from Lemma \ref{MNEstimates}. Hence, we only give an estimate from above.
    \begin{lemma}\textit{(Bound on $m_n(\chat)$)}\label{constructionmnBound}\\
        \input{\CommonPath/lemmas/constructionmnBound}
    \end{lemma}
\begin{proof}
    \input{\CommonPath/proofs/constructionmnBound.tex}

\end{proof}
Lemmas \ref{constructionLeftMaximum} and \ref{constructionmnBound} yield an analogue of Lemma \ref{boundLemmaMNEstimate} provided that $r^m(n)-r_*(n) \to \infty$ as $n\to \infty.$
    \begin{lemma}\textit{(Bound of Lemma \ref{MNEstimates} for $\chat$)}\label{constructionMNEstimate}\\
        \input{\CommonPath/lemmas/constructionMNEstimate}
    \end{lemma}
\begin{proof}
    \input{\CommonPath/proofs/constructionMNEstimate.tex}
\end{proof}
As mentioned above, we also have to do a construction (calling it $\ccirc\rs$ instead of $\chat\rs$) with $\ccirc\rs = \cbar\rs$ for $r>r^*.$ This is just going to be the same construction as for $\chat\rs$, exploiting the symmetry $r \leftrightsquigarrow s.$ To be precise, given $\ccirc\rs$ for $r+s\le n$ and $r^*(n+1)$ well defined, we construct $\ccirc\rs$ for $r+s = n+1$ via
\begin{equation}\label{ccircConstruction}
  \ccirc\rs \coloneqq
  \begin{cases}
    \cbar\rs  &\text{ for all }s<s^*(n+1)\\
    \max\left(\ccirc\rpsm \frac{\ccirc\rms}{\ccirc\rsm},e^{-\delta_1(r+s)}\right) &\text{ for all } s^*(n+1)\le s \le \max(s_m(n),s^*(n+1))\\
    \ccirc\rpsm \min\left(\frac{\cbar\rs}{\cbar\rpsm},1\right) & \text{ for all } s>\max(s_m(n),s^*(n+1)),
  \end{cases}
\end{equation}
where $\max\left(\ccirc\rpsm \frac{\ccirc\rms}{\ccirc\rsm},e^{-\delta_1(r+s)}\right) \coloneqq e^{-\delta_1(r+s)} $ if $r=0$ or $s=0$ and 
\[ s^*(n) =\min\big\{s \, \big|\, \cbar\sn\ge e^{-\delta_1 n}\big\} \text{ and }
 s_m(n)\coloneqq \max\big\{s \,\big|\, \cbar\sn \ge \cbar\kn\big\}.\]
Following the above, we obtain the corresponding Lemmas.
    \begin{lemma}\textit{}\label{constructionIsConcaveCirc}\\
        \input{\CommonPath/lemmas/constructionIsConcaveCirc}
    \end{lemma}
    \begin{lemma}\textit{}\label{constructionLeftMaximumCirc}\\
        \input{\CommonPath/lemmas/constructionLeftMaximumCirc}
    \end{lemma}
    \begin{lemma}\textit{}\label{constructionmnBoundCirc}\\
        \input{\CommonPath/lemmas/constructionmnBoundCirc}
    \end{lemma}
    \begin{lemma}\textit{}\label{constructionMNEstimateCirc}\\
        \input{\CommonPath/lemmas/constructionMNEstimateCirc}
    \end{lemma}
\begin{proof}
  (Of Lemmas \ref{constructionIsConcaveCirc}--\ref{constructionMNEstimateCirc})\\
  If we define $D\rs \coloneqq Q_{s,r}$, then $\ccirc\rs(\monr,\mons,Q\rs) = \chat_{s,r}(\mons,\monr,D\rs)$ follows from the construction (by induction if you want). Furthermore, $D\rs$ satisfies the concavity condition, if and only if $Q\rs$ does. Hence, Lemmas \ref{constructionIsConcaveCirc}--\ref{constructionMNEstimateCirc} reduce to Lemmas \ref{constructionIsConcave}--\ref{constructionMNEstimate}.
\end{proof}

\subsubsection{Application to \texorpdfstring{$\cbar\rs$}{cbar}}\label{subsection:constructionApplication}

In order to turn Subsections \ref{subsection:construction} and \ref{subsection:chatResults} into a useful estimate on the dissipation, we need to assure that whenever the constructions $\chat$ and $\ccirc$ are necessary (according to Lemma \ref{relativeEntropySemicontinuity}), then the assumptions \eqref{conditionGrowthOnLeft}, \eqref{conditionSmallOnLeft} and \eqref{constructionMNEstimate:ass:width}, respectively \eqref{conditionGrowthOnLeftCirc}, \eqref{conditionSmallOnRight} and \eqref{constructionMNEstimateCirc:ass:width} are true. To do so, we need a second generalisation of $\lim\limits_{i \to \infty} \frac{Q_i}{Q_{i+1}} = z_s$, that is slightly different from \eqref{assumptionMn}:\\
For all $0\le\monr,\mons < \infty$ with $\max(\monr,\mons) >0$ we have 
\begin{equation}\label{assumptionQn}
  \frac{\max\limits_{r+s=n}\cbar\rs}{\max\limits_{r+s=n-1}\cbar\rs} \eqqcolon q_n = q_\infty(\monr,\mons)(1+\smallo(1))
\end{equation}
  and for all $\eta >0$ and any $\eps>0$, there is an $M<\infty$, such that $|\smallo(1)| \le \eps$ for all $r+s\ge M$ and $\eta\le \max(\monr,\mons) \le \rho+\sigma$.\\
  Because of the concavity of $\cbar\rs$, we can compare $\frac{\cbar_{r_*,s_*}}{\cbar_{r_*,s_*-1}}$ and $\frac{\cbar_{r_*+1,s_*}}{\cbar_{r_*,s_*}}$ with $q_n.$ This way \eqref{assumptionQn} ties in nicely with the other assumptions of this section.
  We start by showing that the start of the hole coming from Lemma \ref{relativeEntropySemicontinuity} can be bound independent of $\monr,\mons$ as long as they are ``large enough''.
    \begin{lemma}\textit{}\label{maxQuotientLimit}\\
        \input{\CommonPath/lemmas/maxQuotientLimit}
    \end{lemma}
\begin{proof}
    \input{\CommonPath/proofs/maxQuotientLimi.tex}
\end{proof}
    \begin{proposition}\textit{(Convergence of {$\limsup\limits_{r+s\to \infty} \sqrt[r+s}\label{\cbar\rs}\\
        \input{\CommonPath/propositions/\cbar\rs}
\end{proposition}$})]{maxcbarLimit}
\begin{proof}
    \input{\CommonPath/proofs/maxcbarLimit.tex}
\end{proof}
    \begin{proposition}\textit{(Checking the assumptions)}\label{supercritIsApplicable}\\
        \input{\CommonPath/propositions/supercritIsApplicable}
\end{proposition}
\begin{proof}
    \input{\CommonPath/proofs/supercritIsApplicable.tex}
\end{proof}
With this, we finally arrive at the dissipation estimate for supercritical monomers.
    \begin{theorem}\textit{({$\Dnlin \ge \frac 1 K$ for supercritical $\monr,\mons$})}\label{dissipationSupercritical}\\
        \input{\CommonPath/theorems/dissipationSupercritical}
    \end{theorem}
\begin{proof}
    \input{\CommonPath/proofs/dissipationSupercritical.tex}
\end{proof}

%% file: lemmas/constructionIsConcave.tex
Fix $\monr,\mons>0.$ Assume that $Q\rs$ satisfies the concavity condition \eqref{concavityAss} for $r+s \ge M$ and that for $r_* = r_*(n)$ and $s_* = n-r_*$
\begin{equation}\label{conditionGrowthOnLeft}
  \frac{\cbar_{r_*,s_*}}{\cbar_{r_*-1,s_*}} \ge e^{-\delta_1} \text{ holds for all } r_*+s_* > M.
\end{equation}
Then, $\chat\rs$ constructed via \eqref{chatConstruction} starting with $\chat\rs = \cbar\rs$ for $r+s \le M$ satisfies \eqref{constructionConcave} for all $n \ge M$.
Furthermore, the following holds true for all $n\ge M$
\begin{multicols}{2}
  \begin{enumerate}
    \item\label{constructionIsConcave:rstar} $r_*(n+1)\le r_*(n)+1$
    \item\label{constructionIsConcave:le} $\chat\rs \le \cbar\rs$
    \item $r^m(n,\cbar) = r^m(n,\chat)$
    \item\label{constructionIsConcave:large} $\chat\rn \ge e^{-\delta_1n}$ for $r_*(n) \le r\le r^m(n)$.
  \end{enumerate}
\end{multicols}

Note, that \eqref{conditionGrowthOnLeft} in particular asks for $r_*(n)$ to be well defined, i.e. $\{r\,|\, \cbar\rn \ge e^{-\delta_1 n}\} \neq \emptyset$ and \eqref{conditionGrowthOnLeft} is an empty statement if $r_*=0.$

%% file: proofs/constructionIsConcave.tex
Firstly, we see that $r_*(n+1)-1\le r_*(n) $, because for any $r\le r_*(n+1)-1$ and $r+s = n+1$ we have 
\[ \cbar\rms = \frac{\cbar\rms}{\cbar\rs} \cbar\rs \overset{\text{\eqref{concavityAss:old}}} \le \frac{\cbar_{r_*(n+1),s_*(n+1)}}{\cbar_{r_*(n+1)-1,s_*(n+1)}}\cbar\rs \overset{r<r_*}< \frac{\cbar_{r_*(n+1),s_*(n+1)}}{\cbar_{r_*(n+1)-1,s_*(n+1)}} e^{-\delta_1(n+1)}\overset{\text{\eqref{conditionGrowthOnLeft}}}\le e^{-\delta_1n}.\]
Now we will inductively ($n \to n+1$) show the following
\begin{enumerate}[label=\alph*),ref=\alph*)]
  \item \label{constructionConcave:growing} $\frac{\chat\rpsm}{\chat\rs} \ge 1$ for $r+1\le r^m(n,\cbar)$ and $\frac{\chat\rpsm}{\chat\rs} < 1$ for $r+1> r^m(n,\cbar),$ where $r+s = n$ 
  \item \label{constructionConcave:rm}$r^m(n,\cbar) = r^m(n,\chat),$ which we can therefore denote as $r^m(n)$
  \item\label{constructionConcave:bound} $\chat\rn \ge e^{-\delta_1 n}$ for $r_*(n)\le r \le r^m(n)$
  \item\label{constructionConcave:flat} $\frac{\chat\rpsm}{\chat\rs}\le \frac{\cbar\rpsm}{\cbar\rs}$ for all $r+1\le n$ and $r+s = n$
  \item\label{constructionConcave:smaller} $\chat\rs\le \cbar\rs$.
\end{enumerate}

This proof can get a little confusing. We are basically just checking the statements for all possible cases. In order to follow along, it helps to keep the following picture of the construction in mind.
\input{\CommonPath/pictures/constructionFull.tex}

The induction start at $n=M$ is true, because then $\chat=\cbar$.

\underline{Proof of \ref{constructionConcave:growing}}: We will show $\frac{\chat\rpsm}{\chat\rs} \ge 1$ for $r+s = n+1$ in the following three steps
\begin{enumerate}[label=a.\roman*),ref=a.\roman*)]
  \item \label{constructionConcave:growing:i} for all $r+1<r_*(n+1)$
  \item \label{constructionConcave:growing:ii} for all $r_*(n+1)\le r+1 \le r^m(n,\cbar)$
  \item \label{constructionConcave:growing:iii} for $r+1=r^m(n+1,\cbar)$
\end{enumerate}
and then $\frac{\chat\rpsm}{\chat\rs} < 1$ for all $r+1>r^m(n+1,\cbar).$\\
Since $Q\rs$ is concave, so is $\cbar\rs$ and hence \ref{constructionConcave:growing:i} is just $\frac{\chat\rpsm}{\chat\rs} \overset{\text{\constructioni}} = \frac{\cbar\rpsm}{\cbar\rs} \ge 1$ for all $r+1<r_*(n+1).$\\
For $r_*(n+1)\le r+1 \le r^m(n,\cbar)\overset{\text{induction}} = r^m(n,\chat)$, we get 
\[\chat\rpsm = \max\Big(\chat\rs \frac{\chat_{r+1,s-2}}{\chat\rsm},e^{-\delta_1(r+s)}\Big)\overset{r+1\le r^m(n)\le n \implies s-2 \ge 0}\ge \chat\rs \frac{\chat_{r+1,s-2}}{\chat\rsm} \overset{r+1\le r^m(n,\chat)}\ge \chat\rs.\]
If $r^m(c+1,\cbar) = r^m(n)$, then \ref{constructionConcave:growing:iii} is just \ref{constructionConcave:growing:ii}. 
Otherwise, by concavity we already know $r^m(n+1,\cbar) = r^m(n) + 1$. Here we distinguish another two cases, $r^m(n+1,\cbar) = r_*(n+1)$ and $r^m(n+1,\cbar) > r_*(n+1).$
Now, if $r_*(n+1) = r^m(n+1,\cbar)$, we obtain for $r+1 = r^m(n+1)$
\[ \chat_{\rpsm} \overset{\text{\constructionii}}\ge e^{-\delta_1(n+1)} \overset{\text{def. }r_*(n+1)}\ge \cbar\rs \overset{\text{\constructioni}} = \chat\rs.\]
And if $r_*(n+1) < r^m(n+1)$ we get for $r+1 = r^m(n+1,\cbar) > \max(r^m(n),r_*(n+1))$
\[ \chat\rpsm \overset{\text{\constructioniii}}= \chat\rs \min\Big(\underbrace{\frac{\cbar\rpsm}{\cbar\rs}}_{\mathclap{\ge 1 \text{ by def. of }r^m(n+1,\cbar)}},1\Big) = \chat\rs,\]
which shows \ref{constructionConcave:growing:iii}.
Finally, for all $r+1>r^m(n+1,\cbar)\ge\max(r^m(n),r_*(n+1))$ and $r+s = n+1$ we have
\[ \chat\rpsm \overset{\text{\constructioniii}}= \chat\rs \min\Big(\underbrace{\frac{\cbar\rpsm}{\cbar\rs}}_{\mathclap{< 1 \text{ by def. of }r^m(n+1,\cbar)}},1\Big) < \chat\rs.\]
\underline{Proof of \ref{constructionConcave:rm}}: Follows immediately from \ref{constructionConcave:growing}.

\underline{Proof of \ref{constructionConcave:bound}}: By construction we have $\chat\rs \ge e^{-\delta_1(r+s)}$ for $r=r_*$. And then by \ref{constructionConcave:growing} $\chat\rs \ge \chat_{r_*,s_*}$ for $r_*\le r \le r^m$.

\underline{Proof of \ref{constructionConcave:flat}}:
There is nothing to do for $r+1< r_*(n+1)$ and $r+s = n+1$, because it holds with equality. For $r+1 = r_*(n+1)$ and $r+s =n+1$, we have either $\chat\rpsm = e^{-\delta_1(r+s)}$ and then $\frac{\chat\rpsm}{\chat\rs} = \frac{e^{-\delta_1(r+s)}}{\cbar\rs} \le \frac{\cbar\rpsm}{\cbar\rs}$ or 
\begin{equation*}
  \frac{\chat\rpsm}{\chat\rs} = \frac{\chat_{r+1,s-2}}{\chat\rsm} \overset{\text{induction}} \le \frac{\cbar_{r+1,s-2}}{\cbar\rsm} \overset{\cbar\text{ concave}}\le \frac{\cbar\rpsm}{\cbar\rs}.
\end{equation*}
The same calculation holds for $r_*(n+1)<r+1\le r^m(n)$, because by \ref{constructionConcave:growing} and \ref{constructionConcave:bound} we have (note again, that $r+1\le n$ implies $s\ge 2$) 
\[\chat\rpsm \overset{\text{\constructionii}} = \max\left(\chat\rs \frac{\chat_{r+1,s-2}}{\chat\rsm},e^{-\delta_1(r+s)}\right) \overset{\text{\ref{constructionConcave:growing} and \ref{constructionConcave:bound}}} = \chat\rs \frac{\chat_{r+1,s-2}}{\chat\rsm}.\]
Now for $r+1 \ge \max(r^m(n),r_*(n+1)) + 1$ and $r+s=n+1$ we have 
\[ \frac{\chat\rpsm}{\chat\rs}  \overset{\text{\constructioniii}} = \min\Big(\frac{\cbar\rpsm}{\cbar\rs},1\Big) \le \frac{\cbar\rpsm}{\cbar\rs}.\]

\underline{Proof of \ref{constructionConcave:smaller}}:
Because $\chat_{0,n}\le \cbar_{0,n}$, it follows directly from \ref{constructionConcave:flat}.

It remains to show that $\chat\rs$ satisfies the concavity condition \eqref{constructionConcave} for $n\ge M$. We will do so with a two-step induction, where we first show that 
\begin{equation}\label{constructionConcave:monotone}
  \frac{\chat\rsm}{\chat\rms} \le \frac{\chat\rms}{\chat\rmmsp}
\end{equation}
for all $r+s = n$ and then show the concavity condition \eqref{constructionConcave} for $r+s = n+1$, which in particular implies \eqref{constructionConcave:monotone} for all $r+s = n+1$. Hence, the induction start is just \eqref{constructionConcave:monotone} for $r+s=M$, which is true by the concavity of $\cbar$ and it remains to do the induction step, i.e. show \eqref{constructionConcave} for $r+s=n+1$ assuming \eqref{constructionConcave:monotone} for $r+s=n.$ We will do so in the following steps:
\begin{multicols}{2}
\begin{enumerate}[label=\arabic*),ref=\arabic*)]
  \item\label{constructionConcave:concavity:one} $\frac{\chat\rsm}{\chat\rms} \le \frac{\chat\rs}{\chat\rmsp} \le \frac{\chat\rms}{\chat\rmmsp}$ for $r<r_*$
  \item\label{constructionConcave:concavity:two} $\frac{\chat\rsm}{\chat\rms} \le \frac{\chat\rs}{\chat\rmsp}$ for $r_*\le r \le \max(r^m(n),r_*)$
  \item\label{constructionConcave:concavity:three} $\frac{\chat\rs}{\chat\rmsp} \le \frac{\chat\rms}{\chat\rmmsp}$ for $r=r_*$
  \item\label{constructionConcave:concavity:four} \makebox[0.3\textwidth][l]{$\frac{\chat\rs}{\chat\rmsp} \le \frac{\chat\rms}{\chat\rmmsp}$ for $r_*< r \le \max(r^m(n),r_*)$}
  \item\label{constructionConcave:concavity:five} $\frac{\chat\rsm}{\chat\rms} \le \frac{\chat\rs}{\chat\rmsp} \le \frac{\chat\rms}{\chat\rmmsp}$ for $r=r^m$
  \item\label{constructionConcave:concavity:six} $\frac{\chat\rs}{\chat\rmsp} \le \frac{\chat\rms}{\chat\rmmsp}$ for $r^m< r$
  \item\label{constructionConcave:concavity:seven} $\frac{\chat\rsm}{\chat\rms} \le \frac{\chat\rs}{\chat\rmsp}$ for $r^m< r$
\end{enumerate}
\end{multicols}
where $r_* = r_*(n+1)$, $r^m = r^m(n+1)$ and $r+s = n+1 >M.$

For \ref{constructionConcave:concavity:one} we know that $r<r_*$ implies $r\le r_*(n)$ (due to \ref{constructionIsConcave:rstar}) and hence all the $\chat$ appearing in \ref{constructionConcave:concavity:one} are equal to $\cbar,$ apart from $\chat\rsm$, which we can estimate with $\chat\rsm\le \cbar\rsm.$ But $\cbar\rs$ is concave by assumption.\\
From the construction, \ref{constructionConcave:concavity:two} follows immediately.\\
To prove \ref{constructionConcave:concavity:three} let $r=r_*(n+1)$. If $\chat\rmsp \frac{\chat\rsm}{\chat\rms} \ge e^{-\delta_1 n},$  we have $\frac{\chat\rs}{\chat\rmsp} \overset{\text{\constructionii}}= \frac{\chat\rsm}{\chat\rms}\overset{\text{\eqref{constructionConcave:monotone}}} \le \frac{\chat\rms}{\chat\rmmsp}.$ Otherwise we have to show $e^{-\delta_1 n} \le \cbar\rmsp \frac{\chat\rms}{\chat_{r-2,s+1}}$, which follows with $r-1 = r_*(n+1) -1 \overset{\text{\ref{constructionIsConcave:rstar}}} \le r_*(n)$ via
\begin{equation}
  \begin{split}
    \frac{\cbar\rmsp}{\cbar\rmmsp} \chat\rms \overset{\cbar\text{ concave}}\ge \frac{\cbar\rs}{\cbar\rms} \chat\rms=
    \begin{cases}
      \overset{\text{\constructioni}}= \cbar\rs \overset{\text{def. }r_*} \ge e^{-\delta_1 (n+1)} &\text{ if } r-1<r_*(n)\\
      \overset{\text{\ref{constructionConcave:bound}}}\ge \frac{\cbar\rs}{\cbar\rms}e^{-\delta_1 n} \overset{\text{\eqref{conditionGrowthOnLeft}}}\ge e^{-\delta_1(n+1)} &\text{ if }r-1 = r_*(n).
    \end{cases}
  \end{split}
\end{equation}
To show \ref{constructionConcave:concavity:four}, we know that for $r_*< r \le \max(r^m(n),r_*)$ we have $\chat\rmsp \frac{\chat\rsm}{\chat\rms} \overset{\text{\ref{constructionConcave:growing} and \ref{constructionConcave:bound}}}  \ge e^{-\delta_1 n}$ and hence $\frac{\chat\rs}{\chat\rmsp} \overset{\text{\constructionii}}= \frac{\chat\rsm}{\chat\rms}\overset{\text{\eqref{constructionConcave:monotone}}} \le \frac{\chat\rms}{\chat\rmmsp}.$\\
Next we consider \ref{constructionConcave:concavity:five}, i.e. $r=r^m(n+1)$. If $r^m(n+1) = r^m(n)$, \ref{constructionConcave:concavity:five} follows from \ref{constructionConcave:concavity:two} and \ref{constructionConcave:concavity:four}. So assume $r^m(n+1) = r^m(n)+1$. Again if $r^m(n+1) = r_*(n+1)$, it follows from \ref{constructionConcave:concavity:two} and \ref{constructionConcave:concavity:three}. Otherwise we have $r>\max(r^m(n),r_*)$ and then 
\[ \frac{\chat\rsm}{\chat\rms} \overset{r-1 = r^m(n)} < 1\overset{r=r^m(n+1)} \le \frac{\chat\rs}{\chat\rmsp} \overset{\text{\constructioniii}} = \min\Big(\frac{\cbar\rs}{\cbar\rmsp},1\Big) \le 1 \overset{r-1 = r^m(n)}\le \frac{\chat\rms}{\chat\rmmsp}.\]
Because $r>r^m(n+1) \ge \max(r^m(n),r_*)$, we can show \ref{constructionConcave:concavity:six} via 
\[\frac{\chat\rs}{\chat\rmsp} \overset{\text{\constructioniii}} = \frac{\cbar\rs}{\cbar\rmsp} \overset{\cbar\text{ concave}}\le \frac{\cbar\rms}{\cbar\rmmsp} \overset{\text{\ref{constructionConcave:flat}}}\le \frac{\chat\rms}{\chat\rmmsp}.\]
Finally, for \ref{constructionConcave:concavity:seven} we distinguish $n = M$ and $n>M$
\[\frac{\chat\rs}{\chat\rmsp} = \frac{\cbar\rs}{\cbar\rmsp} \overset{\cbar\text{ concave}}\ge \frac{\cbar\rsm}{\cbar\rms}=
\begin{cases}
  \overset{\text{construction start}} = \frac{\chat\rsm}{\chat\rms} &\text{ if } n = M\\
  \ge \min(\frac{\cbar\rsm}{\cbar\rms},1) \overset{\text{\constructioniii}}= \frac{\chat\rsm}{\chat\rms} &\text{ if }n>M,
\end{cases}\]
where we have used $r>r^m(n) \ge \max(r^m(n-1),r_*(n))$ to see that we are in construction.iii in the last equality.

%% file: pictures/constructionFull.tex
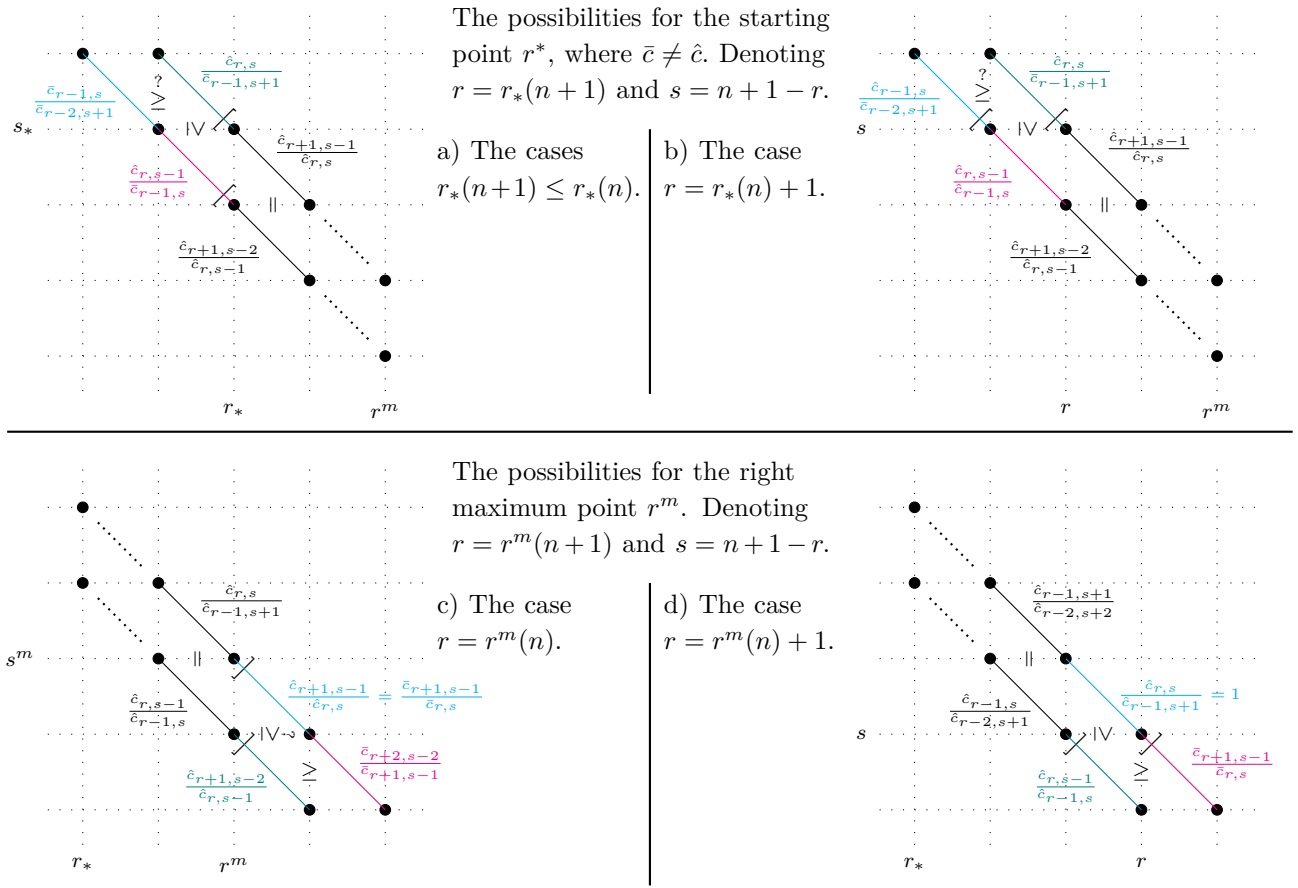
\begin{figure}[!htb]
  \centering
  \begin{tikzpicture}[scale=1.0,>={Stealth[scale=1.2]}]
    \draw[black,thick] (5.5,1.0) -- (5.5,-2.5);
    \draw[black,thick] (-3,-3) -- (14,-3);
    \draw[black,thick] (5.5,-5) -- (5.5,-9);
    \begin{scope}[shift={(0,0)}]
        \draw[black, loosely dotted] (-2,2.5) -- (-2,-2.5);
        \draw[black, loosely dotted] (-1,2.5) -- (-1,-2.5);
        \draw[black, loosely dotted] (0,2.5) -- (0,-2.5)node[below] {\scriptsize{$r_*$}};
        \draw[black, loosely dotted] (1,2.5) -- (1,-2.5);
        \draw[black, loosely dotted] (2,2.5) -- (2,-2.5) node[below] {\scriptsize{$r^m$}};
        \draw[black, loosely dotted] (2.5,-2) -- (-2.5,-2);
        \draw[black, loosely dotted] (2.5,-1) -- (-2.5,-1);
        \draw[black, loosely dotted] (2.5,0) -- (-2.5,0);
        \draw[black, loosely dotted] (2.5,1) -- (-2.5,1)node[left] {\scriptsize{$s_*$}};
        \draw[black, loosely dotted] (2.5,2) -- (-2.5,2);
        \filldraw[black] (0,0) circle (2pt) node {};
        \filldraw[black] (1,0) circle (2pt) node {};
        \filldraw[black] (2,-1) circle (2pt) node {};
        \filldraw[black] (0,1) circle (2pt) node {};
        \filldraw[black] (-1,2) circle (2pt) node {};
        \filldraw[black] (1,-1) circle (2pt) node {};
        \filldraw[black] (2,-2) circle (2pt) node {};
        \filldraw[black] (-1,1) circle (2pt) node {};
        \filldraw[black] (-2,2) circle (2pt) node {};
        \node[left,rotate=-45] (n) at (0,0) {$[$};
        \node[left,rotate=-45] (n+1) at (0,1) {$[$};
        \node (first) at (-1,1.5) {\scriptsize{$\overset{?}\ge$}};
        \node[rotate=270, yscale = 1] (second) at (-0.5,1) {\scriptsize{$\ge$}};
        \node[rotate=270, yscale = 1] (second) at (0.5,0) {\scriptsize{$=$}};
        \draw[cyan, -] (-2,2) -- (-1,1);
        \node (left) at (-2.10,1.35) {\tiny{\color{cyan}$\frac{\cbar\rms}{\cbar_{r-2,s+1}}$}};
        \draw[magenta, -] (-1,1) -- (0,0);
        \node (left) at (-1.0,0.3) {\tiny{\color{magenta}$\frac{\chat\rsm}{\cbar\rms}$}};
        \draw[teal, -] (-1,2) -- (0,1);
        \node (left) at (0.1,1.75) {\tiny{\color{teal}$\frac{\chat\rs}{\cbar\rmsp}$}};
        \draw[black,-] (0,1) -- (1,0);
        \node (left) at (1.1,0.7) {\tiny{\color{black}$\frac{\chat\rpsm}{\chat\rs}$}};
        \draw[black,-] (0,0) -- (1,-1);
        \node (left) at (-0.2,-0.7) {\tiny{\color{black}$\frac{\chat_{r+1,s-2}}{\chat\rsm}$}};
        \draw[black,dotted, thick] (1.2,-1.2) -- (1.8,-1.8);
        \draw[black,dotted, thick] (1.2,-0.2) -- (1.8,-0.8);
        \node[text width=5cm,below right,] at (2.75,2.75) {\small The possibilities for the starting point $r^*$, where $\cbar \neq \chat.$ Denoting $r=r_*(n+1)$ and $s=n+1-r.$};
        \node[text width=2.7cm,below right,] at (2.55,1.0) {\small a) The cases $r_*(n+1) \le r_*(n)$.};
    \end{scope}
    \begin{scope}[shift={(11,0)}]
        \draw[black, loosely dotted] (-2,2.5) -- (-2,-2.5);
        \draw[black, loosely dotted] (-1,2.5) -- (-1,-2.5);
        \draw[black, loosely dotted] (0,2.5) -- (0,-2.5)node[below] {\scriptsize{$r$}};
        \draw[black, loosely dotted] (1,2.5) -- (1,-2.5);
        \draw[black, loosely dotted] (2,2.5) -- (2,-2.5) node[below] {\scriptsize{$r^m$}};
        \draw[black, loosely dotted] (2.5,-2) -- (-2.5,-2);
        \draw[black, loosely dotted] (2.5,-1) -- (-2.5,-1);
        \draw[black, loosely dotted] (2.5,0) -- (-2.5,0);
        \draw[black, loosely dotted] (2.5,1) -- (-2.5,1)node[left] {\scriptsize{$s$}};
        \draw[black, loosely dotted] (2.5,2) -- (-2.5,2);
        \filldraw[black] (0,0) circle (2pt) node {};
        \filldraw[black] (1,0) circle (2pt) node {};
        \filldraw[black] (2,-1) circle (2pt) node {};
        \filldraw[black] (0,1) circle (2pt) node {};
        \filldraw[black] (-1,2) circle (2pt) node {};
        \filldraw[black] (1,-1) circle (2pt) node {};
        \filldraw[black] (2,-2) circle (2pt) node {};
        \filldraw[black] (-1,1) circle (2pt) node {};
        \filldraw[black] (-2,2) circle (2pt) node {};
        \node[left,rotate=-45] (n) at (-1,1) {$[$};
        \node[left,rotate=-45] (n+1) at (0,1) {$[$};
        \node (first) at (-1.1,1.6) {\scriptsize{$\overset ? \ge$}};
        \node[rotate=270, yscale = 1] (second) at (-0.5,1) {\scriptsize{$\ge$}};
        \node[rotate=270, yscale = 1] (second) at (0.5,0) {\scriptsize{$=$}};
        \draw[cyan, -] (-2,2) -- (-1,1);
        \node (left) at (-2.2,1.4) {\tiny{\color{cyan}$\frac{\chat\rms}{\cbar_{r-2,s+1}}$}};
        \draw[magenta, -] (-1,1) -- (0,0);
        \node (left) at (-1.1,0.3) {\tiny{\color{magenta}$\frac{\chat\rsm}{\chat\rms}$}};
        \draw[teal, -] (-1,2) -- (0,1);
        \node (left) at (0.05,1.75) {\tiny{\color{teal}$\frac{\chat\rs}{\cbar\rmsp}$}};
        \draw[black,-] (0,1) -- (1,0);
        \node (left) at (1.1,0.75) {\tiny{\color{black}$\frac{\chat\rpsm}{\chat\rs}$}};
        \draw[black,-] (0,0) -- (1,-1);
        \node (left) at (-0.2,-0.7) {\tiny{\color{black}$\frac{\chat_{r+1,s-2}}{\chat\rsm}$}};
        \draw[black,dotted, thick] (1.2,-1.2) -- (1.8,-1.8);
        \draw[black,dotted, thick] (1.2,-0.2) -- (1.8,-0.8);
        \node[text width=2.7cm,below right,] at (-5.45,1.0) {\small b) The case\\ $r = r_*(n)+1$.};
    \end{scope}
    \begin{scope}[shift={(0,-6)}]
        \draw[black, loosely dotted] (-2,2.5) -- (-2,-2.5)node[below] {\scriptsize{$r_*$}};
        \draw[black, loosely dotted] (-1,2.5) -- (-1,-2.5);
        \draw[black, loosely dotted] (0,2.5) -- (0,-2.5) node[below] {\scriptsize{$r^m$}};
        \draw[black, loosely dotted] (1,2.5) -- (1,-2.5);
        \draw[black, loosely dotted] (2,2.5) -- (2,-2.5);
        \draw[black, loosely dotted] (2.5,-2) -- (-2.5,-2);
        \draw[black, loosely dotted] (2.5,-1) -- (-2.5,-1);
        \draw[black, loosely dotted] (2.5,0) -- (-2.5,0)node[left] {\scriptsize{$s^m$}};
        \draw[black, loosely dotted] (2.5,1) -- (-2.5,1);
        \draw[black, loosely dotted] (2.5,2) -- (-2.5,2);
        \filldraw[black] (0,0) circle (2pt) node {};
        \filldraw[black] (0,-1) circle (2pt) node {};
        \filldraw[black] (1,-2) circle (2pt) node {};
        \filldraw[black] (-1,0) circle (2pt) node {};
        \filldraw[black] (-2,1) circle (2pt) node {};
        \filldraw[black] (1,-1) circle (2pt) node {};
        \filldraw[black] (2,-2) circle (2pt) node {};
        \filldraw[black] (-1,1) circle (2pt) node {};
        \filldraw[black] (-2,2) circle (2pt) node {};
        \node[right,rotate=-45] (n+1) at (0,0) {$]$};
        \node[right,rotate=-45] (n) at (0,-1) {$]$};
        \node[rotate=270] (first) at (0.6,-1) {\scriptsize{$\overset{?}\ge$}};
        \node (second) at (1,-1.5) {\scriptsize{$\ge$}};
        \node[rotate=270, yscale = 1] (equal) at (-0.5,0) {\scriptsize{$=$}};
        \draw[cyan, -] (0,0) -- (1,-1);
        \node (left) at (2,-0.5) {\tiny{\color{cyan}$\frac{\chat\rpsm}{\chat\rs} = \frac{\cbar\rpsm}{\cbar\rs}$}};
        \draw[magenta, -] (1,-1) -- (2,-2);
        \node (left) at (2.2,-1.4) {\tiny{\color{magenta}$\frac{\cbar_{r+2,s-2}}{\cbar\rpsm}$}};
        \draw[teal, -] (0,-1) -- (1,-2);
        \node (left) at (-0.1,-1.7) {\tiny{\color{teal}$\frac{\chat_{r+1,s-2}}{\chat\rsm}$}};
        \draw[black,-] (-1,0) -- (0,-1);
        \node (left) at (0.1,0.75) {\tiny{\color{black}$\frac{\chat\rs}{\chat\rmsp}$}};
        \draw[black,-] (-1,1) -- (0,0);
        \node (left) at (-1.0,-0.7) {\tiny{\color{black}$\frac{\chat\rsm}{\chat\rms}$}};
        \draw[black,dotted, thick] (-1.8,1.8) -- (-1.2,1.2);
        \draw[black,dotted, thick] (-1.8,0.8) -- (-1.2,0.2);
        \node[text width=5cm,below right,] at (2.75,2.75) {\small The possibilities for the right maximum point $r^m$. Denoting $r=r^m(n+1)$ and $s=n+1-r.$};
        \node[text width=2.7cm,below right,] at (2.55,1.0) {\small c) The case \\$r = r^m(n)$.};
    \end{scope}
    \begin{scope}[shift={(11,-6)}]
        \draw[black, loosely dotted] (-2,2.5) -- (-2,-2.5)node[below] {\scriptsize{$r_*$}};
        \draw[black, loosely dotted] (-1,2.5) -- (-1,-2.5);
        \draw[black, loosely dotted] (0,2.5) -- (0,-2.5);
        \draw[black, loosely dotted] (1,2.5) -- (1,-2.5) node[below] {\scriptsize{$r$}};
        \draw[black, loosely dotted] (2,2.5) -- (2,-2.5);
        \draw[black, loosely dotted] (2.5,-2) -- (-2.5,-2);
        \draw[black, loosely dotted] (2.5,-1) -- (-2.5,-1)node[left] {\scriptsize{$s$}};
        \draw[black, loosely dotted] (2.5,0) -- (-2.5,0);
        \draw[black, loosely dotted] (2.5,1) -- (-2.5,1);
        \draw[black, loosely dotted] (2.5,2) -- (-2.5,2);
        \filldraw[black] (0,0) circle (2pt) node {};
        \filldraw[black] (0,-1) circle (2pt) node {};
        \filldraw[black] (1,-2) circle (2pt) node {};
        \filldraw[black] (-1,0) circle (2pt) node {};
        \filldraw[black] (-2,1) circle (2pt) node {};
        \filldraw[black] (1,-1) circle (2pt) node {};
        \filldraw[black] (2,-2) circle (2pt) node {};
        \filldraw[black] (-1,1) circle (2pt) node {};
        \filldraw[black] (-2,2) circle (2pt) node {};
        \node[right,rotate=-45] (n+1) at (1,-1) {$]$};
        \node[right,rotate=-45] (n) at (0,-1) {$]$};
        \node[rotate=270] (first) at (0.5,-1) {\scriptsize{$\ge$}};
        \node (second) at (1,-1.5) {\scriptsize{$\ge$}};
        \node[rotate=270, yscale = 1] (equal) at (-0.5,0) {\scriptsize{$=$}};
        \draw[cyan, -] (0,0) -- (1,-1);
        \node (left) at (1.5,-0.5) {\tiny{\color{cyan}$\frac{\chat\rs}{\chat\rmsp} = 1$}};
        \draw[magenta, -] (1,-1) -- (2,-2);
        \node (left) at (2.2,-1.4) {\tiny{\color{magenta}$\frac{\cbar\rpsm}{\cbar\rs}$}};
        \draw[teal, -] (0,-1) -- (1,-2);
        \node (left) at (0.0,-1.7) {\tiny{\color{teal}$\frac{\chat\rsm}{\chat\rms}$}};
        \draw[black,-] (-1,0) -- (0,-1);
        \node (left) at (0.1,0.75) {\tiny{\color{black}$\frac{\chat\rmsp}{\chat_{r-2,s+2}}$}};
        \draw[black,-] (-1,1) -- (0,0);
        \node (left) at (-1.0,-0.7) {\tiny{\color{black}$\frac{\chat\rms}{\chat_{r-2,s+1}}$}};
        \draw[black,dotted, thick] (-1.8,1.8) -- (-1.2,1.2);
        \draw[black,dotted, thick] (-1.8,0.8) -- (-1.2,0.2);
        \node[text width=2.7cm,below right,] at (-5.45,1.0) {\small d) The case\\ $r = r^m(n)+1$.};
    \end{scope}
  \end{tikzpicture}
  \caption{Pictographic description of the construction and the inequalities we have to proof. The more involved arguments have been marked with a ``$?$'' and will differ for the cases a)--d).}
\end{figure}

%% file: lemmas/constructionLeftMaximum.tex
Given the assumptions and $\chat$ as in Lemma \ref{constructionIsConcave}. Then, we have 
\[ r_m(n+1,\chat) \le \max\big(r_*(n+1), r_m(n,\chat)\big) \text{ for all }n\ge M.\]

%% file: proofs/constructionLeftMaximum.tex
Let $n\ge M$ and $r = \max\big(r_*(n+1), r_m(n,\chat)\big)$, as well as $s=n+1-r$. If $r=n+1$, there is nothing to show. Otherwise, by the concavity of $\chat\rs$, it suffices to show $\frac{\chat\rpsm}{\chat\rs}\le 1.$ We distinguish two cases. \\
If $r+1 > \max(r_*(n+1),r^m(n))$ we obtain 
\[\frac{\chat\rpsm}{\chat\rs}\overset{\text{construction.iii}} = \min\left(1,\frac{\cbar\rpsm}{\cbar\rs}\right) \le 1.\]
On the other hand, if $r+1 \le \max(r_*(n+1),r^m(n))$ we use $r\ge r_*(n+1)$, so that Lemma \ref{constructionIsConcave}.\ref{constructionIsConcave:large} implies $\chat\rs \ge e^{-\delta_1(r+s)}$. Furthermore, $r\ge r_*(n+1)$ together with $r+1 \le \max(r_*(n+1),r^m(n))$ yields $\max(r_*(n+1),r^m(n)) = r^m(n).$ Together with $r+1\le n \implies s\ge2$ we find
\[\chat\rpsm \overset{\text{construction.ii}} = \max\Big(\chat\rs \underbrace{\frac{\chat\rpsmm}{\chat\rsm}}_{\mathclap{\ge 1, \text{ since }r+1\le r^m(n)}},e^{-\delta_1(r+s)}\Big) = \chat\rs {\frac{\chat\rpsmm}{\chat\rsm}},\]
which proves $\displaystyle\frac{\chat\rpsm}{\chat\rs}= \frac{\chat\rpsmm}{\chat\rsm} \overset{r\ge r_m(n)} \le 1.$

%% file: pictures/constructioncHat.tex
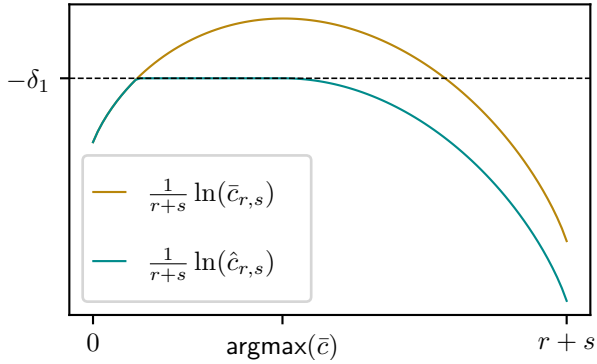
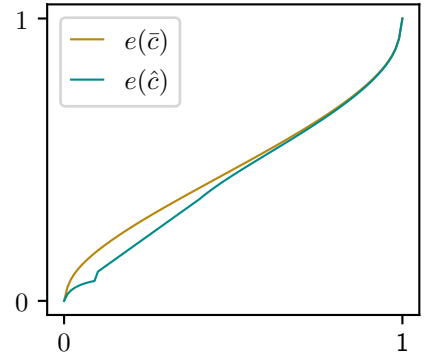
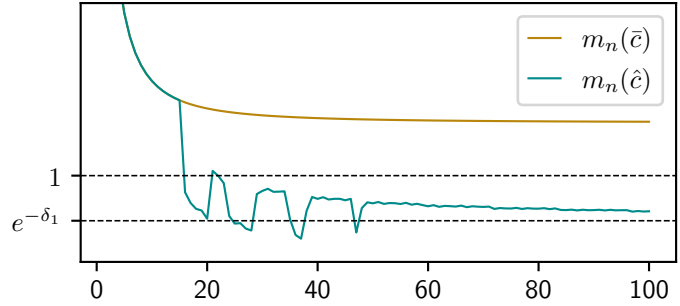
\begin{figure}[H]
  \centering
  \begin{subfigure}{0.55\textwidth}
    \centering
    \input{\CommonPath/pictures/chatConstruction.pgf}
    \caption{Behaviour of $\chat\rs$ and $\cbar\rs$ at $r+s = 100;$ $\chat$ appears to be constant between $r_*$ and $r^m$. }
  \end{subfigure}
  \hfill
  \begin{subfigure}{0.4\textwidth}
    \centering
    \input{\CommonPath/pictures/ehatConstruction.pgf}
    \caption{Lemma \ref{eres} applied to $\cbar$ and $\chat$ yields $e(\cbar)$ and $e(\chat).$ Here shown at $r+s = 99.$}
  \end{subfigure}
  \begin{subfigure}{0.95\textwidth}
    \centering
    \input{\CommonPath/pictures/mhatConstruction.pgf}
    \caption{We have plotted $n\mapsto m_n$ for $\cbar$ and $\chat.$}
  \end{subfigure}
  \caption{The construction from Lemma \ref{constructionIsConcave} for $Q\rs = \binom{r+s}{r}^{0.3} e^{-\frac{5}{100}(r+s-1)^{\frac 2 3}}$, $\monr = 0.8$, $\mons = 0.9,$ $M=15$ and $\delta_1 = 0.04$ up to $r+s = 100.$}
  \label{figure:constructionChat}
\end{figure} 

%% file: pictures/chatConstruction.pgf
\begingroup%
\makeatletter%
\begin{pgfpicture}%
\pgfpathrectangle{\pgfpointorigin}{\pgfqpoint{3.500000in}{2.100000in}}%
\pgfusepath{use as bounding box, clip}%
\begin{pgfscope}%
\pgfsetbuttcap%
\pgfsetmiterjoin%
\definecolor{currentfill}{rgb}{1.000000,1.000000,1.000000}%
\pgfsetfillcolor{currentfill}%
\pgfsetlinewidth{0.000000pt}%
\definecolor{currentstroke}{rgb}{1.000000,1.000000,1.000000}%
\pgfsetstrokecolor{currentstroke}%
\pgfsetdash{}{0pt}%
\pgfpathmoveto{\pgfqpoint{0.000000in}{0.000000in}}%
\pgfpathlineto{\pgfqpoint{3.500000in}{0.000000in}}%
\pgfpathlineto{\pgfqpoint{3.500000in}{2.100000in}}%
\pgfpathlineto{\pgfqpoint{0.000000in}{2.100000in}}%
\pgfpathlineto{\pgfqpoint{0.000000in}{0.000000in}}%
\pgfpathclose%
\pgfusepath{fill}%
\end{pgfscope}%
\begin{pgfscope}%
\pgfsetbuttcap%
\pgfsetmiterjoin%
\definecolor{currentfill}{rgb}{1.000000,1.000000,1.000000}%
\pgfsetfillcolor{currentfill}%
\pgfsetlinewidth{0.000000pt}%
\definecolor{currentstroke}{rgb}{0.000000,0.000000,0.000000}%
\pgfsetstrokecolor{currentstroke}%
\pgfsetstrokeopacity{0.000000}%
\pgfsetdash{}{0pt}%
\pgfpathmoveto{\pgfqpoint{0.437500in}{0.231000in}}%
\pgfpathlineto{\pgfqpoint{3.150000in}{0.231000in}}%
\pgfpathlineto{\pgfqpoint{3.150000in}{1.848000in}}%
\pgfpathlineto{\pgfqpoint{0.437500in}{1.848000in}}%
\pgfpathlineto{\pgfqpoint{0.437500in}{0.231000in}}%
\pgfpathclose%
\pgfusepath{fill}%
\end{pgfscope}%
\begin{pgfscope}%
\pgfsetbuttcap%
\pgfsetroundjoin%
\definecolor{currentfill}{rgb}{0.000000,0.000000,0.000000}%
\pgfsetfillcolor{currentfill}%
\pgfsetlinewidth{0.803000pt}%
\definecolor{currentstroke}{rgb}{0.000000,0.000000,0.000000}%
\pgfsetstrokecolor{currentstroke}%
\pgfsetdash{}{0pt}%
\pgfsys@defobject{currentmarker}{\pgfqpoint{0.000000in}{-0.048611in}}{\pgfqpoint{0.000000in}{0.000000in}}{%
\pgfpathmoveto{\pgfqpoint{0.000000in}{0.000000in}}%
\pgfpathlineto{\pgfqpoint{0.000000in}{-0.048611in}}%
\pgfusepath{stroke,fill}%
}%
\begin{pgfscope}%
\pgfsys@transformshift{0.560795in}{0.231000in}%
\pgfsys@useobject{currentmarker}{}%
\end{pgfscope}%
\end{pgfscope}%
\begin{pgfscope}%
\definecolor{textcolor}{rgb}{0.000000,0.000000,0.000000}%
\pgfsetstrokecolor{textcolor}%
\pgfsetfillcolor{textcolor}%
\pgftext[x=0.560795in,y=0.133778in,,top]{\color{textcolor}{\sffamily\fontsize{10.000000}{12.000000}\selectfont\catcode`\^=\active\def^{\ifmmode\sp\else\^{}\fi}\catcode`\%=\active\def
\end{pgfscope}%
\begin{pgfscope}%
\pgfsetbuttcap%
\pgfsetroundjoin%
\definecolor{currentfill}{rgb}{0.000000,0.000000,0.000000}%
\pgfsetfillcolor{currentfill}%
\pgfsetlinewidth{0.803000pt}%
\definecolor{currentstroke}{rgb}{0.000000,0.000000,0.000000}%
\pgfsetstrokecolor{currentstroke}%
\pgfsetdash{}{0pt}%
\pgfsys@defobject{currentmarker}{\pgfqpoint{0.000000in}{-0.048611in}}{\pgfqpoint{0.000000in}{0.000000in}}{%
\pgfpathmoveto{\pgfqpoint{0.000000in}{0.000000in}}%
\pgfpathlineto{\pgfqpoint{0.000000in}{-0.048611in}}%
\pgfusepath{stroke,fill}%
}%
\begin{pgfscope}%
\pgfsys@transformshift{1.547159in}{0.231000in}%
\pgfsys@useobject{currentmarker}{}%
\end{pgfscope}%
\end{pgfscope}%
\begin{pgfscope}%
\definecolor{textcolor}{rgb}{0.000000,0.000000,0.000000}%
\pgfsetstrokecolor{textcolor}%
\pgfsetfillcolor{textcolor}%
\pgftext[x=1.547159in,y=0.133778in,,top]{\color{textcolor}{\sffamily\fontsize{10.000000}{12.000000}\selectfont\catcode`\^=\active\def^{\ifmmode\sp\else\^{}\fi}\catcode`\%=\active\def
\end{pgfscope}%
\begin{pgfscope}%
\pgfsetbuttcap%
\pgfsetroundjoin%
\definecolor{currentfill}{rgb}{0.000000,0.000000,0.000000}%
\pgfsetfillcolor{currentfill}%
\pgfsetlinewidth{0.803000pt}%
\definecolor{currentstroke}{rgb}{0.000000,0.000000,0.000000}%
\pgfsetstrokecolor{currentstroke}%
\pgfsetdash{}{0pt}%
\pgfsys@defobject{currentmarker}{\pgfqpoint{0.000000in}{-0.048611in}}{\pgfqpoint{0.000000in}{0.000000in}}{%
\pgfpathmoveto{\pgfqpoint{0.000000in}{0.000000in}}%
\pgfpathlineto{\pgfqpoint{0.000000in}{-0.048611in}}%
\pgfusepath{stroke,fill}%
}%
\begin{pgfscope}%
\pgfsys@transformshift{3.026705in}{0.231000in}%
\pgfsys@useobject{currentmarker}{}%
\end{pgfscope}%
\end{pgfscope}%
\begin{pgfscope}%
\definecolor{textcolor}{rgb}{0.000000,0.000000,0.000000}%
\pgfsetstrokecolor{textcolor}%
\pgfsetfillcolor{textcolor}%
\pgftext[x=3.026705in,y=0.133778in,,top]{\color{textcolor}{\sffamily\fontsize{10.000000}{12.000000}\selectfont\catcode`\^=\active\def^{\ifmmode\sp\else\^{}\fi}\catcode`\%=\active\def
\end{pgfscope}%
\begin{pgfscope}%
\pgfsetbuttcap%
\pgfsetroundjoin%
\definecolor{currentfill}{rgb}{0.000000,0.000000,0.000000}%
\pgfsetfillcolor{currentfill}%
\pgfsetlinewidth{0.803000pt}%
\definecolor{currentstroke}{rgb}{0.000000,0.000000,0.000000}%
\pgfsetstrokecolor{currentstroke}%
\pgfsetdash{}{0pt}%
\pgfsys@defobject{currentmarker}{\pgfqpoint{-0.048611in}{0.000000in}}{\pgfqpoint{-0.000000in}{0.000000in}}{%
\pgfpathmoveto{\pgfqpoint{-0.000000in}{0.000000in}}%
\pgfpathlineto{\pgfqpoint{-0.048611in}{0.000000in}}%
\pgfusepath{stroke,fill}%
}%
\begin{pgfscope}%
\pgfsys@transformshift{0.437500in}{1.463232in}%
\pgfsys@useobject{currentmarker}{}%
\end{pgfscope}%
\end{pgfscope}%
\begin{pgfscope}%
\definecolor{textcolor}{rgb}{0.000000,0.000000,0.000000}%
\pgfsetstrokecolor{textcolor}%
\pgfsetfillcolor{textcolor}%
\pgftext[x=0.108217in, y=1.410470in, left, base]{\color{textcolor}{\sffamily\fontsize{10.000000}{12.000000}\selectfont\catcode`\^=\active\def^{\ifmmode\sp\else\^{}\fi}\catcode`\%=\active\def
\end{pgfscope}%
\begin{pgfscope}%
\pgfpathrectangle{\pgfqpoint{0.437500in}{0.231000in}}{\pgfqpoint{2.712500in}{1.617000in}}%
\pgfusepath{clip}%
\pgfsetrectcap%
\pgfsetroundjoin%
\pgfsetlinewidth{0.803000pt}%
\definecolor{currentstroke}{rgb}{0.721569,0.525490,0.043137}%
\pgfsetstrokecolor{currentstroke}%
\pgfsetdash{}{0pt}%
\pgfpathmoveto{\pgfqpoint{0.560795in}{1.130703in}}%
\pgfpathlineto{\pgfqpoint{0.585455in}{1.185953in}}%
\pgfpathlineto{\pgfqpoint{0.610114in}{1.231981in}}%
\pgfpathlineto{\pgfqpoint{0.634773in}{1.272557in}}%
\pgfpathlineto{\pgfqpoint{0.659432in}{1.309226in}}%
\pgfpathlineto{\pgfqpoint{0.684091in}{1.342832in}}%
\pgfpathlineto{\pgfqpoint{0.708750in}{1.373910in}}%
\pgfpathlineto{\pgfqpoint{0.733409in}{1.402827in}}%
\pgfpathlineto{\pgfqpoint{0.758068in}{1.429852in}}%
\pgfpathlineto{\pgfqpoint{0.782727in}{1.455191in}}%
\pgfpathlineto{\pgfqpoint{0.807386in}{1.479005in}}%
\pgfpathlineto{\pgfqpoint{0.832045in}{1.501424in}}%
\pgfpathlineto{\pgfqpoint{0.856705in}{1.522555in}}%
\pgfpathlineto{\pgfqpoint{0.881364in}{1.542487in}}%
\pgfpathlineto{\pgfqpoint{0.906023in}{1.561298in}}%
\pgfpathlineto{\pgfqpoint{0.930682in}{1.579053in}}%
\pgfpathlineto{\pgfqpoint{0.955341in}{1.595808in}}%
\pgfpathlineto{\pgfqpoint{0.980000in}{1.611612in}}%
\pgfpathlineto{\pgfqpoint{1.004659in}{1.626509in}}%
\pgfpathlineto{\pgfqpoint{1.029318in}{1.640539in}}%
\pgfpathlineto{\pgfqpoint{1.053977in}{1.653735in}}%
\pgfpathlineto{\pgfqpoint{1.078636in}{1.666128in}}%
\pgfpathlineto{\pgfqpoint{1.103295in}{1.677745in}}%
\pgfpathlineto{\pgfqpoint{1.127955in}{1.688613in}}%
\pgfpathlineto{\pgfqpoint{1.152614in}{1.698753in}}%
\pgfpathlineto{\pgfqpoint{1.177273in}{1.708187in}}%
\pgfpathlineto{\pgfqpoint{1.201932in}{1.716932in}}%
\pgfpathlineto{\pgfqpoint{1.226591in}{1.725006in}}%
\pgfpathlineto{\pgfqpoint{1.251250in}{1.732425in}}%
\pgfpathlineto{\pgfqpoint{1.275909in}{1.739203in}}%
\pgfpathlineto{\pgfqpoint{1.300568in}{1.745352in}}%
\pgfpathlineto{\pgfqpoint{1.325227in}{1.750886in}}%
\pgfpathlineto{\pgfqpoint{1.349886in}{1.755814in}}%
\pgfpathlineto{\pgfqpoint{1.374545in}{1.760147in}}%
\pgfpathlineto{\pgfqpoint{1.399205in}{1.763895in}}%
\pgfpathlineto{\pgfqpoint{1.423864in}{1.767065in}}%
\pgfpathlineto{\pgfqpoint{1.448523in}{1.769665in}}%
\pgfpathlineto{\pgfqpoint{1.473182in}{1.771703in}}%
\pgfpathlineto{\pgfqpoint{1.497841in}{1.773184in}}%
\pgfpathlineto{\pgfqpoint{1.522500in}{1.774115in}}%
\pgfpathlineto{\pgfqpoint{1.547159in}{1.774500in}}%
\pgfpathlineto{\pgfqpoint{1.571818in}{1.774345in}}%
\pgfpathlineto{\pgfqpoint{1.596477in}{1.773653in}}%
\pgfpathlineto{\pgfqpoint{1.621136in}{1.772428in}}%
\pgfpathlineto{\pgfqpoint{1.645795in}{1.770674in}}%
\pgfpathlineto{\pgfqpoint{1.670455in}{1.768393in}}%
\pgfpathlineto{\pgfqpoint{1.695114in}{1.765587in}}%
\pgfpathlineto{\pgfqpoint{1.719773in}{1.762259in}}%
\pgfpathlineto{\pgfqpoint{1.744432in}{1.758409in}}%
\pgfpathlineto{\pgfqpoint{1.769091in}{1.754039in}}%
\pgfpathlineto{\pgfqpoint{1.793750in}{1.749150in}}%
\pgfpathlineto{\pgfqpoint{1.818409in}{1.743741in}}%
\pgfpathlineto{\pgfqpoint{1.843068in}{1.737812in}}%
\pgfpathlineto{\pgfqpoint{1.867727in}{1.731363in}}%
\pgfpathlineto{\pgfqpoint{1.892386in}{1.724393in}}%
\pgfpathlineto{\pgfqpoint{1.917045in}{1.716900in}}%
\pgfpathlineto{\pgfqpoint{1.941705in}{1.708882in}}%
\pgfpathlineto{\pgfqpoint{1.966364in}{1.700338in}}%
\pgfpathlineto{\pgfqpoint{1.991023in}{1.691264in}}%
\pgfpathlineto{\pgfqpoint{2.015682in}{1.681657in}}%
\pgfpathlineto{\pgfqpoint{2.040341in}{1.671513in}}%
\pgfpathlineto{\pgfqpoint{2.065000in}{1.660829in}}%
\pgfpathlineto{\pgfqpoint{2.089659in}{1.649600in}}%
\pgfpathlineto{\pgfqpoint{2.114318in}{1.637820in}}%
\pgfpathlineto{\pgfqpoint{2.138977in}{1.625483in}}%
\pgfpathlineto{\pgfqpoint{2.163636in}{1.612585in}}%
\pgfpathlineto{\pgfqpoint{2.188295in}{1.599116in}}%
\pgfpathlineto{\pgfqpoint{2.212955in}{1.585070in}}%
\pgfpathlineto{\pgfqpoint{2.237614in}{1.570438in}}%
\pgfpathlineto{\pgfqpoint{2.262273in}{1.555211in}}%
\pgfpathlineto{\pgfqpoint{2.286932in}{1.539379in}}%
\pgfpathlineto{\pgfqpoint{2.311591in}{1.522930in}}%
\pgfpathlineto{\pgfqpoint{2.336250in}{1.505854in}}%
\pgfpathlineto{\pgfqpoint{2.360909in}{1.488136in}}%
\pgfpathlineto{\pgfqpoint{2.385568in}{1.469764in}}%
\pgfpathlineto{\pgfqpoint{2.410227in}{1.450720in}}%
\pgfpathlineto{\pgfqpoint{2.434886in}{1.430987in}}%
\pgfpathlineto{\pgfqpoint{2.459545in}{1.410549in}}%
\pgfpathlineto{\pgfqpoint{2.484205in}{1.389382in}}%
\pgfpathlineto{\pgfqpoint{2.508864in}{1.367466in}}%
\pgfpathlineto{\pgfqpoint{2.533523in}{1.344774in}}%
\pgfpathlineto{\pgfqpoint{2.558182in}{1.321280in}}%
\pgfpathlineto{\pgfqpoint{2.582841in}{1.296952in}}%
\pgfpathlineto{\pgfqpoint{2.607500in}{1.271755in}}%
\pgfpathlineto{\pgfqpoint{2.632159in}{1.245652in}}%
\pgfpathlineto{\pgfqpoint{2.656818in}{1.218599in}}%
\pgfpathlineto{\pgfqpoint{2.681477in}{1.190546in}}%
\pgfpathlineto{\pgfqpoint{2.706136in}{1.161436in}}%
\pgfpathlineto{\pgfqpoint{2.730795in}{1.131205in}}%
\pgfpathlineto{\pgfqpoint{2.755455in}{1.099775in}}%
\pgfpathlineto{\pgfqpoint{2.780114in}{1.067058in}}%
\pgfpathlineto{\pgfqpoint{2.804773in}{1.032945in}}%
\pgfpathlineto{\pgfqpoint{2.829432in}{0.997308in}}%
\pgfpathlineto{\pgfqpoint{2.854091in}{0.959984in}}%
\pgfpathlineto{\pgfqpoint{2.878750in}{0.920768in}}%
\pgfpathlineto{\pgfqpoint{2.903409in}{0.879392in}}%
\pgfpathlineto{\pgfqpoint{2.928068in}{0.835487in}}%
\pgfpathlineto{\pgfqpoint{2.952727in}{0.788519in}}%
\pgfpathlineto{\pgfqpoint{2.977386in}{0.737644in}}%
\pgfpathlineto{\pgfqpoint{3.002045in}{0.681318in}}%
\pgfpathlineto{\pgfqpoint{3.026705in}{0.615768in}}%
\pgfusepath{stroke}%
\end{pgfscope}%
\begin{pgfscope}%
\pgfpathrectangle{\pgfqpoint{0.437500in}{0.231000in}}{\pgfqpoint{2.712500in}{1.617000in}}%
\pgfusepath{clip}%
\pgfsetrectcap%
\pgfsetroundjoin%
\pgfsetlinewidth{0.803000pt}%
\definecolor{currentstroke}{rgb}{0.000000,0.545098,0.545098}%
\pgfsetstrokecolor{currentstroke}%
\pgfsetdash{}{0pt}%
\pgfpathmoveto{\pgfqpoint{0.560795in}{1.130703in}}%
\pgfpathlineto{\pgfqpoint{0.585455in}{1.185953in}}%
\pgfpathlineto{\pgfqpoint{0.610114in}{1.231981in}}%
\pgfpathlineto{\pgfqpoint{0.634773in}{1.272557in}}%
\pgfpathlineto{\pgfqpoint{0.659432in}{1.309226in}}%
\pgfpathlineto{\pgfqpoint{0.684091in}{1.342832in}}%
\pgfpathlineto{\pgfqpoint{0.708750in}{1.373910in}}%
\pgfpathlineto{\pgfqpoint{0.733409in}{1.402827in}}%
\pgfpathlineto{\pgfqpoint{0.758068in}{1.429852in}}%
\pgfpathlineto{\pgfqpoint{0.782727in}{1.455191in}}%
\pgfpathlineto{\pgfqpoint{0.807386in}{1.463232in}}%
\pgfpathlineto{\pgfqpoint{0.832045in}{1.463232in}}%
\pgfpathlineto{\pgfqpoint{0.856705in}{1.463232in}}%
\pgfpathlineto{\pgfqpoint{0.881364in}{1.463232in}}%
\pgfpathlineto{\pgfqpoint{0.906023in}{1.463232in}}%
\pgfpathlineto{\pgfqpoint{0.930682in}{1.463232in}}%
\pgfpathlineto{\pgfqpoint{0.955341in}{1.463232in}}%
\pgfpathlineto{\pgfqpoint{0.980000in}{1.463232in}}%
\pgfpathlineto{\pgfqpoint{1.004659in}{1.463232in}}%
\pgfpathlineto{\pgfqpoint{1.029318in}{1.463232in}}%
\pgfpathlineto{\pgfqpoint{1.053977in}{1.463232in}}%
\pgfpathlineto{\pgfqpoint{1.078636in}{1.463232in}}%
\pgfpathlineto{\pgfqpoint{1.103295in}{1.463232in}}%
\pgfpathlineto{\pgfqpoint{1.127955in}{1.463232in}}%
\pgfpathlineto{\pgfqpoint{1.152614in}{1.463232in}}%
\pgfpathlineto{\pgfqpoint{1.177273in}{1.463232in}}%
\pgfpathlineto{\pgfqpoint{1.201932in}{1.463232in}}%
\pgfpathlineto{\pgfqpoint{1.226591in}{1.463232in}}%
\pgfpathlineto{\pgfqpoint{1.251250in}{1.463232in}}%
\pgfpathlineto{\pgfqpoint{1.275909in}{1.463232in}}%
\pgfpathlineto{\pgfqpoint{1.300568in}{1.463232in}}%
\pgfpathlineto{\pgfqpoint{1.325227in}{1.463232in}}%
\pgfpathlineto{\pgfqpoint{1.349886in}{1.463232in}}%
\pgfpathlineto{\pgfqpoint{1.374545in}{1.463232in}}%
\pgfpathlineto{\pgfqpoint{1.399205in}{1.463232in}}%
\pgfpathlineto{\pgfqpoint{1.423864in}{1.463232in}}%
\pgfpathlineto{\pgfqpoint{1.448523in}{1.463232in}}%
\pgfpathlineto{\pgfqpoint{1.473182in}{1.463232in}}%
\pgfpathlineto{\pgfqpoint{1.497841in}{1.463232in}}%
\pgfpathlineto{\pgfqpoint{1.522500in}{1.463232in}}%
\pgfpathlineto{\pgfqpoint{1.547159in}{1.463232in}}%
\pgfpathlineto{\pgfqpoint{1.571818in}{1.463076in}}%
\pgfpathlineto{\pgfqpoint{1.596477in}{1.462385in}}%
\pgfpathlineto{\pgfqpoint{1.621136in}{1.461160in}}%
\pgfpathlineto{\pgfqpoint{1.645795in}{1.459406in}}%
\pgfpathlineto{\pgfqpoint{1.670455in}{1.457125in}}%
\pgfpathlineto{\pgfqpoint{1.695114in}{1.454319in}}%
\pgfpathlineto{\pgfqpoint{1.719773in}{1.450991in}}%
\pgfpathlineto{\pgfqpoint{1.744432in}{1.447141in}}%
\pgfpathlineto{\pgfqpoint{1.769091in}{1.442771in}}%
\pgfpathlineto{\pgfqpoint{1.793750in}{1.437881in}}%
\pgfpathlineto{\pgfqpoint{1.818409in}{1.432472in}}%
\pgfpathlineto{\pgfqpoint{1.843068in}{1.426544in}}%
\pgfpathlineto{\pgfqpoint{1.867727in}{1.420095in}}%
\pgfpathlineto{\pgfqpoint{1.892386in}{1.413124in}}%
\pgfpathlineto{\pgfqpoint{1.917045in}{1.405631in}}%
\pgfpathlineto{\pgfqpoint{1.941705in}{1.397614in}}%
\pgfpathlineto{\pgfqpoint{1.966364in}{1.389069in}}%
\pgfpathlineto{\pgfqpoint{1.991023in}{1.379995in}}%
\pgfpathlineto{\pgfqpoint{2.015682in}{1.370388in}}%
\pgfpathlineto{\pgfqpoint{2.040341in}{1.360245in}}%
\pgfpathlineto{\pgfqpoint{2.065000in}{1.349561in}}%
\pgfpathlineto{\pgfqpoint{2.089659in}{1.338331in}}%
\pgfpathlineto{\pgfqpoint{2.114318in}{1.326551in}}%
\pgfpathlineto{\pgfqpoint{2.138977in}{1.314215in}}%
\pgfpathlineto{\pgfqpoint{2.163636in}{1.301316in}}%
\pgfpathlineto{\pgfqpoint{2.188295in}{1.287847in}}%
\pgfpathlineto{\pgfqpoint{2.212955in}{1.273801in}}%
\pgfpathlineto{\pgfqpoint{2.237614in}{1.259169in}}%
\pgfpathlineto{\pgfqpoint{2.262273in}{1.243942in}}%
\pgfpathlineto{\pgfqpoint{2.286932in}{1.228110in}}%
\pgfpathlineto{\pgfqpoint{2.311591in}{1.211662in}}%
\pgfpathlineto{\pgfqpoint{2.336250in}{1.194586in}}%
\pgfpathlineto{\pgfqpoint{2.360909in}{1.176868in}}%
\pgfpathlineto{\pgfqpoint{2.385568in}{1.158495in}}%
\pgfpathlineto{\pgfqpoint{2.410227in}{1.139451in}}%
\pgfpathlineto{\pgfqpoint{2.434886in}{1.119719in}}%
\pgfpathlineto{\pgfqpoint{2.459545in}{1.099280in}}%
\pgfpathlineto{\pgfqpoint{2.484205in}{1.078114in}}%
\pgfpathlineto{\pgfqpoint{2.508864in}{1.056197in}}%
\pgfpathlineto{\pgfqpoint{2.533523in}{1.033506in}}%
\pgfpathlineto{\pgfqpoint{2.558182in}{1.010011in}}%
\pgfpathlineto{\pgfqpoint{2.582841in}{0.985683in}}%
\pgfpathlineto{\pgfqpoint{2.607500in}{0.960487in}}%
\pgfpathlineto{\pgfqpoint{2.632159in}{0.934384in}}%
\pgfpathlineto{\pgfqpoint{2.656818in}{0.907331in}}%
\pgfpathlineto{\pgfqpoint{2.681477in}{0.879277in}}%
\pgfpathlineto{\pgfqpoint{2.706136in}{0.850168in}}%
\pgfpathlineto{\pgfqpoint{2.730795in}{0.819936in}}%
\pgfpathlineto{\pgfqpoint{2.755455in}{0.788507in}}%
\pgfpathlineto{\pgfqpoint{2.780114in}{0.755789in}}%
\pgfpathlineto{\pgfqpoint{2.804773in}{0.721677in}}%
\pgfpathlineto{\pgfqpoint{2.829432in}{0.686039in}}%
\pgfpathlineto{\pgfqpoint{2.854091in}{0.648715in}}%
\pgfpathlineto{\pgfqpoint{2.878750in}{0.609500in}}%
\pgfpathlineto{\pgfqpoint{2.903409in}{0.568123in}}%
\pgfpathlineto{\pgfqpoint{2.928068in}{0.524218in}}%
\pgfpathlineto{\pgfqpoint{2.952727in}{0.477251in}}%
\pgfpathlineto{\pgfqpoint{2.977386in}{0.426375in}}%
\pgfpathlineto{\pgfqpoint{3.002045in}{0.370049in}}%
\pgfpathlineto{\pgfqpoint{3.026705in}{0.304500in}}%
\pgfusepath{stroke}%
\end{pgfscope}%
\begin{pgfscope}%
\pgfpathrectangle{\pgfqpoint{0.437500in}{0.231000in}}{\pgfqpoint{2.712500in}{1.617000in}}%
\pgfusepath{clip}%
\pgfsetbuttcap%
\pgfsetroundjoin%
\pgfsetlinewidth{0.602250pt}%
\definecolor{currentstroke}{rgb}{0.000000,0.000000,0.000000}%
\pgfsetstrokecolor{currentstroke}%
\pgfsetdash{{2.220000pt}{0.960000pt}}{0.000000pt}%
\pgfpathmoveto{\pgfqpoint{0.437500in}{1.463232in}}%
\pgfpathlineto{\pgfqpoint{3.150000in}{1.463232in}}%
\pgfusepath{stroke}%
\end{pgfscope}%
\begin{pgfscope}%
\pgfsetrectcap%
\pgfsetmiterjoin%
\pgfsetlinewidth{0.803000pt}%
\definecolor{currentstroke}{rgb}{0.000000,0.000000,0.000000}%
\pgfsetstrokecolor{currentstroke}%
\pgfsetdash{}{0pt}%
\pgfpathmoveto{\pgfqpoint{0.437500in}{0.231000in}}%
\pgfpathlineto{\pgfqpoint{0.437500in}{1.848000in}}%
\pgfusepath{stroke}%
\end{pgfscope}%
\begin{pgfscope}%
\pgfsetrectcap%
\pgfsetmiterjoin%
\pgfsetlinewidth{0.803000pt}%
\definecolor{currentstroke}{rgb}{0.000000,0.000000,0.000000}%
\pgfsetstrokecolor{currentstroke}%
\pgfsetdash{}{0pt}%
\pgfpathmoveto{\pgfqpoint{3.150000in}{0.231000in}}%
\pgfpathlineto{\pgfqpoint{3.150000in}{1.848000in}}%
\pgfusepath{stroke}%
\end{pgfscope}%
\begin{pgfscope}%
\pgfsetrectcap%
\pgfsetmiterjoin%
\pgfsetlinewidth{0.803000pt}%
\definecolor{currentstroke}{rgb}{0.000000,0.000000,0.000000}%
\pgfsetstrokecolor{currentstroke}%
\pgfsetdash{}{0pt}%
\pgfpathmoveto{\pgfqpoint{0.437500in}{0.231000in}}%
\pgfpathlineto{\pgfqpoint{3.150000in}{0.231000in}}%
\pgfusepath{stroke}%
\end{pgfscope}%
\begin{pgfscope}%
\pgfsetrectcap%
\pgfsetmiterjoin%
\pgfsetlinewidth{0.803000pt}%
\definecolor{currentstroke}{rgb}{0.000000,0.000000,0.000000}%
\pgfsetstrokecolor{currentstroke}%
\pgfsetdash{}{0pt}%
\pgfpathmoveto{\pgfqpoint{0.437500in}{1.848000in}}%
\pgfpathlineto{\pgfqpoint{3.150000in}{1.848000in}}%
\pgfusepath{stroke}%
\end{pgfscope}%
\begin{pgfscope}%
\pgfsetbuttcap%
\pgfsetmiterjoin%
\definecolor{currentfill}{rgb}{1.000000,1.000000,1.000000}%
\pgfsetfillcolor{currentfill}%
\pgfsetfillopacity{0.800000}%
\pgfsetlinewidth{1.003750pt}%
\definecolor{currentstroke}{rgb}{0.800000,0.800000,0.800000}%
\pgfsetstrokecolor{currentstroke}%
\pgfsetstrokeopacity{0.800000}%
\pgfsetdash{}{0pt}%
\pgfpathmoveto{\pgfqpoint{0.534722in}{0.300444in}}%
\pgfpathlineto{\pgfqpoint{1.668843in}{0.300444in}}%
\pgfpathquadraticcurveto{\pgfqpoint{1.696621in}{0.300444in}}{\pgfqpoint{1.696621in}{0.328222in}}%
\pgfpathlineto{\pgfqpoint{1.696621in}{1.033843in}}%
\pgfpathquadraticcurveto{\pgfqpoint{1.696621in}{1.061621in}}{\pgfqpoint{1.668843in}{1.061621in}}%
\pgfpathlineto{\pgfqpoint{0.534722in}{1.061621in}}%
\pgfpathquadraticcurveto{\pgfqpoint{0.506944in}{1.061621in}}{\pgfqpoint{0.506944in}{1.033843in}}%
\pgfpathlineto{\pgfqpoint{0.506944in}{0.328222in}}%
\pgfpathquadraticcurveto{\pgfqpoint{0.506944in}{0.300444in}}{\pgfqpoint{0.534722in}{0.300444in}}%
\pgfpathlineto{\pgfqpoint{0.534722in}{0.300444in}}%
\pgfpathclose%
\pgfusepath{stroke,fill}%
\end{pgfscope}%
\begin{pgfscope}%
\pgfsetrectcap%
\pgfsetroundjoin%
\pgfsetlinewidth{0.803000pt}%
\definecolor{currentstroke}{rgb}{0.721569,0.525490,0.043137}%
\pgfsetstrokecolor{currentstroke}%
\pgfsetdash{}{0pt}%
\pgfpathmoveto{\pgfqpoint{0.562500in}{0.871211in}}%
\pgfpathlineto{\pgfqpoint{0.645833in}{0.871211in}}%
\pgfpathlineto{\pgfqpoint{0.729167in}{0.871211in}}%
\pgfusepath{stroke}%
\end{pgfscope}%
\begin{pgfscope}%
\definecolor{textcolor}{rgb}{0.000000,0.000000,0.000000}%
\pgfsetstrokecolor{textcolor}%
\pgfsetfillcolor{textcolor}%
\pgftext[x=0.840278in,y=0.822600in,left,base]{\color{textcolor}{\sffamily\fontsize{10.000000}{12.000000}\selectfont\catcode`\^=\active\def^{\ifmmode\sp\else\^{}\fi}\catcode`\%=\active\def
\end{pgfscope}%
\begin{pgfscope}%
\pgfsetrectcap%
\pgfsetroundjoin%
\pgfsetlinewidth{0.803000pt}%
\definecolor{currentstroke}{rgb}{0.000000,0.545098,0.545098}%
\pgfsetstrokecolor{currentstroke}%
\pgfsetdash{}{0pt}%
\pgfpathmoveto{\pgfqpoint{0.562500in}{0.511456in}}%
\pgfpathlineto{\pgfqpoint{0.645833in}{0.511456in}}%
\pgfpathlineto{\pgfqpoint{0.729167in}{0.511456in}}%
\pgfusepath{stroke}%
\end{pgfscope}%
\begin{pgfscope}%
\definecolor{textcolor}{rgb}{0.000000,0.000000,0.000000}%
\pgfsetstrokecolor{textcolor}%
\pgfsetfillcolor{textcolor}%
\pgftext[x=0.840278in,y=0.462845in,left,base]{\color{textcolor}{\sffamily\fontsize{10.000000}{12.000000}\selectfont\catcode`\^=\active\def^{\ifmmode\sp\else\^{}\fi}\catcode`\%=\active\def
\end{pgfscope}%
\end{pgfpicture}%
\makeatother%
\endgroup%

%% file: pictures/ehatConstruction.pgf
\begingroup%
\makeatletter%
\begin{pgfpicture}%
\pgfpathrectangle{\pgfpointorigin}{\pgfqpoint{2.500000in}{2.100000in}}%
\pgfusepath{use as bounding box, clip}%
\begin{pgfscope}%
\pgfsetbuttcap%
\pgfsetmiterjoin%
\definecolor{currentfill}{rgb}{1.000000,1.000000,1.000000}%
\pgfsetfillcolor{currentfill}%
\pgfsetlinewidth{0.000000pt}%
\definecolor{currentstroke}{rgb}{1.000000,1.000000,1.000000}%
\pgfsetstrokecolor{currentstroke}%
\pgfsetdash{}{0pt}%
\pgfpathmoveto{\pgfqpoint{0.000000in}{0.000000in}}%
\pgfpathlineto{\pgfqpoint{2.500000in}{0.000000in}}%
\pgfpathlineto{\pgfqpoint{2.500000in}{2.100000in}}%
\pgfpathlineto{\pgfqpoint{0.000000in}{2.100000in}}%
\pgfpathlineto{\pgfqpoint{0.000000in}{0.000000in}}%
\pgfpathclose%
\pgfusepath{fill}%
\end{pgfscope}%
\begin{pgfscope}%
\pgfsetbuttcap%
\pgfsetmiterjoin%
\definecolor{currentfill}{rgb}{1.000000,1.000000,1.000000}%
\pgfsetfillcolor{currentfill}%
\pgfsetlinewidth{0.000000pt}%
\definecolor{currentstroke}{rgb}{0.000000,0.000000,0.000000}%
\pgfsetstrokecolor{currentstroke}%
\pgfsetstrokeopacity{0.000000}%
\pgfsetdash{}{0pt}%
\pgfpathmoveto{\pgfqpoint{0.312500in}{0.231000in}}%
\pgfpathlineto{\pgfqpoint{2.250000in}{0.231000in}}%
\pgfpathlineto{\pgfqpoint{2.250000in}{1.848000in}}%
\pgfpathlineto{\pgfqpoint{0.312500in}{1.848000in}}%
\pgfpathlineto{\pgfqpoint{0.312500in}{0.231000in}}%
\pgfpathclose%
\pgfusepath{fill}%
\end{pgfscope}%
\begin{pgfscope}%
\pgfsetbuttcap%
\pgfsetroundjoin%
\definecolor{currentfill}{rgb}{0.000000,0.000000,0.000000}%
\pgfsetfillcolor{currentfill}%
\pgfsetlinewidth{0.803000pt}%
\definecolor{currentstroke}{rgb}{0.000000,0.000000,0.000000}%
\pgfsetstrokecolor{currentstroke}%
\pgfsetdash{}{0pt}%
\pgfsys@defobject{currentmarker}{\pgfqpoint{0.000000in}{-0.048611in}}{\pgfqpoint{0.000000in}{0.000000in}}{%
\pgfpathmoveto{\pgfqpoint{0.000000in}{0.000000in}}%
\pgfpathlineto{\pgfqpoint{0.000000in}{-0.048611in}}%
\pgfusepath{stroke,fill}%
}%
\begin{pgfscope}%
\pgfsys@transformshift{0.400568in}{0.231000in}%
\pgfsys@useobject{currentmarker}{}%
\end{pgfscope}%
\end{pgfscope}%
\begin{pgfscope}%
\definecolor{textcolor}{rgb}{0.000000,0.000000,0.000000}%
\pgfsetstrokecolor{textcolor}%
\pgfsetfillcolor{textcolor}%
\pgftext[x=0.400568in,y=0.133778in,,top]{\color{textcolor}{\sffamily\fontsize{10.000000}{12.000000}\selectfont\catcode`\^=\active\def^{\ifmmode\sp\else\^{}\fi}\catcode`\%=\active\def
\end{pgfscope}%
\begin{pgfscope}%
\pgfsetbuttcap%
\pgfsetroundjoin%
\definecolor{currentfill}{rgb}{0.000000,0.000000,0.000000}%
\pgfsetfillcolor{currentfill}%
\pgfsetlinewidth{0.803000pt}%
\definecolor{currentstroke}{rgb}{0.000000,0.000000,0.000000}%
\pgfsetstrokecolor{currentstroke}%
\pgfsetdash{}{0pt}%
\pgfsys@defobject{currentmarker}{\pgfqpoint{0.000000in}{-0.048611in}}{\pgfqpoint{0.000000in}{0.000000in}}{%
\pgfpathmoveto{\pgfqpoint{0.000000in}{0.000000in}}%
\pgfpathlineto{\pgfqpoint{0.000000in}{-0.048611in}}%
\pgfusepath{stroke,fill}%
}%
\begin{pgfscope}%
\pgfsys@transformshift{2.161932in}{0.231000in}%
\pgfsys@useobject{currentmarker}{}%
\end{pgfscope}%
\end{pgfscope}%
\begin{pgfscope}%
\definecolor{textcolor}{rgb}{0.000000,0.000000,0.000000}%
\pgfsetstrokecolor{textcolor}%
\pgfsetfillcolor{textcolor}%
\pgftext[x=2.161932in,y=0.133778in,,top]{\color{textcolor}{\sffamily\fontsize{10.000000}{12.000000}\selectfont\catcode`\^=\active\def^{\ifmmode\sp\else\^{}\fi}\catcode`\%=\active\def
\end{pgfscope}%
\begin{pgfscope}%
\pgfsetbuttcap%
\pgfsetroundjoin%
\definecolor{currentfill}{rgb}{0.000000,0.000000,0.000000}%
\pgfsetfillcolor{currentfill}%
\pgfsetlinewidth{0.803000pt}%
\definecolor{currentstroke}{rgb}{0.000000,0.000000,0.000000}%
\pgfsetstrokecolor{currentstroke}%
\pgfsetdash{}{0pt}%
\pgfsys@defobject{currentmarker}{\pgfqpoint{-0.048611in}{0.000000in}}{\pgfqpoint{-0.000000in}{0.000000in}}{%
\pgfpathmoveto{\pgfqpoint{-0.000000in}{0.000000in}}%
\pgfpathlineto{\pgfqpoint{-0.048611in}{0.000000in}}%
\pgfusepath{stroke,fill}%
}%
\begin{pgfscope}%
\pgfsys@transformshift{0.312500in}{0.304500in}%
\pgfsys@useobject{currentmarker}{}%
\end{pgfscope}%
\end{pgfscope}%
\begin{pgfscope}%
\definecolor{textcolor}{rgb}{0.000000,0.000000,0.000000}%
\pgfsetstrokecolor{textcolor}%
\pgfsetfillcolor{textcolor}%
\pgftext[x=0.145833in, y=0.251738in, left, base]{\color{textcolor}{\sffamily\fontsize{10.000000}{12.000000}\selectfont\catcode`\^=\active\def^{\ifmmode\sp\else\^{}\fi}\catcode`\%=\active\def
\end{pgfscope}%
\begin{pgfscope}%
\pgfsetbuttcap%
\pgfsetroundjoin%
\definecolor{currentfill}{rgb}{0.000000,0.000000,0.000000}%
\pgfsetfillcolor{currentfill}%
\pgfsetlinewidth{0.803000pt}%
\definecolor{currentstroke}{rgb}{0.000000,0.000000,0.000000}%
\pgfsetstrokecolor{currentstroke}%
\pgfsetdash{}{0pt}%
\pgfsys@defobject{currentmarker}{\pgfqpoint{-0.048611in}{0.000000in}}{\pgfqpoint{-0.000000in}{0.000000in}}{%
\pgfpathmoveto{\pgfqpoint{-0.000000in}{0.000000in}}%
\pgfpathlineto{\pgfqpoint{-0.048611in}{0.000000in}}%
\pgfusepath{stroke,fill}%
}%
\begin{pgfscope}%
\pgfsys@transformshift{0.312500in}{1.774500in}%
\pgfsys@useobject{currentmarker}{}%
\end{pgfscope}%
\end{pgfscope}%
\begin{pgfscope}%
\definecolor{textcolor}{rgb}{0.000000,0.000000,0.000000}%
\pgfsetstrokecolor{textcolor}%
\pgfsetfillcolor{textcolor}%
\pgftext[x=0.145833in, y=1.721738in, left, base]{\color{textcolor}{\sffamily\fontsize{10.000000}{12.000000}\selectfont\catcode`\^=\active\def^{\ifmmode\sp\else\^{}\fi}\catcode`\%=\active\def
\end{pgfscope}%
\begin{pgfscope}%
\pgfpathrectangle{\pgfqpoint{0.312500in}{0.231000in}}{\pgfqpoint{1.937500in}{1.617000in}}%
\pgfusepath{clip}%
\pgfsetrectcap%
\pgfsetroundjoin%
\pgfsetlinewidth{0.803000pt}%
\definecolor{currentstroke}{rgb}{0.721569,0.525490,0.043137}%
\pgfsetstrokecolor{currentstroke}%
\pgfsetdash{}{0pt}%
\pgfpathmoveto{\pgfqpoint{0.400568in}{0.304500in}}%
\pgfpathlineto{\pgfqpoint{0.418182in}{0.375778in}}%
\pgfpathlineto{\pgfqpoint{0.435795in}{0.415582in}}%
\pgfpathlineto{\pgfqpoint{0.453409in}{0.444004in}}%
\pgfpathlineto{\pgfqpoint{0.471023in}{0.467059in}}%
\pgfpathlineto{\pgfqpoint{0.488636in}{0.487088in}}%
\pgfpathlineto{\pgfqpoint{0.506250in}{0.505177in}}%
\pgfpathlineto{\pgfqpoint{0.523864in}{0.521905in}}%
\pgfpathlineto{\pgfqpoint{0.541477in}{0.537615in}}%
\pgfpathlineto{\pgfqpoint{0.559091in}{0.552527in}}%
\pgfpathlineto{\pgfqpoint{0.576705in}{0.566793in}}%
\pgfpathlineto{\pgfqpoint{0.594318in}{0.580522in}}%
\pgfpathlineto{\pgfqpoint{0.611932in}{0.593797in}}%
\pgfpathlineto{\pgfqpoint{0.629545in}{0.606681in}}%
\pgfpathlineto{\pgfqpoint{0.647159in}{0.619225in}}%
\pgfpathlineto{\pgfqpoint{0.664773in}{0.631470in}}%
\pgfpathlineto{\pgfqpoint{0.682386in}{0.643452in}}%
\pgfpathlineto{\pgfqpoint{0.700000in}{0.655197in}}%
\pgfpathlineto{\pgfqpoint{0.717614in}{0.666731in}}%
\pgfpathlineto{\pgfqpoint{0.735227in}{0.678075in}}%
\pgfpathlineto{\pgfqpoint{0.752841in}{0.689247in}}%
\pgfpathlineto{\pgfqpoint{0.770455in}{0.700262in}}%
\pgfpathlineto{\pgfqpoint{0.788068in}{0.711135in}}%
\pgfpathlineto{\pgfqpoint{0.805682in}{0.721878in}}%
\pgfpathlineto{\pgfqpoint{0.823295in}{0.732503in}}%
\pgfpathlineto{\pgfqpoint{0.840909in}{0.743019in}}%
\pgfpathlineto{\pgfqpoint{0.858523in}{0.753435in}}%
\pgfpathlineto{\pgfqpoint{0.876136in}{0.763761in}}%
\pgfpathlineto{\pgfqpoint{0.893750in}{0.774003in}}%
\pgfpathlineto{\pgfqpoint{0.911364in}{0.784168in}}%
\pgfpathlineto{\pgfqpoint{0.928977in}{0.794264in}}%
\pgfpathlineto{\pgfqpoint{0.946591in}{0.804295in}}%
\pgfpathlineto{\pgfqpoint{0.964205in}{0.814268in}}%
\pgfpathlineto{\pgfqpoint{0.981818in}{0.824188in}}%
\pgfpathlineto{\pgfqpoint{0.999432in}{0.834060in}}%
\pgfpathlineto{\pgfqpoint{1.017045in}{0.843888in}}%
\pgfpathlineto{\pgfqpoint{1.034659in}{0.853677in}}%
\pgfpathlineto{\pgfqpoint{1.052273in}{0.863431in}}%
\pgfpathlineto{\pgfqpoint{1.069886in}{0.873154in}}%
\pgfpathlineto{\pgfqpoint{1.087500in}{0.882851in}}%
\pgfpathlineto{\pgfqpoint{1.105114in}{0.892524in}}%
\pgfpathlineto{\pgfqpoint{1.122727in}{0.902177in}}%
\pgfpathlineto{\pgfqpoint{1.140341in}{0.911814in}}%
\pgfpathlineto{\pgfqpoint{1.157955in}{0.921438in}}%
\pgfpathlineto{\pgfqpoint{1.175568in}{0.931053in}}%
\pgfpathlineto{\pgfqpoint{1.193182in}{0.940661in}}%
\pgfpathlineto{\pgfqpoint{1.210795in}{0.950266in}}%
\pgfpathlineto{\pgfqpoint{1.228409in}{0.959872in}}%
\pgfpathlineto{\pgfqpoint{1.246023in}{0.969481in}}%
\pgfpathlineto{\pgfqpoint{1.263636in}{0.979096in}}%
\pgfpathlineto{\pgfqpoint{1.281250in}{0.988721in}}%
\pgfpathlineto{\pgfqpoint{1.298864in}{0.998358in}}%
\pgfpathlineto{\pgfqpoint{1.316477in}{1.008012in}}%
\pgfpathlineto{\pgfqpoint{1.334091in}{1.017684in}}%
\pgfpathlineto{\pgfqpoint{1.351705in}{1.027379in}}%
\pgfpathlineto{\pgfqpoint{1.369318in}{1.037099in}}%
\pgfpathlineto{\pgfqpoint{1.386932in}{1.046849in}}%
\pgfpathlineto{\pgfqpoint{1.404545in}{1.056631in}}%
\pgfpathlineto{\pgfqpoint{1.422159in}{1.066449in}}%
\pgfpathlineto{\pgfqpoint{1.439773in}{1.076306in}}%
\pgfpathlineto{\pgfqpoint{1.457386in}{1.086207in}}%
\pgfpathlineto{\pgfqpoint{1.475000in}{1.096156in}}%
\pgfpathlineto{\pgfqpoint{1.492614in}{1.106156in}}%
\pgfpathlineto{\pgfqpoint{1.510227in}{1.116212in}}%
\pgfpathlineto{\pgfqpoint{1.527841in}{1.126328in}}%
\pgfpathlineto{\pgfqpoint{1.545455in}{1.136509in}}%
\pgfpathlineto{\pgfqpoint{1.563068in}{1.146761in}}%
\pgfpathlineto{\pgfqpoint{1.580682in}{1.157088in}}%
\pgfpathlineto{\pgfqpoint{1.598295in}{1.167495in}}%
\pgfpathlineto{\pgfqpoint{1.615909in}{1.177990in}}%
\pgfpathlineto{\pgfqpoint{1.633523in}{1.188579in}}%
\pgfpathlineto{\pgfqpoint{1.651136in}{1.199268in}}%
\pgfpathlineto{\pgfqpoint{1.668750in}{1.210065in}}%
\pgfpathlineto{\pgfqpoint{1.686364in}{1.220978in}}%
\pgfpathlineto{\pgfqpoint{1.703977in}{1.232016in}}%
\pgfpathlineto{\pgfqpoint{1.721591in}{1.243188in}}%
\pgfpathlineto{\pgfqpoint{1.739205in}{1.254506in}}%
\pgfpathlineto{\pgfqpoint{1.756818in}{1.265980in}}%
\pgfpathlineto{\pgfqpoint{1.774432in}{1.277625in}}%
\pgfpathlineto{\pgfqpoint{1.792045in}{1.289453in}}%
\pgfpathlineto{\pgfqpoint{1.809659in}{1.301482in}}%
\pgfpathlineto{\pgfqpoint{1.827273in}{1.313729in}}%
\pgfpathlineto{\pgfqpoint{1.844886in}{1.326215in}}%
\pgfpathlineto{\pgfqpoint{1.862500in}{1.338964in}}%
\pgfpathlineto{\pgfqpoint{1.880114in}{1.352004in}}%
\pgfpathlineto{\pgfqpoint{1.897727in}{1.365365in}}%
\pgfpathlineto{\pgfqpoint{1.915341in}{1.379085in}}%
\pgfpathlineto{\pgfqpoint{1.932955in}{1.393210in}}%
\pgfpathlineto{\pgfqpoint{1.950568in}{1.407792in}}%
\pgfpathlineto{\pgfqpoint{1.968182in}{1.422898in}}%
\pgfpathlineto{\pgfqpoint{1.985795in}{1.438611in}}%
\pgfpathlineto{\pgfqpoint{2.003409in}{1.455036in}}%
\pgfpathlineto{\pgfqpoint{2.021023in}{1.472315in}}%
\pgfpathlineto{\pgfqpoint{2.038636in}{1.490638in}}%
\pgfpathlineto{\pgfqpoint{2.056250in}{1.510285in}}%
\pgfpathlineto{\pgfqpoint{2.073864in}{1.531685in}}%
\pgfpathlineto{\pgfqpoint{2.091477in}{1.555577in}}%
\pgfpathlineto{\pgfqpoint{2.109091in}{1.583390in}}%
\pgfpathlineto{\pgfqpoint{2.126705in}{1.618425in}}%
\pgfpathlineto{\pgfqpoint{2.144318in}{1.670114in}}%
\pgfpathlineto{\pgfqpoint{2.161932in}{1.774500in}}%
\pgfusepath{stroke}%
\end{pgfscope}%
\begin{pgfscope}%
\pgfpathrectangle{\pgfqpoint{0.312500in}{0.231000in}}{\pgfqpoint{1.937500in}{1.617000in}}%
\pgfusepath{clip}%
\pgfsetrectcap%
\pgfsetroundjoin%
\pgfsetlinewidth{0.803000pt}%
\definecolor{currentstroke}{rgb}{0.000000,0.545098,0.545098}%
\pgfsetstrokecolor{currentstroke}%
\pgfsetdash{}{0pt}%
\pgfpathmoveto{\pgfqpoint{0.400568in}{0.304500in}}%
\pgfpathlineto{\pgfqpoint{0.418182in}{0.339589in}}%
\pgfpathlineto{\pgfqpoint{0.435795in}{0.358399in}}%
\pgfpathlineto{\pgfqpoint{0.453409in}{0.371082in}}%
\pgfpathlineto{\pgfqpoint{0.471023in}{0.380685in}}%
\pgfpathlineto{\pgfqpoint{0.488636in}{0.388392in}}%
\pgfpathlineto{\pgfqpoint{0.506250in}{0.394747in}}%
\pgfpathlineto{\pgfqpoint{0.523864in}{0.400031in}}%
\pgfpathlineto{\pgfqpoint{0.541477in}{0.404401in}}%
\pgfpathlineto{\pgfqpoint{0.559091in}{0.407949in}}%
\pgfpathlineto{\pgfqpoint{0.576705in}{0.456880in}}%
\pgfpathlineto{\pgfqpoint{0.594318in}{0.469447in}}%
\pgfpathlineto{\pgfqpoint{0.611932in}{0.482013in}}%
\pgfpathlineto{\pgfqpoint{0.629545in}{0.494579in}}%
\pgfpathlineto{\pgfqpoint{0.647159in}{0.507146in}}%
\pgfpathlineto{\pgfqpoint{0.664773in}{0.519712in}}%
\pgfpathlineto{\pgfqpoint{0.682386in}{0.532278in}}%
\pgfpathlineto{\pgfqpoint{0.700000in}{0.544845in}}%
\pgfpathlineto{\pgfqpoint{0.717614in}{0.557411in}}%
\pgfpathlineto{\pgfqpoint{0.735227in}{0.569977in}}%
\pgfpathlineto{\pgfqpoint{0.752841in}{0.582544in}}%
\pgfpathlineto{\pgfqpoint{0.770455in}{0.595110in}}%
\pgfpathlineto{\pgfqpoint{0.788068in}{0.607676in}}%
\pgfpathlineto{\pgfqpoint{0.805682in}{0.620243in}}%
\pgfpathlineto{\pgfqpoint{0.823295in}{0.632809in}}%
\pgfpathlineto{\pgfqpoint{0.840909in}{0.645375in}}%
\pgfpathlineto{\pgfqpoint{0.858523in}{0.657942in}}%
\pgfpathlineto{\pgfqpoint{0.876136in}{0.670508in}}%
\pgfpathlineto{\pgfqpoint{0.893750in}{0.683074in}}%
\pgfpathlineto{\pgfqpoint{0.911364in}{0.695641in}}%
\pgfpathlineto{\pgfqpoint{0.928977in}{0.708207in}}%
\pgfpathlineto{\pgfqpoint{0.946591in}{0.720773in}}%
\pgfpathlineto{\pgfqpoint{0.964205in}{0.733340in}}%
\pgfpathlineto{\pgfqpoint{0.981818in}{0.745906in}}%
\pgfpathlineto{\pgfqpoint{0.999432in}{0.758472in}}%
\pgfpathlineto{\pgfqpoint{1.017045in}{0.771039in}}%
\pgfpathlineto{\pgfqpoint{1.034659in}{0.783605in}}%
\pgfpathlineto{\pgfqpoint{1.052273in}{0.796171in}}%
\pgfpathlineto{\pgfqpoint{1.069886in}{0.808738in}}%
\pgfpathlineto{\pgfqpoint{1.087500in}{0.821304in}}%
\pgfpathlineto{\pgfqpoint{1.105114in}{0.833870in}}%
\pgfpathlineto{\pgfqpoint{1.122727in}{0.848365in}}%
\pgfpathlineto{\pgfqpoint{1.140341in}{0.862230in}}%
\pgfpathlineto{\pgfqpoint{1.157955in}{0.875568in}}%
\pgfpathlineto{\pgfqpoint{1.175568in}{0.888463in}}%
\pgfpathlineto{\pgfqpoint{1.193182in}{0.900983in}}%
\pgfpathlineto{\pgfqpoint{1.210795in}{0.913185in}}%
\pgfpathlineto{\pgfqpoint{1.228409in}{0.925119in}}%
\pgfpathlineto{\pgfqpoint{1.246023in}{0.936823in}}%
\pgfpathlineto{\pgfqpoint{1.263636in}{0.948332in}}%
\pgfpathlineto{\pgfqpoint{1.281250in}{0.959676in}}%
\pgfpathlineto{\pgfqpoint{1.298864in}{0.970880in}}%
\pgfpathlineto{\pgfqpoint{1.316477in}{0.981965in}}%
\pgfpathlineto{\pgfqpoint{1.334091in}{0.992952in}}%
\pgfpathlineto{\pgfqpoint{1.351705in}{1.003857in}}%
\pgfpathlineto{\pgfqpoint{1.369318in}{1.014695in}}%
\pgfpathlineto{\pgfqpoint{1.386932in}{1.025479in}}%
\pgfpathlineto{\pgfqpoint{1.404545in}{1.036223in}}%
\pgfpathlineto{\pgfqpoint{1.422159in}{1.046937in}}%
\pgfpathlineto{\pgfqpoint{1.439773in}{1.057632in}}%
\pgfpathlineto{\pgfqpoint{1.457386in}{1.068317in}}%
\pgfpathlineto{\pgfqpoint{1.475000in}{1.079002in}}%
\pgfpathlineto{\pgfqpoint{1.492614in}{1.089696in}}%
\pgfpathlineto{\pgfqpoint{1.510227in}{1.100406in}}%
\pgfpathlineto{\pgfqpoint{1.527841in}{1.111140in}}%
\pgfpathlineto{\pgfqpoint{1.545455in}{1.121908in}}%
\pgfpathlineto{\pgfqpoint{1.563068in}{1.132716in}}%
\pgfpathlineto{\pgfqpoint{1.580682in}{1.143573in}}%
\pgfpathlineto{\pgfqpoint{1.598295in}{1.154486in}}%
\pgfpathlineto{\pgfqpoint{1.615909in}{1.165464in}}%
\pgfpathlineto{\pgfqpoint{1.633523in}{1.176515in}}%
\pgfpathlineto{\pgfqpoint{1.651136in}{1.187648in}}%
\pgfpathlineto{\pgfqpoint{1.668750in}{1.198872in}}%
\pgfpathlineto{\pgfqpoint{1.686364in}{1.210195in}}%
\pgfpathlineto{\pgfqpoint{1.703977in}{1.221630in}}%
\pgfpathlineto{\pgfqpoint{1.721591in}{1.233185in}}%
\pgfpathlineto{\pgfqpoint{1.739205in}{1.244874in}}%
\pgfpathlineto{\pgfqpoint{1.756818in}{1.256709in}}%
\pgfpathlineto{\pgfqpoint{1.774432in}{1.268703in}}%
\pgfpathlineto{\pgfqpoint{1.792045in}{1.280873in}}%
\pgfpathlineto{\pgfqpoint{1.809659in}{1.293234in}}%
\pgfpathlineto{\pgfqpoint{1.827273in}{1.305807in}}%
\pgfpathlineto{\pgfqpoint{1.844886in}{1.318613in}}%
\pgfpathlineto{\pgfqpoint{1.862500in}{1.331676in}}%
\pgfpathlineto{\pgfqpoint{1.880114in}{1.345025in}}%
\pgfpathlineto{\pgfqpoint{1.897727in}{1.358691in}}%
\pgfpathlineto{\pgfqpoint{1.915341in}{1.372715in}}%
\pgfpathlineto{\pgfqpoint{1.932955in}{1.387140in}}%
\pgfpathlineto{\pgfqpoint{1.950568in}{1.402022in}}%
\pgfpathlineto{\pgfqpoint{1.968182in}{1.417428in}}%
\pgfpathlineto{\pgfqpoint{1.985795in}{1.433443in}}%
\pgfpathlineto{\pgfqpoint{2.003409in}{1.450174in}}%
\pgfpathlineto{\pgfqpoint{2.021023in}{1.467764in}}%
\pgfpathlineto{\pgfqpoint{2.038636in}{1.486407in}}%
\pgfpathlineto{\pgfqpoint{2.056250in}{1.506385in}}%
\pgfpathlineto{\pgfqpoint{2.073864in}{1.528137in}}%
\pgfpathlineto{\pgfqpoint{2.091477in}{1.552407in}}%
\pgfpathlineto{\pgfqpoint{2.109091in}{1.580648in}}%
\pgfpathlineto{\pgfqpoint{2.126705in}{1.616205in}}%
\pgfpathlineto{\pgfqpoint{2.144318in}{1.668641in}}%
\pgfpathlineto{\pgfqpoint{2.161932in}{1.774500in}}%
\pgfusepath{stroke}%
\end{pgfscope}%
\begin{pgfscope}%
\pgfsetrectcap%
\pgfsetmiterjoin%
\pgfsetlinewidth{0.803000pt}%
\definecolor{currentstroke}{rgb}{0.000000,0.000000,0.000000}%
\pgfsetstrokecolor{currentstroke}%
\pgfsetdash{}{0pt}%
\pgfpathmoveto{\pgfqpoint{0.312500in}{0.231000in}}%
\pgfpathlineto{\pgfqpoint{0.312500in}{1.848000in}}%
\pgfusepath{stroke}%
\end{pgfscope}%
\begin{pgfscope}%
\pgfsetrectcap%
\pgfsetmiterjoin%
\pgfsetlinewidth{0.803000pt}%
\definecolor{currentstroke}{rgb}{0.000000,0.000000,0.000000}%
\pgfsetstrokecolor{currentstroke}%
\pgfsetdash{}{0pt}%
\pgfpathmoveto{\pgfqpoint{2.250000in}{0.231000in}}%
\pgfpathlineto{\pgfqpoint{2.250000in}{1.848000in}}%
\pgfusepath{stroke}%
\end{pgfscope}%
\begin{pgfscope}%
\pgfsetrectcap%
\pgfsetmiterjoin%
\pgfsetlinewidth{0.803000pt}%
\definecolor{currentstroke}{rgb}{0.000000,0.000000,0.000000}%
\pgfsetstrokecolor{currentstroke}%
\pgfsetdash{}{0pt}%
\pgfpathmoveto{\pgfqpoint{0.312500in}{0.231000in}}%
\pgfpathlineto{\pgfqpoint{2.250000in}{0.231000in}}%
\pgfusepath{stroke}%
\end{pgfscope}%
\begin{pgfscope}%
\pgfsetrectcap%
\pgfsetmiterjoin%
\pgfsetlinewidth{0.803000pt}%
\definecolor{currentstroke}{rgb}{0.000000,0.000000,0.000000}%
\pgfsetstrokecolor{currentstroke}%
\pgfsetdash{}{0pt}%
\pgfpathmoveto{\pgfqpoint{0.312500in}{1.848000in}}%
\pgfpathlineto{\pgfqpoint{2.250000in}{1.848000in}}%
\pgfusepath{stroke}%
\end{pgfscope}%
\begin{pgfscope}%
\pgfsetbuttcap%
\pgfsetmiterjoin%
\definecolor{currentfill}{rgb}{1.000000,1.000000,1.000000}%
\pgfsetfillcolor{currentfill}%
\pgfsetfillopacity{0.800000}%
\pgfsetlinewidth{1.003750pt}%
\definecolor{currentstroke}{rgb}{0.800000,0.800000,0.800000}%
\pgfsetstrokecolor{currentstroke}%
\pgfsetstrokeopacity{0.800000}%
\pgfsetdash{}{0pt}%
\pgfpathmoveto{\pgfqpoint{0.409722in}{1.317509in}}%
\pgfpathlineto{\pgfqpoint{0.975856in}{1.317509in}}%
\pgfpathquadraticcurveto{\pgfqpoint{1.003634in}{1.317509in}}{\pgfqpoint{1.003634in}{1.345287in}}%
\pgfpathlineto{\pgfqpoint{1.003634in}{1.750778in}}%
\pgfpathquadraticcurveto{\pgfqpoint{1.003634in}{1.778556in}}{\pgfqpoint{0.975856in}{1.778556in}}%
\pgfpathlineto{\pgfqpoint{0.409722in}{1.778556in}}%
\pgfpathquadraticcurveto{\pgfqpoint{0.381944in}{1.778556in}}{\pgfqpoint{0.381944in}{1.750778in}}%
\pgfpathlineto{\pgfqpoint{0.381944in}{1.345287in}}%
\pgfpathquadraticcurveto{\pgfqpoint{0.381944in}{1.317509in}}{\pgfqpoint{0.409722in}{1.317509in}}%
\pgfpathlineto{\pgfqpoint{0.409722in}{1.317509in}}%
\pgfpathclose%
\pgfusepath{stroke,fill}%
\end{pgfscope}%
\begin{pgfscope}%
\pgfsetrectcap%
\pgfsetroundjoin%
\pgfsetlinewidth{0.803000pt}%
\definecolor{currentstroke}{rgb}{0.721569,0.525490,0.043137}%
\pgfsetstrokecolor{currentstroke}%
\pgfsetdash{}{0pt}%
\pgfpathmoveto{\pgfqpoint{0.437500in}{1.666088in}}%
\pgfpathlineto{\pgfqpoint{0.520833in}{1.666088in}}%
\pgfpathlineto{\pgfqpoint{0.604167in}{1.666088in}}%
\pgfusepath{stroke}%
\end{pgfscope}%
\begin{pgfscope}%
\definecolor{textcolor}{rgb}{0.000000,0.000000,0.000000}%
\pgfsetstrokecolor{textcolor}%
\pgfsetfillcolor{textcolor}%
\pgftext[x=0.715278in,y=1.617477in,left,base]{\color{textcolor}{\sffamily\fontsize{10.000000}{12.000000}\selectfont\catcode`\^=\active\def^{\ifmmode\sp\else\^{}\fi}\catcode`\%=\active\def
\end{pgfscope}%
\begin{pgfscope}%
\pgfsetrectcap%
\pgfsetroundjoin%
\pgfsetlinewidth{0.803000pt}%
\definecolor{currentstroke}{rgb}{0.000000,0.545098,0.545098}%
\pgfsetstrokecolor{currentstroke}%
\pgfsetdash{}{0pt}%
\pgfpathmoveto{\pgfqpoint{0.437500in}{1.456398in}}%
\pgfpathlineto{\pgfqpoint{0.520833in}{1.456398in}}%
\pgfpathlineto{\pgfqpoint{0.604167in}{1.456398in}}%
\pgfusepath{stroke}%
\end{pgfscope}%
\begin{pgfscope}%
\definecolor{textcolor}{rgb}{0.000000,0.000000,0.000000}%
\pgfsetstrokecolor{textcolor}%
\pgfsetfillcolor{textcolor}%
\pgftext[x=0.715278in,y=1.407787in,left,base]{\color{textcolor}{\sffamily\fontsize{10.000000}{12.000000}\selectfont\catcode`\^=\active\def^{\ifmmode\sp\else\^{}\fi}\catcode`\%=\active\def
\end{pgfscope}%
\end{pgfpicture}%
\makeatother%
\endgroup%

%% file: pictures/mhatConstruction.pgf
\begingroup%
\makeatletter%
\begin{pgfpicture}%
\pgfpathrectangle{\pgfpointorigin}{\pgfqpoint{4.000000in}{1.750000in}}%
\pgfusepath{use as bounding box, clip}%
\begin{pgfscope}%
\pgfsetbuttcap%
\pgfsetmiterjoin%
\definecolor{currentfill}{rgb}{1.000000,1.000000,1.000000}%
\pgfsetfillcolor{currentfill}%
\pgfsetlinewidth{0.000000pt}%
\definecolor{currentstroke}{rgb}{1.000000,1.000000,1.000000}%
\pgfsetstrokecolor{currentstroke}%
\pgfsetdash{}{0pt}%
\pgfpathmoveto{\pgfqpoint{0.000000in}{0.000000in}}%
\pgfpathlineto{\pgfqpoint{4.000000in}{0.000000in}}%
\pgfpathlineto{\pgfqpoint{4.000000in}{1.750000in}}%
\pgfpathlineto{\pgfqpoint{0.000000in}{1.750000in}}%
\pgfpathlineto{\pgfqpoint{0.000000in}{0.000000in}}%
\pgfpathclose%
\pgfusepath{fill}%
\end{pgfscope}%
\begin{pgfscope}%
\pgfsetbuttcap%
\pgfsetmiterjoin%
\definecolor{currentfill}{rgb}{1.000000,1.000000,1.000000}%
\pgfsetfillcolor{currentfill}%
\pgfsetlinewidth{0.000000pt}%
\definecolor{currentstroke}{rgb}{0.000000,0.000000,0.000000}%
\pgfsetstrokecolor{currentstroke}%
\pgfsetstrokeopacity{0.000000}%
\pgfsetdash{}{0pt}%
\pgfpathmoveto{\pgfqpoint{0.500000in}{0.192500in}}%
\pgfpathlineto{\pgfqpoint{3.600000in}{0.192500in}}%
\pgfpathlineto{\pgfqpoint{3.600000in}{1.540000in}}%
\pgfpathlineto{\pgfqpoint{0.500000in}{1.540000in}}%
\pgfpathlineto{\pgfqpoint{0.500000in}{0.192500in}}%
\pgfpathclose%
\pgfusepath{fill}%
\end{pgfscope}%
\begin{pgfscope}%
\pgfsetbuttcap%
\pgfsetroundjoin%
\definecolor{currentfill}{rgb}{0.000000,0.000000,0.000000}%
\pgfsetfillcolor{currentfill}%
\pgfsetlinewidth{0.803000pt}%
\definecolor{currentstroke}{rgb}{0.000000,0.000000,0.000000}%
\pgfsetstrokecolor{currentstroke}%
\pgfsetdash{}{0pt}%
\pgfsys@defobject{currentmarker}{\pgfqpoint{0.000000in}{-0.048611in}}{\pgfqpoint{0.000000in}{0.000000in}}{%
\pgfpathmoveto{\pgfqpoint{0.000000in}{0.000000in}}%
\pgfpathlineto{\pgfqpoint{0.000000in}{-0.048611in}}%
\pgfusepath{stroke,fill}%
}%
\begin{pgfscope}%
\pgfsys@transformshift{0.583395in}{0.192500in}%
\pgfsys@useobject{currentmarker}{}%
\end{pgfscope}%
\end{pgfscope}%
\begin{pgfscope}%
\definecolor{textcolor}{rgb}{0.000000,0.000000,0.000000}%
\pgfsetstrokecolor{textcolor}%
\pgfsetfillcolor{textcolor}%
\pgftext[x=0.583395in,y=0.095278in,,top]{\color{textcolor}{\sffamily\fontsize{10.000000}{12.000000}\selectfont\catcode`\^=\active\def^{\ifmmode\sp\else\^{}\fi}\catcode`\%=\active\def
\end{pgfscope}%
\begin{pgfscope}%
\pgfsetbuttcap%
\pgfsetroundjoin%
\definecolor{currentfill}{rgb}{0.000000,0.000000,0.000000}%
\pgfsetfillcolor{currentfill}%
\pgfsetlinewidth{0.803000pt}%
\definecolor{currentstroke}{rgb}{0.000000,0.000000,0.000000}%
\pgfsetstrokecolor{currentstroke}%
\pgfsetdash{}{0pt}%
\pgfsys@defobject{currentmarker}{\pgfqpoint{0.000000in}{-0.048611in}}{\pgfqpoint{0.000000in}{0.000000in}}{%
\pgfpathmoveto{\pgfqpoint{0.000000in}{0.000000in}}%
\pgfpathlineto{\pgfqpoint{0.000000in}{-0.048611in}}%
\pgfusepath{stroke,fill}%
}%
\begin{pgfscope}%
\pgfsys@transformshift{1.158534in}{0.192500in}%
\pgfsys@useobject{currentmarker}{}%
\end{pgfscope}%
\end{pgfscope}%
\begin{pgfscope}%
\definecolor{textcolor}{rgb}{0.000000,0.000000,0.000000}%
\pgfsetstrokecolor{textcolor}%
\pgfsetfillcolor{textcolor}%
\pgftext[x=1.158534in,y=0.095278in,,top]{\color{textcolor}{\sffamily\fontsize{10.000000}{12.000000}\selectfont\catcode`\^=\active\def^{\ifmmode\sp\else\^{}\fi}\catcode`\%=\active\def
\end{pgfscope}%
\begin{pgfscope}%
\pgfsetbuttcap%
\pgfsetroundjoin%
\definecolor{currentfill}{rgb}{0.000000,0.000000,0.000000}%
\pgfsetfillcolor{currentfill}%
\pgfsetlinewidth{0.803000pt}%
\definecolor{currentstroke}{rgb}{0.000000,0.000000,0.000000}%
\pgfsetstrokecolor{currentstroke}%
\pgfsetdash{}{0pt}%
\pgfsys@defobject{currentmarker}{\pgfqpoint{0.000000in}{-0.048611in}}{\pgfqpoint{0.000000in}{0.000000in}}{%
\pgfpathmoveto{\pgfqpoint{0.000000in}{0.000000in}}%
\pgfpathlineto{\pgfqpoint{0.000000in}{-0.048611in}}%
\pgfusepath{stroke,fill}%
}%
\begin{pgfscope}%
\pgfsys@transformshift{1.733673in}{0.192500in}%
\pgfsys@useobject{currentmarker}{}%
\end{pgfscope}%
\end{pgfscope}%
\begin{pgfscope}%
\definecolor{textcolor}{rgb}{0.000000,0.000000,0.000000}%
\pgfsetstrokecolor{textcolor}%
\pgfsetfillcolor{textcolor}%
\pgftext[x=1.733673in,y=0.095278in,,top]{\color{textcolor}{\sffamily\fontsize{10.000000}{12.000000}\selectfont\catcode`\^=\active\def^{\ifmmode\sp\else\^{}\fi}\catcode`\%=\active\def
\end{pgfscope}%
\begin{pgfscope}%
\pgfsetbuttcap%
\pgfsetroundjoin%
\definecolor{currentfill}{rgb}{0.000000,0.000000,0.000000}%
\pgfsetfillcolor{currentfill}%
\pgfsetlinewidth{0.803000pt}%
\definecolor{currentstroke}{rgb}{0.000000,0.000000,0.000000}%
\pgfsetstrokecolor{currentstroke}%
\pgfsetdash{}{0pt}%
\pgfsys@defobject{currentmarker}{\pgfqpoint{0.000000in}{-0.048611in}}{\pgfqpoint{0.000000in}{0.000000in}}{%
\pgfpathmoveto{\pgfqpoint{0.000000in}{0.000000in}}%
\pgfpathlineto{\pgfqpoint{0.000000in}{-0.048611in}}%
\pgfusepath{stroke,fill}%
}%
\begin{pgfscope}%
\pgfsys@transformshift{2.308813in}{0.192500in}%
\pgfsys@useobject{currentmarker}{}%
\end{pgfscope}%
\end{pgfscope}%
\begin{pgfscope}%
\definecolor{textcolor}{rgb}{0.000000,0.000000,0.000000}%
\pgfsetstrokecolor{textcolor}%
\pgfsetfillcolor{textcolor}%
\pgftext[x=2.308813in,y=0.095278in,,top]{\color{textcolor}{\sffamily\fontsize{10.000000}{12.000000}\selectfont\catcode`\^=\active\def^{\ifmmode\sp\else\^{}\fi}\catcode`\%=\active\def
\end{pgfscope}%
\begin{pgfscope}%
\pgfsetbuttcap%
\pgfsetroundjoin%
\definecolor{currentfill}{rgb}{0.000000,0.000000,0.000000}%
\pgfsetfillcolor{currentfill}%
\pgfsetlinewidth{0.803000pt}%
\definecolor{currentstroke}{rgb}{0.000000,0.000000,0.000000}%
\pgfsetstrokecolor{currentstroke}%
\pgfsetdash{}{0pt}%
\pgfsys@defobject{currentmarker}{\pgfqpoint{0.000000in}{-0.048611in}}{\pgfqpoint{0.000000in}{0.000000in}}{%
\pgfpathmoveto{\pgfqpoint{0.000000in}{0.000000in}}%
\pgfpathlineto{\pgfqpoint{0.000000in}{-0.048611in}}%
\pgfusepath{stroke,fill}%
}%
\begin{pgfscope}%
\pgfsys@transformshift{2.883952in}{0.192500in}%
\pgfsys@useobject{currentmarker}{}%
\end{pgfscope}%
\end{pgfscope}%
\begin{pgfscope}%
\definecolor{textcolor}{rgb}{0.000000,0.000000,0.000000}%
\pgfsetstrokecolor{textcolor}%
\pgfsetfillcolor{textcolor}%
\pgftext[x=2.883952in,y=0.095278in,,top]{\color{textcolor}{\sffamily\fontsize{10.000000}{12.000000}\selectfont\catcode`\^=\active\def^{\ifmmode\sp\else\^{}\fi}\catcode`\%=\active\def
\end{pgfscope}%
\begin{pgfscope}%
\pgfsetbuttcap%
\pgfsetroundjoin%
\definecolor{currentfill}{rgb}{0.000000,0.000000,0.000000}%
\pgfsetfillcolor{currentfill}%
\pgfsetlinewidth{0.803000pt}%
\definecolor{currentstroke}{rgb}{0.000000,0.000000,0.000000}%
\pgfsetstrokecolor{currentstroke}%
\pgfsetdash{}{0pt}%
\pgfsys@defobject{currentmarker}{\pgfqpoint{0.000000in}{-0.048611in}}{\pgfqpoint{0.000000in}{0.000000in}}{%
\pgfpathmoveto{\pgfqpoint{0.000000in}{0.000000in}}%
\pgfpathlineto{\pgfqpoint{0.000000in}{-0.048611in}}%
\pgfusepath{stroke,fill}%
}%
\begin{pgfscope}%
\pgfsys@transformshift{3.459091in}{0.192500in}%
\pgfsys@useobject{currentmarker}{}%
\end{pgfscope}%
\end{pgfscope}%
\begin{pgfscope}%
\definecolor{textcolor}{rgb}{0.000000,0.000000,0.000000}%
\pgfsetstrokecolor{textcolor}%
\pgfsetfillcolor{textcolor}%
\pgftext[x=3.459091in,y=0.095278in,,top]{\color{textcolor}{\sffamily\fontsize{10.000000}{12.000000}\selectfont\catcode`\^=\active\def^{\ifmmode\sp\else\^{}\fi}\catcode`\%=\active\def
\end{pgfscope}%
\begin{pgfscope}%
\pgfsetbuttcap%
\pgfsetroundjoin%
\definecolor{currentfill}{rgb}{0.000000,0.000000,0.000000}%
\pgfsetfillcolor{currentfill}%
\pgfsetlinewidth{0.803000pt}%
\definecolor{currentstroke}{rgb}{0.000000,0.000000,0.000000}%
\pgfsetstrokecolor{currentstroke}%
\pgfsetdash{}{0pt}%
\pgfsys@defobject{currentmarker}{\pgfqpoint{-0.048611in}{0.000000in}}{\pgfqpoint{-0.000000in}{0.000000in}}{%
\pgfpathmoveto{\pgfqpoint{-0.000000in}{0.000000in}}%
\pgfpathlineto{\pgfqpoint{-0.048611in}{0.000000in}}%
\pgfusepath{stroke,fill}%
}%
\begin{pgfscope}%
\pgfsys@transformshift{0.500000in}{0.641667in}%
\pgfsys@useobject{currentmarker}{}%
\end{pgfscope}%
\end{pgfscope}%
\begin{pgfscope}%
\definecolor{textcolor}{rgb}{0.000000,0.000000,0.000000}%
\pgfsetstrokecolor{textcolor}%
\pgfsetfillcolor{textcolor}%
\pgftext[x=0.333333in, y=0.588905in, left, base]{\color{textcolor}{\sffamily\fontsize{10.000000}{12.000000}\selectfont\catcode`\^=\active\def^{\ifmmode\sp\else\^{}\fi}\catcode`\%=\active\def
\end{pgfscope}%
\begin{pgfscope}%
\pgfsetbuttcap%
\pgfsetroundjoin%
\definecolor{currentfill}{rgb}{0.000000,0.000000,0.000000}%
\pgfsetfillcolor{currentfill}%
\pgfsetlinewidth{0.803000pt}%
\definecolor{currentstroke}{rgb}{0.000000,0.000000,0.000000}%
\pgfsetstrokecolor{currentstroke}%
\pgfsetdash{}{0pt}%
\pgfsys@defobject{currentmarker}{\pgfqpoint{-0.048611in}{0.000000in}}{\pgfqpoint{-0.000000in}{0.000000in}}{%
\pgfpathmoveto{\pgfqpoint{-0.000000in}{0.000000in}}%
\pgfpathlineto{\pgfqpoint{-0.048611in}{0.000000in}}%
\pgfusepath{stroke,fill}%
}%
\begin{pgfscope}%
\pgfsys@transformshift{0.500000in}{0.406839in}%
\pgfsys@useobject{currentmarker}{}%
\end{pgfscope}%
\end{pgfscope}%
\begin{pgfscope}%
\definecolor{textcolor}{rgb}{0.000000,0.000000,0.000000}%
\pgfsetstrokecolor{textcolor}%
\pgfsetfillcolor{textcolor}%
\pgftext[x=0.139669in, y=0.354077in, left, base]{\color{textcolor}{\sffamily\fontsize{10.000000}{12.000000}\selectfont\catcode`\^=\active\def^{\ifmmode\sp\else\^{}\fi}\catcode`\%=\active\def
\end{pgfscope}%
\begin{pgfscope}%
\pgfpathrectangle{\pgfqpoint{0.500000in}{0.192500in}}{\pgfqpoint{3.100000in}{1.347500in}}%
\pgfusepath{clip}%
\pgfsetrectcap%
\pgfsetroundjoin%
\pgfsetlinewidth{0.803000pt}%
\definecolor{currentstroke}{rgb}{0.721569,0.525490,0.043137}%
\pgfsetstrokecolor{currentstroke}%
\pgfsetdash{}{0pt}%
\pgfpathmoveto{\pgfqpoint{0.717496in}{1.550000in}}%
\pgfpathlineto{\pgfqpoint{0.727180in}{1.487705in}}%
\pgfpathlineto{\pgfqpoint{0.755937in}{1.365721in}}%
\pgfpathlineto{\pgfqpoint{0.784694in}{1.280376in}}%
\pgfpathlineto{\pgfqpoint{0.813451in}{1.218028in}}%
\pgfpathlineto{\pgfqpoint{0.842208in}{1.170950in}}%
\pgfpathlineto{\pgfqpoint{0.870965in}{1.134459in}}%
\pgfpathlineto{\pgfqpoint{0.899722in}{1.105562in}}%
\pgfpathlineto{\pgfqpoint{0.928479in}{1.082265in}}%
\pgfpathlineto{\pgfqpoint{0.957236in}{1.063193in}}%
\pgfpathlineto{\pgfqpoint{0.985993in}{1.047372in}}%
\pgfpathlineto{\pgfqpoint{1.014750in}{1.034094in}}%
\pgfpathlineto{\pgfqpoint{1.043506in}{1.022834in}}%
\pgfpathlineto{\pgfqpoint{1.072263in}{1.013199in}}%
\pgfpathlineto{\pgfqpoint{1.101020in}{1.004884in}}%
\pgfpathlineto{\pgfqpoint{1.129777in}{0.997654in}}%
\pgfpathlineto{\pgfqpoint{1.158534in}{0.991325in}}%
\pgfpathlineto{\pgfqpoint{1.187291in}{0.985751in}}%
\pgfpathlineto{\pgfqpoint{1.216048in}{0.980812in}}%
\pgfpathlineto{\pgfqpoint{1.244805in}{0.976413in}}%
\pgfpathlineto{\pgfqpoint{1.273562in}{0.972477in}}%
\pgfpathlineto{\pgfqpoint{1.302319in}{0.968938in}}%
\pgfpathlineto{\pgfqpoint{1.331076in}{0.965743in}}%
\pgfpathlineto{\pgfqpoint{1.359833in}{0.962849in}}%
\pgfpathlineto{\pgfqpoint{1.388590in}{0.960216in}}%
\pgfpathlineto{\pgfqpoint{1.417347in}{0.957814in}}%
\pgfpathlineto{\pgfqpoint{1.446104in}{0.955615in}}%
\pgfpathlineto{\pgfqpoint{1.474861in}{0.953597in}}%
\pgfpathlineto{\pgfqpoint{1.503618in}{0.951739in}}%
\pgfpathlineto{\pgfqpoint{1.532375in}{0.950024in}}%
\pgfpathlineto{\pgfqpoint{1.561132in}{0.948438in}}%
\pgfpathlineto{\pgfqpoint{1.589889in}{0.946968in}}%
\pgfpathlineto{\pgfqpoint{1.618646in}{0.945601in}}%
\pgfpathlineto{\pgfqpoint{1.647403in}{0.944329in}}%
\pgfpathlineto{\pgfqpoint{1.676160in}{0.943143in}}%
\pgfpathlineto{\pgfqpoint{1.704917in}{0.942034in}}%
\pgfpathlineto{\pgfqpoint{1.733673in}{0.940997in}}%
\pgfpathlineto{\pgfqpoint{1.762430in}{0.940024in}}%
\pgfpathlineto{\pgfqpoint{1.791187in}{0.939111in}}%
\pgfpathlineto{\pgfqpoint{1.819944in}{0.938252in}}%
\pgfpathlineto{\pgfqpoint{1.848701in}{0.937444in}}%
\pgfpathlineto{\pgfqpoint{1.877458in}{0.936682in}}%
\pgfpathlineto{\pgfqpoint{1.906215in}{0.935963in}}%
\pgfpathlineto{\pgfqpoint{1.934972in}{0.935284in}}%
\pgfpathlineto{\pgfqpoint{1.963729in}{0.934641in}}%
\pgfpathlineto{\pgfqpoint{1.992486in}{0.934033in}}%
\pgfpathlineto{\pgfqpoint{2.021243in}{0.933456in}}%
\pgfpathlineto{\pgfqpoint{2.050000in}{0.932909in}}%
\pgfpathlineto{\pgfqpoint{2.078757in}{0.932390in}}%
\pgfpathlineto{\pgfqpoint{2.107514in}{0.931896in}}%
\pgfpathlineto{\pgfqpoint{2.136271in}{0.931427in}}%
\pgfpathlineto{\pgfqpoint{2.165028in}{0.930980in}}%
\pgfpathlineto{\pgfqpoint{2.193785in}{0.930554in}}%
\pgfpathlineto{\pgfqpoint{2.222542in}{0.930148in}}%
\pgfpathlineto{\pgfqpoint{2.251299in}{0.929761in}}%
\pgfpathlineto{\pgfqpoint{2.280056in}{0.929392in}}%
\pgfpathlineto{\pgfqpoint{2.308813in}{0.929039in}}%
\pgfpathlineto{\pgfqpoint{2.337570in}{0.928702in}}%
\pgfpathlineto{\pgfqpoint{2.366327in}{0.928379in}}%
\pgfpathlineto{\pgfqpoint{2.395083in}{0.928071in}}%
\pgfpathlineto{\pgfqpoint{2.423840in}{0.927775in}}%
\pgfpathlineto{\pgfqpoint{2.452597in}{0.927493in}}%
\pgfpathlineto{\pgfqpoint{2.481354in}{0.927222in}}%
\pgfpathlineto{\pgfqpoint{2.510111in}{0.926962in}}%
\pgfpathlineto{\pgfqpoint{2.538868in}{0.926713in}}%
\pgfpathlineto{\pgfqpoint{2.567625in}{0.926474in}}%
\pgfpathlineto{\pgfqpoint{2.596382in}{0.926244in}}%
\pgfpathlineto{\pgfqpoint{2.625139in}{0.926024in}}%
\pgfpathlineto{\pgfqpoint{2.653896in}{0.925813in}}%
\pgfpathlineto{\pgfqpoint{2.682653in}{0.925609in}}%
\pgfpathlineto{\pgfqpoint{2.711410in}{0.925414in}}%
\pgfpathlineto{\pgfqpoint{2.740167in}{0.925226in}}%
\pgfpathlineto{\pgfqpoint{2.768924in}{0.925045in}}%
\pgfpathlineto{\pgfqpoint{2.797681in}{0.924871in}}%
\pgfpathlineto{\pgfqpoint{2.826438in}{0.924704in}}%
\pgfpathlineto{\pgfqpoint{2.855195in}{0.924543in}}%
\pgfpathlineto{\pgfqpoint{2.883952in}{0.924388in}}%
\pgfpathlineto{\pgfqpoint{2.912709in}{0.924238in}}%
\pgfpathlineto{\pgfqpoint{2.941466in}{0.924094in}}%
\pgfpathlineto{\pgfqpoint{2.970223in}{0.923955in}}%
\pgfpathlineto{\pgfqpoint{2.998980in}{0.923822in}}%
\pgfpathlineto{\pgfqpoint{3.027737in}{0.923693in}}%
\pgfpathlineto{\pgfqpoint{3.056494in}{0.923569in}}%
\pgfpathlineto{\pgfqpoint{3.085250in}{0.923449in}}%
\pgfpathlineto{\pgfqpoint{3.114007in}{0.923333in}}%
\pgfpathlineto{\pgfqpoint{3.142764in}{0.923221in}}%
\pgfpathlineto{\pgfqpoint{3.171521in}{0.923114in}}%
\pgfpathlineto{\pgfqpoint{3.200278in}{0.923010in}}%
\pgfpathlineto{\pgfqpoint{3.229035in}{0.922910in}}%
\pgfpathlineto{\pgfqpoint{3.257792in}{0.922813in}}%
\pgfpathlineto{\pgfqpoint{3.286549in}{0.922719in}}%
\pgfpathlineto{\pgfqpoint{3.315306in}{0.922629in}}%
\pgfpathlineto{\pgfqpoint{3.344063in}{0.922542in}}%
\pgfpathlineto{\pgfqpoint{3.372820in}{0.922458in}}%
\pgfpathlineto{\pgfqpoint{3.401577in}{0.922376in}}%
\pgfpathlineto{\pgfqpoint{3.430334in}{0.922298in}}%
\pgfpathlineto{\pgfqpoint{3.459091in}{0.922222in}}%
\pgfusepath{stroke}%
\end{pgfscope}%
\begin{pgfscope}%
\pgfpathrectangle{\pgfqpoint{0.500000in}{0.192500in}}{\pgfqpoint{3.100000in}{1.347500in}}%
\pgfusepath{clip}%
\pgfsetrectcap%
\pgfsetroundjoin%
\pgfsetlinewidth{0.803000pt}%
\definecolor{currentstroke}{rgb}{0.000000,0.545098,0.545098}%
\pgfsetstrokecolor{currentstroke}%
\pgfsetdash{}{0pt}%
\pgfpathmoveto{\pgfqpoint{0.717496in}{1.550000in}}%
\pgfpathlineto{\pgfqpoint{0.727180in}{1.487705in}}%
\pgfpathlineto{\pgfqpoint{0.755937in}{1.365721in}}%
\pgfpathlineto{\pgfqpoint{0.784694in}{1.280376in}}%
\pgfpathlineto{\pgfqpoint{0.813451in}{1.218028in}}%
\pgfpathlineto{\pgfqpoint{0.842208in}{1.170950in}}%
\pgfpathlineto{\pgfqpoint{0.870965in}{1.134459in}}%
\pgfpathlineto{\pgfqpoint{0.899722in}{1.105562in}}%
\pgfpathlineto{\pgfqpoint{0.928479in}{1.082265in}}%
\pgfpathlineto{\pgfqpoint{0.957236in}{1.063193in}}%
\pgfpathlineto{\pgfqpoint{0.985993in}{1.047372in}}%
\pgfpathlineto{\pgfqpoint{1.014750in}{1.034094in}}%
\pgfpathlineto{\pgfqpoint{1.043506in}{0.554876in}}%
\pgfpathlineto{\pgfqpoint{1.072263in}{0.499394in}}%
\pgfpathlineto{\pgfqpoint{1.101020in}{0.468493in}}%
\pgfpathlineto{\pgfqpoint{1.129777in}{0.460206in}}%
\pgfpathlineto{\pgfqpoint{1.158534in}{0.417193in}}%
\pgfpathlineto{\pgfqpoint{1.187291in}{0.666524in}}%
\pgfpathlineto{\pgfqpoint{1.216048in}{0.639973in}}%
\pgfpathlineto{\pgfqpoint{1.244805in}{0.603005in}}%
\pgfpathlineto{\pgfqpoint{1.273562in}{0.431095in}}%
\pgfpathlineto{\pgfqpoint{1.302319in}{0.391874in}}%
\pgfpathlineto{\pgfqpoint{1.331076in}{0.394328in}}%
\pgfpathlineto{\pgfqpoint{1.359833in}{0.365994in}}%
\pgfpathlineto{\pgfqpoint{1.388590in}{0.355945in}}%
\pgfpathlineto{\pgfqpoint{1.417347in}{0.545904in}}%
\pgfpathlineto{\pgfqpoint{1.446104in}{0.562095in}}%
\pgfpathlineto{\pgfqpoint{1.474861in}{0.573316in}}%
\pgfpathlineto{\pgfqpoint{1.503618in}{0.557316in}}%
\pgfpathlineto{\pgfqpoint{1.532375in}{0.558369in}}%
\pgfpathlineto{\pgfqpoint{1.561132in}{0.559941in}}%
\pgfpathlineto{\pgfqpoint{1.589889in}{0.414728in}}%
\pgfpathlineto{\pgfqpoint{1.618646in}{0.333115in}}%
\pgfpathlineto{\pgfqpoint{1.647403in}{0.313274in}}%
\pgfpathlineto{\pgfqpoint{1.676160in}{0.457743in}}%
\pgfpathlineto{\pgfqpoint{1.704917in}{0.530657in}}%
\pgfpathlineto{\pgfqpoint{1.733673in}{0.520731in}}%
\pgfpathlineto{\pgfqpoint{1.762430in}{0.529374in}}%
\pgfpathlineto{\pgfqpoint{1.791187in}{0.517369in}}%
\pgfpathlineto{\pgfqpoint{1.819944in}{0.520913in}}%
\pgfpathlineto{\pgfqpoint{1.848701in}{0.521030in}}%
\pgfpathlineto{\pgfqpoint{1.877458in}{0.513399in}}%
\pgfpathlineto{\pgfqpoint{1.906215in}{0.520767in}}%
\pgfpathlineto{\pgfqpoint{1.934972in}{0.345888in}}%
\pgfpathlineto{\pgfqpoint{1.963729in}{0.472642in}}%
\pgfpathlineto{\pgfqpoint{1.992486in}{0.502792in}}%
\pgfpathlineto{\pgfqpoint{2.021243in}{0.497495in}}%
\pgfpathlineto{\pgfqpoint{2.050000in}{0.504431in}}%
\pgfpathlineto{\pgfqpoint{2.078757in}{0.494422in}}%
\pgfpathlineto{\pgfqpoint{2.107514in}{0.498974in}}%
\pgfpathlineto{\pgfqpoint{2.136271in}{0.498104in}}%
\pgfpathlineto{\pgfqpoint{2.165028in}{0.494009in}}%
\pgfpathlineto{\pgfqpoint{2.193785in}{0.500040in}}%
\pgfpathlineto{\pgfqpoint{2.222542in}{0.490440in}}%
\pgfpathlineto{\pgfqpoint{2.251299in}{0.495010in}}%
\pgfpathlineto{\pgfqpoint{2.280056in}{0.488382in}}%
\pgfpathlineto{\pgfqpoint{2.308813in}{0.482191in}}%
\pgfpathlineto{\pgfqpoint{2.337570in}{0.487950in}}%
\pgfpathlineto{\pgfqpoint{2.366327in}{0.479168in}}%
\pgfpathlineto{\pgfqpoint{2.395083in}{0.484182in}}%
\pgfpathlineto{\pgfqpoint{2.423840in}{0.482653in}}%
\pgfpathlineto{\pgfqpoint{2.452597in}{0.480691in}}%
\pgfpathlineto{\pgfqpoint{2.481354in}{0.485540in}}%
\pgfpathlineto{\pgfqpoint{2.510111in}{0.477442in}}%
\pgfpathlineto{\pgfqpoint{2.538868in}{0.482142in}}%
\pgfpathlineto{\pgfqpoint{2.567625in}{0.480159in}}%
\pgfpathlineto{\pgfqpoint{2.596382in}{0.478762in}}%
\pgfpathlineto{\pgfqpoint{2.625139in}{0.480505in}}%
\pgfpathlineto{\pgfqpoint{2.653896in}{0.468376in}}%
\pgfpathlineto{\pgfqpoint{2.682653in}{0.472953in}}%
\pgfpathlineto{\pgfqpoint{2.711410in}{0.471002in}}%
\pgfpathlineto{\pgfqpoint{2.740167in}{0.470455in}}%
\pgfpathlineto{\pgfqpoint{2.768924in}{0.473811in}}%
\pgfpathlineto{\pgfqpoint{2.797681in}{0.468093in}}%
\pgfpathlineto{\pgfqpoint{2.826438in}{0.472175in}}%
\pgfpathlineto{\pgfqpoint{2.855195in}{0.469862in}}%
\pgfpathlineto{\pgfqpoint{2.883952in}{0.469654in}}%
\pgfpathlineto{\pgfqpoint{2.912709in}{0.472241in}}%
\pgfpathlineto{\pgfqpoint{2.941466in}{0.467265in}}%
\pgfpathlineto{\pgfqpoint{2.970223in}{0.470919in}}%
\pgfpathlineto{\pgfqpoint{2.998980in}{0.463117in}}%
\pgfpathlineto{\pgfqpoint{3.027737in}{0.462231in}}%
\pgfpathlineto{\pgfqpoint{3.056494in}{0.464514in}}%
\pgfpathlineto{\pgfqpoint{3.085250in}{0.460517in}}%
\pgfpathlineto{\pgfqpoint{3.114007in}{0.464131in}}%
\pgfpathlineto{\pgfqpoint{3.142764in}{0.461605in}}%
\pgfpathlineto{\pgfqpoint{3.171521in}{0.462246in}}%
\pgfpathlineto{\pgfqpoint{3.200278in}{0.463923in}}%
\pgfpathlineto{\pgfqpoint{3.229035in}{0.460434in}}%
\pgfpathlineto{\pgfqpoint{3.257792in}{0.463689in}}%
\pgfpathlineto{\pgfqpoint{3.286549in}{0.460922in}}%
\pgfpathlineto{\pgfqpoint{3.315306in}{0.461741in}}%
\pgfpathlineto{\pgfqpoint{3.344063in}{0.461751in}}%
\pgfpathlineto{\pgfqpoint{3.372820in}{0.454170in}}%
\pgfpathlineto{\pgfqpoint{3.401577in}{0.457423in}}%
\pgfpathlineto{\pgfqpoint{3.430334in}{0.454761in}}%
\pgfpathlineto{\pgfqpoint{3.459091in}{0.456028in}}%
\pgfusepath{stroke}%
\end{pgfscope}%
\begin{pgfscope}%
\pgfpathrectangle{\pgfqpoint{0.500000in}{0.192500in}}{\pgfqpoint{3.100000in}{1.347500in}}%
\pgfusepath{clip}%
\pgfsetbuttcap%
\pgfsetroundjoin%
\pgfsetlinewidth{0.602250pt}%
\definecolor{currentstroke}{rgb}{0.000000,0.000000,0.000000}%
\pgfsetstrokecolor{currentstroke}%
\pgfsetdash{{2.220000pt}{0.960000pt}}{0.000000pt}%
\pgfpathmoveto{\pgfqpoint{0.500000in}{0.406839in}}%
\pgfpathlineto{\pgfqpoint{3.600000in}{0.406839in}}%
\pgfusepath{stroke}%
\end{pgfscope}%
\begin{pgfscope}%
\pgfpathrectangle{\pgfqpoint{0.500000in}{0.192500in}}{\pgfqpoint{3.100000in}{1.347500in}}%
\pgfusepath{clip}%
\pgfsetbuttcap%
\pgfsetroundjoin%
\pgfsetlinewidth{0.602250pt}%
\definecolor{currentstroke}{rgb}{0.000000,0.000000,0.000000}%
\pgfsetstrokecolor{currentstroke}%
\pgfsetdash{{2.220000pt}{0.960000pt}}{0.000000pt}%
\pgfpathmoveto{\pgfqpoint{0.500000in}{0.641667in}}%
\pgfpathlineto{\pgfqpoint{3.600000in}{0.641667in}}%
\pgfusepath{stroke}%
\end{pgfscope}%
\begin{pgfscope}%
\pgfsetrectcap%
\pgfsetmiterjoin%
\pgfsetlinewidth{0.803000pt}%
\definecolor{currentstroke}{rgb}{0.000000,0.000000,0.000000}%
\pgfsetstrokecolor{currentstroke}%
\pgfsetdash{}{0pt}%
\pgfpathmoveto{\pgfqpoint{0.500000in}{0.192500in}}%
\pgfpathlineto{\pgfqpoint{0.500000in}{1.540000in}}%
\pgfusepath{stroke}%
\end{pgfscope}%
\begin{pgfscope}%
\pgfsetrectcap%
\pgfsetmiterjoin%
\pgfsetlinewidth{0.803000pt}%
\definecolor{currentstroke}{rgb}{0.000000,0.000000,0.000000}%
\pgfsetstrokecolor{currentstroke}%
\pgfsetdash{}{0pt}%
\pgfpathmoveto{\pgfqpoint{3.600000in}{0.192500in}}%
\pgfpathlineto{\pgfqpoint{3.600000in}{1.540000in}}%
\pgfusepath{stroke}%
\end{pgfscope}%
\begin{pgfscope}%
\pgfsetrectcap%
\pgfsetmiterjoin%
\pgfsetlinewidth{0.803000pt}%
\definecolor{currentstroke}{rgb}{0.000000,0.000000,0.000000}%
\pgfsetstrokecolor{currentstroke}%
\pgfsetdash{}{0pt}%
\pgfpathmoveto{\pgfqpoint{0.500000in}{0.192500in}}%
\pgfpathlineto{\pgfqpoint{3.600000in}{0.192500in}}%
\pgfusepath{stroke}%
\end{pgfscope}%
\begin{pgfscope}%
\pgfsetrectcap%
\pgfsetmiterjoin%
\pgfsetlinewidth{0.803000pt}%
\definecolor{currentstroke}{rgb}{0.000000,0.000000,0.000000}%
\pgfsetstrokecolor{currentstroke}%
\pgfsetdash{}{0pt}%
\pgfpathmoveto{\pgfqpoint{0.500000in}{1.540000in}}%
\pgfpathlineto{\pgfqpoint{3.600000in}{1.540000in}}%
\pgfusepath{stroke}%
\end{pgfscope}%
\begin{pgfscope}%
\pgfsetbuttcap%
\pgfsetmiterjoin%
\definecolor{currentfill}{rgb}{1.000000,1.000000,1.000000}%
\pgfsetfillcolor{currentfill}%
\pgfsetfillopacity{0.800000}%
\pgfsetlinewidth{1.003750pt}%
\definecolor{currentstroke}{rgb}{0.800000,0.800000,0.800000}%
\pgfsetstrokecolor{currentstroke}%
\pgfsetstrokeopacity{0.800000}%
\pgfsetdash{}{0pt}%
\pgfpathmoveto{\pgfqpoint{2.803766in}{1.009509in}}%
\pgfpathlineto{\pgfqpoint{3.502778in}{1.009509in}}%
\pgfpathquadraticcurveto{\pgfqpoint{3.530556in}{1.009509in}}{\pgfqpoint{3.530556in}{1.037287in}}%
\pgfpathlineto{\pgfqpoint{3.530556in}{1.442778in}}%
\pgfpathquadraticcurveto{\pgfqpoint{3.530556in}{1.470556in}}{\pgfqpoint{3.502778in}{1.470556in}}%
\pgfpathlineto{\pgfqpoint{2.803766in}{1.470556in}}%
\pgfpathquadraticcurveto{\pgfqpoint{2.775988in}{1.470556in}}{\pgfqpoint{2.775988in}{1.442778in}}%
\pgfpathlineto{\pgfqpoint{2.775988in}{1.037287in}}%
\pgfpathquadraticcurveto{\pgfqpoint{2.775988in}{1.009509in}}{\pgfqpoint{2.803766in}{1.009509in}}%
\pgfpathlineto{\pgfqpoint{2.803766in}{1.009509in}}%
\pgfpathclose%
\pgfusepath{stroke,fill}%
\end{pgfscope}%
\begin{pgfscope}%
\pgfsetrectcap%
\pgfsetroundjoin%
\pgfsetlinewidth{0.803000pt}%
\definecolor{currentstroke}{rgb}{0.721569,0.525490,0.043137}%
\pgfsetstrokecolor{currentstroke}%
\pgfsetdash{}{0pt}%
\pgfpathmoveto{\pgfqpoint{2.831544in}{1.358088in}}%
\pgfpathlineto{\pgfqpoint{2.914877in}{1.358088in}}%
\pgfpathlineto{\pgfqpoint{2.998211in}{1.358088in}}%
\pgfusepath{stroke}%
\end{pgfscope}%
\begin{pgfscope}%
\definecolor{textcolor}{rgb}{0.000000,0.000000,0.000000}%
\pgfsetstrokecolor{textcolor}%
\pgfsetfillcolor{textcolor}%
\pgftext[x=3.109322in,y=1.309477in,left,base]{\color{textcolor}{\sffamily\fontsize{10.000000}{12.000000}\selectfont\catcode`\^=\active\def^{\ifmmode\sp\else\^{}\fi}\catcode`\%=\active\def
\end{pgfscope}%
\begin{pgfscope}%
\pgfsetrectcap%
\pgfsetroundjoin%
\pgfsetlinewidth{0.803000pt}%
\definecolor{currentstroke}{rgb}{0.000000,0.545098,0.545098}%
\pgfsetstrokecolor{currentstroke}%
\pgfsetdash{}{0pt}%
\pgfpathmoveto{\pgfqpoint{2.831544in}{1.148398in}}%
\pgfpathlineto{\pgfqpoint{2.914877in}{1.148398in}}%
\pgfpathlineto{\pgfqpoint{2.998211in}{1.148398in}}%
\pgfusepath{stroke}%
\end{pgfscope}%
\begin{pgfscope}%
\definecolor{textcolor}{rgb}{0.000000,0.000000,0.000000}%
\pgfsetstrokecolor{textcolor}%
\pgfsetfillcolor{textcolor}%
\pgftext[x=3.109322in,y=1.099787in,left,base]{\color{textcolor}{\sffamily\fontsize{10.000000}{12.000000}\selectfont\catcode`\^=\active\def^{\ifmmode\sp\else\^{}\fi}\catcode`\%=\active\def
\end{pgfscope}%
\end{pgfpicture}%
\makeatother%
\endgroup%

%% file: lemmas/constructionmnBound.tex
Given the assumptions and $\chat$ as in Lemma \ref{constructionIsConcave}. Furthermore, assume
\begin{equation}\label{conditionSmallOnLeft}
  \frac{\cbar_{r_*,s_*+1}}{\cbar_{r_*,s_*}}  \le e^{-\eps} \text{ for all } r_*+s_* \ge M 
\end{equation}
for some $0<\eps\le \delta_1$. Then, we have 
\begin{equation*}
  m_{n+1}(\chat) \coloneqq \frac{\sumrs[n+1][] \chat\rs}{\sumrs[n][] \chat\rs} \le e^{-\eps} \frac{r^m(n,\chat) - r_m(n,\chat)+2}{r^m(n,\chat)-r_m(n,\chat)+1} \text{ for all } n \ge M.
\end{equation*}

%% file: proofs/constructionmnBound.tex
We will denote $r^m(n) = r^m(n,\chat)$ and $s^m(n) = n-r^m(n)$ in this proof and fix $n\ge M$.

We start by showing $\frac{\chat\rsp}{\chat\rs}\le e^{-\eps}$ for all $\max(0,r_*(n+1)-1)\le r \le \max(r^m,r_*(n+1))$ and $r+s =n.$ We use an induction $(r-1,s+1)\mapsto (r,s)$.\\
\underline{Induction start}: Let $r= \max(0,r_*(n+1)-1)$ and $r+s = n.$ Then we have 
$\chat\rsp \le e^{-\delta_1 (n+1)},$ because if $r_*(n+1)>0$, it is true by the definition of $r_*$ and if $r_*(n+1) = 0$, it holds with equality by the way we defined \constructionii{} for $r=0.$ From this we find
\[ \frac{\chat\rsp}{\chat\rs} = 
\begin{cases}
  =\frac{\chat\rsp}{\cbar\rs} \overset{\text{Lemma \ref{constructionIsConcave}.\ref{constructionIsConcave:le}}}\le \frac{\cbar\rsp}{\cbar\rs} \overset{\cbar{\text{ concave}}} \le \frac{\cbar_{r_*(n),s_*(n)+1}}{\cbar_{r_*(n),s_*(n)}}\overset{\text{\eqref{conditionSmallOnLeft}}} \le e^{-\eps} &\text{ if } r<r_*(n)\\
  \le \frac{e^{-\delta_1(n+1)}}{\chat\rs} \overset{\text{Lemma \ref{constructionIsConcave}.\ref{constructionIsConcave:large}}} \le e^{-\delta_1}\le e^{-\eps}& \text{ if }r = r_*(n)
\end{cases}
\]
Due to Lemma \ref{constructionIsConcave}.\ref{constructionIsConcave:rstar} we know $r\le r_*(n)$ and thus, these are all the possible cases.\\
\underline{Induction step $(r-1,s+1) \mapsto (r,s)$}:
For $r_*(n+1) \le \max(0,r_*(n+1)-1)+1 \le r \le \max(r^m,r_*(n+1))$ and $r+s=n$ we find
\begin{equation}
  \frac{\chat\rsp}{\chat\rs} \overset{\text{\constructionii}}= 
  \begin{cases}
    =\frac{e^{-\delta_1(n+1)}}{\chat\rs} \overset{\text{Lemma \ref{constructionIsConcave}.\ref{constructionIsConcave:large}: }} \le e^{-\delta_1}\le e^{-\eps} &\text{ if } \chat\rsp =  e^{-\delta_1(n+1)},\\
    = \frac{\chat\rmspp}{\chat\rmsp} \overset{\text{induction}}\le e^{-\eps} &\text{ else}.
  \end{cases}
\end{equation}
Next, we show $\displaystyle \frac{\max\limits_{r+s = n+1} \chat\rs}{\max\limits_{r+s=n} \chat\rs} \le e^{-\eps}.$ By the concavity, we know $r^m(n+1) \in \{r^m(n),r^m(n+1)\}$. If $r^m(n+1) = r^m(n),$ this is just $\frac{\max\limits_{r+s = n+1} \chat\rs}{\max\limits_{r+s=n} \chat\rs} = \frac{\chat_{r^m(n),s^m(n)+1}}{\chat_{r^m(n),s^m(n)}}$, which we have estimated above, because $\max(0,r_*(n+1)-1)) \overset{\text{Lemma \ref{constructionIsConcave}.\ref{constructionIsConcave:rstar}}} \le \max(0,r_*(n)) \le r^m(n)$. And if $r^m(n+1) = r^m(n)+1$, there are two cases to distinguish. If $r^m(n+1) = r_*(n+1),$ then we estimate
\begin{equation}
  \begin{split}
    &\frac{\max\limits_{r+s = n+1} \chat\rs}{\max\limits_{r+s=n} \chat\rs} = \frac{\chat_{r^m(n)+1,s^m(n)}}{\chat_{r^m(n),s^m(n)}} \overset{\text{\constructionii}}=\\&=
    \begin{cases}
      \frac{e^{-\delta_1(n+1)}}{\chat_{r^m(n),s^m(n)}} \overset{\chat_{r^m(n),s^m(n)} \ge e^{-\delta_1n}} \le e^{-\delta_1} & \text{ if } \chat_{r^m(n)+1,s^m(n)} = e^{-\delta_1 (n+1)},\\
      \frac{\chat_{r^m(n),s^m(n)+1}}{\chat_{r^m(n),s^m(n)}} \frac{\chat_{r^m(n)+1,s^m(n)-1}}{\chat_{r^m(n),s^m(n)}}\overset{\text{def. of }r^m}\le \frac{\chat_{r^m(n),s^m(n)+1}}{\chat_{r^m(n),s^m(n)}}\overset{\text{above}}\le e^{-\eps} &\text{ else}.
    \end{cases}
  \end{split}
\end{equation}
Otherwise we have $r^m(n)+1 = r^m(n+1)>\max(r^m(n),r_*(n+1))$ and hence 
\begin{equation}
  \frac{\max\limits_{r+s = n+1} \chat\rs}{\max\limits_{r+s=n} \chat\rs} = \frac{\chat_{r^m(n)+1,s^m(n)}}{\chat_{r^m(n),s^m(n)}} \overset{\text{\constructioniii}}=\underbrace{\frac{\chat_{r^m(n),s^m(n)+1}}{\chat_{r^m(n),s^m(n)}}}_{\le e^{-\eps}}\min\Big(\frac{\cbar_{r^m(n)+1,s^m(n)}}{\cbar_{r^m(n),s^m(n)}},1\Big)\le e^{-\eps}.
\end{equation}
From here, we calculate
\begin{equation*}
  \begin{split}
    \sumrs[n+1][]\chat\rs
    = \sum\limits_{\substack{r+s = n \\ r<r_m(n)}}\chat\rsp
    +\sum\limits_{\substack{r+s = n \\ r_m(n)\le r \le r^m(n)+1}}\chat\rsp
    +\sum\limits_{\substack{r+s = n \\ r^m(n)<r}}\chat\rps
  \end{split}
\end{equation*}
and estimate the three sums individually. First, we have
\begin{equation*}
  \sum\limits_{\substack{r+s = n \\ r<r_m(n)}}\chat\rsp = 
  \sum\limits_{\substack{r+s = n \\ r<r_m(n)}}\underbrace{\frac{\chat\rsp}{\chat\rs}}_{\overset{\text{concavity}}\le \frac{\chat_{r^m(n),n+1-r^m(n)}}{\chat_{r^m(n),n-r^m(n)}}\le \frac{\chat_{r^m(n+1),n+1-r^m(n+1)}}{\chat_{r^m(n),n-r^m(n)}} \le e^{-\eps}} \chat\rs
  \le e^{-\eps}  \sum\limits_{\substack{r+s = n \\ r<r_m(n)}}\chat\rs .
\end{equation*}
And similarly, we find
\begin{equation*}
  \sum\limits_{\substack{r+s = n \\ r^m(n)<r}}\chat\rps = 
  \sum\limits_{\substack{r+s = n \\ r^m(n)<r}}\underbrace{\frac{\chat\rps}{\chat\rs}}_{\overset{\text{concavity}}\le \frac{\chat_{r^m(n)+1,n-r^m(n)}}{\chat_{r^m(n),n-r^m(n)}}\le \frac{\chat_{r^m(n+1),n+1-r^m(n+1)}}{\chat_{r^m(n),n-r^m(n)}} \le e^{-\eps}} \chat\rs
  \le e^{-\eps}  \sum\limits_{\substack{r+s = n \\ r^m(n)<r}}\chat\rs .
\end{equation*}
Lastly, due to ${\frac{\max\limits_{r+s = n+1} \chat\rs}{\max\limits_{r+s=n} \chat\rs}}\le e^{-\eps},$
\begin{equation*}
  \begin{split}
    \sum\limits_{\substack{r+s = n \\ r_m(n)\le r^m(n)+1}}\chat\rsp 
    &\le (r^m(n)-r_m(n) + 2) \max\limits_{r+s = n+1} \chat\rs \le   (r^m(n)-r_m(n) + 2)e^{-\eps} \max\limits_{r+s=n}\chat\rs\\
    &=  \frac{r^m(n)-r_m(n) + 2}{r^m(n)-r_m(n)+1}e^{-\eps}\sum\limits_{\substack{r+s = n \\ r_m(n)\le r \le r^m(n)}}\chat\rs 
  \end{split}
\end{equation*}
holds true.

%% file: lemmas/constructionMNEstimate.tex
Given the assumptions and $\chat$ as in Lemma \ref{constructionmnBound} and $\monr \le \rho,$ $\mons\le \sigma$. Furthermore, assume that there is a sequence $(g_n)\subset \N$, such that $g_n \to \infty$ for $n\to\infty$ satisfying 
\begin{equation}\label{constructionMNEstimate:ass:width}
  r^m(n) - r_*(n) \ge g_n.
\end{equation}
Then, there is a constant $C_{g_n,\eps,M,\rho,\sigma,Q\rs}$ depending only on $\eps$ from Lemma \ref{constructionmnBound}, $M$ from Lemma \ref{constructionIsConcave}, as well as $g_n,\rho,\sigma,Q\rs$, such that 
\[ \summ[r+s][N] \prod\limits_{n=r+s+1}^m m_n(\chat) \le C_{g_n,\eps,M,\rho,\sigma,Q\rs}\quad \text{ and }\quad \sum\limits_{m=2}^{r+s} \prod\limits_{n=m+1}^{r+s} m_n(\chat) \le C_{g_n,\eps,M,\rho,\sigma,Q\rs}.\]
Note, that \eqref{constructionMNEstimate:ass:width} only needs to hold for $n\ge M$, i.e. when $r_*$ is well defined.

%% file: proofs/constructionMNEstimate.tex
It is enough to show $m_n(\chat)\le e^{-\frac \eps 2}$ for large $n$ and bound $m_n(\chat)$ for small $n$. In view of Lemma \ref{constructionmnBound} we need to estimate $r^m(n)-r_m(n)$. Because $g_n\to \infty$ we can fix a $M_1>M$, such that $\frac{g_n+2}{g_n+1} \le e^{\frac \eps 2}$ for all $n\ge M_1$. Furthermore, we can fix an $M_2>M_1$ such that $\frac{g_{M_2}-M_1 + 2}{g_{M_2}-M_1 + 1} \le e^{\frac \eps 2}.$ Let $M_3 \coloneqq \min\left(M_2,\min\{n\ge M_1\,|\, r_m(n) \le r_*(n)\}\right)$ ($\min(\emptyset)\coloneqq \infty$). \\
\underline{Claim 1}: $\displaystyle \frac{r^m(M_3) - r_m(M_3) + 2}{r^m(M_3)-r_m(M_3) + 1}\le e^{\frac \eps 2}.$\\
If $M_2 \ge \min\{n\ge M_1\,|\, r_m(n) \le r_*(n)\}$, from the choice of $M_1$, we find $\frac{r^m(M_3)-r_m(M_3)+2}{r^m(M_3)-r_m(M_3)+1}\le \frac{g_{M_3}+2}{g_{M_3}+1} \le e^{\frac \eps 2}$.\\
On the other hand, if $M_2 < \min\{n\ge M_1\,|\, r_m(n) \le r_*(n)\}$, we have $r_m(n) > r_*(n)$ for all $M_1 \le n \le M_2$. Putting this in Lemma \ref{constructionLeftMaximum}, we obtain for all $M_1\le n \le M_2$ that $r_m(n) \le r_m(n-1).$ So inductively, we find $r_m(M_2) \le r_m(M_1) \le M_1$. Combining this with $r^m(M_2) \ge g_{M_2}$, we obtain
\[ \frac{r^m(M_3) - r_m(M_3) + 2}{r^m(M_3)-r_m(M_3) + 1} = \frac{r^m(M_2) - r_m(M_2) + 2}{r^m(M_2)-r_m(M_2) + 1} \le \frac{g_{M_2} - M_1 + 2}{g_{M_2}-M_1 + 1} \overset{\text{choice of }M_2}\le e^{\frac \eps 2}.\]
\underline{Claim 2}: For all $n\ge M_3$ we have $\displaystyle r^m(n) - r_m(n) \ge \inf\limits_{k \in \{M_3,\ldots,n\}}\Big(r^m(k) - r_*(k), r^m(M_3) - r_m(M_3)\Big).$\\
This follows from an induction. For $n=M_3$ the claim is true. The induction step $n\mapsto n+1$ is as follows. If $r_m(n+1) \le r_*(n+1)$, the statement is of course true. If $r_m(n+1) > r_*(n+1)$, Lemma \ref{constructionLeftMaximum} shows again $r_m(n+1)\le r_m(n)$. The concavity given by Lemma \ref{constructionIsConcave} implies $r^m(n+1) \ge r^m(n)$. Combining this yields
\begin{equation*}
  \begin{split}
    r^m(n+1) - r_m(n+1) &\ge r^m(n) - r_m(n) \overset{\text{ind.}} \ge \inf\limits_{k \in \{M_3,\ldots,n\}}\Big(r^m(k) - r_*(k), r^m(M_3) - r_m(M_3)\Big)\\
    &\ge \inf\limits_{k \in \{M_3,\ldots,n+1\}}\Big(r^m(k) - r_*(k), r^m(M_3) - r_m(M_3)\Big).
  \end{split}
\end{equation*}
Combining Claim 1, Claim 2 and the choice of $M_1$, we have shown, that for all $n\ge M_3$ we have $\frac{r^m(n) - r_m(n) + 2}{r^m(n)-r_m(n) + 1}\le e^{\frac \eps 2}$. Lemma \ref{constructionmnBound} implies $m_n(\chat) \le e^{-\frac \eps 2}$ for all $n> M_3.$\\
For $M<n\le M_3$, Lemma \ref{constructionmnBound} yields $m_n \le 2 e^{-\eps}$. And for $1<n\le M$, we can use 
\[m_n \overset{\text{def.}} = \frac{\sumrs[n][]\cbar\rs}{\sumrs[n-1][] \cbar\rs} \le \max(\monr,\mons) \frac{\sumrs[n][] Q\rs}{\min(Q_{0,n-1},Q_{n-1,0})}.\]

%% file: lemmas/constructionIsConcaveCirc.tex
Fix $\monr,\mons>0$. Assume that $Q\rs$ satisfies the concavity condition for $r+s \ge M$ and that for $s^* = s^*(n)$ and $r^* = n-s^*$
\begin{equation}\label{conditionGrowthOnLeftCirc}
  \frac{\cbar_{r^*,s^*}}{\cbar_{r^*,s^*-1}} \ge e^{-\delta_1} \text{ holds for all } r^*+s^*>M.
\end{equation}
Then, $\ccirc\rs$ constructed via \eqref{ccircConstruction} starting with $\ccirc\rs = \cbar\rs$ for $r+s \le M$ satisfies \eqref{constructionConcave} for all $r+s \ge M$.
Furthermore, $\ccirc\rs \le \cbar\rs$, $s_m(n,\cbar) = s_m(n,\ccirc)$ and $\ccirc\sn \ge e^{-\delta_1n}$ for $s^*(n) \le s \le s_m(n)$ hold true.

Note, that \eqref{conditionGrowthOnLeftCirc} in particular asks for $s^*(n)$ to be well defined, i.e. $\{s\,|\, \cbar\sn \ge e^{-\delta_1 n}\} \neq \emptyset$ and \eqref{conditionGrowthOnLeftCirc} is an empty statement if $s^*=0$.

%% file: lemmas/constructionLeftMaximumCirc.tex
Given the assumptions and $\ccirc$ as in Lemma \ref{constructionIsConcaveCirc}. Then, we have for $s^m(n) \coloneqq \min\big\{s \,\big|\, \cbar\sn \ge \cbar\kn\big\}$
\[ s^m(n+1,\ccirc) \le \max\big(s^*(n+1), s^m(n,\ccirc)\big) \text{ for all }n\ge M.\]

%% file: lemmas/constructionmnBoundCirc.tex
Given the assumptions and $\ccirc$ as in Lemma \ref{constructionIsConcaveCirc}. Furthermore, assume
\begin{equation}\label{conditionSmallOnRight}
  \frac{\cbar_{r^*+1,s^*}}{\cbar_{r^*,s^*}}  \le e^{-\eps} \text{ for all } r_*+s_* \ge M 
\end{equation}
for some $0<\eps\le \delta_1$. Then, we have 
\begin{equation*}
  m_{n+1}(\ccirc) \coloneqq \frac{\sumrs[n+1][] \ccirc\rs}{\sumrs[n][] \ccirc\rs} \le e^{-\eps} \frac{r^m(n,\ccirc) - r_m(n,\ccirc)+2}{r^m(n,\ccirc)-r_m(n,\ccirc)+1} \text{ for all } n \ge M.
\end{equation*}

%% file: lemmas/constructionMNEstimateCirc.tex
Given the assumptions and $\ccirc$ as in Lemma \ref{constructionmnBoundCirc} and $\monr \le \rho,$ $\mons\le \sigma$. Furthermore, assume that there is a sequence $(g_n)\subset \N$, such that $g_n \to \infty$ for $n\to\infty$ satisfying 
\begin{equation}\label{constructionMNEstimateCirc:ass:width}
  s_m(n) - s^*(n) \ge g_n.
\end{equation}
Then, there is a constant $C_{g_n,\eps,M,\rho,\sigma,Q\rs}$ depending only on $\eps$ from Lemma \ref{constructionmnBoundCirc}, $M$ from Lemma \ref{constructionIsConcaveCirc}, as well as $g_n,\rho,\sigma,Q\rs$, such that 
\[ \summ[r+s][N] \prod\limits_{n=r+s+1}^m m_n(\ccirc) \le C_{g_n,\eps,M,\rho,\sigma,Q\rs}\quad \text{ and }\quad \sum\limits_{m=2}^{r+s} \prod\limits_{n=m+1}^{r+s} m_n(\ccirc) \le C_{g_n,\eps,M,\rho,\sigma,Q\rs}.\]

%% file: lemmas/maxQuotientLimit.tex
Assume $Q\rs$ satisfies \eqref{assumptionQn}. Then, we already have 
\begin{equation}
  \lim\limits_{n \to \infty} q_n = \limsup_{r+s \to \infty} \sqrt[r+s]{\cbar\rs}.
\end{equation}

%% file: proofs/maxQuotientLimi.tex
The proof is the essentially the same as for Lemma \ref{mnLimit}.

%% file: proofs/maxcbarLimit.tex
From Lemma \ref{maxQuotientLimit}, we know that $q_\infty(\monr,\mons)$ is a monotonically increasing function in both $\monr$ and $\mons$. For any $\eps>0$, we fix an $M<\infty$, such that $q_\infty e^{-\eps} \le q_n \le q_\infty e^\eps$ for all $n\ge M$.

\textbf{Step 1:} $\displaystyle \sqrt[n]{\max\limits_{r+s=n} \cbar\rs} \le q_\infty e^{2\eps}$ for $n$ large independent of $\monr,\mons$.

From the monotonicity we find $q_\infty(\monr,\mons) \ge q_\infty(0,\mons) = \mons q_\infty(0,1)$ and $q_\infty(\monr,\mons) \ge \monr q_\infty(1,0)$. Lemma \ref{maxQuotientLimit} together with \eqref{assCoeff} implies, that $q_\infty(1,0)>0$ and $q_\infty(0,1)>0$. For $n\ge M$ we find 
\begin{equation*}
  \begin{split}
    \max\limits_{r+s=n} \cbar\rs = 
    \max\limits_{r+s=M} \cbar\rs \prod\limits_{k=M+1}^n q_k  
    \le q_\infty(\monr,\mons)^{n-M} e^{\eps(n-M)} \max\limits_{r+s=M} \cbar\rs.
  \end{split}
\end{equation*}
So we need to estimate $q_\infty(\monr,\mons)^{-M} e^{-\eps M} \max\limits_{r+s=M} \cbar\rs$ independent of $\monr,\mons$. We can do that via 
\begin{equation*}
  \begin{split}
    \max\limits_{r+s=M} \cbar\rs\le \max(\monr,\mons)^M \max\limits_{r+s=M} Q\rs \text{ and } q_\infty(\monr,\mons) \ge \max(\monr,\mons)\min(q_\infty(1,0),q_\infty(0,1)).
  \end{split}
\end{equation*}
Putting this together we find 
\begin{equation*}
  \sqrt[n]{\max\limits_{r+s=n} \cbar\rs} 
  \le q_\infty e^\eps \underbrace{\sqrt[n]{e^{-\eps M} \max\limits_{r+s=M} Q\rs\big[\min(q_\infty(1,0),q_\infty(0,1))\big]^{-M}}}_{\to 1 \text{ as }n\to\infty}.
\end{equation*}
\textbf{Step 2:} $\displaystyle \sqrt[n]{\max\limits_{r+s=n} \cbar\rs} \ge q_\infty e^{-2\eps}$ for $n$ large independent of $\monr,\mons$.

Once more, we find from the monotonicity 
$$q_\infty(\monr,\mons) \le q_\infty(\max(\monr,\mons),\max(\monr,\mons)) = \max(\monr,\mons) q_\infty(1,1).$$ 
Again Lemma \ref{maxQuotientLimit} together with \eqref{assCoeff} implies $q_\infty(1,1) < \infty$. And again we find for $n\ge M$
\[ \max\limits_{r+s=n} \cbar\rs \ge  q_\infty^{n-M} e^{-\eps(n-M)} \max\limits_{r+s=M}\cbar\rs.\]
Next, we will use $\max\limits_{r+s=M}\cbar\rs \ge \max(\cbar_{M,0},\cbar_{0,M}) \ge \max(\monr,\mons)^M \min(Q_{M,0},Q_{0,M})$ to find 
\begin{equation*}
  \sqrt[n]{\max\limits_{r+s=n} \cbar\rs} 
  \ge q_\infty e^{-\eps} \underbrace{\sqrt[n]{e^{\eps M} \min(Q_{M,0},Q_{0,M}) q_\infty(1,1)^{-M}}}_{\to 1 \text{ as }n\to\infty}.
\end{equation*}

%% file: propositions/supercritIsApplicable.tex
Assume that $Q\rs$ satisfies \eqref{assCoeff}, \eqref{concavityAss:old} for all $n \ge L$ and \eqref{assumptionQn}. Then, there is a constant $C<\infty$, such that $\frac{Q\rs}{Q\rps} \le C$ and $\frac{Q\rs}{Q\rsp}\le C.$\\
If we furthermore assume \eqref{assumption:boundaryComb}, then for fixed $\eta,\delta_1,\rho,\sigma > 0$, there is an $M\in\N$ and a $(g_n)\subset \N$ such that
\begin{enumerate}
  \item\label{supercritIsApplicable:i} all the assumptions of Lemmas \ref{constructionIsConcave}, \ref{constructionmnBound} and \ref{constructionMNEstimate} are true for any $c\in X^+_{\rho,\sigma}$ with $\min(\monr,\mons) \ge \eta$, $q_\infty(\monr,\mons) \ge e^{-\frac{\delta_1}2}$ and $q_\infty(0,\mons) \le e^{-\frac 3 4 {\delta_1}}$
    and 
  \item\label{supercritIsApplicable:ii} all the assumptions of Lemmas \ref{constructionIsConcaveCirc}, \ref{constructionmnBoundCirc} and \ref{constructionMNEstimateCirc} are true for any $c\in X^+_{\rho,\sigma}$ with $\min(\monr,\mons) \ge \eta$, $q_\infty(\monr,\mons) \ge e^{-\frac{\delta_1}2}$ and $q_\infty(\monr,0) \le e^{-\frac 3 4 {\delta_1}}.$
\end{enumerate}

%% file: proofs/supercritIsApplicable.tex
Assume that $Q\rs$ satisfies \eqref{assCoeff}, \eqref{concavityAss:old} for all $n \ge L$ and \eqref{assumptionQn}. Then we have for $r+s \ge L$ already $\displaystyle \frac{Q\rps}{Q\rs} \overset{\text{\eqref{concavityAss:old}}} \ge \frac{Q_{r+s+1,0}}{Q_{r+s,0}} = q_{r+s+1}(1,0)$. Furthermore, due to \eqref{assumptionQn} we find for $r+s$ sufficiently large 
\[q_{r+s+1}(1,0) \overset{\text{\eqref{assumptionQn}}}\ge \frac 1 2 q_\infty(1,0) \overset{\text{Lemma \ref{maxQuotientLimit}}} = \frac 1 2 \limsup_{n\to \infty} \sqrt[n]{Q_{n,0}} \ge \frac 1 2 \liminf_{r+s \to \infty} \sqrt[r+s]{Q\rs}.\]
Now \eqref{assCoeff} yields the desired bound on $\frac{Q\rs}{Q\rps}.$ Similarly, we have for large $r+s$
\[ \frac{Q\rsp}{Q\rs} \overset{\text{\eqref{concavityAss:old}}}\ge \frac{Q_{0,r+s+1}}{Q_{0,r+s}} \ge \frac 1 2 q_\infty(0,1) \ge \frac 1 2 \liminf_{r+s\to \infty}\sqrt[r+s]{Q\rs} .\]
Next, we want to show \ref{supercritIsApplicable:i}, so assume that $Q\rs$ and $c\rs$ are as in \ref{supercritIsApplicable:i}. First we will find an $M$, such that 
\begin{enumerate}[label=\alph*),ref=\alph*)]
  \item\label{supercritIsApplicable:rstar} $r_*(\cbar,n)$ is well defined for $n\ge M$ and
  \item\label{supercritIsApplicable:large} $\frac{\cbar_{r_*,s_*}}{\cbar_{r_*-1,s_*}} \ge e^{-\delta_1}$ for all $r_*+s_* > M$,
  \item\label{supercritIsApplicable:small} $\frac{\cbar_{r_*,s_*+1}}{\cbar_{r_*,s_*}}\le e^{-\frac{\delta_1}{4}}$ for all $r_*+s_* \ge M$.
\end{enumerate}
Point \ref{supercritIsApplicable:rstar} follows from Proposition \ref{maxcbarLimit}, because $q_\infty \ge e^{-\frac{\delta_1}2}.$\\
To show \ref{supercritIsApplicable:large}, let $M\ge L$ be such that $q_n(\monr,\mons) \ge q_\infty e^{-\frac {\delta_1} 2}$ for all $n\ge M.$ Then, we have for any $n>M$ and $r_* = r_*(n) > 0$ and $s_* = n-r_*$
\[ \frac{\cbar_{r_*,s_*}}{\cbar_{r_*-1,s_*}} \overset{\text{\eqref{concavityAss:old}}}\ge \frac{\cbar_{r^m(n),n-r^m(n)}}{\cbar_{r^m(n)-1,n-r^m(n)}} \ge q_{n}\ge e^{-\delta_1}.\]
Next, we prove \ref{supercritIsApplicable:small}. Let $n\ge L$ be large enough, so that $\max\limits_{r+s =n} \sqrt[n]{\cbar\rs} > e^{-\delta_1}$. If $r_*(n) = 0$, then $\frac{\cbar_{r_*,s_*+1}}{\cbar_{r_*,s_*}} = q_n(0,\mons)$ and for $n$ large enough due to \eqref{assumptionQn}, we have $q_n(0,\mons) \le e^{ \frac{\delta_1}2} q_\infty(0,\mons) \le e^{-\frac{\delta_1}4}.$
 Otherwise let $\eps \in (0,1)$ be such that 
 \[\max\limits_{r+s = r_*+s_*}\big\{(\eps \monr)^r \mons^s Q\rs\big\} = e^{-\delta_1 (r_*+s_*)}.\] 
 Because $\eps \le 1$, we get 
 \[\tilde r \coloneqq \max\big\{ 0\le r\le n \, \big| \, (\eps \monr)^r \mons^{n-r} Q\rn = e^{-\delta_1 n}\big\} \ge r_*(n).\]
 From here we find for $n$ large enough, due to \eqref{assumptionQn}, Lemma \ref{maxQuotientLimit} and Proposition \ref{maxcbarLimit}
\begin{equation*}
  \begin{split}
    \frac{\cbar_{r_*,s_*+1}}{\cbar_{r_*,s_*}} 
    \overset{\text{\eqref{concavityAss:old}}}&\le \frac{\cbar_{\tilde r,n+1-\tilde r}}{\cbar_{\tilde r,n-\tilde r}}
    \le q_{n+1}(\eps \monr, \mons)
    \overset{\text{\eqref{assumptionQn}}}\le e^{\frac{\delta_1}{4}} q_\infty(\eps \monr,\mons)\\
    \overset{\text{Lemma \ref{maxQuotientLimit}}}& = e^{\frac{\delta_1}{4}} \limsup_{r+s \to \infty} \sqrt[r+s]{(\eps \monr)^r \mons^s Q\rs}
    \overset{\text{Proposition \ref{maxcbarLimit}}} \le e^{\frac{\delta_1}{2}} \sqrt[n]{\max\limits_{r+s = n}\big\{(\eps \monr)^r \mons^s Q\rs\big\}} = e^{-\frac  {\delta_1}2}.
  \end{split}
\end{equation*}
Finally, we will show that there is a sequence $(g_n)\subset \N$ such that $g_n \to \infty$ as $n \to \infty$ satisfying $r^m(n) - r_*(n) \ge g_n$ (for $n$ large enough so that $r_*$ is well defined).\\
By Proposition \ref{maxcbarLimit}, there is a $M_0$, such that for all $n\ge M_0$, we have $c_{0,n} \le e^{-\big(\frac 3 4 - \frac 1 {16}\big)\delta_1 n}$ and $\max\limits_{r+s = n} \cbar\rs \ge e^{-\big(\frac 1 2 + \frac 1 {16}\big)\delta_1 n}.$ Now, we distinguish the two cases $r_*(n) = 0$ and $r_*(n) >0$. If $r_*(n) = 0$, we calculate for $n\ge M_0$
\begin{equation*}
  \begin{split}
      e^{n \frac{\delta_1} 8 } 
      &= \frac{e^{-\left(\frac 1 2 + \frac 1 {16}\right) \delta_1 n}}{e^{-\left(\frac 3 4 - \frac 1 {16}\right) \delta_1 n}}
      \le \frac{\cbar_{r^m(n),n-r^m(n)}}{\cbar_{0,n}}
      \overset{\text{denote }r^m = r^m(n)}= \left(\frac \monr \mons\right)^{r^m}\frac{Q_{r^m,n-r^m}}{Q_{0,n}}\\
      &= \left(\frac \monr \mons\right)^{r^m}\frac{Q_{1,n-1}}{Q_{0,n}} \cdot \ldots \cdot \frac{Q_{r^m,n-r^m}}{Q_{r^m-1,n-r^m+1}}.
  \end{split}
\end{equation*}
Next, we use \eqref{assumption:boundaryComb} to estimate 
$\displaystyle \frac{Q\rpsm}{Q\rs}  = \underbrace{\frac{Q\rpsm}{Q\rsm}}_{\overset{\text{\eqref{assumption:boundaryComb}}} \le \lambda n} \underbrace{\frac{Q\rsm}{Q\rs}}_{\overset{\text{above}} \le C}$ for $r+s =n$ and $s>0$.
Combining this with $\monr\le \rho$ and $\mons \ge \eta$, we arrive at 
\[e^{n \frac{\delta_1}8} \le \left(\frac{C\lambda \rho n}{\eta}\right)^{r^m}.\]
Taking the logarithm on both sides, we see $r^m(n) \ge \frac{\delta_1} 8 \frac{n}{\ln\left(\frac{C\lambda \rho n}\eta\right)} \to \infty.$\\
If $r_*(n) >0$ we proceed similarly, 
\begin{equation*}
  \begin{split}
  e^{n \frac{7}{16}\delta_1 } 
      &= \frac{e^{-\left(\frac 1 2 + \frac 1 {16}\right) \delta_1 n}}{e^{-\delta_1 n}}
      \le \frac{\cbar_{r^m(n),n-r^m(n)}}{\cbar_{r_*(n)-1,n-r_*(n)+1}}
      \overset{\text{denote }r_* = r_*(n)}= \left(\frac \monr \mons\right)^{r^m-r_*+1}\frac{Q_{r^m,n-r^m}}{Q_{r_*-1,n-r_*+1}}\\
      &\overset{\text{as above}}\le \left(\frac{C\lambda \rho n}{\eta}\right)^{r^m-r_*+1}.
  \end{split}
\end{equation*}
Again, taking the logarithm yields $r^m(n)- r_*(n) \ge \frac{7\delta_1} {16} \frac{n}{\ln\left(\frac{C\lambda \rho n}\eta\right)} - 1 \to \infty.$

The proof of \ref{supercritIsApplicable:ii} is analogous. (Alternatively \ref{supercritIsApplicable:ii} reduces to \ref{supercritIsApplicable:i} with the symmetry $r \leftrightsquigarrow s$.)

%% file: theorems/dissipationSupercritical.tex
Assume $Q\rs$ satisfies \eqref{assCoeff}, \eqref{assumption:boundaryComb}, \eqref{assumptionQn} and \eqref{concavityAss:old} for all $n \ge L$ and fix $\rho,\sigma>0$ as well as $\eta,\delta>0.$ Then, there exist constants $K<\infty$ and $\delta_2,\delta_3>0$ as well as $M<\infty$, such that for all $c\in X^+_{\rho,\sigma}$ and $M\le N\in \N$ with 
\[ q_\infty(\monr,\mons)>e^{-\delta_3},\, \min(\monr,\mons) \ge \eta, \, \sumrs[N+1] (r+s) c\rs \le \delta_2 \text{ and } H[c|\cinfzw] > \delta,\]
where $\cinfzw$ satisfies \eqref{steadyStateLimitEq}, we have $\displaystyle \Dnlin(c) \ge \frac 1 K.$

%% file: proofs/dissipationSupercritical.tex
By Lemma \ref{relativeEntropySemicontinuity}, we can find an $M_1\in\N$ and $\delta_1,\delta_2>0$, such that for all $N>M_1$ and any $c\in X^+_{\rho,\sigma}$ with $\monr,\mons>0$, $H[c|\cinfzw]>\delta$ and $\sumrs[N+1] (r+s) c\rs \le \delta_2$, we have $H^\Gamma[c|\cbar] \ge \frac{\delta} 2,$ where 
\[ \Gamma \coloneqq \left\{ (r,s) \in \Omega \, \middle| \, r+s \le M_1 \text{ or }  r+s \le N \text{ with } \ln\Big(\cbar\rs^{\frac 1 {r+s}}\Big) < -\delta_1\right\}.\]
Now, we distinguish four cases.\\
\textbf{Case 1: $\min(\monr,\mons)\ge \eta$ and $q_\infty(\monr,\mons)> e^{-\frac {\delta_1} 2}$ and $q_\infty(0,\mons) \le e^{-\frac 3 4 {\delta_1}}$ and $q_\infty(\monr,0) > e^{-\frac 3 4 {\delta_1}}.$}

Let $M_2$ be from Proposition \ref{supercritIsApplicable}. Proposition \ref{maxcbarLimit} together with $q_\infty(\monr,0) > e^{-\frac 3 4 {\delta_1}}$ implies for an $M_3$ that $\cbar_{n,0} > e^{-\delta_1 n}$ for all $n\ge M_3.$ Setting $M \coloneqq \max(M_1,M_2,M_3)$ implies $\Gamma \subset \{r+s \le N\,|\, r+s\le M \text{ or }r < r_*\}.$
Let $\chat\rs$ be constructed according to \eqref{chatConstruction} starting at $r+s = M$.
In particular $\cbar\rs = \chat\rs$ on $\Gamma$ and therefore $\displaystyle H^\Gamma[c|\cbar] \le \sumrs[2][N] \chat\rs \Psi(u\rs)$. Now, Lemma \ref{constructionIsConcave} tells us that we can use Theorem \ref{hardyInequality} and Lemma \ref{MNEstimates} for $\chat\rs$. And Lemma \ref{constructionMNEstimate} bounds the constant from Lemma \ref{MNEstimates}. So we obtain for $\er\rs,\es\rs$ defined via $\chat\rs$ and a constant $C$
\begin{equation*}
  \begin{split}
    \sumrs[2][N] \chat\rs \Psi(u\rs) 
    \le C \sumrs[2][N] \chat\rs \left(\er\rs \Bigl(\delr \sqrt{\Psi(u\rs)}\Bigr)^2 +  \es\rs \Bigl(\dels \sqrt{\Psi(u\rs)}\Bigr)^2\right).
  \end{split}
\end{equation*}
Now, for $r+s \ge M$, we bound
\begin{equation*}
  \begin{split}
    \er\rs \chat\rs &\le \chat\rms \left(\er\rs \frac{\chat\rs}{\chat\rms} + \es\rmsp \frac{\chat\rmsp}{\chat\rms}\right) = \chat\rms m_{r+s}(\chat)\\
    \overset{\text{Lemma \ref{constructionmnBound}}}&\le 2  \chat\rms 
    \overset{\text{Lemma \ref{constructionIsConcave}.\ref{constructionIsConcave:le}}}\le  2  \cbar\rms 
    \overset{\monr \ge \eta}\le \frac{2 }{\eta}\monr \cbar\rms .
  \end{split}
\end{equation*}
and similarly $\es\rs \chat\rs \le \frac 2 \eta \mons \cbar\rsm.$ From here, we can proceed as in the proof of Theorem \ref{dissipationSubcritical}, to find 
\[ \frac \delta 2 \le H^\Gamma \le \sumrs[2][N]\chat\rs \Psi(u\rs) \le C_{\delta_1,\eta,Q\rs,\rho,\sigma} \Dnlin.\]
\textbf{Case 2: $\min(\monr,\mons)\ge \eta$ and $q_\infty(\monr,\mons)> e^{-\frac {\delta_1} 2}$ and $q_\infty(0,\mons) > e^{-\frac 3 4{\delta_1}}$ and $q_\infty(\monr,0) \le e^{-\frac 3 4{\delta_1}}.$}

This is the same as Case 1, using $\ccirc\rs$ instead of $\chat\rs.$

\textbf{Case 3: $\min(\monr,\mons)\ge \eta$ and $q_\infty(\monr,\mons)> e^{-\frac {\delta_1} 2}$ and $q_\infty(0,\mons) \le e^{-\frac 3 4{\delta_1}}$ and $q_\infty(\monr,0) \le e^{-\frac 3 4 {\delta_1}}.$}

Let $M_2$ again be from Propositions \ref{supercritIsApplicable} and set $M= \max(M_1,M_2).$ Let $\chat\rs$ be constructed according to \eqref{chatConstruction} and $\ccirc\rs$ according to \eqref{ccircConstruction} both starting at $M$. Then we have $\cbar\rs \le \chat\rs + \ccirc\rs$ on $\Gamma$. And as in Case 3 and Case 4, there are constants $\hat{C}_{\delta_1,\eta,Q\rs,\rho,\sigma}, \mathring{C}_{\delta_1,\eta,Q\rs,\rho,\sigma}<\infty$, such that 
\[ \sumrs[2][N] \chat\rs \Psi(u\rs) \le \hat{C}_{\delta_1,\eta,Q\rs,\rho,\sigma} \Dnlin
  \quad\text{ and } \quad
\sumrs[2][N] \ccirc\rs \Psi(u\rs) \le \mathring{C}_{\delta_1,\eta,Q\rs,\rho,\sigma}\Dnlin.\]
Combining this, we obtain
\[ \frac \delta 2 \le H^\Gamma \le \sumrs[2][N] \chat\rs \Psi(u\rs) + 
\sumrs[2][N] \ccirc\rs \Psi(u\rs)
\le (\hat{C}_{\delta_1,\eta,Q\rs,\rho,\sigma} + \mathring{C}_{\delta_1,\eta,Q\rs,\rho,\sigma}) \Dnlin.\]

\textbf{Case 4: $\min(\monr,\mons)\ge \eta$ and $q_\infty(\monr,\mons)> e^{-\frac {\delta_1} 2}$ and $q_\infty(0,\mons) > e^{-\frac 3 4{\delta_1}}$ and $q_\infty(\monr,0) > e^{-\frac 3 4 {\delta_1}}.$}

Due to Proposition \ref{maxcbarLimit}, we can fix a $M>L$, such that $\cbar_{n,0} > e^{-\delta_1 n}$ and $\cbar_{0,n} > e^{-\delta_1 n}$ for all $n > M$. Since the concavity \eqref{concavityAss:old} implies $\min\limits_{r+s = n}\{\cbar\rs\} = \min(\cbar_{n,0},\cbar_{0,n})$, we obtain $\Gamma \subset\{r+s\le M\}$.
Furthermore, we can crudely estimate for $n<M$
\begin{equation*}
  \sumrs[n+1][] \cbar\rs
  \le \sumrs[n][] \cbar\rps + \cbar \rsp
  = \sumrs[n][] \cbar\rs \underbrace{\Big(\frac{\cbar\rps}{\cbar\rs} + \frac{\cbar\rsp}{\cbar\rs}\Big)}_{\mathclap{\le (\monr + \mons) \max\limits_{r+s <M} \Big\{ \frac{Q\rps}{Q\rs},\frac{Q\rsp}{Q\rs}\Big\}}}.
\end{equation*}
This implies for $2\le n\le M$ the estimate $m_n \le (\rho+\sigma) \max\limits_{r+s <M} \Big\{ \frac{Q\rps}{Q\rs},\frac{Q\rsp}{Q\rs}\Big\}\eqqcolon C_{\rho,\sigma,Q\rs,\delta_1}$. Now, Lemma \ref{MNEstimates} implies
\[\mathbb  E[\Trs \cbar\kl] f\rs \mathbb E\left[\Prs \frac 1 {f\kl \cbar\kl}\right] \le C^2 \left(\sum\limits_{k=0}^{M} C_{\rho,\sigma,Q\rs,\delta_1}^k\right)^2,\]
which is bounded independent of $N$ and $c\rs.$ From here we have another constant $C_{\rho,\sigma,Q\rs,\delta_1},$ such that 
\begin{equation*}
  \begin{split}
    \frac \delta 2 \le H^\Gamma \le \sumrs[2][M]\cbar\rs \Psi(u\rs)
    \le C_{\rho,\sigma,Q\rs,\delta_1} D^{M,\text{lin}}
    \overset{M\le N} &\le C_{\rho,\sigma,Q\rs,\delta_1} \Dnlin.
  \end{split}
\end{equation*}
This finishes the proof if we set $\delta_3 = \frac {\delta_1} 2$ and take $M$ as the maximal value from the cases 1--4.

%% file: MainTheorem.tex
\section{Main Theorem}\label{chapter:mainTheorem}

This chapter is devoted to prove Theorems \ref{mainTheorem} and \ref{dissipationEstimateIntro} from the introduction. First we state and prove the analogous results under the assumptions found in Chapter \ref{chapter:dissipationEstimate}. Then, we derive these assumptions from \eqref{assumption:coagulationGrowth}--\eqref{assumption:thermo:higherOrder} and characterise any escaping mass through the binding energy.

\subsection{Proof Under Direct Assumptions on \texorpdfstring{$Q\rs$}{Q}}\label{section:proofUnderDirectAssumptions}

We start by adding the dissipation estimates from Chapter \ref{chapter:dissipationEstimate}.
    \begin{theorem}\textit{(Dissipation estimate)}\label{dissipationEstimate}\\
        \input{\CommonPath/theorems/dissipationEstimate}
    \end{theorem}
\begin{proof}
    \input{\CommonPath/proofs/dissipationEstimate.tex}
\end{proof}
If we combine this with the timescale from Lemma \ref{timescale}, we find a solution that minimises the entropy.
    \begin{theorem}\textit{(Long time behaviour)}\label{longtimeMain}\\
        \input{\CommonPath/theorems/longtimeMain}
    \end{theorem}
\begin{proof}
    \input{\CommonPath/proofs/longtimeMain.tex}
\end{proof}
Clearly, \eqref{assumption:coagulationGrowth} is not optimal and may be replaced by 
$\inf\limits_{N\in 3\N}\Bigg(\frac{\inf\limits_{r+s<N} \Big\{\frac {\ar\rs}{r+s+1},\frac{\as\rs}{r+s+1}\Big\}}{\sup\limits_{r+s+1\ge \frac N 3} \Big\{\frac {\ar\rs}{r+s},\frac{\as\rs}{r+s}\Big\}}\Bigg) >0.$ However \eqref{assumption:coagulationGrowth} is easier to interpret physically as a straight forward adaptation of \eqref{coagulationPhysical}. 
Let us finish this section with a discussion of some coefficients that are covered by Theorem \ref{longtimeMain}, but do not satisfy \eqref{assumption:thermo}--\eqref{assumption:thermo:higherOrder}.
    \begin{example}\textit{(Multiplicative coefficients)}\label{multiplicativeCoeff}\\
        \input{\CommonPath/examples/multiplicativeCoeff}
    \end{example}

\subsection{Reduction to \texorpdfstring{$\Phi(\xi)$}{Phi} and Concentration on the Critical Line}\label{section:reductionToPhi}

We want to finish our discussion of the long time behaviour of \eqref{mcBD} by making the above results interpretable in a physically relevant context.
As we have seen in Subsection \ref{subsection:physicalConsiderations}, we expect ``$Q\rs \approx e^{(r+s)\Phi(\frac r {r+s})}$'', where $\Phi(\xi)$ is the thermodynamically dominating part of the Gibbs free energy. Hence, a general theory should place conditions on $\Phi$ rather then $Q\rs$. We will show below that Section \ref{section:proofUnderDirectAssumptions} can be put into this framework by deducing the assumptions \eqref{assCoeff}, \eqref{assumption:boundaryComb}, \eqref{assumptionMn}, \eqref{assumptionQn} and \eqref{concavityAss:old} from the assumptions \eqref{assumption:thermo:concave} and \eqref{assumption:thermo:boundary} through the connection \eqref{assumption:thermo} and \eqref{assumption:thermo:higherOrder}. Then we deduce some qualitative results of the long time behaviour with respect to the binding energy of the limit point.

As we have alluded to, when we introduced \eqref{assumptionMn}, the behaviour of $m_n$ should be due to Laplace's method. We start with a precise adaptation to the sums appearing in $m_n$, that does not depend on $\monr,\mons.$ We are essentially proving, that if $\Phi^{\prime\prime} \le -\gamma$, then $\sumrs[n][] \cbar\rs$ becomes a Dirac on the maximum.
    \begin{lemma}\textit{(Laplace's method)}\label{cbarDirac}\\
        \input{\CommonPath/lemmas/cbarDirac}
    \end{lemma} 
Note that if we put $Q\rs = e^{(r+s) \Phi\big(\frac r {r+s}\big)}$ into \eqref{QsecondDeriv}, the second derivative of $\Phi$ appears naturally due to $\left( \frac{Q\rs Q\rs}{Q\rpsm Q\rmsp}\right)^{r+s} = e^{(r+s)^2\left(2\Phi\big(\frac r {r+s}\big)- \Phi\big(\frac {r+1} {r+s}\big) - \Phi\big(\frac {r-1} {r+s}\big) \right)} \approx e^{- \Phi^{\prime\prime}\big(\frac r {r+s}\big)}$. 
\begin{proof}
    \input{\CommonPath/proofs/cbarDirac.tex}
\end{proof}
With this, we are in the position to prove that \eqref{assumption:coagulationGrowth}--\eqref{assumption:thermo:higherOrder} imply the assumptions from Theorem \ref{longtimeMain}.
Note, that if the limits in \eqref{assumption:thermo} exist, then $\Phi \in C^0([0,1])$ implies $\lim\limits_{\xi \to 1} (1-\xi)\Phi^\prime(\xi) = 0 = \lim\limits_{\xi \to 0} \xi \Phi^\prime(\xi).$
Also, given \eqref{assumption:thermo}, we see that \eqref{assumption:thermo:boundary} is equivalent to $\xi \varphi(\xi) \in C^0([0,1])$ and $(1-\xi)\psi(\xi) \in C^0([0,1]).$
    \begin{proposition}\textit{}\label{mainAssumptions}\\
        \input{\CommonPath/propositions/mainAssumptions}
\end{proposition}
\begin{proof}
    \input{\CommonPath/proofs/mainAssumptions.tex}
\end{proof}
Finally, we want to characterise the long time behaviour through the binding energy 
\begin{equation}\label{definition:bindingEnergy}
  G_{z,w}(\xi) \coloneqq \xi \ln z + (1-\xi) \ln w + \Phi(\xi) \text{ for } \xi\in[0,1],
\end{equation}
of the limit point $(z,w) \leftarrow (\monr(t),\mons(t))$. To do so, we first describe the region of existence through $G_{z,w}$ and $\varphi,\psi.$
    \begin{corollary}\textit{($\Ex$ through $\Phi$)}\label{regionOfExistencePhi}\\
        \input{\CommonPath/corollaries/regionOfExistencePhi}
    \end{corollary}
\begin{proof}
    \input{\CommonPath/proofs/regionOfExistencePhi.tex}
\end{proof}
\begin{example}\hphantom{0pt}\\
If we apply Corollary \ref{regionOfExistencePhi} to Example \ref{standardCoeffApplicable} and use $z \varphi(\xi_\crit) = w \psi(\xi_\crit) \iff \xi_\crit =\frac{(\lambda z)^{\frac 1 \beta}}{(\lambda z)^{\frac 1 \beta} + (\mu w)^{\frac 1 \beta}}$ and hence 
\[ z \xi_\crit \varphi(\xi_\crit) + w (1-\xi_\crit) \psi(\xi_\crit) = \left((\lambda z)^{\frac 1 \beta} + (\mu w)^{\frac 1 \beta}\right)^\beta,\]
we find $\interior{\Ex} = \left\{\left((\lambda z^{\frac 1 \beta} + (\mu w)^{\frac 1 \beta}\right)^\beta < 1 \right\}$.
\end{example}

Next, we can also describe $(z,w)\in \Ex$  satisfying \eqref{steadyStateLimitEq} in terms of the binding energy, which tells us that mass can only be lost (in $\rho,\sigma$ space) in a single direction, which is given by the maximum point $\xi_\crit$ of $G_{z,w}$. To be precise, we have the following.
    \begin{proposition}\textit{(Characterisation of the limit points)}\label{criticalDirection}\\
        \input{\CommonPath/propositions/criticalDirection}
\end{proposition}
\begin{proof}
    \input{\CommonPath/proofs/criticalDirection.tex}
\end{proof}
Lastly, we can also show that any mass escaping to infinity has to concentrate in $\Omega$ around a line given by $\frac r {r+s} = \xi_\crit,$ where $\xi_\crit(z,w)$ is given as in Corollary \ref{regionOfExistencePhi} with $(z,w) \in \Ex$ satisfying \eqref{steadyStateLimitEq}.
Note that the concentration in phase space seems to be a common feature for multidimensional coagulation equations, as it was also observed for example in \cite{ferreiramcCoagulation} for mass conserving multicomponent (pure) coagulation equations and in \cite{cristianCoagulationNonSphere} for a (pure) coagulation equation that incorporated the geometry of the particles.
    \begin{proposition}\textit{(Concentration on the critical line)}\label{massConcentrationOnCritDirection}\\
        \input{\CommonPath/propositions/massConcentrationOnCritDirection}
\end{proposition}
\begin{proof}
    \input{\CommonPath/proofs/massConcentrationOnCritDirection.tex}
\end{proof}

\subsection{Remark on the One-Component System}\label{section:improvedOC}

Our proof strategy may also be used to improve the assumptions on the coefficients in the classical one-component system. To our knowledge, the best result so far is from \cite{slemrod:longTime}. After fixing a typo in the statement, it reads:
    \begin{theorem}\textit{(One-component long time behaviour, \cite{slemrod:longTime})}\label{longtimeSlemrod}\\
        \input{\CommonPath/theorems/longtimeSlemrod}
    \end{theorem}
If we follow along the arguments of chapters 2 and 3, we can replace \eqref{assumption:ocBD} with 
\begin{equation}\label{assumption:ocBD:new}
  \lim\limits_{r \to \infty} \frac{Q_r}{Q_{r+1}} = z_s \quad\text{ and }\quad
  \inf\limits_{N\in 3\N}\Bigg\{\,\frac{\inf\limits_{r<N} \frac {a_r}{r+1}}{\sup\limits_{r+1\ge \frac N 3} \frac {a_r}{r}}\,\Bigg\} >0.
\end{equation}
This is nice, because it separates the assumption on the energy $Q$ from those on the dynamics $a$. In particular, if we assume $\lim\limits_{r \to \infty} \frac{Q_r}{Q_r+1} = z_s$, as is often the case (e.g. \cite{canizo:SpectralGab,canizo:EDEstimate}), then $a_r = (1+\indicator{r \in 2\N}) r^\alpha$ violates \eqref{assumption:ocBD}, but not \eqref{assumption:ocBD:new}. To prove \eqref{assumption:ocBD:new}, one can redo the calculations above in the one-component setting. This is pretty straight forward. Section \ref{section:logSobolevInequality} reduces to the proof of Theorem \ref{hardyInequality}, where all the steps no longer need further explanations. Section \ref{section:supercriticality} becomes obsolete, because we can just cut at a fixed value $M$, i.e. consider $H^M[c|\cbar]$ if $\frac {c_1} {z_s} \ge e^{-\delta_1}$ and $H^N$ else. Alternatively, we can reduce the one-component system given by $A_r,Q_r$ to the two-component system, as was already pointed out in \cite{dunwell:phd}, via 
\[ \ar\rs \coloneqq \as\rs \coloneqq A_{r+s} \text{ and }Q\rs = \binom{r+s}{r}Q_{r+s}.\]
Then, if $c\rs$ solves the two-component system it follows that $x_m = \sumrs[m][]\cbar\rs$ solves the one-component system. Finally, the assumptions of Proposition \ref{mainAssumptions} and Theorem \ref{longtimeMain} follow from \eqref{assumption:ocBD:new}.

%% file: theorems/dissipationEstimate.tex
Assume $Q\rs$ satisfies \eqref{assCoeff}, \eqref{assumption:boundaryComb}, \eqref{assumptionMn}, \eqref{assumptionQn} and \eqref{concavityAss:old} for all $n \ge L$ and fix $\rho,\sigma>0$ and $\delta>0.$
Then, there exist constants $K<\infty$ and $\delta_2>0$ as well as $M<\infty$, such that for all $c\in X^+_{\rho,\sigma}$ and $M\le N\in \N$ with 
\[  \sumrs[N+1] (r+s) c\rs \le \delta_2 \text{ and } H[c|\cinfzw] > \delta,\]
where $\cinfzw$ satisfies \eqref{steadyStateLimitEq}, we have $\displaystyle \Dnlin(c) \ge \frac 1 K.$

%% file: proofs/dissipationEstimate.tex
By Lemma \ref{mnLimit} and Proposition \ref{maxcbarLimit} we have $m_\infty(\monr,\mons) = q_\infty(\monr,\mons).$ Hence, Theorem \ref{dissipationEstimate} follows immediately from Theorems \ref{dissipationSmallMon}, \ref{dissipationSupercritical} and \ref{dissipationSubcritical}.

%% file: theorems/longtimeMain.tex
Assume $\ar\rs,\as\rs$ satisfies \eqref{assumption:coagulationGrowth} and $\br\rps,\bs\rsp$ is given by \eqref{assumption:detailedBalance}, where $Q\rs$ satisfies \eqref{assCoeff}, \eqref{assumption:boundaryComb}, \eqref{assumptionMn}, \eqref{assumptionQn} and \eqref{concavityAss:old} for all $n\ge L$.  Then, for any initial data with masses $\rho,\sigma>0,$ there is a solution $c$ to \eqref{mcBD} satisfying 
  \[ H[c(t)|\cinfzw] \to 0 \text{ as }t\to\infty,\]
  where $\cinfzw$ is the steady state satisfying \eqref{steadyStateLimitEq}.

%% file: proofs/longtimeMain.tex
Let $c(t)$ be the solution found in Theorem \ref{entropyInequality}. Then, $H[c(t)|\cinfzw]\ge 0$ is monotonically decreasing.
We will show via contradiction, that $H[c|\cinfzw]\to 0$. So, let us assume that there is a $\delta>0$, such that $H[c|\cinfzw] > \delta$ for all times.

According to Theorem \ref{dissipationEstimate}, there are constants $K<\infty,$ $M\in\N$ and $\delta_2>0$ such that for all $N>M$ with $\sumrs[N+1](r+s)c\rs \le \delta_2$, we have $\Dnlin \ge \frac 1 K.$
Now, let us fix an $\eps>0$ satisfying for $c_\alpha,C_\alpha$ from \eqref{assumption:coagulationGrowth}
\[\displaystyle 0 < \eps < \left(\frac {\delta_2}{2}\right)^2 \frac{c_\alpha}{8\cdot 3^{1-\alpha}\cdot C_\alpha(\rho+\sigma)^2 H[c(0)|\cinfzw]}\frac 1 K.\]
Since $H[c(t)|\cinfzw]$ is a monotonically decreasing function bounded from below, we can find a $T<\infty$, such that 
\[ H[c(T)|\cinfzw] - \lim\limits_{t \to \infty} H[c(t)|\cinfzw] < \eps.\]
Now, let $N \in 3\N$ be so large, that $\frac N 3 > M$ and $\sumrs[\frac N 3] (r+s) c\rs(T) < \frac {\delta_2} 2.$ Then by Lemma \ref{timescale}, for all $\displaystyle t\in [T,T+ N^{1-\alpha} \left(\frac {\delta_2}{2}\right)^2 \frac{1}{4\cdot 3^{1-\alpha}\cdot C_\alpha(\rho+\sigma)^2 H[c(0)|\cinfzw]}]$, we have 
\begin{equation*}
  \begin{split}
    \sumrs[N] (r+s)c\rs(t)
    &\le \sumrs[\frac N 3] (r+s)c\rs(T) 
    + 2(\rho + \sigma)\sqrt{\int_{T}^{t} D(c(\tau))\dd \tau} \sqrt{(t-T)\sup\limits_{r+s+1\ge \frac N 3}\left\{\frac {\ar\rs}{r+s},\frac{\as\rs}{r+s}\right\}}\\
    &\le \frac{\delta_2} 2 + 2(\rho + \sigma) \underbrace{\sqrt{\int_{T}^{t} D(c(\tau))\dd \tau}}_{\le \sqrt{H[c(0)|\cinfzw]}}
    \Bigg((t-T)\underbrace{\sup\limits_{r+s+1\ge \frac N 3}\Big\{\frac {\ar\rs}{r+s},\frac{\as\rs}{r+s}\Big\}}_{\overset{\eqref{assumption:coagulationGrowth}}\le C_\alpha \frac{3^{1-\alpha}}{N^{1-\alpha}}}\Bigg)^{\frac 1 2}\\
  \overset{t-T\le N^{1-\alpha} \left(\frac {\delta_2}{2}\right)^2 \frac{1}{4\cdot 3^{1-\alpha}\cdot C_\alpha(\rho+\sigma)^2 H[c(0)|\cinfzw]}}&\le \frac{\delta_2}2 + \frac{\delta_2}2.
  \end{split}
\end{equation*}
In particular, $\Dnlin \ge \frac 1 K$ holds true on that time interval.\\
Next, from \eqref{assumption:coagulationGrowth} we find $\displaystyle \inf\limits_{r+s<N} \Big\{\frac{\ar\rs}{r+s+1},\frac{\as\rs}{r+s+1}\Big\} \ge \frac{c_\alpha }2 \frac{1}{N^{1-\alpha}}$ and hence $\displaystyle D\ge \frac{c_\alpha }2 \frac{1}{N^{1-\alpha}}\Dnlin$.
Finally, we can put this together, to find the contradiction
\begin{equation*}
  \begin{split}
    \eps \overset{\text{def. }T}&> \int_T^\infty D(\tau)\dd\tau
    \ge \frac{c_\alpha }2 \frac{1}{N^{1-\alpha}}\int_T^{T+ N^{1-\alpha} \left(\frac {\delta_2}{2}\right)^2 \frac{1}{4\cdot 3^{1-\alpha}\cdot C_\alpha(\rho+\sigma)^2 H[c(0)|\cinfzw]}} \Dnlin(\tau)\dd\tau\\
    &\ge \frac{c_\alpha }2 \frac{1}{N^{1-\alpha}}\Bigg(N^{1-\alpha} \left(\frac {\delta_2}{2}\right)^2 \frac{1}{4\cdot 3^{1-\alpha}\cdot C_\alpha(\rho+\sigma)^2 H[c(0)|\cinfzw]}\Bigg) \frac 1 K 
    \overset{\text{def. }\eps}> \eps.
  \end{split}
\end{equation*}

%% file: examples/multiplicativeCoeff.tex
Consider the coefficients from Example \ref{standardCoeffApplicable}, when we set $\beta = 0,$ i.e. $Q\rs = \lambda^r\mu^s e^{-(v_1r+v_2s)^\frac 2 3}$, where $\lambda,\mu,v_1,v_2 >0.$ These do not satisfy the strong concavity condition \eqref{assumption:thermo:concave}. In fact for arbitrary $v_1,v_2$, they also violate the milder concavity assumption \eqref{concavityAss:old}. However, if $v_1=v_2,$ one can check by hand that the assumptions of Theorem \ref{longtimeMain} are still satisfied, with $m_\infty(\monr,\mons) = \max(\lambda\monr,\mu\mons)$. Even though, we are not aware of a physically relevant system, that these coefficients describe, they are mathematically pleasing due to their multiplicative structure. In fact, in an early numerical discussion of the two-component Becker--Döring system \cite{soheiliNumericalmcBD}, they consider exactly these coefficients. Now, Theorem \ref{longtimeMain} tell us, that there is a solution minimising $H[c|\cinfzw]$, where $\cinfzw$ is determined by \eqref{steadyStateLimitEq}. Just as in Theorem \ref{mainTheorem}, we can characterise the lost Type I and Type II mass through this limit point.
To do so, we first note that $\Ex = [0,\frac 1 \lambda] \times [0,\frac 1 \mu]$. Now it suffices to observe the following two implications, where $\rho,\sigma>0$ are the initial masses.
\begin{equation}
  \begin{split}
    &\text{ If }\rho(\cinfzw)<\rho, \text{ then } z = \frac 1 \lambda \text{, because otherwise }(u,w) \text{ with }u\in{(z,\frac 1 \lambda)} \text{ is a competitor violating \eqref{steadyStateLimitEq}}.\\
    &\text{ If }\sigma(\cinfzw)<\sigma, \text{ then } w = \frac 1 \mu \text{, because otherwise }(z,v) \text{ with }v\in{(w,\frac 1 \mu)} \text{ is a competitor violating \eqref{steadyStateLimitEq}}.
  \end{split}
\end{equation}
This allows us to draw the analogous picture to Figure \ref{longtimePicture}.
\input{\CommonPath/pictures/longTimeMultiplicative.tex}

%% file: pictures/longTimeMultiplicative.tex
\newcommand{\lam}{2} 
\newcommand{\muVal}{3} 
\newcommand{\zs}{1/\lam}
\newcommand{\ws}{1/\muVal}
\newcommand{\purplez}{0.4}
\newcommand{\purplew}{0.85}
\newcommand{\bluez}{0.7} 
\newcommand{\redw}{0.8} 
\newcommand{\scale}{0.2}
\begin{figure}[H]
  \centering 
  \begin{tikzpicture}[
          scale=1,
          declare function={
            xln(\x) = ifthenelse(\x == 0, 0, \x*ln(\x));
            f(\x,\k) = 
              \x^1*1*exp(-(1+\k)^(2/3))+
              \x^2*2*exp(-(2+\k)^(2/3))+
              \x^3*3*exp(-(3+\k)^(2/3))+
              \x^4*4*exp(-(4+\k)^(2/3))+
              \x^5*5*exp(-(5+\k)^(2/3))+
              \x^6*6*exp(-(6+\k)^(2/3))+
              \x^7*7*exp(-(7+\k)^(2/3))+
              \x^8*8*exp(-(8+\k)^(2/3))+
              \x^9*9*exp(-(9+\k)^(2/3))+
              \x^10*10*exp(-(10+\k)^(2/3))+
              \x^11*11*exp(-(11+\k)^(2/3))+
              \x^12*12*exp(-(12+\k)^(2/3))+
              \x^13*13*exp(-(13+\k)^(2/3))+
              \x^14*14*exp(-(14+\k)^(2/3))+
              \x^15*15*exp(-(15+\k)^(2/3))+
              \x^16*16*exp(-(16+\k)^(2/3))+
              \x^17*17*exp(-(17+\k)^(2/3))+
              \x^18*18*exp(-(18+\k)^(2/3))+
              \x^19*19*exp(-(19+\k)^(2/3))+
              \x^20*20*exp(-(20+\k)^(2/3))+
              \x^21*22*exp(-(21+\k)^(2/3))+
              \x^22*22*exp(-(22+\k)^(2/3));
            rho(\z,\w) = 
              \w^0 * f(\z,0)+
              \w^1 * f(\z,1)+
              \w^2 * f(\z,2)+
              \w^3 * f(\z,3)+
              \w^4 * f(\z,4)+
              \w^5 * f(\z,5)+
              \w^6 * f(\z,6)+
              \w^7 * f(\z,7)+
              \w^8 * f(\z,8)+
              \w^9 * f(\z,9)+
              \w^10 * f(\z,10)+
              \w^11 * f(\z,11)+
              \w^12 * f(\z,12)+
              \w^13 * f(\z,13)+
              \w^14 * f(\z,14)+
              \w^15 * f(\z,15)+
              \w^16 * f(\z,16)+
              \w^17 * f(\z,17)+
              \w^18 * f(\z,18)+
              \w^19 * f(\z,19)+
              \w^20 * f(\z,20)+
              \w^21 * f(\z,21)+
              \w^22 * f(\z,22);
            sigma(\z,\w) = rho(\w,\z);
          }
      ]
      \begin{scope}[shift={(6,0)}]
        \draw[->] (0, 0) -- (3.5, 0) node[below] {$\rho$};
        \draw[->] (0, 0) -- (0, 3.5) node[left] {$\sigma$};
        \draw[scale={\scale}, domain=0:1, smooth, variable=\w, LightCoral] plot ( {rho(1,\w)}, {sigma(1,\w)});
        \draw[scale={\scale}, domain=0:1, smooth, variable=\z, LightSlateBlue] plot ( {rho(\z,1)}, {sigma(\z,1)});
        \foreach \xi in {{\bluez},{\redw},0.5,0.88,0.94,0.98}{
        \draw[scale={\scale}, densely dashed,  LightSlateBlue, very thin] ({rho(\xi,1)},{sigma(\xi,1}) -- ({rho(\xi,1)},3.5/\scale);
        \draw[scale={\scale}, densely dashed,  LightCoral, very thin] ({rho(1,\xi)},{sigma(1,\xi}) -- (3.5/\scale,{sigma(1,\xi)});
        }
        \foreach \xi in {0,0.1,0.2,...,1}{
        \draw[scale={\scale}, densely dashed,  LightSeaGreen, very thin] ({rho(1,1)},{sigma(1,1}) -- ({(1-\xi)*3.5/\scale+ \xi*rho(1,1)},{\xi*3.5/\scale+ (1-\xi)*sigma(1,1)});
        }
        \draw[-{Latex}] ({\scale*rho(1,\redw)},{\scale*sigma(1,\redw)}) -- ({\scale*rho(1,\redw)+0.4},{\scale*sigma(1,\redw)}) node[above right = -2.5pt] {$v_{\text{\color{LightCoral}crit}}$};
        \draw[-{Latex}] ({\scale*rho(\bluez,1)},{\scale*sigma(\bluez,1)}) -- ({\scale*rho(\bluez,1)},{\scale*sigma(\bluez,1)+0.4}) node[above right = -3pt] {$v_{\text{\color{LightSlateBlue}crit}}$};
        \filldraw[LightCoral] ({\scale*rho(1,\redw)},{\scale*sigma(1,\redw)}) circle (1pt);
        \filldraw[LightSlateBlue] ({\scale*rho(\bluez,1)},{\scale*sigma(\bluez,1)}) circle (1pt);
        \filldraw[LightSeaGreen] ({\scale*rho(1,1)},{\scale*sigma(1,1)}) circle (1pt);
        \filldraw[DarkKhaki] ({\scale*rho(\purplez,\purplew)},{\scale*sigma(\purplez,\purplew)}) circle (1pt);
      \end{scope}
      \draw[<->] (3.5,2) to[bend left=30] (5,2);
      \node at (4.25,2.5){$\rho(\cinfzw),\sigma(\cinfzw)$};
      \begin{scope}[shift={(0,0)},scale=3.5/0.6]
          \node at (0.3,0.15) (a) {$\Ex$}; 
          \draw[->] (0, 0) -- (0.6, 0) node[below] {$z$};
          \draw[->] (0, 0) -- (0, 0.6) node[left] {$w$};
          \draw[scale=1, domain=0:\zs, smooth, variable=\z, LightSlateBlue] plot ({\z}, {\ws});
          \draw[scale=1, domain=0:\ws, smooth, variable=\w, LightCoral] plot ({\zs}, {\w});
          \filldraw[LightSeaGreen] (\zs,\ws) circle (0.6/3.5pt);
          \filldraw[LightSlateBlue] ({\bluez*\zs},\ws) circle (0.6/3.5pt);
          \filldraw[LightCoral] (\zs,{\ws*\redw}) circle (0.6/3.5pt);
          \filldraw[DarkKhaki] ({\zs*\purplez},{\ws*\purplew}) circle (0.6/3.5pt);
      \end{scope}
      \draw[<->] (9.5,2) to[bend left=30] (11,2);
      \node at (10.25,2.75){\footnotesize$v_\crit \coloneqq \begin{pmatrix}
        \xi_\crit \\ 1-\xi_\crit
      \end{pmatrix}$};
      \begin{scope}[shift={(12,3)}]
          \node at (-0.2,-0.3) (a) {$0$}; 
          \node at (-0.2,-0.65) (a) {\,\rotatebox{90}{$=$}};
          \node at (-0.2,-0.9) (a) {$\xi_\text{\color{LightSlateBlue}crit}$};
          \draw[->] (-0.2, 0) -- (3.5, 0) node[below] {$\xi$};
          \draw[->] (0, -3) -- (0, 0.5) node[left] {$G_{z,w}(\xi)$};
          \draw[dotted] (3, -3) -- (3, 0.5);
          \draw[scale=3, domain=0:1, smooth, variable=\x, LightCoral] plot ({\x}, {(1-\x)*ln(\redw)});
          \draw[scale=3, domain=0:1, smooth, variable=\x, LightSlateBlue] plot ({\x}, {\x * ln(\bluez)});
          \draw[scale=3, domain=0:1, smooth, variable=\x, LightSeaGreen] plot ({\x}, {0});
          \draw[scale=3, domain=0:1, smooth, variable=\x, DarkKhaki] plot ({\x}, {\x * ln(\purplez) + (1-\x)*ln(\purplew)});
          \draw[-] (3,0) -- (3,-0.1) node[below] {$\xi_\text{\color{LightCoral}crit}$};
      \end{scope}
  \end{tikzpicture}
  \caption{All initial conditions with masses on the dashed lines will converge to the anchor of their line. The slopes of the dashed lines are determined by $G_{z,w}^{-1}(\{0\}),$ where the binding energy $G_{z,w}$ is given as $G_{z,w}(\xi) = \xi \ln (\lambda z) + (1-\xi) \ln(\mu w).$}
  \label{longtimePictureMultiplicative}
\end{figure}
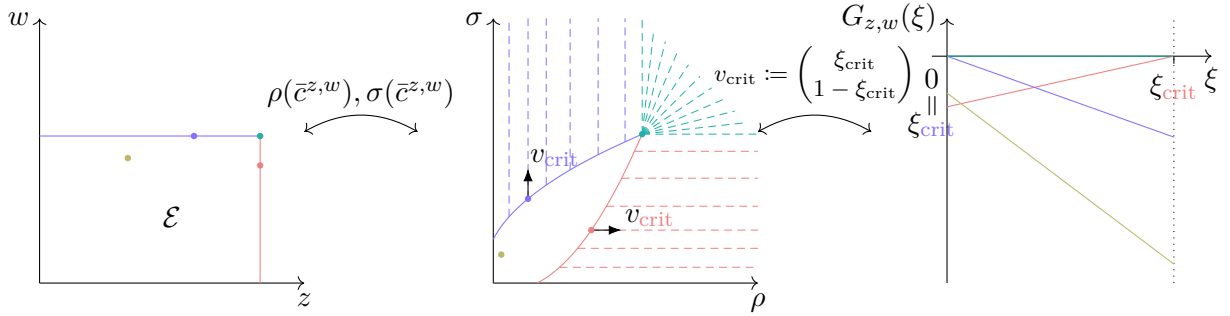

%% file: lemmas/cbarDirac.tex
Assume that there is a $c>0$, such that for all $0 <  r< n$ and $s=n-r$ we have 
\begin{equation}\label{QsecondDeriv}
\left( \frac{Q\rs Q\rs}{Q\rpsm Q\rmsp}\right)^{r+s} \ge (1+c).
\end{equation}
Then, given $\delta>0$, there is a constant $C_\delta$, such that 
\begin{equation*}
  \frac{\sum\limits_{\substack{r+s = n \\ r< r_m(n) - \delta n}} \cbar\rs}{
    \sum\limits_{\substack{r+s = n \\  r_m(n) - \delta n \le r \le r_m}} \cbar\rs} \le  \frac{C_\delta}{n} \quad\text{ and }\quad
\frac{\sum\limits_{\substack{r+s = n \\ r> r^m(n) + \delta n}} \cbar\rs}{
    \sum\limits_{\substack{r+s = n \\  r_m \le r \le r_m(n) + \delta n }} \cbar\rs} \le  \frac{C_\delta}{n},
\end{equation*}
where $\displaystyle r_m(n) = \min\Big\{r \,\Big|\, \cbar\rn \ge \cbar\kn \Big\}$ and $\displaystyle r^m(n) = \max\Big\{r \,\Big|\, \cbar\rn \ge \cbar\kn \Big\}$ and $\monr,\mons$ are arbitrary with $\max(\monr,\mons)>0$.

%% file: proofs/cbarDirac.tex
In this proof, for a given $0\le k,r\le n$ we always denote $s= n-r$ and $l=n-k$.

First, we will show that for $r<r_m - \delta n - 1$, we have a constant $ 0< q_\delta < 1$ such that  $\frac{\cbar\rs}{\cbar\rpsm} \le q_\delta.$ This is seen through a telescope product
\begin{equation*}
  \begin{split}
    \frac{\cbar\rs}{\cbar\rpsm}
  \overset{s_m \coloneqq n-r_m}&= \frac{\cbar_{r_m-1,s_m+1}}{\cbar_{r_m,s_m}}\frac{Q_{r_m,s_m} Q\rs}{Q_{r_m-1,s_m+1}Q\rpsm} \le \frac{Q_{r_m,s_m} Q\rs}{Q_{r_m-1,s_m+1}Q\rpsm}\\
  &= \prod\limits_{k=r+1}^{r_m-1} \frac{Q\kmlp Q\kplm}{Q\kl Q\kl}
  = \Bigg(\prod\limits_{k=r+1}^{r_m-1} \left(\frac{Q\kmlp Q\kplm}{Q\kl Q\kl}\right)^n\Bigg)^{\frac 1 n}\\
  \overset{\text{\eqref{QsecondDeriv}}} &\le  \Bigg(\prod\limits_{k=r+1}^{r_m-1} \frac 1 {1+c}\Bigg)^{\frac 1 n}
   = \left(\frac 1 {1+c}\right)^{\frac{r_m - 1 -r}{n}}\\
   \overset{r \le r_m - \delta n -1}&\le \left(\frac 1 {1+c} \right)^\delta \eqqcolon q_\delta.
  \end{split}
\end{equation*}
Now, let $k = \lceil r_m - \delta n -1 \rceil $. If $k < 0$, there is nothing to do. Otherwise, we estimate
\begin{equation*}
  \begin{split}
    \sum\limits_{\substack{r+s = n \\ r< r_m(n) - \delta n}} \cbar\rs
    &= \sum\limits_{\substack{r+s = n \\ r \le k}} \cbar\rs
    = \cbar\kl \sum\limits_{\substack{r+s = n \\ r \le  k}} \frac{\cbar\rs}{\cbar\kl}\\
    &= \cbar\kl \sum\limits_{\substack{r+s = n \\ r \le k}} \prod\limits_{m=r}^{k-1} \frac{\cbar_{m,n-m}}{\cbar_{m+1,n-m-1}}\\
    \overset{k-1 < r_m-\delta n -1}&\le \cbar\kl \sum\limits_{\substack{r+s = n \\ r \le k}} q_\delta^{k-r} 
    \overset{\text{geo. series}} \le \cbar\kl \frac{1}{1-q_\delta}.
  \end{split}
\end{equation*}
Next, we write 
\begin{equation*}
  \begin{split}
    \cbar\kl &= \frac{1}{r_m-k} \sum\limits_{r=k+1}^{r_m} \cbar\kl
    = \frac{1}{r_m-k} \sum\limits_{r=k+1}^{r_m} \cbar\rs \frac{\cbar\kl}{\cbar\rs}
    = \frac{1}{r_m-k} \sum\limits_{r=k+1}^{r_m} \cbar\rs \prod\limits_{m=k}^{r-1} \frac{\cbar_{m,n-m}}{\cbar_{m+1,n-m-1}}.
  \end{split}
\end{equation*}
And once more, we can estimate for $m<r_m$ as above 
\[\frac{\cbar_{m,n-m}}{\cbar_{m+1,n-m-1}} = \frac{\cbar_{r_m-1,s_m+1}}{\cbar_{r_m,s_m}} \prod\limits_{k=m+1}^{r_m-1} \frac{Q\kmlp Q\kplm}{Q\kl Q\kl}\le 1.\]
Putting everything together with $r_m - k  \ge \delta n,$ we obtain
\begin{equation*}
  \begin{split}
    \sum\limits_{\substack{r+s = n \\ r< r_m(n) - \delta n}} \cbar\rs
    \le \frac {1} {\delta (1-q_\delta) }\frac 1 n
    \sum\limits_{\substack{r+s = n \\  r_m(n) - \delta n \le r \le r_m}} \cbar\rs.
  \end{split}
\end{equation*}
The second inequality is completely analogous. For $r> r^m + \delta n+1$ we write 
\[\frac{\cbar\rs}{\cbar\rmsp} = \frac{\cbar_{r^m+1,s^m-1}}{\cbar_{r^m,s^m}} \prod\limits_{k=r^m+1}^{r-1} \frac{Q\kmlp Q\kplm}{Q\kl Q\kl} \le q_\delta.\]
And for $k = \lfloor r^m + \delta n + 1\rfloor $, we find $\sum\limits_{\substack{r+s = n \\ r> r_m(n) - \delta n}} \cbar\rs \le \cbar\kl \frac{1}{1-q_\delta}.$ Lastly, the result follows from $\displaystyle \cbar\kl = \frac{1}{k-r^m} \sum\limits_{r=r^m}^{k-1} \cbar\kl \le \frac{1}{k-r^m} \sum\limits_{r=r^m}^{k-1} \cbar\rs$.

%% file: propositions/mainAssumptions.tex
Assume $Q\rs$ satisfies the assumptions \eqref{assumption:thermo}--\eqref{assumption:thermo:higherOrder}. Then, $Q\rs$ already satisfies \eqref{assCoeff}, \eqref{assumption:boundaryComb}, \eqref{assumptionMn}, \eqref{assumptionQn} and \eqref{concavityAss:old} for all $n\ge L$. Furthermore, we have 
\[ m_\infty(\monr,\mons) = \exp\left({\max\limits_{\xi \in [0,1]}\Big( \xi \ln \monr + (1-\xi) \ln \mons + \Phi(\xi)}\Big)\right).\]

%% file: proofs/mainAssumptions.tex
\underline{Claim 1}: {\eqref{assCoeff} is true, i.e. $ \lim\limits_{N \to \infty} \inf\limits_{r+s\ge N} Q\rs^{\frac 1 {r+s}} > 0 \quad\text{and}\quad \lim\limits_{N \to \infty} \sup\limits_{r+s \ge N} Q\rs^{\frac 1 {r+s}} < \infty.$}

Let $C<\infty$ be such that $f\rs\le C$ and $g\rs \le C.$ We calculate inductively for any $1<r,s$
\begin{equation*}
  \begin{split}
    Q\rs 
    &= \min \left(\frac {Q\rs}{Q\rms} Q\rms, \frac{Q\rs}{Q\rsm} Q\rsm \right)\\
    &\le \max\limits_{r+s<\infty}\{f\rs,g\rs\} \min \left(\varphi\left(\frac{r}{r+s}\right)Q\rms,  \psi\left(\frac r {r+s}\right)Q\rsm \right)\\
    &\le C \max\limits_{k+l = r+s -1}\{Q\kl\}  \min \left(\varphi\left(\frac{r}{r+s}\right),  \psi\left(\frac r {r+s}\right)\right).
  \end{split}
\end{equation*}
Because of the opposite sign in front of the derivative in $\varphi(\xi) = e^{\Phi(\xi) + (1-\xi)\Phi^\prime(\xi)}$ and $\psi(\xi) = e^{\Phi(\xi) - \xi \Phi^\prime(\xi)}$, we have $\min(\varphi,\psi) \le e^{||\Phi||_\infty}.$ Furthermore, we have $\frac {Q_{r,0}}{Q_{r-1,0}} = f_{r-1,0} \varphi(1) = f_{r-1,0}e^{\Phi(1)} \le C e^{||\Phi||_\infty}$ and $\frac {Q_{0,s}}{Q_{0,s-1}} = g_{0,s-1} \psi(0) = g_{0,s-1}e^{\Phi(0)} \le C e^{||\Phi||_\infty}$. Putting this together, we have shown for any $N$ that 
\[ \max\limits_{r+s = N} Q\rs \le Ce^{||\Phi||_\infty} \max\limits_{r+s = N-1} Q\rs \overset{\text{ind.}}\le \left(Ce^{||\Phi||_\infty}\right)^{N-1} \max(Q_{1,0},Q_{0,1}).\]
In particular, we have $\lim\limits_{N \to \infty} \sup\limits_{r+s \ge N} Q\rs^{\frac 1 {r+s}}\le C e^{||\Phi||_\infty}.$ Similarly, let $c>0$, be such that $f\rs \ge c$ and $g\rs \ge c.$ Then, we have 
  \[ Q\rs = \max \left(\frac {Q\rs}{Q\rms} Q\rms, \frac{Q\rs}{Q\rsm} Q\rsm \right)
  \ge c \min\limits_{k+l = r+s -1}\{Q\kl\}  \max \left(\varphi\left(\frac{r}{r+s}\right),  \psi\left(\frac r {r+s}\right)\right).\]
And therefore, $\lim\limits_{N \to \infty} \inf\limits_{r+s\ge N} Q\rs^{\frac 1 {r+s}} \ge c e^{-||\Phi||_\infty}.$

\underline{Claim 2}: {\eqref{assumption:boundaryComb} is true, i.e. $\frac{Q\rs}{Q\rps} \ge \frac 1 \lambda \frac {r+1}{r+s+1}$ and $\frac{Q\rs}{Q\rsp} \ge \frac 1 \mu \frac {s+1}{r+s+1}$ for all $r+s\ge 1$ and some $\lambda,\mu < \infty.$}

We have 
\begin{equation*}
  \begin{split}
    \frac{Q\rs}{Q\rps}&= \frac 1 {f\rs \varphi\left(\frac{r+1}{r+s+1}\right)}
    = \frac{r+1}{r+s+1} \frac 1 {f\rs \frac{r+1}{r+s+1} \varphi\left(\frac{r+1}{r+s+1}\right)}
    \overset{f\rs \le C} \ge  \frac{r+1}{r+s+1} \frac 1 {C \frac{r+1}{r+s+1} \varphi\left(\frac{r+1}{r+s+1}\right)}\\
    \overset{\text{\eqref{assumption:thermo:boundary}}}&\ge  \frac{r+1}{r+s+1} \frac 1 {C ||\xi \varphi(\xi)||_\infty}
  \end{split}
\end{equation*}
and
\begin{equation*}
  \begin{split}
    \frac{Q\rs}{Q\rsp}&= \frac 1 {g\rs \psi\left(\frac{r}{r+s+1}\right)}
    = \frac{s+1}{r+s+1} \frac 1 {g\rs \frac{s+1}{r+s+1} \psi\left(\frac{r}{r+s+1}\right)}
    \overset{g\rs \le C} \ge  \frac{s+1}{r+s+1} \frac 1 {C \frac{s+1}{r+s+1} \psi\left(\frac{r}{r+s+1}\right)}\\
    \overset{\text{\eqref{assumption:thermo:boundary}}}&\ge  \frac{s+1}{r+s+1} \frac 1 {C ||(1-\xi) \psi(\xi)||_\infty}.
  \end{split}
\end{equation*}

\underline{Claim 3}: {There is an $L\in\N$, such that \eqref{concavityAss:old} is true for all $n\ge L$ (which is equivalent to \eqref{concavityAss}), i.e. $\frac{Q\rpsm}{Q\rs} \le \frac{Q\rps}{Q\rsp} \le \frac{Q\rs}{Q\rmsp}$ for all $r+s = n$.}

Let $F(\xi) = \Phi(\xi) + (1-\xi) \Phi^\prime(\xi).$ Then we have $F^\prime(\xi) = (1-\xi) \Phi^{\prime\prime}(\xi)\overset{\text{\eqref{assumption:thermo:concave}}} \le -\gamma.$ And therefore, for any $r>0$ and $s$, we find 
\begin{equation*}
  \begin{split}
    \frac{Q\rps Q\rmsp}{Q\rs Q \rsp}
    &= \frac{\varphi\left(\frac{r+1}{r+s+1}\right) f\rs}{\varphi\left(\frac{r}{r+s+1}\right) f\rmsp}
    = e^{F\left(\frac{r+1}{r+s+1}\right)-F\left(\frac{r}{r+s+1}\right)} \frac{f\rs}{f\rmsp}\\
    \overset{\text{mean value theorem}}&\le  e^{-\gamma \frac{1}{r+s+1}}  \frac{f\rs}{f\rmsp}
    = \left( e^{-\gamma}  \left(\frac{f\rs}{f\rmsp}\right)^{r+s+1}\right)^{\frac 1 {r+s+1}}.
  \end{split}
\end{equation*}
By \eqref{assumption:thermo:higherOrder}, we can fix an $L$, such that for $r+s \ge L$, we have $\left(\frac{f\rs}{f\rmsp}\right)^{r+s+1}\le e^\gamma$ and therefore $\frac{Q\rps Q\rmsp}{Q\rs Q \rsp}\le 1$. Analogously, let $P(\xi) = \Phi(\xi) - \xi \Phi^\prime(\xi)$ and hence $P^\prime(\xi) = - \xi \Phi^{\prime\prime}(\xi) \ge \gamma.$ Then for any $s>0$ and $r$, we obtain
\begin{equation*}
  \begin{split}
    \frac{Q\rps Q\rs}{Q\rpsm Q \rsp}
    &= \frac{\psi\left(\frac{r+1}{r+s+1}\right) g\rpsm}{\psi\left(\frac{r}{r+s+1}\right) g\rs}
    = e^{P\left(\frac{r+1}{r+s+1}\right)-P\left(\frac{r}{r+s+1}\right)} \frac{g\rpsm}{g\rs }\\
    \overset{\text{mean value theorem}}&\ge  e^{\gamma \frac{1}{r+s+1}}  \frac{g\rpsm}{g\rs}
    \overset{\text{ for }r+s\text{ large via \eqref{assumption:thermo:higherOrder}}}\ge 1.
  \end{split}
\end{equation*}

Now, we turn our attention to $m_n.$ By the assumption \eqref{assumption:thermo:concave}, the functions $\varphi$ and $\psi$ are strictly decreasing/increasing. We will distinguish the three cases, where $\monr \varphi \ge \mons \psi$, $\monr \varphi \le \mons \psi$ or $\monr \varphi(\xi_\crit) = \mons \psi(\xi_\crit)$, for a unique $\xi_\crit \in (0,1).$ If $\monr \varphi \ge \mons \psi$, we set $\xi_\crit \coloneqq 1$ and if $\monr \varphi \le \mons \psi$, we set $\xi_\crit = 0.$

\underline{Claim 4}: Given $0\le \monr \le \rho,$ and $0\le \mons\le \sigma$ with $\max(\monr,\mons) > 0$, we have $\lim\limits_{n \to \infty} m_n(\monr,\mons) = \monr \xi_\crit \varphi(\xi_\crit) + \mons (1-\xi_\crit) \psi(\xi_\crit)$ and the speed of convergence is independent of $\monr,\mons.$

First, we write $m_{n+1}$ as 
\begin{equation*}
  \begin{split}
    m_{n+1} = \frac{\sumrs[n+1][] \cbar\rs}{\sumrs[n][] \cbar\rs}
    =  \frac{\sumrs[n][]\left(\monr \frac{r+1}{r+s+1}\frac{Q\rps}{Q\rs} + \mons \frac{s+1}{r+s+1} \frac{Q\rsp}{Q\rs}\right) \cbar\rs}{\sumrs[n][] \cbar\rs}.
  \end{split}
\end{equation*}
So, it suffices to show that $\frac{\sumrs[n][]\frac{r+1}{r+s+1}\frac{Q\rps}{Q\rs}  \cbar\rs}{\sumrs[n][] \cbar\rs} \to \xi_\crit \varphi(\xi_\crit)$ and $\frac{\sumrs[n][] \frac{s+1}{r+s+1} \frac{Q\rsp}{Q\rs} \cbar\rs}{\sumrs[n][] \cbar\rs} \to (1-\xi_\crit) \psi(\xi_\crit)$ and the speed of convergence is independent of $\monr,\mons.$ To do so, we want to apply Lemma \ref{cbarDirac}, which we can, because for $0<r<n$ and $r+s = n$
\begin{equation*}
  \begin{split}
    \Bigg(\frac{Q\rs Q\rs}{Q\rpsm Q\rmsp}\Bigg)^{r+s}&= 
    \Bigg(\frac{Q\rs}{Q\rps}\frac{ Q\rs}{Q\rsp} \frac{Q\rps}{Q\rpsm}\frac{Q\rsp}{ Q\rmsp}\Bigg)^{r+s} \\
    \overset{\text{\eqref{assumption:thermo}}}& = \underbrace{\Bigg(\frac{\varphi\left(\frac{r}{r+s+1}\right)\psi\left(\frac{r+1}{r+s+1}\right)}{\varphi\left(\frac{r+1}{r+s+1}\right)\psi\left(\frac{r}{r+s+1}\right)}\Bigg)^{r+s}}_{\overset{\text{def. }\varphi,\psi} = e^{(r+s)\left(\Phi^\prime\left(\frac r {r+s+1}\right) - \Phi^\prime\left(\frac{r+1}{r+s+1}\right)\right)}}  \underbrace{\Bigg(\frac{f\rmsp g\rpsm}{f\rs g\rs}\Bigg)^{r+s}}_{\overset{\text{ for }r+s\text{ large via \eqref{assumption:thermo:higherOrder}}}\ge e^{-\gamma}}\\
      \overset{\text{\eqref{assumption:thermo:concave}}}&\ge e^{\frac{r+s}{r+s+1}2\gamma} e^{-\gamma}
      \overset{r+s \text{ large}}\ge e^{\frac \gamma 2}.
  \end{split}
\end{equation*}
We only show the first limit, because the second is analogous. Let $\eps>0$ be arbitrary and $\delta>0$, be such that $|\xi \varphi(\xi) - \xi_\crit \varphi(\xi_\crit)|< \eps$ for all $|\xi-\xi_\crit| \le 3 \delta$, which exists independent of $\xi_\crit$, because of \eqref{assumption:thermo:boundary}. Next, we decompose the sum 
\begin{equation*}
  \begin{split}
    \sumrs[n][]&\frac{r+1}{r+s+1} \frac{Q\rps}{Q\rs} \cbar\rs 
    = \sumrs[n][] \frac{r+1}{r+s+1}\varphi\left(\frac{r+1}{r+s+1}\right) f\rs \cbar\rs \\
    &= \sumrs[n][] \frac{r+1}{r+s+1}\varphi\left(\frac{r+1}{r+s+1}\right) \cbar\rs 
    + \sumrs[n][] \underbrace{\frac{r+1}{r+s+1}\varphi\left(\frac{r+1}{r+s+1}\right)}_{|\cdot|\overset{\text{ by \eqref{assumption:thermo:boundary}}}\le C } \underbrace{(f\rs-1)}_{|\cdot| \le \eps \text{ for }r+s\text{ large}}\cbar\rs 
  \end{split}
\end{equation*}
and then further split the first sum into 
\[\sumrs[n][] = \sum\limits_{\substack{r+s = n \\ r< r_m(n) - \delta n}} 
    + \sum\limits_{\substack{r+s = n \\  r_m(n) - \delta n \le r \le r^m(n) + \delta n}}
    +\sum\limits_{\substack{r+s = n \\ r> r^m(n) + \delta n}}.\]
Since $\xi \varphi(\xi)$ is bounded by \eqref{assumption:thermo:boundary}, Lemma \ref{cbarDirac} yields a constant $C_\delta$, such that 
\[
  \left|\left(\sum\limits_{\substack{r+s = n \\ r< r_m(n) - \delta n}} 
  +\sum\limits_{\substack{r+s = n \\ r> r^m(n) + \delta n}}\right) \frac{r+1}{r+s+1}\varphi\left(\frac{r+1}{r+s+1}\right)\cbar\rs\right| \le \underbrace{\frac{C_\delta} n}_{\le \eps \text{ for }n\text{ large}}  \sumrs[n][]\cbar\rs.
\]
Next, we show for large $n$, that  we have $|\frac{r+1}{r+s+1}\varphi(\frac{r+1}{r+s+1}) - \xi_\crit \varphi(\xi_\crit)| \le \eps$ for $r_m(n) - \delta n \le r \le r^m(n) + \delta n$. By the definition of $\delta$, it suffices to show, that for such $r$, we have $|\frac{r+1}{r+s+1} - \xi_\crit| \le 3 \delta.$ To do so, fix $r\coloneq \lfloor (\xi_\crit + \delta)(n+1)\rfloor$. We claim, that $r^m(n) \le r.$ If $r\ge n$, there is nothing to prove. Otherwise, we know that $\xi_\crit <1$ and $s = n-r >0$ and therefore 
\begin{equation*}
  \begin{split}
    \frac{\cbar\rpsm}{\cbar\rs}
    &= \frac \monr \mons \frac{Q\rpsm Q\rps}{Q\rps Q\rs}
    = \frac \monr \mons \frac{\varphi\left(\frac {r+1}{r+s+1}\right)}{\psi\left(\frac{r+1}{r+s+1}\right)} \frac{f\rs}{g\rpsm}
    = \underbrace{\frac{\monr \varphi(\xi_\crit)}{\mons \psi(\xi_\crit)}}_{\overset{\xi_\crit < 1} \le 1}\frac{\varphi\left(\frac {r+1}{r+s+1}\right)\psi(\xi_\crit)}{\psi\left(\frac{r+1}{r+s+1}\right)\varphi(\xi_\crit)} \frac{f\rs}{g\rpsm}\\
    \overset{\frac{\varphi(\xi)}{\psi(\xi)} = e^{\Phi^\prime(\xi)}}&\le 
    e^{\Phi^\prime\left(\frac {r+1}{r+s+1}\right) - \Phi^\prime(\xi_\crit)} \frac{f\rs}{g\rpsm}
    \overset{\text{\eqref{assumption:thermo:concave}}\implies \Phi^{\prime\prime} \le -2\gamma}\le e^{-2\gamma \left(\frac{r+1}{n+1} - \xi_\crit\right)}  \frac{f\rs}{g\rpsm}\\
    \overset{\frac{r+1}{n+1} - \xi_\crit \ge \delta \text{ by choice of }r}&{\le} e^{-2\gamma \delta}  \frac{f\rs}{g\rpsm}.
  \end{split}
\end{equation*}
Since $f\rs \to 1$ and $g\rs \to 1$, we have for sufficiently large $n$, $\frac{\cbar\rpsm}{\cbar\rs} <1$ and so $r_m \le r,$ i.e. $\frac{r^m+1}{n+1} - \xi_\crit \le \frac{r}{n+1} -\xi_\crit + \frac{1}{n+1}\le 2 \delta$, if $\frac{1}{n+1}\le \delta.$ In particular, for any $r\le r^m+\delta n$, we have $\frac{r+1}{n+1} - \xi_\crit \le \frac{r^m + 1}{n+1} - \xi_\crit + \delta \frac{n}{n+1}\le 3 \delta.$ In analogous fashion, we can set $r\coloneqq \lfloor (\xi_\crit - \delta)(n+1)\rfloor$ and show, that $r\le r_m$. Again, if $r\le 0$, there is nothing to do, and otherwise we have $\xi_\crit >0$ and 
\[ \frac{\cbar\rmsp}{\cbar\rs} \le e^{-2\gamma\delta}\frac{g\rs}{f\rmsp},\]
which is smaller then one, for $n$ large. Hence for $r\ge r_m - \delta n $, we obtain 
$\xi_\crit - \frac{r+1}{n+1} \le \xi_\crit - \frac{r_m + 1}{n+1} + \delta \frac{n}{n+1}\le 3\delta.$ Because $\xi \varphi(\xi)$ is bounded, it remains to show
\[\frac{\sum\limits_{\substack{r+s = n \\  r_m(n) - \delta n \le r \le r^m(n) + \delta n}}\cbar\rs}{\sumrs[n][] \cbar\rs} \to 1,\]
which follows from Lemma \ref{cbarDirac}.

\underline{Claim 5}: {$\monr \xi_\crit \varphi(\xi_\crit) + \mons (1-\xi_\crit)\psi(\xi_\crit) = \exp\left(\max\limits_{\xi\in[0,1]} \Big(\xi \ln\monr + (1-\xi)\ln\mons + \Phi(\xi)\Big)\right)$ and \eqref{assumptionMn} holds true.}

For the first part, we have to distinguish the three cases for $\xi_\crit.$ If $\monr \varphi \ge \mons \psi$, then $0\le (1-\xi) \mons \psi \le (1-\xi) \monr \varphi(\xi) \to 0$ as $\xi \to 1$. So $\monr \xi_\crit \varphi(\xi_\crit) + \mons (1-\xi_\crit) \psi(\xi_\crit) = \monr \varphi(1) = e^{\ln \monr + \Phi(1)}.$ On the other hand $\odv{}{\xi} [\xi \ln \monr + (1-\xi) \ln \mons + \Phi(\xi)] = \ln\left(\frac{\monr \varphi(\xi)}{\mons \psi(\xi)}\right) \ge 0$ implies $\max\limits_{\xi \in [0,1]} \xi \ln \monr + (1-\xi)\ln\mons +\Phi(\xi) = \ln \monr + \Phi(1).$
Similarly $\monr \varphi\le \mons \psi$ yields $\monr \xi_\crit\varphi(\xi_\crit) + \mons(1-\xi_\crit)\psi(\xi_\crit)= e^{\ln \mons + \Phi(0)} = e^{\max \xi \ln \mons + (1-\xi)\monr +\Phi(\xi)}.$
Now, if $\xi_\crit \in (0,1)$, with $1 = \frac{\monr \varphi(\xi_\crit)}{\mons \psi(\xi_\crit)} = e^{\ln\frac \monr \mons + \Phi^\prime(\xi_\crit)}$, then 
\[ \odv{}{\xi}\left[\xi \ln \monr + (1-\xi)\ln\mons + \Phi(\xi)\right]\Big|_{\xi = \xi_\crit} = 0.\]
Since, $\xi \ln \monr + (1-\xi) \ln \mons + \Phi(\xi)$ is strict concave, we have a unique maximum at $\xi_\crit.$ Furthermore, 
\[ \monr \varphi(\xi_\crit) = \mons \psi(\xi_\crit) = e^{\ln \mons + \Phi(\xi_\crit)  - \xi_\crit \Phi^\prime(\xi_\crit)}\overset{\Phi^\prime(\xi_\crit) = \ln \frac \mons \monr} = e^{\xi_\crit \ln \monr + (1-\xi_\crit)\ln\mons + \Phi(\xi_\crit)}.\]
Finally, \eqref{assumptionMn} follows from $m_n(\monr,\mons) = m_\infty \left( 1 + \frac{m_n - m_\infty}{m_\infty}\right),$ because for fixed $\eta,\rho,\sigma >0$ and any $\eta \le \max(\monr,\mons)\le \rho+\sigma$, the maximum $\max\limits_{\xi\in[0,1]}\Big(\xi \ln \monr + (1-\xi) \ln \mons + \Phi(\xi)\Big)$ is bounded from below and above, since $\Phi \in C^0([0,1]).$ 

\underline{Claim 6}: {\eqref{assumptionQn} holds true.}

Fix $0< \eta,\rho,\sigma$. We will show that there is a constant $C_{\eta,\rho,\sigma,Q\rs},$ such that for any $\eps >0$ and $n$ large enough, we have 
\begin{equation*}
  \begin{split}
    &\frac{\max\limits_{k+l = n+1} \cbar\kl}{\max\limits_{k+l = n}\cbar\kl}\ge
    \left(\xi_\crit \varphi(\xi_\crit) \monr + (1-\xi_\crit) \psi(\xi_\crit) \mons\right)\frac{n+1}{n+2}(1- C\eps)\text{ and}\\
    &\frac{\max\limits_{k+l = n+1} \cbar\kl}{\max\limits_{k+l = n}\cbar\kl}
    \le \left(\xi_\crit \varphi(\xi_\crit) \monr + (1-\xi_\crit) \psi(\xi_\crit) \mons\right)(1+ C\eps),
  \end{split}
\end{equation*}
for all $\eta \le \max(\monr,\mons) \le \rho+\sigma.$
So let $\eps>0$ be arbitrary and fix again $\delta>0$, such that $|\xi \varphi(\xi) - \xi_\crit \varphi(\xi_\crit)| < \eps$ and $|(1-\xi) \psi(\xi) - (1-\xi_\crit)\psi(\xi_\crit)| < \eps$ for all $|\xi-\xi_\crit| \le 2 \delta.$
As seen in Claim 4, we have for $n$ large $\frac{r^m(n)}{n+1} \le \xi_\crit + \delta$ and $\frac{r_m(n)+1}{n+1} \ge \xi_\crit - \delta.$ In particular, we find $|\frac{r^m(n) + 1}{n+1} -\xi_\crit |\le 2 \delta$ and $|\frac{r^m(n)}{n+1} - \xi_\crit| \le 2 \delta$ for large $n.$ Now we fix $r\coloneqq r^m(n)$ and $s= n-r$. By the concavity (true by Claim 3), we know that $r^m(n+1) \in \{r,r+1\}.$ We show the case $r^m(n+1) = r+1,$ as the case $r^m(n+1) = r$ is completely analogous. Then, we have 
\begin{equation*}
  \begin{split}
    \frac{\max\limits_{k+l = n+1} \cbar\kl}{\max\limits_{k+l = n}\cbar\kl}
    \overset{\text{def. r}}&= \frac{\cbar\rps}{\cbar\rs}
    = \frac{r+s+1}{r+s+2}\left(\frac{r+1}{r+s+1} \frac{\cbar\rps}{\cbar\rs} + \frac{s+1}{r+s+1} \frac{\cbar\rps}{\cbar\rs}\right)\\
    \overset{\cbar\rps \ge \cbar\rsp}& \ge \frac{r+s+1}{r+s+2}\left(\frac{r+1}{r+s+1} \frac{\cbar\rps}{\cbar\rs} + \frac{s+1}{r+s+1} \frac{\cbar\rsp}{\cbar\rs}\right)\\
    &= \frac{r+s+1}{r+s+2}\left(\frac{r+1}{r+s+1} \varphi\left(\frac{r+1}{r+s+1}\right) \monr f\rs + \frac{s+1}{r+s+1} \psi\left(\frac{r}{r+s+1}\right) \mons g\rs\right).
  \end{split}
\end{equation*}
Because of \eqref{assumption:thermo:boundary}, there is a $C<\infty$, such that $|\xi\varphi(\xi)|\le C $ and $|(1-\xi)\psi(\xi)|\le C$. So for $n$ so large, that $|f\rs-1|<\eps$ and $|g\rs-1|<\eps$, we obtain
\begin{equation*}
  \begin{split}
    &\frac{r+1}{r+s+1} \varphi\left(\frac{r+1}{r+s+1}\right) \monr f\rs + \frac{s+1}{r+s+1} \psi\left(\frac{r}{r+s+1}\right) \mons g\rs\\
    &{\quad\ge} -C(\monr+\mons) \eps + \frac{r+1}{r+s+1} \varphi\left(\frac{r+1}{r+s+1}\right) \monr + \frac{s+1}{r+s+1} \psi\left(\frac{r}{r+s+1}\right) \mons \\
    \overset{|\frac{r}{r+s+1}-\xi_\crit|\le 2\delta\text{ and }|\frac{r+1}{r+s+1}-\xi_\crit|\le 2\delta}&{\quad\ge} -(C+1)(\monr+\mons) \eps +\xi_\crit \varphi(\xi_\crit) \monr + (1-\xi_\crit) \psi(\xi_\crit) \mons.
  \end{split}
\end{equation*}
By Claim 5, $\xi_\crit \varphi(\xi_\crit) \monr + (1-\xi_\crit) \psi(\xi_\crit) \mons$ is bounded from below for all $\monr,\mons$ satisfying $\max(\monr,\mons)\ge \eta$, so we have shown for all $n$ large enough 
\[
    \frac{\max\limits_{k+l = n+1} \cbar\kl}{\max\limits_{k+l = n}\cbar\kl}
  \ge \left(\xi_\crit \varphi(\xi_\crit) \monr + (1-\xi_\crit) \psi(\xi_\crit) \mons\right)\frac{n+1}{n+2}(1- C\eps),
\]
for some constant $C<\infty$ independent of $\monr,\mons.$ Now, we show the reverse inequality. If $r = n$, we have  $\frac{\max\limits_{k+l = n+1} \cbar\kl}{\max\limits_{k+l = n}\cbar\kl} = \frac{\cbar_{n+1,0}}{\cbar_{n,0}} = f_{n,0} \monr e^{\Phi(1)} \le f_{n,0} e^{\max\limits_{\xi\in[0,1]} \xi \ln\monr + (1-\xi)\ln \mons + \Phi(\xi)}.$ If $r<n$, we estimate 
\begin{equation*}
  \begin{split}
    \frac{\max\limits_{k+l = n+1} \cbar\kl}{\max\limits_{k+l = n}\cbar\kl}
    &= \frac{\cbar\rps}{\cbar\rs} = \frac{r+1}{r+s+1} \frac{\cbar\rps}{\cbar\rs} + \frac {s}{r+s+1} \frac{\cbar\rps}{\cbar\rs}\\
    &\le \frac{r+1}{r+s+1} \frac{\cbar\rps}{\cbar\rs} + \frac {s}{r+s+1} \frac{\cbar\rps}{\cbar\rpsm}.
  \end{split}
\end{equation*}
Again, taking $r+s$ so large that $|f\rs-1|<\eps$ and $|g\rpsm -1 |<\eps$, we find a constant $C>0$ independent of $\monr,\mons$, such that 
\begin{equation*}
  \begin{split}
    &\frac{r+1}{r+s+1} \frac{\cbar\rps}{\cbar\rs} + \frac {s}{r+s+1} \frac{\cbar\rps}{\cbar\rpsm} \\
    &\quad\le C\eps + \frac{r+1}{r+s+1} \varphi\left(\frac{r+1}{r+s+1}\right)\monr + \frac {s}{r+s+1}\psi\left(\frac{r+1}{r+s+1}\right)\mons\\
    \overset{|\frac{r+1}{r+s+1}-\xi_\crit|\le 2\delta}&{\quad\le} C\eps +\xi_\crit \varphi(\xi_\crit) \monr + (1-\xi_\crit) \psi(\xi_\crit) \mons + (\monr+\mons)\eps.
  \end{split}
\end{equation*}
Once more, $\xi_\crit \varphi(\xi_\crit) \monr + (1-\xi_\crit) \psi(\xi_\crit) \mons$ is bounded from below for $\max(\monr,\mons)\ge \eta$, so we have found a constant $C<\infty$ such that for $n$ large
\[
    \frac{\max\limits_{k+l = n+1} \cbar\kl}{\max\limits_{k+l = n}\cbar\kl}
  \le \left(\xi_\crit \varphi(\xi_\crit) \monr + (1-\xi_\crit) \psi(\xi_\crit) \mons\right)(1+ C\eps).
\]

%% file: corollaries/regionOfExistencePhi.tex
Assume that $Q\rs$ satisfies \eqref{assumption:thermo}--\eqref{assumption:thermo:higherOrder}. Then we have 
\[ \interior{\Ex} = \left\{(z,w) \in \R^2_{\ge0}\,\middle|\, \max\limits_{\xi\in[0,1]}G_{z,w}(\xi) < 0\right\} = \Big\{z \xi_\crit \varphi(\xi_\crit) + w (1-\xi_\crit) \psi(\xi_\crit) < 1 \Big\},\]
where $\xi_\crit$ is defined through the three cases $\xi_\crit = 1$ if $z \varphi \ge w \psi$, $\xi_\crit =0$ if $z\varphi \le w \psi$ and $z \varphi(\xi_\crit) = w \varphi(\xi_\crit)$ for a unique $\xi_\crit\in(0,1).$

%% file: proofs/regionOfExistencePhi.tex
The first equality follows from Proposition \ref{mainAssumptions} and Lemma \ref{mnLimit}, whereas the second equality is Claim 5 from the proof of Proposition \ref{mainAssumptions}.

%% file: propositions/criticalDirection.tex
Assume $Q\rs$ satisfies \eqref{assumption:thermo}--\eqref{assumption:thermo:higherOrder} and fix $\rho,\sigma>0.$ Let $(z,w)\in \Ex$ satisfy \eqref{steadyStateLimitEq}. Then, we have
\[\text{either  }\rho(\cinfzw) = \rho \text{ and }\sigma(\cinfzw) = \sigma 
\text{ or } G_{z,w}\left(\frac{\rho-\rho(\cinfzw)}{\rho+\sigma - \rho(\cinfzw) - \sigma(\cinfzw)}\right) = 0.\]

%% file: proofs/criticalDirection.tex
Let $(z,w) \in \Ex$ satisfies $\eqref{steadyStateLimitEq}$. Now, there exists a sequence $c^n \in X^+_{\rho,\sigma},$ such that $c^n \weakstar \cinfzw$ for $n\to \infty,$ e.g. by Theorem \ref{longtimeMain}. Hence $\rho(\cinfzw)\le \rho$ and $\sigma(\cinfzw) \le \sigma$ due to Fatou's Lemma.\\
Now we will argue by contradiction. So assume that $\rho + \sigma - \rho(\cinfzw) - \sigma(\cinfzw) > 0$ and $G_{z,w}(\xi_\crit) \neq 0$, where 
\[\xi_\crit \coloneqq \frac{\rho-\rho(\cinfzw)}{\rho+\sigma - \rho(\cinfzw) - \sigma(\cinfzw)}.\]
Since $(z,w)\in \Ex$, Corollary \ref{regionOfExistencePhi} implies $G_{z,w} (\xi) \le 0$ for all $\xi \in [0,1]$ and in particular $G_{z,w}(\xi_\crit) < 0.$\\
\underline{Claim}:  There are $u,v>0$, such that $G_{u,v}(\xi_\crit) > G_{z,w}(\xi_\crit)$ and $\max\limits_{\xi\in[0,1]}G_{u,v}(\xi) < 0.$\\
We distinguish the three cases $\xi_\crit \in (0,1)$, $\xi_\crit = 1$ and $\xi_\crit = 0.$
If $\xi_\crit \in (0,1)$, then we fix $\frac u v$ via $\ln \frac u v = \Phi^\prime(\xi_\crit).$ Hence $G_{u,v}^\prime(\xi_\crit) = 0$ and by the concavity $G_{u,v}$ has a unique maximum at $\xi_\crit.$ So it suffices to choose $v$, such that $G_{u,v}(\xi_\crit) \in (G_{z,w}(\xi_\crit),0).$ If $\xi_\crit=1$, we choose $u>0,$ such that $\ln(u) + \Phi(1) \in (G_{z,w}(1),0)$ and then $0<v$ so small that $G_{u,v} < 0.$ Similarly, if $\xi_\crit = 0$, we choose $v,$ such that $\ln(v) + \Phi(0) \in (G_{z,w}(0),0)$ and then $u>0$ so small that $G_{u,v} < 0.$ This finishes the proof of the claim.

Since $\max\limits_{\xi\in[0,1]} G_{u,v} < 0,$ by Corollary \ref{regionOfExistencePhi}, we have $(u,v) \in \Ex.$ Now, we can calculate 
\begin{equation}
  \begin{split}
    \ln\left(\frac z u\right)\left(\rho - \rho(\cinfzw)\right) + \ln\left(\frac w v\right) \left(\sigma - \sigma(\cinfzw)\right)
    = \frac{G_{z,w}(\xi_\crit) - G_{u,v}(\xi_\crit)}{\rho+\sigma - \rho(\cinfzw) - \sigma(\cinfzw)} < 0,
  \end{split}
\end{equation}
in contradiction with \eqref{steadyStateLimitEq}.

%% file: propositions/massConcentrationOnCritDirection.tex
Assume $Q\rs$ satisfies \eqref{assumption:thermo}--\eqref{assumption:thermo:higherOrder} and fix $\rho,\sigma>0.$ Let $(z,w)\in\Ex$ satisfy \eqref{steadyStateLimitEq} and $(c^n) \subset X^+_{\rho,\sigma}$ be a minimising sequence of $H[c^n|\cinfzw]$. Then, for any $M\in \N$ and $\delta>0$, we have as $n\to\infty$
\begin{equation}
  \sumgamma[\Gamma_{M,\delta}] (r+s) |c^n\rs - \cinfzw\rs| \to 0, \text{ where }
  \Gamma_{M,\delta} \coloneqq \left\{(r,s)\in \Omega\,\middle|\, r+s\le M \text{ or }\frac r {r+s} \not \in B_\delta (\xi_\crit)\right\}
\end{equation}
and $\xi_\crit$ is given by $(z,w).$

%% file: proofs/massConcentrationOnCritDirection.tex
In this proof, for given $r,k,n$ we always denote $s=n-r$ and $l=n-k.$

By Theorem \ref{steadyStateMinimising} it suffices to show that there is a fixed $\eps>0$, such that $\frac{\ln (\cinfzw\rs)}{r+s} \le -\eps$ for all $\frac r {r+s} \not \in B_\delta(\xi_\crit)$ with $r+s$ sufficiently large. To do so, first recall from Claim 4 of the proof of Proposition \ref{mainAssumptions}, that for any $\delta_1>0$ and all $r_m(n,\cinfzw) - \delta_1 n \le r \le r^m(n,\cinfzw)+\delta_1 n$ we have
\begin{equation}\label{equation:massConcentrationOnCritDirection:1}
  \left | \frac{r+1}{n+1} - \xi_\crit\right| \le 3 \delta_1 \text{ if } n\text{ is sufficiently large.}
\end{equation}
Next, we have seen in the proof of Lemma \ref{cbarDirac} that for any $\delta_2>0$, there is a $q_{\delta_2}\in(0,1)$, such that 
\[ \frac{\cinfzw\rs}{\cinfzw\rpsm} \le q_{\delta_2} \text{ for all } r<r_m - \delta_2 n -1 
\text{ and }\frac{\cinfzw\rs}{\cinfzw\rmsp} \le q_{\delta_2} \text{ for all } r>r^m + \delta_2 n +1 .\]
In particular, we find for $k_1\coloneqq \lceil r_m-\delta_2n-1 \rceil$ and $k_2 \coloneqq \lfloor r^m+\delta_2 n +1 \rfloor$ the estimates
\[ \cinfzw\rs \le q_{\delta_2}^{k_1-r} \underbrace{\cinfzw_{k_1,l_1}}_{\mathclap{\le 1 \text{ for }n \text{ large enough}}} \text{ for all }r\le k_1 \text{ and }
\cinfzw\rs \le q_{\delta_2}^{r-k_2} \underbrace{\cinfzw_{k_2,l_2}}_{\mathclap{\le 1 \text{ for }n \text{ large enough}}} \text{ for all }r\ge k_2.\]
Hence, if $n$ is large enough, we obtain for any $\delta_3>0$ 
\[\frac{\ln( \cinfzw\rs)}{r+s} \le \frac{k_1-r}{r+s} \ln( q_{\delta_2}) \le \delta_3 \ln (q_{\delta_2})\text{ for all }r\le k_1-\delta_3 n
\text{ and }\frac{\ln (\cinfzw\rs)}{r+s} \le \delta_3 \ln (q_{\delta_2})\text{ for all }r\ge k_2+\delta_3 n.\]
Now, it remains to chose $\delta_1,\delta_2,\delta_3$, such that for large $n$, we have $r\le k_1-\delta_3 n$, whenever $\frac{r}{r+s} \le \xi_\crit - \delta$ and $r\ge k_2 + \delta_3n$, whenever $\frac{r}{r+s} \ge \xi_\crit + \delta.$ To prove this, we will show $\frac {k_1}n - \delta_3 \ge \xi_\crit - \delta$ and $\frac{k_2}{n}+\delta_3 \le \xi_\crit + \delta$ as follows
\[\frac{k_1}{n} -\delta_3 \overset{\text{def. }k_1}\ge \frac{r_m - \delta_2 n -1} n - \delta_3 
\ge \frac{r_m+1}{n+1} -\delta_2 - \frac 2 n -\delta_3
\overset{r=r_m\text{ in \eqref{equation:massConcentrationOnCritDirection:1}}} \ge \xi_\crit - 3 \delta_1 -\delta_2 - \frac 2 n - \delta_3\]
and
\[\frac{k_2}{n} +\delta_3 \overset{\text{def. }k_2}\le \frac{r^m + \delta_2 n +1} n + \delta_3 
  \overset{\frac r n \le \frac {r+1}{n+1}}\le \frac{r^m+1}{n+1} +\delta_2 + \frac 1 n +\delta_3
\overset{r=r^m\text{ in \eqref{equation:massConcentrationOnCritDirection:1}}} \le \xi_\crit + 3 \delta_1 +\delta_2 + \frac 1 n + \delta_3.\]
So we can finish the proof by choosing $\delta_1,\delta_2,\delta_3>0$ small enough  and $n$ large enough to ensure $3 \delta_1 +\delta_2 + \frac 2 n + \delta_3 \le \delta$.

%% file: theorems/longtimeSlemrod.tex
Assume that $a_r,b_r>0$ and $a_r,b_r \in \bigO(r)$ as well as $\lim\limits_{r \to \infty} \sqrt[r]{Q_r} = \frac 1 {z_s}\in(0,\infty).$ Furthermore, assume that 
\begin{equation}\label{assumption:ocBD}
  \text{ if }0\le z < z_s \text{ then } a_r z \le b_r \text{ for }r \text{ sufficiently large.}
\end{equation}
Suppose $c^0 \in X^+$ and that $c$ is the only solution \eqref{ocBD}--\eqref{ocBDMass} with initial data $c^0,$ then \eqref{ocLongTime} holds true.

%% file: appendix.tex
\section*{Appendix}
\addcontentsline{toc}{section}{Appendix}
\setcounter{section}{4}
In this appendix, we want to discuss the astonishing fact, that even though Remark \ref{hardy} together with \eqref{QThermoDom} suggest, that we can prove Theorem \ref{hardyInequality} by integrating along lines with fixed values $\frac x {x+y},$ the proof demands to integrate along curved lines, where the curvature is depending on the binding energy of the quasi steady state $G(\xi) \coloneqq \xi\ln \monr + (1-\xi)\ln\mons + \Phi(\xi).$

As discussed in the beginning of Subsection \ref{subsection:application} we need to find $\er\rs,\es\rs$ satisfying an estimate of the form 
\begin{equation}\label{appendix:detailedBound}
  \frac{\er\rps}{\er\rps+\es\rps} \frac{\cbar\rps}{\cbar\rs}  + \frac{\es\rsp}{\er\rsp+\es\rsp} \frac{\cbar\rsp}{\cbar\rs} \overset{!}<  1, \text{ whenever } \monr,\mons \in \interior{\Ex}.
\end{equation}
If $Q\rs$ is given by \eqref{QThermoDom}, we find for large $r+s$ and $\xi \coloneqq \frac {r} {r+s}$ 
\begin{equation}
  \begin{split}
    \frac{\cbar\rps}{\cbar\rs} &\approx e^{\delr (r+s+1) G\left(\frac{r+1}{r+s+1}\right)}
    \approx e^{ G\left(\frac{r+1}{r+s+1}\right) + (1-\frac{r+1}{r+s+1}) G^\prime\left(\frac{r+1}{r+s+1}\right)}
    \approx e^{ G\left(\xi\right) + (1 - \xi) G^\prime\left(\xi\right)}
    \text{ and }\\
    \frac{\cbar\rsp}{\cbar\rs} &\approx e^{\dels (r+s+1) G\left(\frac{r}{r+s+1}\right)}
    \approx e^{ G\left(\frac{r}{r+s+1}\right) -\frac{r}{r+s+1} G^\prime\left(\frac{r}{r+s+1}\right)} 
    \approx e^{ G\left(\xi\right) - \xi G^\prime\left(\xi\right)}.
  \end{split}
\end{equation}
Next, we can assume without loss of generality $\er\rs + \es\rs =1$ (by setting $\tilde \er\rs = \frac{\er\rps}{\er\rps+\es\rps}$ and $\tilde \es\rsp = \frac{\es\rsp}{\er\rsp+\es\rsp}$). If we furthermore assume some continuity on $\er$, such that $\es\rsp \approx \es\rps$, we have reduced \eqref{appendix:detailedBound} to finding an $e(\xi),$ such that 
\begin{equation}\label{appendix:equationE}
  e(\xi)e^{ G\left(\xi\right) + (1 - \xi) G^\prime\left(\xi\right)} + (1-e(\xi)) e^{ G\left(\xi\right) - \xi G^\prime\left(\xi\right)} < 1, \text{ whenever } \max\limits_{\xi \in [0,1]}G(\xi) < 0.
\end{equation}
    \begin{remark}\textit{({$e(\xi) = \xi$} does not work)}\label{straightLineIsNotGood}\\
        \input{\CommonPath/remarks/straightLineIsNotGood}
    \end{remark}
Next, we will show how $\mathbb E[\Prs g\kl]$ corresponds to a line integral. (Certainly, $\mathbb E[\Trs g\kl]$ can be understood in a similar fashion.)
    \begin{remark}\textit{(Coordinate transform due to $\er\rs,\es\rs$)}\label{eresContinuous}\\
        \input{\CommonPath/remarks/eresContinuous}

    \end{remark}
Remarks \ref{straightLineIsNotGood} and \ref{eresContinuous} show that Remark \ref{hardy} does not tell the full story. Hence, it should be acceptable that we have introduced the concavity assumption \eqref{concavityAss:old} even though from Remark \ref{hardy} it may appear unnecessary. Of course, the mismatch between Remarks \ref{straightLineIsNotGood}, \ref{eresContinuous} and \ref{hardy} is a result of the mismatched scaling behaviour of $g\left(\frac k N,\frac l N\right)$ from \ref{eresContinuous} and the exponential $\cbar\rs \approx e^{(k+l) G(\frac k {k+l})}.$

Let us finally comment on the form of $\er\rs,\es\rs$ found in Lemma \ref{eres}. Since $m_n \to e^{\max G}$, we are essentially solving \eqref{appendix:equationE} with right hand side equal to $e^{\max G},$ which we can rearrange to find 
\[ e(\xi) = \frac{ e^{\max G} - e^{G(\xi)-\xi G^\prime(\xi)}}{e^{G(\xi)+(1 - \xi) G^\prime(\xi)}-e^{G(\xi)-\xi G^\prime(\xi)}}.\]
In particular, we need $G$ to be concave to obtain $e(\xi) \in [0,1]$ for arbitrary values of $\monr,\mons.$

%% file: remarks/straightLineIsNotGood.tex
With this, we can see that \eqref{exponentialBehaviourOfPrsTrs} may fail for $\er\rs = \frac r {r+s},$ even for concave $\Phi(\xi).$ If we set 
\[\xi_\crit \coloneqq \argmax\limits_{\xi\in[0,1]} G(\xi) \text{ and } H(\xi)\coloneqq \xi e^{ G\left(\xi\right) + \hat \xi G^\prime\left(\xi\right)} + (1-\xi) e^{ G\left(\xi\right) - \xi G^\prime\left(\xi\right)},\]
we have $H(\xi_\crit) = e^{G(\xi_\crit)} < 1.$ But at the same time 
\[ \odv{}{\xi} H(\xi) = e^{G(\xi)} (e^{\hat \xi G^\prime(\xi)} - e^{-\xi G^\prime(\xi)})\underbrace{(1+\xi\hat\xi \Phi^{\prime\prime}(\xi))}_{<0 \text{ if }\Phi^{\prime\prime} \text{ is sufficiently small}}\]
has a sign change at $\xi_\crit$ from negative to positive. Hence, if $\Phi^{\prime\prime}$ is sufficiently small, we may find an interval $[\xi_1,\xi_2] \subset (0,1)$, such that $H(\xi) \ge q>1$ on $[\xi_1,\xi_2].$
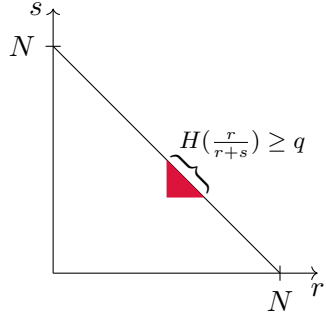
\begin{figure}[!htb]
  \centering
      \begin{tikzpicture}
        \begin{scope}
          \draw[->] (0, 0) -- (3.5, 0) node[below] {$r$};
          \draw[->] (0, 0) -- (0, 3.5) node[left] {$s$};
          \draw[scale=1, domain=0:3, smooth, variable=\x, black] plot ({\x}, {3-\x});
          \draw (3,0.1) -- (3,-0.1) node[below] {$N$};
          \draw (0.1,3) -- (-0.1,3) node[left] {$N$};
          \fill[Crimson] (1.5,1.5) -- (2,1) -- (1.5,1) -- cycle;
          \node[rotate=45] at (1.85,1.35) {\Big \}};
          \node at (2.5,1.7) {\footnotesize $H(\frac r {r+s}) \ge q$};
        \end{scope}
      \end{tikzpicture}
  \caption{We define $\omega$ as the red triangle.}
  \label{figure:straightIsLineNotGood}
\end{figure}

If we define $\omega\subset \{r+s\le N\}$ as the red triangle in Figure \ref{figure:straightIsLineNotGood}, we have due to \eqref{formulaTrs} for any $(r,s)\in\omega$
\[ \mathbb E[\Trs \cbar\kl] \ge \cbar\rs \sum\limits_{m=0}^{N-r-s} q^{m}.\]
Since the height of the red triangle goes to infinity as $N\to\infty$, we find for some $r,s$ depending on $N$
\[ \frac{\mathbb E[\Trs \cbar\kl]}{\cbar\rs} \to \infty \text{ as }N\to\infty.\]
In particular, allowing general $\er\rs,\es\rs$ in Section \ref{section:logSobolevInequality} was necessary.

%% file: remarks/eresContinuous.tex
\newcommand{\qqrs}{q^{r,s}}
As discussed in \eqref{appendix:equationE}, it (roughly) suffices to consider $\er\rs = \er(\frac r {r+s})$ with $\es\rs = 1 -\er\rs.$ So we will assume that in the following.

Now, we can understand the behaviour of $\mathbb E[\Prs g\kl]$, if we rescale $r,s$ by $\frac 1 N$ and consider the limit $N\to\infty$. To be precise, let us define the measure $q^N$ as follows. For any $g(x,y) \in C^1(\{(x,y) \mid 0 \le x,y \le 1 \text{ and } x+y \le 1\}$ we set 
\begin{equation}\label{eresContinuous:defqN}
  \iint\limits_{0\le x+y\le 1} g(x,y) \dd q^N(x,y) \coloneqq \frac 1 N \mathbb E\left[\Prs g\left(\frac k N, \frac l N \right)\right], \text{ where } r \approx x_0 N, s\approx y_0 N.
\end{equation}
Next, we realise that for $\qqrs_{u,v} \coloneqq \mathbb E[\Prs \delta_{k,l = u,v}]$ we get 
\begin{equation}\label{eresContinuous:defqrs}
  \mathbb E[\Prs g\kl] = \sumkl[2] \qqrs\kl g\kl.
\end{equation}
Clearly $\qqrs\kl = 0$ if either $k>r$ or $l>s$  and $\qqrs\rs =1$. Otherwise, by Lemma \ref{treeExpressions}.\ref{treeExpressions:derivativePrs}, we obtain that $\qqrs\kl = \er\rs q^{r-1,s}\kl + \es\rs q^{r,s-1}\kl.$ Subsequently, an induction over the distance from $(k,l)$ to $(r,s)$ yields 
\begin{equation*}
  \qqrs\kl = \er\kpl \qqrs\kpl + \es\klp \qqrs\klp \text{ for all }(k,l) \neq (r,s) \text{ and }k+l>1.
\end{equation*}
With this we can obtain an equation for the (weak* in the sense of measures) limit $q\coloneqq \lim\limits_{N \to \infty} q^N$. For any $g\kl = g(\frac k N,\frac l N)$ we have
\begin{equation*}
  \begin{split}
    0 &= \sumkl[2][r+s-1] \bigl(\qqrs\kl - \er\kpl \qqrs\kpl - \es\klp \qqrs\klp) g\kl\\
      &= -\qqrs\rs + \sumkl[2][] \qqrs\kl\big(\er\kl g\kml + \es\kl g\klm\big) + \sumkl[2][r+s] \qqrs\kl\bigl(\er\kl \delk g\kl + \es\kl \dell g\kl\bigl)\\
        \overset{\delk g\kl \approx \frac 1 N \partial_x g(\frac k N,\frac l N)} &\approx -\qqrs\rs + \sumkl[2][] \qqrs\kl\big(\er\kl g\kml + \es\kl g\klm\big)\\
        &\qquad+ \frac 1 N \sumkl[2] \qqrs\kl\left(\er\left(\frac k {k+l}\right) \partial_x g\left(\frac k N,\frac l N\right) + \es\left(\frac k {k+l}\right) \partial_y g\left(\frac k N, \frac l N\right)\right)\\
        \overset{\text{\eqref{eresContinuous:defqN} and \eqref{eresContinuous:defqrs}}} & =  -\qqrs\rs + \sumkl[2][] \qqrs\kl\big(\er\kl g\kml + \es\kl g\klm\big) \\
        &\qquad+ \iint\limits_{0\le x+y\le 1} \er(x,y) \partial_x g(x,y) + \es(x,y) \partial_y g(x,y) \dd q^N(x,y)\\ 
        \overset{\sumkl[2] \qqrs\kl = 1}&\to g(0,0) - g(x_0,y_0) + \iint\limits_{0\le x+y\le 1} \er(x,y) \partial_x g(x,y) + \es(x,y) \partial_y g(x,y) \dd q(x,y)\\
                                \overset{(\frac{x}{x+y},x+y)=(\xi,\tau)}& = g(0,0) - g(x_0,y_0) + \int_0^1\int_0^1 \tau\left( \er(\xi) \partial_x g(\xi,\tau) + \es(\xi) \partial_y g(\xi,\tau)\right)\dd q(\xi,\tau)\\
                                                                          &= g(0,0) - g(x_0,y_0) + \int_0^1\int_0^1 \biggl[\tau \partial_\tau g + (\er(\xi)-\xi)\partial_\xi g\biggr] \dd q(\xi,\tau).
  \end{split}
\end{equation*}
In other words, $q(\xi,\tau)$ is a distributional solution of 
\begin{equation}\label{weakTransport}
  \begin{cases}
    \tau \partial_\tau q + (\er(\xi)-\xi)\partial_\xi q + \er^\prime q = 0 \text{ for } (\xi,\tau) \in [0,1] \times [0,x_0+y_0]\\
    (x_0+y_0)q(\dd \xi,x_0+y_0) = \delta_{\xi = \frac{x_0}{x_0+y_0}}.
  \end{cases}
\end{equation}
But this solution just transports and scales the Dirac along the characteristics, i.e. 
\begin{equation*}
  q(\xi,\tau) = f(\tau) \delta_{\xi = h(\tau)},
\end{equation*}
where $h(\tau)$ is a solution to 
\[ \tau h^\prime = \er(h)-h \text{ with } h(x_0+y_0) = \frac{x_0}{x_0+y_0}.\]
To summarise, $\mathbb E[\Prs g\kl]$ is roughly an integral along the lines given by $h$.

Note, that we can solve \eqref{weakTransport} for $e(\xi) = \xi$ explicitly by 
$\displaystyle q(\dd\xi,\dd\tau) = \delta_{\xi = \frac{x_0}{x_0+y_0}} \frac {\dd\tau} \tau,$ for $\tau \le x_0+y_0.$
Hence, we have 
\begin{equation}
  \begin{split}
    \frac 1 N \mathbb E[\Prs g\kl] &\approx \iint\limits_{0\le x+y\le 1} g(x,y) \dd q(x,y)
    \overset{(\frac{x}{x+y},x+y)=(\xi,\tau)}= \int_0^1 \int_0^1 g(\xi,\tau) \tau \dd q(\xi,\tau)\\
    &= \int_0^{x_0+y_0} g\left(\frac{x_0}{x_0+y_0},\tau\right)\dd\tau.
  \end{split}
\end{equation}

%% file: acknowledgements.tex
\section*{Acknowledgements}
\addcontentsline{toc}{section}{Acknowledgements}
The author gratefully acknowledges the financial support of the Deutsche Forschungsgemeinschaft (DFG, German Research Foundation) through the collaborative research centre ``Analysis of criticality: from complex phenomena to models and estimates'' (CRC 1720, Project-ID 539309657) and the Bonn International Graduate School of Mathematics at the Hausdorff Center for Mathematics (EXC 2047/2, Project-ID 390685813).\\
I would like to thank Barbara Niethammer for numerous helpful discussions and guidance throughout this project, as well as Juan Velázquez for his intuitive insights.